\documentclass[11pt,reqno,a4paper, oneside]{amsart}

\usepackage{amsmath,amsthm,amsfonts,amssymb}
\usepackage{array}
\usepackage[labelformat=simple]{subcaption}
\renewcommand\thesubfigure{(\alph{subfigure})}
\usepackage{thmtools}
 \usepackage{cite}
\usepackage{calc}
\usepackage{enumitem}
\usepackage[top=0.85in, bottom=0.85in, right = 1in, left = 1in]{geometry}
\usepackage{graphicx}
\usepackage{tikz}
\usetikzlibrary{decorations.pathreplacing,angles,quotes}
\usetikzlibrary{positioning}
\usepackage{nameref}
\usepackage[colorlinks=true, citecolor = blue, linkcolor = teal,pdfencoding=auto, psdextra]{hyperref}
\usepackage[nameinlink]{cleveref}
\usepackage{mathtools,xfrac}
\usepackage{multirow}
\usepackage{caption}
\usepackage{comment}
\usepackage{breqn}
\usepackage{lmodern}
\allowdisplaybreaks
\allowbreak

\usepackage{titlesec}
\titleformat{\section}
  {\centering\scshape\fontsize{12}{15}}
  {\thesection}
  {1em}
  {}
\titleformat{\subsection}
  {\normalfont\fontsize{12}{17}\bfseries}
  {\thesubsection}
  {1em}
  {}
\titleformat{\subsubsection}
  {\normalfont\fontsize{11}{14}\bfseries}
  {\thesubsubsection}
  {1em}
  {}

\pgfdeclarelayer{background}
\pgfsetlayers{background,main}

\newdimen\zerolinewidth

\tikzset{
  zero line width/.code={\zerolinewidth=\pgflinewidth
    \tikzset{line width=0cm}},
  use line width/.code={\tikzset{line width=\the\zerolinewidth}},
  draw with anchor in boundary/.style={
    zero line width,
    postaction={draw,use line width},
  },
}

\tikzstyle{vertex}=[circle, draw, fill=black, inner sep=0.7pt, outer sep=0pt, anchor=center]
\tikzstyle{edge}=[draw,line width = 0.4pt]

\theoremstyle{plain}
\newtheorem{theorem}{Theorem}[section]
\newtheorem{proposition}[theorem]{Proposition}

\newtheorem{conj}[theorem]{Conjecture}

\newtheorem{claim}[theorem]{Claim}
\newtheorem{rem}[theorem]{Remark}
\newtheorem{ques}[theorem]{Question}
\theoremstyle{definition}
\newtheorem{definition}[theorem]{Definition}

\numberwithin{equation}{section}

\newcommand{\N}{\mathbb{N}}
\newcommand{\M}{\mathsf{M}}
\newcommand{\Z}{\mathbb{Z}}
\newcommand{\Q}{\mathbb{Q}}
\newcommand{\R}{\mathbb{R}}
\newcommand{\Sp}{\mathbb{S}}
\DeclarePairedDelimiter\floor{\lfloor}{\rfloor}
\DeclarePairedDelimiter\ceil{\lceil}{\rceil}
\newcommand{\lk}[2]{{\rm lk}_{#1}(#2)}
\newcommand{\st}[2]{{\rm st}_{#1}(#2)}
\newcommand{\skel}[2]{{\rm skel}_{#1}(#2)}
\newcommand{\Kd}[1]{{\mathcal K}(#1)}
\newcommand{\bs}{\backslash}
\newcommand{\ol}{\overline}
\newcommand{\meets}{\leftrightarrow}
\newcommand{\nmeets}{\nleftrightarrow}
\newcommand{\TPSS}{S^{\hspace{.2mm}2} \mbox{$\times
\hspace{-2.6mm}_{-}$} \, S^{\hspace{.1mm}1}}
\newcommand{\ind}{\mathrm{Ind}}
\newcommand{\on}[2]{N_{#1}(#2)}   \newcommand{\cn}[2]{N_{#1}[#2]}   \DeclarePairedDelimiterX\set[1]\lbrace\rbrace{\def\given{\;\delimsize\vert\;}#1}
\renewcommand{\mod}[1]{\,({\mathrm{mod\;\,} #1})}
\newcommand{\n}{N}
\newcommand{\dimax}[1]{d_{\mathit{max}}(\ind(#1))}
\newcommand{\dimin}[1]{d_{\mathit{min}}(\ind(#1))}
\newcommand{\eqcolon}{\mathrel{\resizebox{\widthof{$\mathord{=}$}}{\height}{ $\!\!=\!\!\resizebox{1.2\width}{0.8\height}{\raisebox{0.23ex}{$\mathop{:}$}}\!\!$ }}}
\newcommand{\coloneq}{\mathrel{\resizebox{\widthof{$\mathord{=}$}}{\height}{ $\!\!\resizebox{1.2\width}{0.8\height}{\raisebox{0.23ex}{$\mathop{:}$}}\!\!=\!\!$ }}}

\renewcommand*{\thesubfigure}{(\alph{subfigure})}
\newcommand{\samir}[1]{\textcolor{green!65!black}{#1}}
\newcommand{\sourav}[1]{\textcolor{violet}{#1}}
\newcommand{\raju}[1]{\textcolor{orange!80!black}{#1}}
\newcommand{\sagar}[1]{\textcolor{brown!60!black}{#1}}
\newcommand{\red}[1]{\textcolor{red!80!black}{#1}}
\newcommand{\blue}[1]{\textcolor{blue!80!black}{#1}}

\Crefname{claim}{Claim}{Claims}

\makeatletter
\@namedef{subjclassname@2020}{\textup{2020} Mathematics Subject Classification}
\makeatother
\date{}

\begin{document}

\title{On the matching complexes of the categorical product of path graphs}

\author{Raju Kumar Gupta}
	\address{Department of Mathematics, Indian Institute of Technology (IIT) Madras, India}
	\email{rajukrg3217@gmail.com}

 \author{Sourav Sarkar}
	\address{Department of Mathematics, Indian Institute of Technology (IIT) Madras, India}
	\email{sarkarsourav610@gmail.com}

 \author{Sagar S. Sawant}
	\address{Department of Mathematics, Indian Institute of Technology (IIT) Madras,  India}
	\email{sagar@smail.iitm.ac.in}

	\author{Samir Shukla}
	\address{School of Mathematical and Statistical Sciences, Indian Institute of Technology (IIT) Mandi,  India}
	\email{samir@iitmandi.ac.in}

\begin{abstract}
		The matching complex $\mathsf{M}(G)$ of a graph $G$ is a simplicial complex whose simplices are matchings in $G$. These complexes appear in various places and found applications in many areas of mathematics including computational geometry, representation theory, combinatorics, etc. 
        In this article, we consider the matching complexes of the categorical product $P_n \times P_m$ of path graphs $P_n$ and $P_m$. For $m = 1$, $P_n \times P_m$ is a discrete graph and therefore its matching complex is the void complex. For $m = 2$, $\M(P_n \times P_m)$ has been proved to be homotopy equivalent to a wedge of spheres by Kozlov. We show that for $n \geq 2$ and $3 \leq m \leq 5$, the matching complex of $P_n \times P_m$ is homotopy equivalent to a wedge of spheres. For $m =3$, we explicitly compute the number and dimension of spheres appearing in the wedge. Furthermore, for $m \in \{4, 5\}$,  we provide the minimum and maximum dimensions of spheres appearing in the wedge in the homotopy type of $\mathsf{M}(P_n \times P_m)$.   
	\end{abstract}
 

\maketitle
  
\section{Introduction}
 A {\it matching} in a simple graph $G$ is a collection of pairwise disjoint edges. The {\it matching complex} $\M(G)$ of a graph $G$ is the simplicial complex whose vertices are the edges, and simplicies are the matchings in $G$. In topological combinatorics, it is an important problem to study the topology of matching complexes of graphs. Matching complexes first appeared in the thesis of Garst \cite{Gar79}, where they obtained the matching complexes of complete bipartite graphs  $K_{m, n}$ as special coset complexes of symmetric groups $S_n$.  Thereafter, matching complexes have appeared and found applications in various places.  The matching complexes of complete graphs were studied in the paper of Bouc  \cite{Bou92} in connection with Brown complexes and Quillen complexes.  In \cite{Bjorner1992},  Bj{\"o}rner et al., described the matching complexes of complete bipartite graphs as chessboard complexes to study the geometry of non-attacking rook configurations on a general $m\times n$ chessboard.  In \cite{Rade92},  $\M(K_{m, n})$ appeared as the complex of all partial injective functions from $\{1, \ldots, m\}$ to $\{1, \ldots, n\}$ in connection to the problem of estimating the number of halving hyperplanes of a finite set of points in Euclidean space. 
Other than the complete graphs and complete bipartite graphs, the classes of graphs for which matching complexes have been studied include; paths and cycles \cite{Kozlov99}, trees \cite{MMTD2008, Marija2022}, polygonal tilings \cite{Matsushita2022,Bayer2023}, honeycomb graphs \cite{MMTD2008}, etc.   For more details on these complexes, we refer to a survey article by Wachs \cite{Wachs03} and Chapter 11 of the book \cite{Jonsson2008} by Jonsson.

\begin{definition}
    The {\it cartesian product} $G \Box H$ of graphs $G$ and $H$ is the graph with vertex set $V(G) \times  V(H)$, where any two vertices $(u, v)$  and $(u', v')$ form an edge if and only if either $u = u'$ and $vv' \in E(H)$, or $v = v'$ and $uu' \in E(G)$. 
\end{definition}

\begin{definition}
	The {\it categorical product} $G \times H$ of graphs $G$ and $H$ is the graph with vertex set $V(G) \times  V(H)$, where any two vertices $(u, v)$  and $(u', v')$ form an edge if and only if $uu' \in E(G)$ and $vv' \in E(H)$. 
\end{definition}

For $n \geq 1$, let  $P_n$ denote the path graph on $n$ vertices.  The matching complexes of the cartesian product of path graphs have been studied by various authors. In \cite{Kozlov99},  Kozlov studied  $\M(P_n \Box P_1)$.   In \cite{Jon05}, Jonsson studied the homotopical depth and topological connectivity of matching complexes of the cartesian product of path graphs for certain cases. Furthermore, Braun and Hough investigated the homology of $\mathsf{M}(P_n \Box P_2)$ in \cite{BH17}, and its homotopy type was examined by Matsushita in \cite{Mat19}. 
%In \cite{3xngrids}, Goyal et al. showed that $\mathsf{M}(P_n \Box P_3)$ is homotopy equivalent to a wedge of spheres. However, for $m \geq 4$, the homotopy type of $\mathsf{M}(P_n \Box P_m)$ remains unknown. 

Motivated by the cartesian product of path graphs, in this article, we consider the matching complexes of the categorical product  $P_n \times P_m$ of path graphs  (see \Cref{fig:cat-cart-paths}). Since  $P_n \times P_1$  does not contain any edge, $\M(P_n \times P_1)$ is a void complex. Further, $P_n \times P_2$ is a disjoint union of two path graphs, and therefore $\M(P_n \times P_2)$ is homotopy equivalent to a wedge of spheres from \cite{Kozlov99}. 
For $3 \leq m \leq 5$, we compute the homotopy type of the matching complexes of $P_n \times P_m$. 

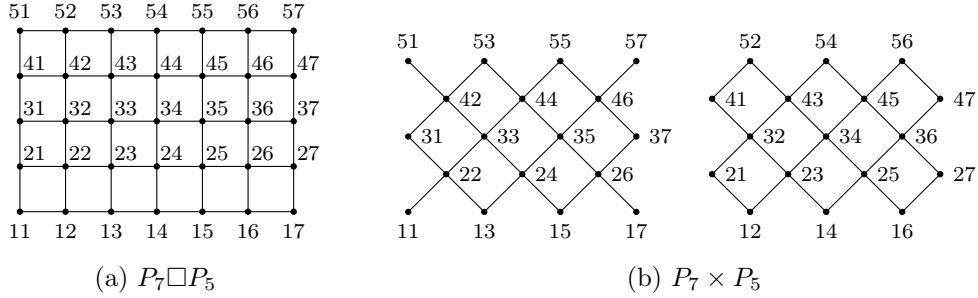
\begin{figure}[h!]
    \centering
    \begin{subfigure}[b]{0.28\textwidth}
    \begin{tikzpicture}[scale=0.5]
        \foreach \x in {1,2,...,7}
        {
        \node[vertex, label = {-90:\tiny{$1\x$}}] (1\x) at (\x,1) {};
        }
        \foreach \x in {1,2,...,7}
        {
        \node[vertex, label = {[label distance = -4pt]30:\tiny{$2\x$}}] (2\x) at (\x,2) {};
        }
        \foreach \x in {1,2,...,7}
        {
        \node[vertex, label = {[label distance = -4pt]30:\tiny{$3\x$}}] (3\x) at (\x,3) {};
        }
        \foreach \x in {1,2,...,7}
        {
        \node[vertex, label = {[label distance = -4pt]30:\tiny{$4\x$}}] (4\x) at (\x,4) {};
        }
        \foreach \x in {1,2,...,7}
        {
        \node[vertex, label = {90:\tiny{$5\x$}}] (5\x) at (\x,5) {};
        }

        \foreach \y [evaluate = \y as \ynext using {int(\y+1)}] in {1,2,...,6}
        {
        \draw[edge] (1\y) -- (1\ynext) ;
        \draw[edge] (2\y) -- (2\ynext) ;
        \draw[edge] (3\y) -- (3\ynext) ;
        \draw[edge] (4\y) -- (4\ynext) ;
        \draw[edge] (5\y) -- (5\ynext) ;
        }

        \foreach \x in {1,2,3,4,5,6,7}
        {
        \foreach \y [evaluate = \y as \ynext using {int(\y+1)}] in {1,2,3,4}
        {
        \draw[edge] (\y\x) -- (\ynext\x) ;
        }
        }
        
    \end{tikzpicture}  
    \caption{\small{$P_7 \Box P_5$}}
    \end{subfigure}
    \hspace{0.05\textwidth}
    \begin{subfigure}[b]{0.55\textwidth}
    \begin{tikzpicture}[scale=0.4]
        \foreach \x in {1,3,5,7}
        {
        \node[vertex, label = {-90:\tiny{$1\x$}}] (1\x) at (\x,1) {};
        }
        \foreach \x in {2,4,6}
        {
        \node[vertex, label = {0:\tiny{$2\x$}}] (2\x) at (\x,2) {};
        }
        \foreach \x in {1,3,5,7}
        {
        \node[vertex, label = {0:\tiny{$3\x$}}] (3\x) at (\x,3) {};
        }
        \foreach \x in {2,4,6}
        {
        \node[vertex, label = {0:\tiny{$4\x$}}] (4\x) at (\x,4) {};
        }
        \foreach \x in {1,3,5,7}
        {
        \node[vertex, label = {90:\tiny{$5\x$}}] (5\x) at (\x,5) {};
        }

        \foreach \y [evaluate = \y as \ynext using {int(\y+1)}] in {1,3,5}
        {
        \draw[edge] (5\y) -- (4\ynext) ;
        \draw[edge] (3\y) -- (2\ynext) ;
        
        \draw[edge] (1\y) -- (2\ynext) ;
        \draw[edge] (3\y) -- (4\ynext) ;
        }

        \foreach \y [evaluate = \y as \ynext using {int(\y+1)}] in {2,4,6}
        {
        \draw[edge] (4\y) -- (5\ynext) ;
        \draw[edge] (2\y) -- (3\ynext) ;
        
        \draw[edge] (2\y) -- (1\ynext) ;
        \draw[edge] (4\y) -- (3\ynext) ;
        }

    \end{tikzpicture}
    \hspace{1mm}
    \begin{tikzpicture}[scale=0.4]
        \foreach \x in {2,4,6}
        {
        \node[vertex, label = {-90:\tiny{$1\x$}}] (1\x) at (\x,1) {};
        }
        \foreach \x in {1,3,5,7}
        {
        \node[vertex, label = {0:\tiny{$2\x$}}] (2\x) at (\x,2) {};
        }
        \foreach \x in {2,4,6}
        {
        \node[vertex, label = {0:\tiny{$3\x$}}] (3\x) at (\x,3) {};
        }
        \foreach \x in {1,3,5,7}
        {
        \node[vertex, label = {0:\tiny{$4\x$}}] (4\x) at (\x,4) {};
        }
        \foreach \x in {2,4,6}
        {
        \node[vertex, label = {90:\tiny{$5\x$}}] (5\x) at (\x,5) {};
        }

        \foreach \y [evaluate = \y as \ynext using {int(\y+1)}] in {2,4,6}
        {
        \draw[edge] (5\y) -- (4\ynext) ;
        \draw[edge] (3\y) -- (2\ynext) ;
        
        \draw[edge] (1\y) -- (2\ynext) ;
        \draw[edge] (3\y) -- (4\ynext) ;
        }

        \foreach \y [evaluate = \y as \ynext using {int(\y+1)}] in {1,3,5}
        {
        \draw[edge] (4\y) -- (5\ynext) ;
        \draw[edge] (2\y) -- (3\ynext) ;
        
        \draw[edge] (2\y) -- (1\ynext) ;
        \draw[edge] (4\y) -- (3\ynext) ;
        }

    \end{tikzpicture}
    \caption{\small{$P_7 \times P_5$}}
    \end{subfigure}
   \caption{\small{Cartesian and categorical product of $P_7$ and $P_5$}}
    \label{fig:cat-cart-paths}
\end{figure}

\begin{theorem} \label{theorem:combine}
	For $n \geq 2$ and $3 \leq m \leq 5$, $\M(P_n \times P_m)$ is homotopy equivalent to a wedge of spheres. 
	\end{theorem} 
	
We explicitly compute the dimension and number of spheres in the homotopy type of $\M(P_n \times P_3)$ (see \Cref{PnxP3}). For $m \in \{4, 5\}$, we give the minimum and maximum dimension of spheres appearing in the wedge in homotopy type of $\M(P_n \times P_m)$.

\begin{proposition}\label{proposition:dim_PnP5}
    For  $n \geq 3$, let $d_{max}(\M(P_n \times P_5))$ and $d_{min}(\M(P_n \times P_5))$ denote the maximum and minimum dimension of spheres appearing in the wedge in the homotopy type of $\M(P_n \times P_5)$, respectively. Then, $d_{max}(\M(P_n \times P_5))$ and $d_{min}(\M(P_n \times P_5))$ are given in \Cref{tab:intro_dimensionPnP5}. 
\end{proposition}
\begin{table}[h!]
	{\tiny
		\setlength{\tabcolsep}{6pt} % horizontal spacing
		\renewcommand{\arraystretch}{1.5} % vertical spacing
		\centering
		\begin{tabular}{|c|c|c|c|c|c|}
			\hline
			\multicolumn{3}{|c|}{$d_{\max}(\M(P_n \times P_5))$} & \multicolumn{3}{c|}{$d_{\min}(\M(P_n \times P_5))$} \\
			\hline
			$k \equiv$ & $n = 2k+2$ & $n = 2k+3$ & $k \equiv$ & $n = 2k+2$ & $n = 2k+3$ \\
			\hline
			$0 \mod{4}$ & $7\big(\frac{k}{2}\big)+1$ & $7\big(\frac{k}{2}\big)+3$ & $0 \mod{5}$ & $16\big(\frac{k}{5}\big)+3$ & $16\big(\frac{k}{5}\big)+4$ \\
			\hline
			$1 \mod{4}$ & $7\big(\frac{k-1}{2}\big)+5$ & $7\big(\frac{k-1}{2}\big)+7$ & $1 \mod{5}$ & $16\big(\frac{k-1}{5}\big)+5$ & $16\big(\frac{k-1}{5}\big)+7$ \\
			\hline
			$2 \mod{4}$ & $7\big(\frac{k-2}{2}\big)+9$ & $7\big(\frac{k-2}{2}\big)+11$ & $2 \mod{5}$ & $16\big(\frac{k-2}{5}\big)+9$ & $16\big(\frac{k-2}{5}\big)+10$ \\
			\hline
			$3 \mod{4}$ & $7\big(\frac{k-3}{2}\big)+13$ & $7\big(\frac{k+1}{2}\big)$ & $3 \mod{5}$ & $16\big(\frac{k-3}{5}\big)+11$ & $16\big(\frac{k-3}{5}\big)+13$ \\
			\hline
			{} & & & $4 \mod{5}$ & $16\big(\frac{k-4}{5}\big)+15$ & $16\big(\frac{k+1}{5}\big)$ \\
			\hline
		\end{tabular}
		\caption{Maximum and minimum dimension of spheres in $\M(P_n \times P_5)$}
	\label{tab:intro_dimensionPnP5}}
\end{table}

\begin{table}[h!]
	{\tiny
		\setlength{\tabcolsep}{7pt} % horizontal spacing
		\renewcommand{\arraystretch}{1.4} % vertical spacing
		\centering
		\begin{tabular}{|c|c|c|}
			\hline
			$k = 7 q+r$, $0 \leq r \leq 6$ & $n = 2k+2$ & $n = 2k+3$ \\
			\hline
			$r = 0$ & $2(k+3q)+1$ & $2(k+3q)+1$  \\
			\hline
			$r = 1$ & $2(k+3q)+1$ & $2(k+3q)+3$ \\
			\hline
			$r = 2$ & $2(k+3q)+3$ & $2(k+3q)+3$ \\
			\hline
			$r = 3,4$ & $2(k+3q)+3$ & $2(k+3q)+5$ \\
			\hline
			$r = 5,6$ & $2(k+3q)+5$ & $2(k+3q)+7$ \\
			\hline
		\end{tabular}
		
		\caption{Maximum dimension of spheres in $\M(P_n \times P_4)$}}
	\label{tab:intro_dimensionPnP4}
\end{table}

\begin{proposition} \label{proposition:dim_PnP4}
    For $n \geq 3$, the minimum dimension of a sphere in the homotopy type of $\M(P_n \times P_4)$ is $2k+1$ if $n = 2k+2$, and $2k+3$ if $n = 2k+3$. The maximum dimension of a sphere in the homotopy type of $\M(P_n \times P_4)$  is given in \Cref{tab:intro_dimensionPnP4}. 
\end{proposition}

 The {\it line graph} $L(G)$ of a graph $G$ is the graph whose vertices are the edges of $G$, and any two such distinct vertices $uv$ and $u'v'$ are joined by an edge in $L(G)$ if and only if they have a common endpoint, {\it i.e.}, if $\{u, v\} \cap \{u', v'\}  \neq \emptyset$. For a graph $G$, a subset $A \subseteq V(G)$ is said to be {\it independent} in $G$ if the induced subgraph on $A$ does not contain any edge. The {\it independence complex} $\ind(G)$ of $G$ is a simplicial complex on $V(G)$ whose simplices are the independent sets in $G$.  Observe that for any graph $G$, the matching complex $\M(G)$ is the same as the $\ind(L(G))$.

 To prove \Cref{theorem:combine}, we show that  $\ind(L(P_n \times P_m))$ is a wedge of spheres up to homotopy equivalence. The main idea used in this proof is to reduce the graph $L(P_n \times P_m)$ into smaller graphs by the operations; deletions of vertices, deletions of edges, and additions of edges.  During this process of reduction, we come across thirteen new classes of graphs in the case $m = 5$, four new classes of graphs in the case $m = 4$, and one new class of graphs in the case $m=3$. We establish recursive relations among the independence complex of  $L(P_n \times P_m)$ and independence complexes of these new classes of graphs. 

 This paper is organized as follows: In \Cref{section:prel}, we present the necessary preliminaries on graph theory and simplicial complexes. In \Cref{section:PnP3}, we determine the homotopy type of the matching complexes of $P_n \times P_3$, showing that it is a wedge of spheres, and we explicitly compute both the number and the dimensions of the spheres appearing in the wedge. In \Cref{section:PnP4}, we compute the homotopy type of the matching complexes of $P_n \times P_4$. The graph $P_n \times P_4$ is disconnected and consists of two connected components. Regardless of whether $n$ is odd or even, these components are isomorphic. For $n \geq 0$, we denote the connected components of the line graphs of $P_{2n+2}\times P_4$ and $P_{2n+3}\times P_4$ by $\Gamma_{n,4}$ (see \Cref{fig:Gn4}) and $\widetilde{\Gamma}_{n,4}$ (see \Cref{fig:tildeGamman4}), respectively.  
In \Cref{PnP4Graphs}, we formally define the graphs $\Gamma_{n,4}$, $\widetilde{\Gamma}_{n,4}$, and four intermediary graphs needed to compute the homotopy types of the independence complexes of $\Gamma_{n,4}$ and $\widetilde{\Gamma}_{n,4}$. In \Cref{subsection:Gamma4 and Lembda4}, we establish recursive relations between their independence complexes. Then, in \Cref{PnP4homotopy},  we conclude the homotopy type of the matching complexes of $P_n \times P_4$, showing that it is a wedge of spheres, and we determine the maximum and minimum dimensions of spheres that occur in the wedge in $\M(P_n\times P_4)$.

\Cref{section:PnP5} is majorly dedicated to the  matching complexes of  $P_n \times P_5$ and is divided into three subsections. The graph $L(P_n \times P_5)$ is disconnected and has exactly two components. For $k \geq 0$, if $n = 2k+2$, then $L(P_n \times P_5)$ consists of two copies of the graph $\Gamma_{k, 5}$ (see \Cref{fig:Gn}); and if $n = 2k+3$, then $L(P_n \times P_5) = \Lambda_{k, 5} \sqcup \widetilde{\Gamma}_{k, 5}$ (see \Cref{fig:Lambda-n} and \Cref{fig:G'n}).    We start   \Cref{section:PnP5} with the definitions of the graphs  $\Gamma_{k, 5}, \Lambda_{k, 5}$ and $ \widetilde{\Gamma}_{k, 5}$.  In \Cref{subsection:Gamma and Lembda}, we compute the independence complexes of $\Gamma_{k, 5}$ and $ \Lambda_{k, 5}$. We introduce  $6$ new classes of graphs (which we need in the computation of $\ind(\Gamma_{k, 5})$ and $\ind(\Lambda_{k, 5})$), and establish recursive relations among independence complexes of graphs $\Gamma_{k,5}, \Lambda_{k, 5}$ and these $6$ classes of graphs.   
  In \Cref{subsection:GammaTilde},  we compute the independence complex of the graph $\widetilde{\Gamma}_{k, 5}$. Here again,  we define another $7$ class of graphs (which we need in the computation of $\ind(\widetilde{\Gamma}_{k, 5})$), and establish recursive relations among independence complexes of these $7$ new classes of graphs and   $\ind(\widetilde{\Gamma}_{k, 5})$.  
    In \Cref{subsection:PnP5dimensionbound}, we conclude the homotopy type of $\M(P_n \times P_5)$ and determine the minimum and maximum dimensions of spheres appearing in the wedge in $\M(P_n \times P_5)$. Finally, in  \Cref{section:futureplan}, we pose a few questions and a conjecture, which naturally arise from the work done in this paper.

\section{Preliminaries}\label{section:prel}

 A {\em graph} $G$ is an ordered pair $(V(G), E(G))$, where $V(G)$ is a finite set of \emph{vertices} and $E(G) \subseteq \binom{V(G)}{2}$ is the set of \emph{edges} of $G$, consisting of $2$-element subsets of $V(G)$.
 For the sake of brevity, we denote an edge $\{u,v\}$ by $uv$. Given a graph $G=(V(G),E(G))$ and a vertex $u$ in $V(G)$, the {\em open neighborhood} of $u$ is defined as the collection $\{v\in V(G) : uv\in E(G)\}$ and it is denoted by  $N(u)$.  
 The {\em closed neighborhood} of a vertex $u$ in the graph $G$ is defined as $\{u\}\cup N(u)$, and is denoted by $\n[u]$. 
 A map $f \colon V (G) \rightarrow V (H)$ is said to be a {\em graph homomorphism} if $f(v)f(w) \in E(H)$ for all $vw \in E(G)$. A bijective graph homomorphism is called an {\it isomorphism}. 
 If such an isomorphism exists between two graphs $G$ and $H$, they are called {\em isomorphic}, denoted as $G \cong H$. Let $S=\{v_1,\dots, v_r\}$ be a collection of vertices in $V(G)$. 
 Then, $G\setminus S$ denotes the induced subgraph of $G$ on the vertex set $V(G)\setminus S$. 
 In particular, if $S$ is the singleton set $\{v_1\}$, then we write the graph $G\setminus S$ as $G\setminus v_1$. 
 Similarly, for a collection of edges $T=\{e_1,\dots,e_s\}$, by $G+T$ and $G-T$ we denote the graphs on the vertex set $V(G)$ with edge set $E(G)\cup T$ and $E(G)\setminus T$, respectively.  
 For more details about the graphs, we refer the reader to \cite{west}.
 
An {\em (abstract) simplicial complex} $\Delta$ on a vertex set $V$ is a collection of subsets of $V$ such that if $\sigma\in \Delta$ and $\tau$ is a subset of $\sigma$, then $\tau\in\Delta$. If $\sigma\in\Delta$ and $|\sigma|=k+1$, then we say $\sigma$ is a simplex of dimension $k$ or a $k$-simplex. We further assume that the empty set is present in every simplicial complex as the only simplex of dimension $-1$. The $0$-dimensional simplices are called the vertices of $\Delta$, and the set of all vertices of a simplicial complex $\Delta$ is denoted by $V(\Delta)$.

 A {\em subcomplex} of a simplicial complex is a sub-collection of simplices that itself forms a simplicial complex. If $\Delta$ is a simplicial complex and $\sigma$ is a simplex in $\Delta$, then the {\em link} and {\em deletion} of $\sigma$ are defined as the subcomplexes $\{\tau\in\Delta : \sigma\cup\tau\in\Delta \hspace{.2cm}\text{and}\hspace{.1cm} \sigma\cap\tau=\emptyset\}$ and $\{\tau\in\Delta : \sigma\nsubseteq\tau\}$, respectively. 
   The {\em join} of two simplicial complexes $\Delta_1$ and $\Delta_2$, where $V(\Delta_1)\cap V(\Delta_2)=\emptyset$, is the simplicial complex $\Delta_1\ast\Delta_2\coloneq\{\sigma\cup\tau : \sigma\in\Delta_1\hspace{.2cm}\text{and}\hspace{.1cm}\tau\in\Delta_2\}$. In particular, if $a,b$ are vertices that are not part of a simplicial complex $\Delta$, then the simplicial complexes $\{a\}\ast\Delta$ and $(\{a\}\ast\Delta)\cup(\{b\}\ast\Delta)$ are called the {\em cone} and the {\em suspension} of $\Delta$ with the apexes $a$ and $\{a,b\}$, respectively. We denote the suspension of $\Delta$ with apex $\{a,b\}$ by $\Sigma_{a,b}\Delta$ (or $\Sigma\Delta$). 
 
 In this article,  we consider any simplicial
complex as a topological space, namely its geometric realization. For the definition of
geometric realization, we refer the reader to \cite{kozlovbook}. For details about topological background, one may refer to \cite{Hatcher}.

%The following observation follows from the definition of the join of two simplicial complexes:

\begin{proposition}{\label{prop: join}}
    Let $G_1 \sqcup G_2$ denote the disjoint union of two graphs, $G_1$ and $G_2$. Then
    $\ind(G_1 \sqcup G_2) \simeq \ind(G_1) \ast \ind(G_2).$
\end{proposition}

Clearly, for each vertex $v\in V(G)$, the link and the deletion of the simplex $\{v\}$ in $\ind(G)$ are the same as $\ind(G\setminus \n[v])$ and $\ind(G\setminus v)$, respectively. Observe that $\ind(G)$ is the same as the simplicial complex $(\{v\}\ast \ind(G\setminus \n[v]))\cup \ind(G\setminus v)$. Further, from \cite{Adamaszek}, we have the following:

\begin{proposition}[{\rm Proposition 3.1, \cite{Adamaszek}}]\label{Link and Deletion} 
    Let $v$ be a vertex of a graph $G$. If the inclusion $\ind(G\setminus \n[v]) \xhookrightarrow {} \ind(G \setminus v)$ is null-homotopic, then we have
    $$\ind(G)\simeq \ind(G \setminus v) \vee \Sigma \ind(G\setminus N [v]).$$
\end{proposition}

%The following propositions exhibit homotopy equivalence between the independence complex of a graph $G$ and the independence complexes of graphs derived from $G$ through certain graph operations; involving deletion of vertices, edges, or vertex-neighborhoods.

\begin{proposition}[{\rm Lemma 2.4, \cite{Engstrom}}] 
\label{Folding Lemma} 
    Let $v$ and $w$ be a pair of distinct vertices of $ G$ with $N(v) \subseteq N(w)$. Then    $\ind(G)\simeq {} \ind(G \setminus w)$.
\end{proposition}

\begin{proposition} [{\rm Lemma 2.5, \cite{Engstrom}}] \label{Simplicial Vertex Lemma} 
    Let $G$ be a graph and $v$ be a simplicial vertex of $G$, that is, all neighbors of $v$ are pairwise adjacent. Let $N(v) = \{w_1,w_2,\dots,w_k\}$. Then
    $\ind(G) \simeq \bigvee_{i=1}^{k}\Sigma\ind(G\setminus N [w_i]).$
\end{proposition}

\begin{proposition}[{\rm Proposition 3.4, \cite{Adamaszek}}] \label{Edge deletion 1}
    Let $G$ be a graph and let $uv$ be an edge in $G$. If $\ind(G\setminus \n[u, v])$ is contractible, then 
    $\ind(G)\simeq \ind(G \setminus uv).$
\end{proposition}

\begin{definition}
    A triplet of distinct vertices $u,v$ and $x$ in a graph $G$ satisfying $x \notin N(u) \cup N(v)$ and $N(x) \subseteq N(u) \cup N(v)$ is called an \emph{edge-invariant triplet} in $G$, and denoted by $[u,v;x]$.
\end{definition}

Based on the \Cref{Edge deletion 1}, we have the following. Suppose $G$ is a graph that consists of vertices $\{u,v\}$ and some vertex $x$ (depending on $u$ and $v$) such that $[u,v;x]$ is an edge invariant triplet in $G$. Then $x$ is an isolated vertex in $G \setminus \n[u,v]$, which implies that $\ind(G\setminus \n[u, v])$ is contractible. Thus, we have 
$$
\ind(G) \simeq \ind(G\triangle uv),
$$
where $V(G \triangle uv) = V(G)$, and $E(G \triangle uv)$ is given by the symmetric difference $E(G) \triangle \{uv\}.$

% $x \notin N(u) \cup N(v)$, and $N(x) \subseteq N(u) \cup N(v)$. 

For $n \geq 1$, the path graph $P_n$ is a graph with $V (P_n) = \{v_i \mid 1\leq i\leq n\}$ and $E(P_n) = \{v_iv_{i+1} \mid  1\leq i \leq n-1\}$.
 \begin{proposition} [{\rm Proposition 4.6, \cite{Kozlov99}}] \label{Ind(path)}
For $n \geq 1$, 
\begin{center}
            $\ind(P_n) \simeq \begin{cases}
            \Sp^{k-1}  & \text{if} \ n = 3k, 3k-1,\\
            \text{a point} & \text{if} \ n = 3k+1.\\
                       \end{cases}$
  \end{center}
 \end{proposition}
 
For $n \geq 1$, the cycle graph $\mathcal{C}_n$ is a graph with $V (\mathcal{C}_n) = \{v_i \mid 1\leq i\leq n\}$ and $E(\mathcal{C}_n) = \{v_iv_{i+1} \mid 1\leq i \leq n-1\}\cup \{v_1v_n\}$.
\begin{proposition} [{\rm Proposition 5.2, \cite{Kozlov99}}] \label{Ind(Cycle)}
For $n \geq 3$,  
\begin{center}
            $\ind(\mathcal{C}_n) \simeq \begin{cases}
            \Sp^{k-1} \vee S^{k-1} & \text{if} \ n = 3k,\\
             \Sp^{k-1} & \text{if} \ n = 3k \pm 1.\\
                       \end{cases}$
  \end{center}
 \end{proposition}

\section{\texorpdfstring{Matching complexes of $P_n \times P_3$}{Matching complexes of Pn × P3}} \label{section:PnP3}
In this section, we determine the homotopy type of the matching complex of  $P_n\times P_3$.
The graph $P_n\times P_3$ is disconnected and has two connected components. When $n$ is even, then these components are isomorphic. However, for $n$ odd, the components are not isomorphic.
For $n \geq 0$, we denote the (isomorphic) connected components of the line graphs of $P_{2n+2}\times P_3$ by $\Gamma_{n,3}$ (see \Cref{fig:Gamman3}). Furthermore, for $n \geq 0$,  we denote the (non-isomorphic) connected components of the line graphs of  $P_{2n+3}\times P_3$ by $\Lambda_{n,3}$ and $\widetilde{\Gamma}_{n,3}$ (see \Cref{fig:Lambdan3,,fig:Lambdan3tilde}), respectively. Therefore, we have the following:
$$
\M(P_{2n+2}\times P_3) \simeq \ind(\Gamma_{n,3}) * \ind(\Gamma_{n,3}) \ \text{and} \ \M(P_{2n+3}\times P_3) \simeq \ind(\Lambda_{n,3}) * \ind(\widetilde{\Gamma}_{n,3}).
$$

% \begin{equation*}
	%     \M(P_{2n+2}\times P_3) \simeq \ind(\Gamma_{n,3}) * \ind(\Gamma_{n,3}), 
	% \end{equation*}
% \begin{equation*}
	% \M(P_{2n+3}\times P_3) \simeq \ind(\Lambda_{n,3}) * \ind(\widetilde{\Lambda}_{n,3}).
	% \end{equation*}
During the computation of the homotopy type of the independence complexes of graphs $\Gamma_{n,3}, \Lambda_{n,3}$, and $\widetilde{\Gamma}_{n,3}$ for $n\geq 0$, we encounter an intermediary graph, denoted as $A_{n,3}$ (see \Cref{fig:An3}). Using the independence complex of $A_{n,3}$, we illustrate that the independence complexes of $\Gamma_{n,3}, \Lambda_{n,3}$, and $\widetilde{\Gamma}_{n,3}$ are the wedges of spheres. \\

%\subsection{\texorpdfstring{Graph $\Gamma_{n,3}$}{Graph Γn,3}}

\noindent (i) \textbf{Graph }$\Gamma_{n,3}$  \\
For $n\geq 0$, we define the graph $\Gamma_{n,3}$ as follows:
\begin{equation*}
	\begin{split}
		V(\Gamma_{n,3}) = & \set{g_{ij} : i=1,2 \text{ and } j = {1,2,\dots,2n{+}1}},\\
		E(\Gamma_{n,3}) = & \ \set{g_{1j}g_{2j} : j=1,2,\dots,2n{+}1} 
		\cup \set{g_{ij}g_{i\,j+1} : i=1,2 \text{ and } j=1,2,\dots,2n} \\
		& \cup \set{g_{1j}g_{2\,j+1}, g_{1\,j+1}g_{2\,j} : \; j=2,4,\dots,2n}.
	\end{split}
\end{equation*}

\begin{figure}[ht]
	\centering	 
	\begin{subfigure}[b]{0.45\textwidth}
		\centering
		\begin{tikzpicture}[scale=0.3]

			\begin{scope}
				\foreach \x in {1,2,...,8}
				{
					\node[vertex] (1\x) at (2*\x,3*2) {};
				}
				\foreach \x in {1,2,...,8}
				{
					\node[vertex] (2\x) at (2*\x,2*2) {};
				}

				\foreach \y [evaluate={\y as \x using int({\y+1})}] in {1,2,...,4,6,7} 
				{
					\draw[edge] (1\y) -- (1\x);
					\draw[edge] (2\y) -- (2\x);
				}
				
				\foreach \z in {1,2,3,4,5,...,8}
				{\draw[edge] (1\z) -- (2\z);
				}
				
				\foreach \z [evaluate={\znext=int({1+\z})}] in {2,4,7}
				{
					\draw[edge] (1\z) -- (2\znext);
				}
				\foreach \z [evaluate={\znext=int({\z-1})}] in {3,5,8}
				{
					\draw[edge] (1\z) -- (2\znext);
				}
				
				\coordinate (c15) at (15);
				\coordinate (c16) at (16);

				\draw[dotted, thick,shorten >=4pt, shorten <=4pt] (10,5) -- (12,5);

				\draw[decoration={brace,raise=5pt,mirror},decorate]
				(21) -- node[below=6pt] {$1$}(23);
				\draw[decoration={brace,raise=5pt,mirror},decorate]
				(23) -- node[below=6pt] {$2$}(25);
				\draw[decoration={brace,raise=5pt,mirror},decorate]
				(26) -- node[below=6pt] {$n$}(28);
				
				\node[label={[label distance=-3pt]90:\scriptsize{$g_{11}$}}] at (11) {};
				\foreach \x in {2,...,5}
				{
					\node[label={[label distance=-3pt]90:\scriptsize{$g_{1\x}$}}] at (1\x) {};
				}
				\node[label={[label distance=-3pt]180:\scriptsize{$g_{21}$}}] at (21) {};

				\node[label={[label distance=-3pt]90:\scriptsize{$g_{1\,{2n+1}}$}}] at (18) {};
				\node[label={[label distance=-3pt]90:\scriptsize{$\cdots$}}] at (16) {};
				
				\node[label={[label distance=-3pt]0:\scriptsize{$g_{2\,{2n+1}}$}}] at (28) {};
				% \node[] at (7,0) {(b) $\Gamma_{n,3}$ for $n\geq 1$};
				
			\end{scope}

			% \begin{scope}[shift = {(-5,0)}]
				%             \foreach \x in {1,2}
				%             {\node[vertex] (x\x1) at (0,2*(4-\x) {};}
				
				%             \foreach \x [evaluate = \x as \y using int(\x+1)] in {1}
				%             {\draw[edge] (x\x1) -- (x\y1) ;}
				
				%             \foreach \y in {1,2}
				%             {
					%             \node[label={[label distance=-3pt]0:\scriptsize{$g_{\y 1}$}}] at (x\y1) {};
					%             }
				%            \node[] at (0.8,0) {(a) $\Gamma_{0,3}$};
				%         \end{scope}
			
		\end{tikzpicture}
		\caption{$\Gamma_{n,3}$}
		\label{fig:Gamman3}
	\end{subfigure}
	\begin{subfigure}[b]{0.45\textwidth}
		\centering
		\begin{tikzpicture}[scale=0.3]
			\begin{scope}
				\foreach \x in {1,2,...,8}
				{
					\node[vertex] (1\x) at (2*\x,3*2) {};
				}
				\foreach \x in {1,2,...,8}
				{
					\node[vertex] (2\x) at (2*\x,2*2) {};
				}

				\foreach \y [evaluate={\y as \x using int({\y+1})}] in {1,2,...,4,6,7} 
				{
					\draw[edge] (1\y) -- (1\x);
					\draw[edge] (2\y) -- (2\x);
				}
				
				\foreach \z in {1,2,3,4,5,...,8}
				{\draw[edge] (1\z) -- (2\z);
				}
				
				\foreach \z [evaluate={\znext=int({1+\z})}] in {2,4,7}
				{
					\draw[edge] (1\z) -- (2\znext);
				}
				\foreach \z [evaluate={\znext=int({\z-1})}] in {3,5,8}
				{
					\draw[edge] (1\z) -- (2\znext);
				}
				
				\coordinate (c15) at (15);
				\coordinate (c16) at (16);

				\draw[dotted, thick,shorten >=4pt, shorten <=4pt] (10,5) -- (12,5);

				\draw[decoration={brace,raise=5pt,mirror},decorate]
				(21) -- node[below=6pt] {$1$}(23);
				\draw[decoration={brace,raise=5pt,mirror},decorate]
				(23) -- node[below=6pt] {$2$}(25);
				\draw[decoration={brace,raise=5pt,mirror},decorate]
				(26) -- node[below=6pt] {$n$}(28);
				
				\node[label={[label distance=-3pt]90:\scriptsize{$g_{11}$}}] at (11) {};
				\foreach \x in {2,...,5}
				{
					\node[label={[label distance=-3pt]90:\scriptsize{$g_{1\x}$}}] at (1\x) {};
				}

				\node[label={[label distance=-3pt]90:\scriptsize{$g_{1\,{2n+1}}$}}] at (18) {};
				\node[label={[label distance=-3pt]90:\scriptsize{$\cdots$}}] at (16) {};
				
				\node[label={[label distance=-3pt]0:\scriptsize{$g_{2\,{2n+1}}$}}] at (28) {};

			\end{scope}
			
			\begin{scope}[shift = {(-2,0)}]
				\node[vertex] (a1) at (2,4) {};
				
				\draw[edge] (a1) -- (11);
				\draw[edge] (a1) -- (21);
				
				\node[label={[label distance=-3pt]180:\scriptsize{$a_{1}$}}] at (a1) {};
				% \node[] at (8,1) {(b) $A_{n,3}$ for $n\geq 1$};
			\end{scope}

			% \begin{scope}[shift = {(-5,0)}] 
				
				%             \node[vertex] (a1) at (-2,4) {};
				
				%             \draw[edge] (a1) -- (x11);
				%             \draw[edge] (a1) -- (x21);

				%             \foreach \x in {1,2}
				%             {\node[vertex] (x\x1) at (0,2*(4-\x) {};}
				
				%             \foreach \x [evaluate = \x as \y using int(\x+1)] in {1}
				%             {\draw[edge] (x\x1) -- (x\y1) ;}
				
				%             \foreach \y in {1,2}
				%             {
					%             \node[label={[label distance=-3pt]0:\scriptsize{$g_{\y 1}$}}] at (x\y1) {};
					%             }
				
				%              \node[label={[label distance=-7pt]10:\scriptsize{$g_{21}$}}] at (21) {};
				%             \node[label={[label distance=-3pt]180:\scriptsize{$a_{1}$}}] at (a1) {};
				%        \node[] at (0,1) {(a) $A_{0,3}$};
				
				%         \end{scope}
			% \label{subfig:A03}
			
		\end{tikzpicture}
		\caption{ $A_{n,3}$} \label{fig:An3}	
	\end{subfigure}
    \caption{}
\end{figure}

\begin{claim}{\label{claim:Gamman3}}
	For $n\geq 0$,
	$\ind(\Gamma_{n,3}) \simeq \Sp^n$.
\end{claim}
\begin{proof}
	Since $N(g_{11}) \subseteq N(g_{22})$ in $\Gamma_{n,3}$ and $N(g_{21}) \subseteq N(g_{12})$ in $\Gamma_{n,3}\setminus g_{22}$, using \Cref{Folding Lemma}, we get that $ \ind(\Gamma_{n,3}) \simeq \ind(\Gamma_{n,3} \setminus \{g_{12},g_{22}\}).$
	It is straightforward to see that $\Gamma_{n,3} \setminus \{g_{12},g_{22}\} $ has two connected components; one component is an edge $g_{11}g_{21}$, and the other is isomorphic to $\Gamma_{n-1,3}$. Therefore, $ \ind(\Gamma_{n,3}) \simeq \Sigma\ind(\Gamma_{n-1,3})$. Consequently, we get $ \ind(\Gamma_{n,3}) \simeq \Sigma^n\ind(\Gamma_{0,3}).$
	Since $\Gamma_{0,3} \cong K_2$, the result follows by induction.
\end{proof}
%\subsection{\texorpdfstring{Graph $A_{n,3}$}{Graph An,3}}

\noindent (ii) \textbf{Graph }$A_{n,3}$  \\
For $n\geq 0$, we define the graph $A_{n,3}$ as follows:

$ V(A_{n,3}) =  V(\Gamma_{n,3})\cup\{a_1\}$ and 
$E(A_{n,3}) =  E(\Gamma_{n,3})\cup \{a_1g_{11}, a_1g_{21}\}.  $

%\begin{equation*}
%\begin{split}
%V(A_{n,3}) = & V(\Gamma_{n,3})\cup\{a_1\},\\
% E(A_{n,3}) = & E(\Gamma_{n,3})\cup \{a_1g_{11}, a_1g_{21}\}.    
% \end{split}
%\end{equation*}

\begin{claim}\label{lemma: An3}
For $n \geq 0$, $ \ind(A_{n,3}) \simeq \bigvee^{2}\Sigma \ind(A_{n-1,3}) \simeq \bigvee^{2^{n+1}} \Sp^n.$
\end{claim}
\begin{proof}
Since $a_1$ is a simplicial vertex in the graph $A_{n,3}$ (see \Cref{fig:An3}) with $N(a_1)=\{g_{11},g_{21}\}$, from \Cref{Simplicial Vertex Lemma}, we have $ \ind(A_{n,3})\simeq \Sigma \ind(A_{n,3}\setminus \n[g_{11}])\vee \Sigma\ind(A_{n,3}\setminus \n[g_{21}]).$
Since both the graphs $A_{n,3}\setminus \n[g_{11}]$ and $A_{n,3}\setminus \n[g_{21}]$ are isomorphic to $A_{n-1,3}$, we have $\ind(A_{n,3})\simeq \bigvee^{2}\Sigma \ind(A_{n-1,3})$. Proceeding in this way for $n$ steps, we obtain $\mathrm{Ind}(A_{n,3}) \simeq \bigvee^{2^n} \Sigma^n \mathrm{Ind}(A_{0,3})$. Observe that the graph $A_{0,3}$ is isomorphic to $K_3$, and hence $\ind(A_{0,3})\simeq \Sp^0\vee\Sp^0$. Therefore, $\mathrm{Ind}(A_{n,3}) \simeq {\bigvee}^{2^n} \Sigma^n (\mathbb{S}^0 \vee \mathbb{S}^0) \simeq {\bigvee}^{2^n} (\mathbb{S}^n\vee S^n)$. This completes the proof.
\end{proof}
%\subsection{\texorpdfstring{Graphs $\Lambda_{n,3}$ and $\widetilde{\Gamma}_{n,3}$}{Graphs Λn,3 and  ̃Λn,3}}

\noindent (iii) {\bf Graphs $\Lambda_{n,3}$ and $\widetilde{\Gamma}_{n,3}$}  \\
For $n\geq 0$, we define the graphs $\Lambda_{n,3}$ (see \Cref{fig:Lambdan3}) and $\widetilde{\Gamma}_{n,3}$ (see \Cref{fig:Lambdan3tilde}) as follows:

$V(\Lambda_{n,3}) =  V(\Gamma_{n,3})\cup\{g'_{1},g'_{2}\}$,  
$E(\Lambda_{n,3}) =  E(\Gamma_{n,3})\cup \{g'_{1}g_{11}, g'_{2}g_{21}, g'_{1}g'_{2},g'_{1}g_{21},g'_{2}g_{11}\}$, and 
%$V(\widetilde{\Gamma}_{n,3}) =  V(\Gamma_{n,3})\cup\{g_{12n+2},g_{22n+2}\}$,
% $E(\widetilde{\Gamma}_{n,3}) =  E(\Gamma_{n,3})\cup \{g_{12n+1}g_{12n+2}, g_{22n+1}g_{22n+2}, g_{12n+2}g_{22n+2}\}$.
\begin{equation*}
\begin{split} V(\widetilde{\Gamma}_{n,3}) = & V(\Gamma_{n,3})\cup\{g_{12n+2},g_{22n+2}\},\\
	E(\widetilde{\Gamma}_{n,3}) = & E(\Gamma_{n,3})\cup \{g_{12n+1}g_{12n+2}, g_{22n+1}g_{22n+2}, g_{12n+2}g_{22n+2}\}.    
\end{split}
\end{equation*}
	\begin{figure}[ht]
		\centering
		\begin{subfigure}[b]{0.4\textwidth}
			\begin{tikzpicture}[scale=0.3]
				\begin{scope}
					\foreach \x in {1,2,...,8}
					{
						\node[vertex] (1\x) at (3*\x,3*2) {};
					}
					\foreach \x in {1,2,...,8}
					{
						\node[vertex] (2\x) at (3*\x,2*2) {};
					}

					\foreach \y [evaluate={\y as \x using int({\y+1})}] in {1,2,...,4,6,7} 
					{
						\draw[edge] (1\y) -- (1\x);
						\draw[edge] (2\y) -- (2\x);
					}
					
					\foreach \z in {1,2,3,4,5,...,8}
					{\draw[edge] (1\z) -- (2\z);
					}
					
					\foreach \z [evaluate={\znext=int({1+\z})}] in {2,4,7}
					{
						\draw[edge] (1\z) -- (2\znext);
					}
					\foreach \z [evaluate={\znext=int({\z-1})}] in {3,5,8}
					{
						\draw[edge] (1\z) -- (2\znext);
					}
					
					\coordinate (c15) at (15);
					\coordinate (c16) at (16);

					\draw[dotted, thick,shorten >=4pt, shorten <=4pt] (15.5,5) -- (17.5,5);

					\draw[decoration={brace,raise=5pt,mirror},decorate]
					(21) -- node[below=6pt] {$1$}(23);
					\draw[decoration={brace,raise=5pt,mirror},decorate]
					(23) -- node[below=6pt] {$2$}(25);
					\draw[decoration={brace,raise=5pt,mirror},decorate]
					(26) -- node[below=6pt] {$n$}(28);
					
					\node[label={[label distance=-3pt]90:\scriptsize{$g_{11}$}}] at (11) {};
					\foreach \x in {2,...,5}
					{
						\node[label={[label distance=-3pt]90:\scriptsize{$g_{1\x}$}}] at (1\x) {};
					}

					\node[label={[label distance=-3pt]90:\scriptsize{$g_{1\,{2n+1}}$}}] at (18) {};
					\node[label={[label distance=-3pt]90:\scriptsize{$\cdots$}}] at (16) {};
					
					\node[label={[label distance=-3pt]0:\scriptsize{$g_{2\,{2n+1}}$}}] at (28) {};

				\end{scope}
				
				\begin{scope}[shift={(2,0)}]
					\node[vertex] (e1) at (-2,6) {} ;
					\node[vertex] (e2) at (-2,4) {} ;

					\draw[edge] (e1) -- (e2) ;
					\draw[edge] (e1) -- (11) ;
					\draw[edge] (e2) -- (21) ;
					\draw[edge] (e1) -- (21) ;
					\draw[edge] (e2) -- (11) ;
					
					\node[label={[label distance=-3pt]90:\scriptsize{$g'_{1}$}}] at (e1) {};
					\node[label={[label distance=-3pt]270:\scriptsize{$g'_{2}$}}] at (e2) {};
					% \node[] at (8,1) {(b) $\Lambda_{n,3}$ for $n\geq 1$};
				\end{scope}

			\end{tikzpicture}
			\caption{ $\Lambda_{n,3}$}
			\label{fig:Lambdan3}	
		\end{subfigure} 
	
		\begin{subfigure}[b]{0.4\textwidth}
			\begin{tikzpicture}[scale=0.3]
				\begin{scope}
					\foreach \x in {1,2,...,8,9}
					{
						\node[vertex] (1\x) at (3*\x,3*2) {};
					}
					\foreach \x in {1,2,...,8,9}
					{
						\node[vertex] (2\x) at (3*\x,2*2) {};
					}

					\foreach \y [evaluate={\y as \x using int({\y+1})}] in {1,2,...,4,6,7,8} 
					{
						\draw[edge] (1\y) -- (1\x);
						\draw[edge] (2\y) -- (2\x);
					}
					
					\foreach \z in {1,2,3,4,5,...,8,9}
					{\draw[edge] (1\z) -- (2\z);
					}
					
					\foreach \z [evaluate={\znext=int({1+\z})}] in {2,4,7}
					{
						\draw[edge] (1\z) -- (2\znext);
					}
					\foreach \z [evaluate={\znext=int({\z-1})}] in {3,5,8}
					{
						\draw[edge] (1\z) -- (2\znext);
					}
					
					\coordinate (c15) at (15);
					\coordinate (c16) at (16);

					\draw[dotted, thick,shorten >=4pt, shorten <=4pt] (15.5,5) -- (17.5,5);

					\draw[decoration={brace,raise=5pt,mirror},decorate]
					(21) -- node[below=8pt] {$1$}(23);
					\draw[decoration={brace,raise=5pt,mirror},decorate]
					(23) -- node[below=6pt] {$2$}(25);
					\draw[decoration={brace,raise=5pt,mirror},decorate]
					(26) -- node[below=6pt] {$n$}(28);
					
					\node[label={[label distance=-3pt]90:\scriptsize{$g_{11}$}}] at (11) {};
					\foreach \x in {2,...,5}
					{
						\node[label={[label distance=-3pt]90:\scriptsize{$g_{1\x}$}}] at (1\x) {};
					}
					\node[label={[label distance=-3pt]180:\scriptsize{$g_{21}$}}] at (21) {};

					\node[label={[label distance=-3pt]90:\scriptsize{$g_{1\,{2n+2}}$}}] at (19) {};
					\node[label={[label distance=-3pt]90:\scriptsize{$\cdots$}}] at (16) {};
					
					\node[label={[label distance=-3pt]270:\scriptsize{$g_{2\,{2n+2}}$}}] at (29) {};
					% \node[] at (8,1) {(b) $\widetilde{\Gamma}_{n,3}$ for $n\geq 1$};
					
				\end{scope}

			\end{tikzpicture}
			\caption{$\widetilde{\Gamma}_{n,3}$} \label{fig:Lambdan3tilde}	
		\end{subfigure}
		\caption{}
	\end{figure}

\begin{claim}\label{Lambdan3}
For $n\geq0$, $ \ind(\Lambda_{n,3}) \simeq \bigvee^{2^{n+2}-1} \Sp^n.$
\end{claim}
\begin{proof}
The proof is by induction on $n$. Clearly,  $\ind(\Lambda_{0,3}) \simeq \bigvee^{3}\Sp^0$. Assume that $n \geq 1$ and the \Cref{Lambdan3} holds for all $0 \leq k < n$. Since $g'_{1}$ is a simplicial vertex and $N(g'_1)=\{g'_2, g_{11},g_{21}\}$ in the graph $\Lambda_{n,3}$ (see \Cref{fig:Lambdan3}), from \Cref{Simplicial Vertex Lemma}, we have 
\begin{equation*}
	\begin{split}
		\ind(\Lambda_{n,3}) & \simeq \Sigma\ind(\Lambda_{n,3} \setminus \n[g'_2]) \vee \Sigma\ind(\Lambda_{n,3} \setminus \n[g_{11}]) \vee \Sigma\ind(\Lambda_{n,3} \setminus \n[g_{21}]). 
	\end{split}
\end{equation*}
Observe that, the graph $\Lambda_{n,3} \setminus \n[g'_2]$ is isomorphic to $\Lambda_{n-1,3}$, whereas both the graphs $\Lambda_{n,3} \setminus \n[g_{11}]$ and $\Lambda_{n,3} \setminus \n[g_{21}]$ are isomorphic to $A_{n-1,3}$. Therefore, we get $$\ind(\Lambda_{n,3})\simeq \Sigma\ind(\Lambda_{n-1,3}) \vee \Sigma\ind(A_{n-1,3}) \vee \Sigma\ind(A_{n-1,3}).$$ Now, from the induction hypothesis and \Cref{lemma: An3}, we conclude that
%\begin{equation*}
%\begin{split}
$\ind(\Lambda_{n,3})  \simeq \bigvee^{2^{n+2}-1}\Sp^{n}$
% \bigvee^{2^{n+1}-1}(\Sigma \Sp^{n-1})\bigvee^{2^{n}} (\Sigma \Sp^{n-1})\bigvee^{2^{n}} (\Sigma \Sp^{n-1} )
% \simeq .$
%\end{split}
%\end{equation*}
\end{proof}

\begin{claim}\label{eq:Lambdatilde_n}
For $n \geq 0 $, $\ind(\widetilde{\Gamma}_{n,3}) \simeq \Sp^{n}$.
\end{claim}
\begin{proof}
The proof for $\ind(\widetilde{\Gamma}_{n,3})$ is essentially verbatim to the proof of \Cref{claim:Gamman3}. 
% It is evident that the graphs $\Gamma_{n,3}$ and $\widetilde{\Gamma}_{n,3}$ exhibit similarities (see \Cref{fig:Gamman3} and \Cref{fig:Lambdan3tilde}). Moreover, it is clear from the edge set of $\widetilde{\Gamma}_{n,3}$ that the graph $\widetilde{\Gamma}_{n,3}$ can be obtained by the addition of three edges at the rightmost end to the graph $\Gamma_{n,3}$. Therefore, the proof for $\ind(\widetilde{\Gamma}_{n,3})$ is verbatim to the proof of \Cref{claim:Gamman3}.   
\end{proof}

\begin{theorem}\label{PnxP3}
For $n\geq 0$,  $\M(P_{2n+2} \times P_3) \simeq \Sp^{2n+1}$ and $ \M(P_{2n+3} \times P_3) \simeq \bigvee^{2^{n+2} - 1} \Sp^{2n+1}.$
\end{theorem}
\begin{proof}
Since the line graph of the categorical product $P_{2n+2} \times P_3$ is the disjoint union of two copies of $\Gamma_{n,3}$, using \Cref{prop: join} and \Cref{claim:Gamman3}, we get 
$\M(P_{2n+2} \times P_3)\simeq \ind(\Gamma_{n,3})\ast \ind(\Gamma_{n,3}) \simeq \Sp^{2n+1}.$
Similarly, the line graph of the categorical product $P_{2n+3} \times P_3$ is the disjoint union of $\Lambda_{n,3}$ and $\widetilde{\Gamma}_{n,3}$. Therefore, using \Cref{prop: join}, \Cref{Lambdan3,,eq:Lambdatilde_n} we get  
$\M(P_{2n+3} \times P_3) \simeq \ind(\Lambda_{n,3})\ast \ind(\widetilde{\Gamma}_{n,3})\simeq \bigvee^{2^{n+2} - 1} \Sp^{2n+1}.$
\end{proof}

\section{Matching complexes of \texorpdfstring{$P_n \times P_4$}{Pn x P4}} \label{section:PnP4}
In this section, we aim to determine the homotopy type of the matching complex of the categorical product $P_n\times P_4$.
The graph $P_n\times P_4$ is disconnected, and contains two connected components. Regardless of whether
$n$ is odd or even, these respective components are isomorphic.
For $n \geq 0$, we denote the connected components of the line graphs of $P_{2n+2}\times P_4$ and $P_{2n+3}\times P_4$ by $\Gamma_{n,4}$ (see \Cref{fig:Gn4}) and $\widetilde{\Gamma}_{n,4}$ (see \Cref{fig:tildeGamman4}), respectively. Therefore, we have the following:
\begin{equation*}
	\M(P_{2n+2}\times P_4) \simeq \ind(\Gamma_{n,4}) * \ind(\Gamma_{n,4}), 
\end{equation*}
\begin{equation*}
	\M(P_{2n+3}\times P_4) \simeq \ind(\widetilde{\Gamma}_{n,4}) * \ind(\widetilde{\Gamma}_{n,4}).
\end{equation*}

During the computation of the homotopy type of the independence complexes of graphs $\Gamma_{n,4}$ and $\widetilde{\Gamma}_{n,4}$ for $n\geq 0$, we encounter several intermediary graphs denoted as $A_{n,4}, \widetilde{A}_{n,4}, B_{n,4}$, and $\widetilde{B}_{n,4}$. We show that the independence complexes of $\Gamma_{n,4}$ and $\widetilde{\Gamma}_{n,4}$ are wedges of suspensions of these intermediary graphs. By induction on $n$, we show that the independence complexes of $\Gamma_{n,4}$, $\widetilde{\Gamma}_{n,4}$, and the intermediary graphs are homotopy equivalent to a wedge of spheres. The graphs $A_{n,4}, \widetilde{A}_{n,4}, B_{n,4}$, and $\widetilde{B}_{n,4}$ are shown in \Cref{fig:An4,,fig:Bn4,,fig:Cn4,,fig:Dn4}, respectively. In the following subsection, we define the graphs $\Gamma_{n,4}$, $\widetilde{\Gamma}_{n,4}$, along with four intermediary graphs. In the subsequent subsections, we then determine the homotopy types of the independence complexes of $\Gamma_{n,4}$, $\widetilde{\Gamma}_{n,4}$, and these intermediary graphs.

\subsection{Definitions of Graphs $\Gamma_{n,4}, \widetilde{\Gamma}_{n,4}$ and Other Intermediary Graphs }\label{PnP4Graphs}

%Graph and figure Gamman4

\noindent (i){\textbf{ Graph} $\Gamma_{n,4}$}

For $n\geq 0$, we define the graph $\Gamma_{n,4}$ as follows:
\begin{equation*}
	\begin{split}
		V(\Gamma_{n,4}) = & \set{g_{ij} : i=1,2,3 \  \text{ and }j = {1,2,\dots,2n{+}1}},\\
		E(\Gamma_{n,4}) = & \ \set{g_{ij}g_{i+1\,j} : i=1,2 \text{ and } j=1,2,\dots,2n{+}1} \cup \set{g_{ij}g_{i\,j+1} : i=1,2,3 \text{ and } j=1,2,\dots,2n} \\
		& \cup \set{g_{ij}g_{i+1\,j+1}, g_{ij+1}g_{i+1\,j} :  i=1,2, \; j=1,2,\dots,2n,  \text{ and } i+j\equiv 1 \mod{2}}.
	\end{split}
\end{equation*}

\noindent (ii) {\textbf{Graph} $\widetilde{\Gamma}_{n,4}$}

For $n\geq 0$, we define the graph $\widetilde{\Gamma}_{n,4}$ as follows:
\begin{equation*}
	\begin{split}
		V(\widetilde{\Gamma}_{n,4}) = & V(\Gamma_{n,4})\cup\{g_{1\,2n+2},g_{2\,2n+2},g_{3\,2n+2}\},\\
		E(\widetilde{\Gamma}_{n,4}) = & E(\Gamma_{n,4})\cup \{g_{1\,2n+2}g_{2\,2n+2}, g_{2\,2n+2}g_{3\,2n+2}, g_{1\,2n+1}g_{1\,2n+2}\}\\
		& \cup \{g_{2\,2n+1}g_{2\,2n+2}, g_{3\,2n+1}g_{3\,2n+2},g_{2\,2n+2}g_{3\,2n+1},g_{3\,2n+2}g_{2\,2n+1}\}.
	\end{split}
\end{equation*}

\begin{figure}[ht]
	\begin{subfigure}[b]{0.4\textwidth}
		\centering
		\begin{tikzpicture}[scale=0.3]
			\begin{scope}
				\foreach \x in {1,2,...,8}
				{
					\node[vertex] (1\x) at (2*\x,3*2) {};
				}
				\foreach \x in {1,2,...,8}
				{
					\node[vertex] (2\x) at (2*\x,2*2) {};
				}
				\foreach \x in {1,2,...,8}
				{
					\node[vertex] (3\x) at (2*\x,1*2) {};
				}

				\foreach \y [evaluate={\y as \x using int({\y+1})}] in {1,2,...,4,6,7} 
				{
					\draw[edge] (1\y) -- (1\x);
					\draw[edge] (2\y) -- (2\x);
					\draw[edge] (3\y) -- (3\x);
				}
				
				\foreach \z in {1,2,3,4,5,...,8}
				{\draw[edge] (1\z) -- (2\z);
					\draw[edge] (2\z) -- (3\z);
				}
				
				\foreach \z [evaluate={\znext=int({1+\z})}] in {2,4,7}
				{
					\draw[edge] (1\z) -- (2\znext);
				}
				\foreach \z [evaluate={\znext=int({\z-1})}] in {3,5,8}
				{
					\draw[edge] (1\z) -- (2\znext);
				}
				
				\foreach \z [evaluate={\znext=int({1+\z})}] in {1,3,6}
				{
					\draw[edge] (2\z) -- (3\znext);
				}
				\foreach \z [evaluate={\znext=int({\z-1})}] in {2,4,7}
				{
					\draw[edge] (2\z) -- (3\znext);
				}
				\coordinate (c15) at (15);
				\coordinate (c16) at (16);

				\draw[dotted, thick,shorten >=4pt, shorten <=4pt] (10,5) -- (12,5);
				\draw[dotted, thick,shorten >=4pt, shorten <=4pt] (10,3) -- (12,3);
				
				\draw[decoration={brace,raise=5pt,mirror},decorate]
				(2,1) -- node[below=6pt] {\small $1$}(6,1);
				\draw[decoration={brace,raise=5pt,mirror},decorate]
				(6,1) -- node[below=6pt] {\small $2$}(10,1);
				\draw[decoration={brace,raise=5pt,mirror},decorate]
				(12,1) -- node[below=6pt] {\small $n$}(16,1);
				
				\foreach \x in {1,2,...,5}
				{
					\node[label={[label distance=-3pt]90:\scriptsize{$g_{1\x}$}}] at (1\x) {};
					\node[label={[label distance=-3pt]-90:\scriptsize{$g_{3\x}$}}] at (3\x) {};
				}
				\node[label={[label distance=-3pt]180:\scriptsize{$g_{21}$}}] at (21) {};
				\node[label={[label distance=-3pt]90:\scriptsize{$\cdots$}}] at (16) {};
				\node[label={[label distance=-3pt]-90:\scriptsize{$\cdots$}}] at (36) {};
				
				\node[label={[label distance=-3pt]90:\scriptsize{$g_{1\,{2n+1}}$}}] at (18) {};
				\node[label={[label distance=-3pt]0:\scriptsize{$g_{2\,{2n+1}}$}}] at (28) {};
				\node[label={[label distance=-3pt]0:\scriptsize{$g_{3\,{2n+1}}$}}] at (38) {};
				
				\node[label={[label distance=-8pt]135:\scriptsize{$g_{22}$}}] at (22) {};
				\node[label={[label distance=-8pt]-135:\scriptsize{$g_{23}$}}] at (23) {};
				\node[label={[label distance=-8pt]135:\scriptsize{$g_{24}$}}] at (24) {};
				\node[label={[label distance=-8pt]-135:\scriptsize{$g_{25}$}}] at (25) {};
				
				% \node [] at (10,-2) {(b) $\Gamma_{n,4}$ for $n\geq 1$};
			\end{scope}
			
			% \begin{scope}[shift = {(-5,0)}]
				%             \foreach \x in {1,2,3}
				%             {\node[vertex] (x\x1) at (0,2*(4-\x) {};}
				
				%             \foreach \x [evaluate = \x as \y using int(\x+1)] in {1,2}
				%             {\draw[edge] (x\x1) -- (x\y1) ;}
				
				%             \foreach \y in {1,2,3}
				%             {
					%             \node[label={[label distance=-3pt]0:\scriptsize{$g_{\y 1}$}}] at (x\y1) {};
					%             }
				%            \node [] at (0,-2) {(a) $\Gamma_{0,4}$};
				%         \end{scope}
			
		\end{tikzpicture}
		\caption{$\Gamma_{n,4}$} \label{fig:Gn4}	
	\end{subfigure} 
	\begin{subfigure}[b]{0.4\textwidth}
		\begin{tikzpicture}[scale=0.3]
			\begin{scope}[shift={(-1,0)}]
				\foreach \x in {1,2,...,8,9}
				{
					\node[vertex] (1\x) at (2*\x,3*2) {};
				}
				\foreach \x in {1,2,...,8,9}
				{
					\node[vertex] (2\x) at (2*\x,2*2) {};
				}
				\foreach \x in {1,2,...,8,9}
				{
					\node[vertex] (3\x) at (2*\x,1*2) {};
				}

				\foreach \y [evaluate={\y as \x using int({\y+1})}] in {1,2,...,4,6,7,8} 
				{
					\draw[edge] (1\y) -- (1\x);
					\draw[edge] (2\y) -- (2\x);
					\draw[edge] (3\y) -- (3\x);
				}
				
				\foreach \z in {1,2,3,4,5,...,8,9}
				{\draw[edge] (1\z) -- (2\z);
					\draw[edge] (2\z) -- (3\z);
				}
				
				\foreach \z [evaluate={\znext=int({1+\z})}] in {2,4,7}
				{
					\draw[edge] (1\z) -- (2\znext);
				}
				\foreach \z [evaluate={\znext=int({\z-1})}] in {3,5,8}
				{
					\draw[edge] (1\z) -- (2\znext);
				}
				
				\foreach \z [evaluate={\znext=int({1+\z})}] in {1,3,6}
				{
					\draw[edge] (2\z) -- (3\znext);
				}
				\foreach \z [evaluate={\znext=int({\z-1})}] in {2,4,7}
				{
					\draw[edge] (2\z) -- (3\znext);
				}
				\coordinate (c15) at (15);
				\coordinate (c16) at (16);
				
				\draw[edge] (28) -- (39);
				\draw[edge] (29) -- (38);

				\draw[dotted, thick,shorten >=4pt, shorten <=4pt] (10,5) -- (12,5);
				\draw[dotted, thick,shorten >=4pt, shorten <=4pt] (10,3) -- (12,3);
				
				\draw[decoration={brace,raise=5pt,mirror},decorate]
				(2,1) -- node[below=6pt] {\small $1$}(6,1);
				\draw[decoration={brace,raise=5pt,mirror},decorate]
				(6,1) -- node[below=6pt] {\small $2$}(10,1);
				\draw[decoration={brace,raise=5pt,mirror},decorate]
				(12,1) -- node[below=6pt] {\small $n$}(16,1);
				
				\foreach \x in {1,2,...,5}
				{
					\node[label={[label distance=-3pt]90:\scriptsize{$g_{1\x}$}}] at (1\x) {};
					\node[label={[label distance=-3pt]-90:\scriptsize{$g_{3\x}$}}] at (3\x) {};
				}
				\node[label={[label distance=-3pt]180:\scriptsize{$g_{21}$}}] at (21) {};
				\node[label={[label distance=-3pt]90:\scriptsize{$\cdots$}}] at (16) {};
				\node[label={[label distance=-3pt]-90:\scriptsize{$\cdots$}}] at (36) {};
				
				\node[label={[label distance=-3pt]90:\scriptsize{$g_{1\,{2n+2}}$}}] at (19) {};
				\node[label={[label distance=-3pt]0:\scriptsize{$g_{2\,{2n+2}}$}}] at (29) {};
				\node[label={[label distance=-3pt]0:\scriptsize{$g_{3\,{2n+2}}$}}] at (39) {};
				
				\node[label={[label distance=-8pt]135:\scriptsize{$g_{22}$}}] at (22) {};
				\node[label={[label distance=-8pt]-135:\scriptsize{$g_{23}$}}] at (23) {};
				\node[label={[label distance=-8pt]135:\scriptsize{$g_{24}$}}] at (24) {};
				\node[label={[label distance=-8pt]-135:\scriptsize{$g_{25}$}}] at (25) {};
				
				% \node [] at (8,-2) {(b) $\widetilde{\Gamma}_{n,4}$ for $n\geq 1$};
			\end{scope}
			
			% \begin{scope}[shift = {(-4,0)}]
				%             \foreach \x in {1,2,3}
				%             {\node[vertex] (x\x1) at (0,2*(4-\x) {};}
				
				%             \foreach \x [evaluate = \x as \y using int(\x+1)] in {1,2}
				%             {\draw[edge] (x\x1) -- (x\y1) ;}
				
				%             \node[vertex] (e1) at (-2,6) {} ;
				%             \node[vertex] (e2) at (-2,4) {} ;
				%             \node[vertex] (e3) at (-2,2) {} ;

				%             \draw[edge] (e1) -- (e2) ;
				%             \draw[edge] (e2) -- (e3) ;
				%             \draw[edge] (e1) -- (x11) ;
				%             \draw[edge] (e2) -- (x21) ;
				%             \draw[edge] (e3) -- (x31) ;
				%             \draw[edge] (e1) -- (x21) ;
				%             \draw[edge] (e2) -- (x11) ;
				
				%             \foreach \y in {1,2,3}
				%             {
					%             \node[label={[label distance=-3pt]0:\scriptsize{$g_{\y 2}$}}] at (x\y1) {};
					%             }
				
				%             \node[label={[label distance=-3pt]180:\scriptsize{$g_{11}$}}] at (e1) {};
				%             \node[label={[label distance=-3pt]180:\scriptsize{$g_{21}$}}] at (e2) {};
				%             \node[label={[label distance=-3pt]180:\scriptsize{$g_{31}$}}] at (e3) {};
				%         \node [] at (-1,-2) {(a) $\widetilde{\Gamma}_{0,4}$};
				%         \end{scope}
		\end{tikzpicture}
		\caption{$\widetilde{\Gamma}_{n,4}$} \label{fig:tildeGamman4}	
	\end{subfigure}
	\caption{}
\end{figure}

%Graph and figure An4

\noindent (iii) {\textbf{Graph} $A_{n,4}$}

For $n\geq 0$, we define the graph $A_{n,4}$ as follows:

$V(A_{n,4}) =  V(\Gamma_{n,4})\cup\{a_1\}$ and $
E(A_{n,4}) =  E(\Gamma_{n,4})\cup \{a_1g_{11}, a_1g_{21}\}.$

%\begin{equation*}
%\begin{split}
%  V(A_{n,4}) = & V(\Gamma_{n,4})\cup\{a_1\},\\
%  E(A_{n,4}) = & E(\Gamma_{n,4})\cup \{a_1g_{11}, a_1g_{21}\}.    
% \end{split}
%\end{equation*}

\medskip

%\begin{equation*}
%\begin{split}
%     V(\widetilde{A}_{n,4}) = & V(\widetilde{\Gamma}_{n,4})\cup\{a_1\},\\
%      E(\widetilde{A}_{n,4}) = & E(\widetilde{\Gamma}_{n,4})\cup \{a_1g_{11},a_1g_{21} \}.     
%   \end{split}
%\end{equation*}

%figure An4

\begin{figure}[h]
\begin{subfigure}[]{0.4\textwidth}
	\centering
	\begin{tikzpicture}[scale=0.3]
		\begin{scope}
			\foreach \x in {1,2,...,8}
			{
				\node[vertex] (1\x) at (2*\x,3*2) {};
			}
			\foreach \x in {1,2,...,8}
			{
				\node[vertex] (2\x) at (2*\x,2*2) {};
			}
			\foreach \x in {1,2,...,8}
			{
				\node[vertex] (3\x) at (2*\x,1*2) {};
			}

			\foreach \y [evaluate={\y as \x using int({\y+1})}] in {1,2,...,4,6,7} 
			{
				\draw[edge] (1\y) -- (1\x);
				\draw[edge] (2\y) -- (2\x);
				\draw[edge] (3\y) -- (3\x);
			}
			
			\foreach \z in {1,2,3,4,5,...,8}
			{\draw[edge] (1\z) -- (2\z);
				\draw[edge] (2\z) -- (3\z);
			}
			
			\foreach \z [evaluate={\znext=int({1+\z})}] in {2,4,7}
			{
				\draw[edge] (1\z) -- (2\znext);
			}
			\foreach \z [evaluate={\znext=int({\z-1})}] in {3,5,8}
			{
				\draw[edge] (1\z) -- (2\znext);
			}
			
			\foreach \z [evaluate={\znext=int({1+\z})}] in {1,3,6}
			{
				\draw[edge] (2\z) -- (3\znext);
			}
			\foreach \z [evaluate={\znext=int({\z-1})}] in {2,4,7}
			{
				\draw[edge] (2\z) -- (3\znext);
			}
			\coordinate (c15) at (15);
			\coordinate (c16) at (16);

			\draw[dotted, thick,shorten >=4pt, shorten <=4pt] (10,5) -- (12,5);
			\draw[dotted, thick,shorten >=4pt, shorten <=4pt] (10,3) -- (12,3);
			
			\draw[decoration={brace,raise=5pt,mirror},decorate]
			(2,1) -- node[below=6pt] {\small $1$}(6,1);
			\draw[decoration={brace,raise=5pt,mirror},decorate]
			(6,1) -- node[below=6pt] {\small $2$}(10,1);
			\draw[decoration={brace,raise=5pt,mirror},decorate]
			(12,1) -- node[below=6pt] {\small $n$}(16,1);
			
			\foreach \x in {1,2,...,5}
			{
				\node[label={[label distance=-3pt]90:\scriptsize{$g_{1\x}$}}] at (1\x) {};
				\node[label={[label distance=-3pt]-90:\scriptsize{$g_{3\x}$}}] at (3\x) {};
			}
           
			\node[label={[label distance=-3pt]180:\scriptsize{$g_{21}$}}] at (21) {};
			\node[label={[label distance=-3pt]90:\scriptsize{$\cdots$}}] at (16) {};
			\node[label={[label distance=-3pt]-90:\scriptsize{$\cdots$}}] at (36) {};
			
			\node[label={[label distance=-3pt]90:\scriptsize{$g_{1\,{2n+1}}$}}] at (18) {};
			\node[label={[label distance=-3pt]0:\scriptsize{$g_{2\,{2n+1}}$}}] at (28) {};
			\node[label={[label distance=-3pt]0:\scriptsize{$g_{3\,{2n+1}}$}}] at (38) {};
			
			\node[label={[label distance=-8pt]135:\scriptsize{$g_{22}$}}] at (22) {};
			\node[label={[label distance=-8pt]-135:\scriptsize{$g_{23}$}}] at (23) {};
			\node[label={[label distance=-8pt]135:\scriptsize{$g_{24}$}}] at (24) {};
			\node[label={[label distance=-8pt]-135:\scriptsize{$g_{25}$}}] at (25) {};
			% \node[] at (8,-2) {(b) $A_{n,4}$ for $n\geq 1$};
		\end{scope}

		\begin{scope}[shift={(2,0)}]
			\node[vertex] (a) at (-2,6) {} ;
			
			\draw[edge] (a) -- (11) ;
			\draw[edge] (a) -- (21) ;
			
			 \node[label={[label distance=-3pt]180:\scriptsize{$a_1$}}] at (a) {};   
		\end{scope}
		
		% \begin{scope}[shift = {(-5,0)}]
			%             \foreach \x in {1,2,3}
			%             {\node[vertex] (x\x1) at (0,2*(4-\x) {};}
			
			%             \foreach \x [evaluate = \x as \y using int(\x+1)] in {1,2}
			%             {\draw[edge] (x\x1) -- (x\y1) ;}
			
			%             \node[vertex] (a) at (-2,6) {} ;
			
			%             \draw[edge] (a) -- (x11) ;
			%             \draw[edge] (a) -- (x21) ;

			%             \node[label={[label distance=-3pt]180:\scriptsize{$a_{1}$}}] at (a) {};   
			%             \foreach \y in {1,2,3}
			%             {
				%             \node[label={[label distance=-3pt]0:\scriptsize{$g_{\y 1}$}}] at (x\y1) {};
				%             }
			%             \node[] at (-1,-2) {(a) $A_{0,4}$};
			%         \end{scope}
		
	\end{tikzpicture}
	\caption{$A_{n,4}$} \label{fig:An4}	
\end{subfigure}
\begin{subfigure}[]{0.4\textwidth}
	\centering
	\begin{tikzpicture}[scale=0.3]
		\begin{scope}
			\foreach \x in {1,2,...,8,9}
			{
				\node[vertex] (1\x) at (2*\x,3*2) {};
			}
			\foreach \x in {1,2,...,8,9}
			{
				\node[vertex] (2\x) at (2*\x,2*2) {};
			}
			\foreach \x in {1,2,...,8,9}
			{
				\node[vertex] (3\x) at (2*\x,1*2) {};
			}

			\foreach \y [evaluate={\y as \x using int({\y+1})}] in {1,2,...,4,6,7} 
			{
				\draw[edge] (1\y) -- (1\x);
				\draw[edge] (2\y) -- (2\x);
				\draw[edge] (3\y) -- (3\x);
			}
			
			\foreach \z in {1,2,3,4,5,...,8}
			{\draw[edge] (1\z) -- (2\z);
				\draw[edge] (2\z) -- (3\z);
			}
			
			\foreach \z [evaluate={\znext=int({1+\z})}] in {2,4,7}
			{
				\draw[edge] (1\z) -- (2\znext);
			}
			\foreach \z [evaluate={\znext=int({\z-1})}] in {3,5,8}
			{
				\draw[edge] (1\z) -- (2\znext);
			}
			
			\foreach \z [evaluate={\znext=int({1+\z})}] in {1,3,6}
			{
				\draw[edge] (2\z) -- (3\znext);
			}
			\foreach \z [evaluate={\znext=int({\z-1})}] in {2,4,7}
			{
				\draw[edge] (2\z) -- (3\znext);
			}
			\coordinate (c15) at (15);
			\coordinate (c16) at (16);

			%Additional right part
			
			\draw[edge] (19) -- (29);
			\draw[edge] (29) -- (39);
			\draw[edge] (19) -- (18);
			\draw[edge] (29) -- (28);
			\draw[edge] (29) -- (38);
			\draw[edge] (39) -- (38);
			\draw[edge] (39) -- (28);

			\draw[dotted, thick,shorten >=4pt, shorten <=4pt] (10,5) -- (12,5);
			\draw[dotted, thick,shorten >=4pt, shorten <=4pt] (10,3) -- (12,3);
			
			\draw[decoration={brace,raise=5pt,mirror},decorate]
			(2,1) -- node[below=6pt] {\small $1$}(6,1);
			\draw[decoration={brace,raise=5pt,mirror},decorate]
			(6,1) -- node[below=6pt] {\small $2$}(10,1);
			\draw[decoration={brace,raise=5pt,mirror},decorate]
			(12,1) -- node[below=6pt] {\small $n$}(16,1);
			
			\foreach \x in {1,2,...,5}
			{
				\node[label={[label distance=-3pt]90:\scriptsize{$g_{1\x}$}}] at (1\x) {};
				\node[label={[label distance=-3pt]-90:\scriptsize{$g_{3\x}$}}] at (3\x) {};
			}
			\node[label={[label distance=-3pt]180:\scriptsize{$g_{21}$}}] at (21) {};
			\node[label={[label distance=-3pt]90:\scriptsize{$\cdots$}}] at (16) {};
			\node[label={[label distance=-3pt]-90:\scriptsize{$\cdots$}}] at (36) {};
			
			\node[label={[label distance=-3pt]90:\scriptsize{$g_{1\,{2n+2}}$}}] at (19) {};
			\node[label={[label distance=-3pt]0:\scriptsize{$g_{2\,{2n+2}}$}}] at (29) {};
			\node[label={[label distance=-3pt]0:\scriptsize{$g_{3\,{2n+2}}$}}] at (39) {};
			
			\node[label={[label distance=-8pt]135:\scriptsize{$g_{22}$}}] at (22) {};
			\node[label={[label distance=-8pt]-135:\scriptsize{$g_{23}$}}] at (23) {};
			\node[label={[label distance=-8pt]135:\scriptsize{$g_{24}$}}] at (24) {};
			\node[label={[label distance=-8pt]-135:\scriptsize{$g_{25}$}}] at (25) {};
			% \node[] at (8,-2) {(b) $\widetilde{A}_{n,4}$ for $n\geq 1$};
		\end{scope}

		\begin{scope}[shift={(2,0)}]
			\node[vertex] (a1) at (-2,6) {} ;
			
			\draw[edge] (a1) -- (11) ;
			\draw[edge] (a1) -- (21) ;
			
			\node[label={[label distance=-3pt]180:\scriptsize{$a_1$}}] at (a1) {};   
		\end{scope}
		
		% \begin{scope}[shift = {(-7,0)}]
			%             \foreach \x in {1,2,3}
			%             {\node[vertex] (x\x1) at (0,2*(4-\x) {};
				%             \node[vertex] (x\x2) at (2,2*(4-\x) {};}
			
			%             \foreach \x [evaluate = \x as \y using int(\x+1)] in {1,2}
			%             {\draw[edge] (x\x1) -- (x\y1) ;
				%             \draw[edge] (x\x2) -- (x\y2) ;}
			
			%             \foreach \x in {1,2,3}
			%             {
				%             \draw[edge] (x\x1) -- (x\x2) ;
				%             }
			
			%             \node[vertex] (a1) at (-2,6) {} ;
			
			%             \draw[edge] (a1) -- (x11) ;
			%             \draw[edge] (a1) -- (x21) ;
			%             \draw[edge] (x21) -- (x32) ;
			%             \draw[edge] (x31) -- (x22) ;

			%             \node[label={[label distance=-3pt]180:\scriptsize{$a_1$}}] at (a1) {};   
			%             \node[label={[label distance=-3pt]90:\scriptsize{$g_{11}$}}] at (x11) {};
			%             \node[label={[label distance=-3pt]180:\scriptsize{$g_{2 1}$}}] at (x21) {};
			%             \node[label={[label distance=-3pt]180:\scriptsize{$g_{31}$}}] at (x31) {};
			%             \foreach \y in {1,2,3}
			%             {
				%             \node[label={[label distance=-3pt]0:\scriptsize{$g_{\y 2}$}}] at (x\y2) {};
				%             }
			%             \node[] at (-1,-2) {(a) $\widetilde{A}_{0,4}$};
			%         \end{scope}
		
	\end{tikzpicture}
	\caption{$\widetilde{A}_{n,4}$} \label{fig:Bn4}	
\end{subfigure}
\caption{}
\end{figure}

\noindent (iv) {\textbf{Graph} $\widetilde{A}_{n,4}$}

For $n\geq 0$, we define the graph $\widetilde{A}_{n,4}$ as follows:

$ V(\widetilde{A}_{n,4}) =  V(\widetilde{\Gamma}_{n,4})\cup\{a_1\}$ and $
E(\widetilde{A}_{n,4}) =  E(\widetilde{\Gamma}_{n,4})\cup \{a_1g_{11},a_1g_{21} \}.$

\medskip

\noindent (v) \textbf{Graph }$B_{n,4}$

For $n\geq 0$, we define the graph $B_{n,4}$ as follows:

$V(B_{n,4}) =  V(\Gamma_{n,4})\cup\{b_1\}$ and 
$E(B_{n,4}) =  E(\Gamma_{n,4})\cup \{b_1g_{11}, b_1g_{21}, b_1g_{31}\}$.  

%\begin{equation*}
%  \begin{split}
% V(B_{n,4}) = & V(\Gamma_{n,4})\cup\{b_1\},\\
%E(B_{n,4}) = & E(\Gamma_{n,4})\cup \{b_1g_{11}, b_1g_{21}, b_1g_{31}\}.     
%  \end{split}
%\end{equation*}
\medskip

\noindent(vi) \textbf{Graph }$\widetilde{B}_{n,4}$

For $n\geq 0$, we define the graph $\widetilde{B}_{n,4}$ as follows:

$V(\widetilde{B}_{n,4}) =  V(\widetilde{\Gamma}_{n,4})\cup\{b_1\}$ and 
$E(\widetilde{B}_{n,4}) = E(\widetilde{\Gamma}_{n,4})\cup \{b_1g_{11},b_1g_{21},b_1g_{31}\}$.

%\begin{equation*}
%\begin{split}
%  V(\widetilde{B}_{n,4}) = & V(\widetilde{\Gamma}_{n,4})\cup\{b_1\},\\
%   E(\widetilde{B}_{n,4}) = & E(\widetilde{\Gamma}_{n,4})\cup \{b_1g_{11},b_1g_{21},b_1g_{31}\}.     
%\end{split}
%\end{equation*}

	\begin{figure}[ht]
		\begin{tikzpicture}[scale=0.3]
			\begin{scope}
				\foreach \x in {1,2,...,8}
				{
					\node[vertex] (1\x) at (3*\x,3*2) {};
				}
				\foreach \x in {1,2,...,8}
				{
					\node[vertex] (2\x) at (3*\x,2*2) {};
				}
				\foreach \x in {1,2,...,8}
				{
					\node[vertex] (3\x) at (3*\x,1*2) {};
				}

				\foreach \y [evaluate={\y as \x using int({\y+1})}] in {1,2,...,4,6,7} 
				{
					\draw[edge] (1\y) -- (1\x);
					\draw[edge] (2\y) -- (2\x);
					\draw[edge] (3\y) -- (3\x);
				}
				
				\foreach \z in {1,2,3,4,5,...,8}
				{\draw[edge] (1\z) -- (2\z);
					\draw[edge] (2\z) -- (3\z);
				}
				
				\foreach \z [evaluate={\znext=int({1+\z})}] in {2,4,7}
				{
					\draw[edge] (1\z) -- (2\znext);
				}
				\foreach \z [evaluate={\znext=int({\z-1})}] in {3,5,8}
				{
					\draw[edge] (1\z) -- (2\znext);
				}
				
				\foreach \z [evaluate={\znext=int({1+\z})}] in {1,3,6}
				{
					\draw[edge] (2\z) -- (3\znext);
				}
				\foreach \z [evaluate={\znext=int({\z-1})}] in {2,4,7}
				{
					\draw[edge] (2\z) -- (3\znext);
				}
				\coordinate (c15) at (15);
				\coordinate (c16) at (16);

				\draw[dotted, thick,shorten >=4pt, shorten <=4pt] (15.5,5) -- (17.5,5);
				\draw[dotted, thick,shorten >=4pt, shorten <=4pt] (15.5,3) -- (17.5,3);
				
					\draw[decoration={brace,raise=5pt,mirror},decorate]
				(3,1) -- node[below=6pt] {\small $1$}(9,1);
				\draw[decoration={brace,raise=5pt,mirror},decorate]
				(9,1) -- node[below=6pt] {\small $2$}(15,1);
				\draw[decoration={brace,raise=5pt,mirror},decorate]
				(18,1) -- node[below=6pt] {\small $n$}(24,1);
				\foreach \x in {1,2,...,5}
				{
					\node[label={[label distance=-3pt]90:\scriptsize{$g_{1\x}$}}] at (1\x) {};
					\node[label={[label distance=-3pt]-90:\scriptsize{$g_{3\x}$}}] at (3\x) {};
				}
				\node[label={[label distance=-8pt]-110:\scriptsize{$g_{21}$}}] at (21) {};
				\node[label={[label distance=-3pt]90:\scriptsize{$\cdots$}}] at (16) {};
				\node[label={[label distance=-3pt]-90:\scriptsize{$\cdots$}}] at (36) {};
				
				\node[label={[label distance=-3pt]90:\scriptsize{$g_{1\,{2n+1}}$}}] at (18) {};
				\node[label={[label distance=-3pt]0:\scriptsize{$g_{2\,{2n+1}}$}}] at (28) {};
				\node[label={[label distance=-3pt]0:\scriptsize{$g_{3\,{2n+1}}$}}] at (38) {};
				
				\node[label={[label distance=-8pt]135:\scriptsize{$g_{22}$}}] at (22) {};
				\node[label={[label distance=-8pt]-135:\scriptsize{$g_{23}$}}] at (23) {};
				\node[label={[label distance=-8pt]135:\scriptsize{$g_{24}$}}] at (24) {};
				\node[label={[label distance=-8pt]-135:\scriptsize{$g_{25}$}}] at (25) {};
				
			\end{scope}
			\begin{scope}[shift={(2,0)}]
				\node[vertex] (c1) at (-2,4) {} ;
				
				\draw[edge] (c1) -- (11) ;
				\draw[edge] (c1) -- (21) ;
				\draw[edge] (c1) -- (31) ;
				
				\node[label={[label distance=-3pt]180:\scriptsize{$b_1$}}] at (c1) {};   
				% \node[] at (6,-2) {(b) $B_{n,4}$ for $n\geq 1$};
			\end{scope}

		\end{tikzpicture}
		\caption{$B_{n,4}$} \label{fig:Cn4}	
	\end{figure}

	\begin{figure}[ht]
		\begin{tikzpicture}[scale=0.3]
			\begin{scope}
				\foreach \x in {1,2,...,8,9}
				{
					\node[vertex] (1\x) at (3*\x,3*2) {};
				}
				\foreach \x in {1,2,...,8,9}
				{
					\node[vertex] (2\x) at (3*\x,2*2) {};
				}
				\foreach \x in {1,2,...,8,9}
				{
					\node[vertex] (3\x) at (3*\x,1*2) {};
				}

				\foreach \y [evaluate={\y as \x using int({\y+1})}] in {1,2,...,4,6,7,8} 
				{
					\draw[edge] (1\y) -- (1\x);
					\draw[edge] (2\y) -- (2\x);
					\draw[edge] (3\y) -- (3\x);
				}
				
				\foreach \z in {1,2,3,4,5,...,8,9}
				{\draw[edge] (1\z) -- (2\z);
					\draw[edge] (2\z) -- (3\z);
				}
				
				\foreach \z [evaluate={\znext=int({1+\z})}] in {2,4,7}
				{
					\draw[edge] (1\z) -- (2\znext);
				}
				\foreach \z [evaluate={\znext=int({\z-1})}] in {3,5,8}
				{
					\draw[edge] (1\z) -- (2\znext);
				}
				
				\foreach \z [evaluate={\znext=int({1+\z})}] in {1,3,6,8}
				{
					\draw[edge] (2\z) -- (3\znext);
				}
				\foreach \z [evaluate={\znext=int({\z-1})}] in {2,4,7,9}
				{
					\draw[edge] (2\z) -- (3\znext);
				}
				\coordinate (c15) at (15);
				\coordinate (c16) at (16);

		\draw[dotted, thick,shorten >=4pt, shorten <=4pt] (15.5,5) -- (17.5,5);
	\draw[dotted, thick,shorten >=4pt, shorten <=4pt] (15.5,3) -- (17.5,3);

					\draw[decoration={brace,raise=5pt,mirror},decorate]
				(3,1) -- node[below=6pt] {\small $1$}(9,1);
				\draw[decoration={brace,raise=5pt,mirror},decorate]
				(9,1) -- node[below=6pt] {\small $2$}(15,1);
				\draw[decoration={brace,raise=5pt,mirror},decorate]
				(18,1) -- node[below=6pt] {\small $n$}(24,1);
				
				\foreach \x in {1,2,...,5}
				{
					\node[label={[label distance=-3pt]90:\scriptsize{$g_{1\x}$}}] at (1\x) {};
					\node[label={[label distance=-3pt]-90:\scriptsize{$g_{3\x}$}}] at (3\x) {};
				}
				\node[label={[label distance=-8pt]-110:\scriptsize{$g_{21}$}}] at (21) {};
				\node[label={[label distance=-3pt]90:\scriptsize{$\cdots$}}] at (16) {};
				\node[label={[label distance=-3pt]-90:\scriptsize{$\cdots$}}] at (36) {};
				
				\node[label={[label distance=-3pt]90:\scriptsize{$g_{1\,{2n+2}}$}}] at (19) {};
				\node[label={[label distance=-3pt]0:\scriptsize{$g_{2\,{2n+2}}$}}] at (29) {};
				\node[label={[label distance=-3pt]0:\scriptsize{$g_{3\,{2n+2}}$}}] at (39) {};
				
				\node[label={[label distance=-8pt]135:\scriptsize{$g_{22}$}}] at (22) {};
				\node[label={[label distance=-8pt]-135:\scriptsize{$g_{23}$}}] at (23) {};
				\node[label={[label distance=-8pt]135:\scriptsize{$g_{24}$}}] at (24) {};
				\node[label={[label distance=-8pt]-135:\scriptsize{$g_{25}$}}] at (25) {};
				
			\end{scope}
			\begin{scope}[shift={(2,0)}]
				\node[vertex] (d1) at (-2,4) {} ;
				
				\draw[edge] (d1) -- (11) ;
				\draw[edge] (d1) -- (21) ;
				\draw[edge] (d1) -- (31) ;
				
				\node[label={[label distance=-3pt]180:\scriptsize{$b_1$}}] at (d1) {};   
				% \node[] at (6,-2) {\small (b) $\widetilde{B}_{n,4}$ for $n\geq 1$};
			\end{scope}

		\end{tikzpicture}
		\caption{$\widetilde{B}_{n,4}$} \label{fig:Dn4}	
	\end{figure}

\subsection{\texorpdfstring{Independence complexes of  $\Gamma_{n,4}$ and $\widetilde{\Gamma}_{n,4}$}{Independence complexes of Γn,4 and Λn,5}}{\label{subsection:Gamma4 and Lembda4}}

In this section,  we  establish recursive relations among the independence complexes of the graphs in 
$\{\Gamma_{n,4},\widetilde{\Gamma}_{n,4},A_{n,4}, \widetilde{A}_{n,4}, B_{n,4},\widetilde{B}_{n,4}\}$.

\begin{claim}\label{Gamman4} For $n\geq 0$,
	
	\begin{equation*}
		\ind(\Gamma_{n,4}) \simeq \begin{cases}
			\mathbb{S}^0 & \text{if $n =0$,} \\
			\mathbb{S}^1\vee \mathbb{S}^1  & \text{if $n =1$,} \\
			\Sigma \ind(A_{n-1,4}) \vee \Sigma^{2} \ind(\widetilde{\Gamma}_{n-2,4})
			& \text{if $n \geq 2$.}
		\end{cases}
	\end{equation*}
	
\end{claim}

\begin{proof}
	Since $\Gamma_{0,4}$ is a path on three vertices (see \Cref{fig:Gn4}), from \Cref{Ind(path)}, we get $\ind(\Gamma_{0,4})\simeq \mathbb{S}^0$. Now, consider $n\geq 1$. Since $N(g_{11}) \subseteq N(g_{22})$ in the graph $\Gamma_{n,4}$, we have $\ind(\Gamma_{n,4}) \simeq \ind(\Gamma_{n,4} \setminus g_{22})$. Let $\Gamma_{n,4}^1 \coloneq\Gamma_{n,4} \setminus g_{22}$. Since $g_{31}$ is a simplicial vertex in $\Gamma_{n,4}^1$ with $N(g_{31})=\{g_{21},g_{32}\}$,  \Cref{Simplicial Vertex Lemma} implies that $ \ind(\Gamma_{n,4}^1) \simeq \Sigma \ind(\Gamma_{n,4}^1 \setminus \n[g_{21}]) \vee \Sigma \ind(\Gamma_{n,4}^1 \setminus \n[g_{32}])$. Observe that the graph $\Gamma_{n,4}^1 \setminus \n[g_{21}]$ is isomorphic to $A_{n-1,4}$ (see \Cref{fig:An4}). Therefore, $\ind(\Gamma_{n,4}^1)\simeq \Sigma \ind(A_{n-1,4}) \vee \Sigma \ind(\Gamma_{n,4}^1 \setminus \n[g_{32}])$. In case of $n=1$, the graph $\Gamma_{1,4}^1 \setminus \n[g_{32}]$ is isomorphic to $A_{0,4}$ (see \Cref{fig:An4}). Therefore, $\ind(\Gamma_{1,4}^1)\simeq \Sigma \ind(A_{0,4}) \vee \Sigma \ind(A_{0,4})$. Using \Cref{An in pnp4}, we get $\ind(\Gamma_{1,4})\simeq \mathbb{S}^1\vee \mathbb{S}^1$.  
	
	Now, consider $n\geq 2$. Since $g_{11}$ is a simplicial vertex and $g_{12}$ is the only neighbor of $g_{11}$ in the graph $\Gamma_{n,4}^1 \setminus \n[g_{32}]$, we get  $ \ind(\Gamma_{n,4}^1 \setminus \n[g_{32}])\simeq \Sigma\ind(\Gamma_{n,4}^1 \setminus \n[g_{32},g_{12}])$. Observe that  the graph $\Gamma_{n,4}^1 \setminus \n[g_{32},g_{12}]$ is isomorphic to $\widetilde{\Gamma}_{n-2,4}$ ( see \Cref{fig:tildeGamman4}) and therefore  $\ind(\Gamma_{n,4}^1 \setminus \n[g_{32}]) \simeq \Sigma\ind(\widetilde{\Gamma}_{n-2,4})$. Thus, $\ind(\Gamma_{n,4}) \simeq \Sigma \ind(A_{n-1,4}) \vee \Sigma^{2} \ind(\widetilde{\Gamma}_{n-2,4})$.
\end{proof}

\begin{claim}\label{Gamma'n4} For $n\geq 0$,
	\begin{equation*}
		\ind(\widetilde{\Gamma}_{n,4}) \simeq \begin{cases}
			\mathbb{S}^1 & \text{if $n =0$,} \\
            \vee^3\mathbb{S}^2 & \text{if $n =1$,} \\
			\Sigma \ind(\widetilde{A}_{n-1,4}) \vee \Sigma^{2}\ind(\Gamma_{n-2,4})
			& \text{if $n \geq 2$.}
		\end{cases}
	\end{equation*}
\end{claim}
\begin{proof}
	Since $N(g_{11})\subseteq N(g_{22})$ in the graph $\widetilde{\Gamma}_{n,4}$ (see \Cref{fig:tildeGamman4}), $\ind(\widetilde{\Gamma}_{n,4}) \simeq \ind(\widetilde{\Gamma}_{n,4} \setminus g_{22})$. If  $n=0$, then  $N(g_{12})\subseteq N(g_{21})$ in $\widetilde{\Gamma}_{0,4} \setminus g_{22}$. Hence $\ind(\widetilde{\Gamma}_{0,4}) \simeq \ind(\widetilde{\Gamma}_{0,4} \setminus \{g_{21},g_{22}\})$. Now, $\widetilde{\Gamma}_{0,4} \setminus \{g_{21},g_{22}\}$ is the disjoint union of two edges $g_{11}g_{12}$ and $g_{31}g_{32}$. Therefore, $\ind(\widetilde{\Gamma}_{0,4}) \simeq \mathbb{S}^1$.

If $n=1$, then $N(g_{14})\subseteq N(g_{23})$ in the graph $\widetilde{\Gamma}_{1,4}\setminus g_{22}$. Thus, $\ind(\widetilde{\Gamma}_{1,4}) \simeq \ind(\widetilde{\Gamma}_{1,4} \setminus \{g_{22},g_{23}\})$. Let $\widetilde{\Gamma}_{1,4}^1:=\widetilde{\Gamma}_{1,4} \setminus \{g_{22},g_{23}\}$. Then $g_{31}$ is a simplicial vertex in $\widetilde{\Gamma}_{1,4}^1$ with $N(g_{31})=\{g_{21}, g_{32}\}$. From \Cref{Simplicial Vertex Lemma}, $\ind(\widetilde{\Gamma}_{1,4}^1)\simeq \Sigma \ind(\widetilde{\Gamma}_{1,4}^1\setminus \n[g_{21}])\vee \Sigma \ind(\widetilde{\Gamma}_{1,4}^1\setminus \n[g_{32}])$. Since the graph $\widetilde{\Gamma}_{1,4}^1\setminus \n[g_{32}]$ is a path on $6$ vertices, $\ind(\widetilde{\Gamma}_{1,4}^1\setminus \n[g_{32}])\simeq \Sp^1$. In the graph 
$\widetilde{\Gamma}_{1,4}^1\setminus \n[g_{21}]$, we have $N(g_{12})\subseteq N(g_{14})$. Thus from \Cref{Folding Lemma}, $\ind(\widetilde{\Gamma}_{1,4}^1\setminus \n[g_{21}])\simeq \ind((\widetilde{\Gamma}_{1,4}^1\setminus \n[g_{21}])\cup\{g_{14}\})$. Observe that $(\widetilde{\Gamma}_{1,4}^1\setminus \n[g_{21}])\cup\{g_{14}\}$ is disjoint union of an edge $\{g_{12}g_{13}\}$ and a cycle of length $3$. This implies that $\ind((\widetilde{\Gamma}_{1,4}^1\setminus \n[g_{21}])\cup\{g_{14}\})\simeq \Sp^1 \vee \Sp^1$. Therefore, $\ind(\widetilde{\Gamma}_{1,4}^1)\simeq \bigvee^{3} \Sp^2$. This implies that $\ind(\widetilde{\Gamma}_{1,4})\simeq \vee^{3} \Sp^2$.
   
	Now, let  $n\geq 2$. Following similar reduction to the \Cref{Gamman4}, we get 
	%\begin{equation*}  
	% \begin{split}
		$ \ind(\widetilde{\Gamma}_{n,4}) \simeq \Sigma \ind(\widetilde{A}_{n-1,4}) \vee \Sigma^{2}\ind(\Gamma_{n-2,4}).$
		%\end{split}        
		% \end{equation*}
	%This completes the proof.
\end{proof}

\begin{claim} For $n \geq 0$, \label{An in pnp4} 
\begin{equation*}
	\ind(A_{n,4}) \simeq \begin{cases}
		\mathbb{S}^0 & \text{if $n =0$,} \\
		\mathbb{S}^1\vee \mathbb{S}^2 & \text{if $n =1$,} \\
		\Sigma \ind(A_{n-1,4}) \vee \Sigma^{2}\ind(\widetilde{A}_{n-2,4})\vee \Sigma^{3}\ind (B_{n-2,4})
		& \text{if $n \geq 2$.}
	\end{cases}
\end{equation*}
\end{claim}
\begin{proof}

\begin{figure}[h!]
	\centering
	\begin{subfigure}[b]{0.3\textwidth}
		\begin{tikzpicture}[scale=0.3]
			\begin{scope}
				\foreach \x in {3,4,...,5,6}
				{
					\node[vertex] (1\x) at (2*\x,3*2) {};
				}
				\foreach \x in {2,...,6}
				{
					\node[vertex] (2\x) at (2*\x,2*2) {};
				}
				\foreach \x in {1,2,...,6}
				{
					\node[vertex] (3\x) at (2*\x,1*2) {};
				}

				\foreach \y [evaluate={\y as \x using int({\y+1})}] in {3,5} 
				{
					\draw[edge] (1\y) -- (1\x);
					
				}
				\foreach \y [evaluate={\y as \x using int({\y+1})}] in {2,3,5} 
				{
					\draw[edge] (2\y) -- (2\x);
				}
				\foreach \y [evaluate={\y as \x using int({\y+1})}] in {1,2,3,5} 
				{
					\draw[edge] (3\y) -- (3\x);
				}

				\foreach \z in {3,4,5,6}
				{\draw[edge] (1\z) -- (2\z);
					\draw[edge] (2\z) -- (3\z);
				}
				\draw[edge] (22) -- (32);
				\draw[edge] (15) -- (26);
				
				\foreach \z [evaluate={\znext=int({\z-1})}] in {3,6}
				{
					\draw[edge] (1\z) -- (2\znext);
				}
				
				\foreach \z [evaluate={\znext=int({1+\z})}] in {3}
				{
					\draw[edge] (2\z) -- (3\znext);
				}
				\foreach \z [evaluate={\znext=int({\z-1})}] in {2,4}
				{
					\draw[edge] (2\z) -- (3\znext);
				}
				\coordinate (c15) at (15);
				\coordinate (c16) at (16);

				\draw[dotted, thick,shorten >=4pt, shorten <=4pt] (8,5) -- (10,5);
				\draw[dotted, thick,shorten >=4pt, shorten <=4pt] (8,3) -- (10,3);

				\node[label={[label distance=-3pt]90:\scriptsize{$g_{22}$}}] at (22) {};
				\foreach \x in {3,4}
				{
					\node[label={[label distance=-3pt]90:\scriptsize{$g_{1\x}$}}] at (1\x) {};
				}
				
				\node[label={[label distance=-3pt]90:\scriptsize{$g_{1\,{2n+1}}$}}] at (16) {};

				\node[label={[label distance=-3pt]270:\scriptsize{$g_{3\,{2n+1}}$}}] at (36) {};
				
				\node[label={[label distance=-3pt]270:\scriptsize{$g_{32}$}}] at (32) {};
				\node[label={[label distance=-3pt]180:\scriptsize{$g_{31}$}}] at (31) {};
			\end{scope}
		\end{tikzpicture}
		\caption{$A_{n,4}^1$}
		\label{subfig:An41}  
	\end{subfigure}      
	\begin{subfigure}[b]{0.3\textwidth}
		\begin{tikzpicture}[scale=0.3]    
			\begin{scope}
				\foreach \x in {3,...,5,6,7}
				{
					\node[vertex] (1\x) at (2*\x,3*2) {};
				}
				\foreach \x in {5,...,7}
				{
					\node[vertex] (2\x) at (2*\x,2*2) {};
				}
				\node[vertex] (23) at (6,2*2) {};
				\foreach \x in {4,...,7}
				{
					\node[vertex] (3\x) at (2*\x,1*2) {};
				}

				\foreach \y [evaluate={\y as \x using int({\y+1})}] in {3,...,4,6} 
				{
					\draw[edge] (1\y) -- (1\x);
					
				}
				\foreach \y [evaluate={\y as \x using int({\y+1})}] in {6} 
				{
					\draw[edge] (2\y) -- (2\x);
				}
				\foreach \y [evaluate={\y as \x using int({\y+1})}] in {4,6} 
				{
					\draw[edge] (3\y) -- (3\x);
				}

				\foreach \z in {5,...,7}
				{\draw[edge] (1\z) -- (2\z);
					\draw[edge] (2\z) -- (3\z);
				}
				\draw[edge] (13) -- (23);
				\foreach \z [evaluate={\znext=int({1+\z})}] in {4,6}
				{
					\draw[edge] (1\z) -- (2\znext);
				}
				
				\draw[edge] (17) -- (26);
				\draw[edge] (23) -- (34);

				\coordinate (c15) at (15);
				\coordinate (c16) at (16);

				\draw[dotted, thick,shorten >=4pt, shorten <=4pt] (10,5) -- (12,5);
				\draw[dotted, thick,shorten >=4pt, shorten <=4pt] (10,3) -- (12,3);

				\node[label={[label distance=-3pt]180:\scriptsize{$g_{23}$}}] at (23) {};
				\foreach \x in {3,...,5}
				{
					\node[label={[label distance=-3pt]90:\scriptsize{$g_{1\x}$}}] at (1\x) {};
				}
				
				\node[label={[label distance=-3pt]90:\scriptsize{$g_{1\,{2n+1}}$}}] at (17) {};

				\node[label={[label distance=-3pt]270:\scriptsize{$g_{3\,{2n+1}}$}}] at (37) {};
				
				\node[label={[label distance=-3pt]270:\scriptsize{$g_{34}$}}] at (34) {};
				\node[label={[label distance=-3pt]180:\scriptsize{$g_{25}$}}] at (25) {};
				
			\end{scope}
		\end{tikzpicture}
		\caption{$A_{n,4}^2$}
		\label{subfig:An42}
	\end{subfigure}
	\begin{subfigure}[b]{0.3\textwidth}
		\begin{tikzpicture}[scale=0.3]    
			\begin{scope}
				\foreach \x in {3,...,5,6,7}
				{
					\node[vertex] (1\x) at (2*\x,3*2) {};
				}
				\foreach \x in {5,...,7}
				{
					\node[vertex] (2\x) at (2*\x,2*2) {};
				}
				\node[vertex] (23) at (6,2*2) {};
				\foreach \x in {4,...,7}
				{
					\node[vertex] (3\x) at (2*\x,1*2) {};
				}

				\foreach \y [evaluate={\y as \x using int({\y+1})}] in {3,...,4,6} 
				{
					\draw[edge] (1\y) -- (1\x);
					
				}
				\foreach \y [evaluate={\y as \x using int({\y+1})}] in {6} 
				{
					\draw[edge] (2\y) -- (2\x);
				}
				\foreach \y [evaluate={\y as \x using int({\y+1})}] in {4,6} 
				{
					\draw[edge] (3\y) -- (3\x);
				}

				\foreach \z in {5,...,7}
				{\draw[edge] (1\z) -- (2\z);
					\draw[edge] (2\z) -- (3\z);
				}
				\draw[edge] (13) -- (23);
				\foreach \z [evaluate={\znext=int({1+\z})}] in {4,6}
				{
					\draw[edge] (1\z) -- (2\znext);
				}
				
				\draw[edge] (17) -- (26);
				\draw[edge] (23) -- (34);
				\draw[edge] (34) -- (15);
				\draw[edge] (34) -- (25);
				\coordinate (c15) at (15);
				\coordinate (c16) at (16);

				\draw[dotted, thick,shorten >=4pt, shorten <=4pt] (10,5) -- (12,5);
				\draw[dotted, thick,shorten >=4pt, shorten <=4pt] (10,3) -- (12,3);

				\node[label={[label distance=-3pt]180:\scriptsize{$g_{23}$}}] at (23) {};
				\foreach \x in {3,...,5}
				{
					\node[label={[label distance=-3pt]90:\scriptsize{$g_{1\x}$}}] at (1\x) {};
				}
				
				\node[label={[label distance=-3pt]90:\scriptsize{$g_{1\,{2n+1}}$}}] at (17) {};

				\node[label={[label distance=-3pt]270:\scriptsize{$g_{3\,{2n+1}}$}}] at (37) {};
				
				\node[label={[label distance=-3pt]270:\scriptsize{$g_{34}$}}] at (34) {};
				\node[label={[label distance=-3pt]180:\scriptsize{$g_{25}$}}] at (25) {};
				
			\end{scope}
		\end{tikzpicture}
		\caption{$A_{n,4}^3$}
		\label{subfig:An43}
	\end{subfigure}
	\caption{}
\end{figure} 
Since $\n(g_{31})\subseteq \n(a_1)\cap \n(g_{11})$ in the graph $A_{0,4}$ (see \Cref{fig:An4}), using \Cref{Folding Lemma}, we get  $\ind(A_{0,4})\simeq  \ind(A_{0,4}\setminus \{a_1,g_{11}\})\simeq \mathbb{S}^0$. Let $n\geq 1$.  Since $a_1$ is a simplicial vertex in $A_{n,4}$ with $\n(a_1)=\{g_{11},g_{21}\}$, from \Cref{Simplicial Vertex Lemma}, we get
\begin{align} \label{eq:An41}
	\ind(A_{n,4}) & \simeq \Sigma\ind(A_{n,4}\setminus \n[g_{11}])\vee \Sigma \ind(A_{n,4}\setminus \n[g_{21}]) \nonumber \\
	& \simeq  \Sigma \ind(A_{n,4}\setminus \n[g_{11}]) \vee  \Sigma \ind (A_{n-1,4}), 
\end{align} 
where the second homotopy equivalence follows from the observation that $A_{n,4}\setminus \n[g_{21}]$ is isomorphic to  $A_{n-1,4}$. Let $A_{n,4}^1\coloneq A_{n,4}\setminus \n[g_{11}]$ (see \Cref{subfig:An41}). Since $g_{31}$ is a simplicial vertex in the graph $A_{n,4}^1$ with $N(g_{31})=\{g_{22},g_{32}\}$, we have 
\begin{equation}\label{eq:An42}
	\ind(A_{n,4}^1)\simeq  \Sigma\ind(A_{n,4}^1\setminus \n[g_{22}])\vee \Sigma \ind(A_{n,4}^1\setminus \n[g_{32}]).
\end{equation}
In particular, for $n=1$, the graph $A_{1,4}^1\setminus \n[g_{22}]$ consists of a single vertex $g_{33}$, and $A_{1,4}^1\setminus \n[g_{32}]$ is a graph consisting of exactly one edge $g_{13}g_{23}$. Therefore, $\ind(A_{1,4}^1)\simeq \mathbb{S}^1 $, and it follows that $\ind(A_{1,4})\simeq\mathbb{S}^1\vee \mathbb{S}^2$. 

Now, consider $n\geq 2$. Observe that the graph $A_{n,4}^1\setminus \n[g_{22}]$ is isomorphic to $\widetilde{A}_{n-2,4}$ (see \Cref{fig:Bn4}). Therefore, $\ind(A_{n,4}^1)\simeq \Sigma \ind (\widetilde{A}_{n-2,4}) \vee \Sigma \ind(A_{n,4}^1\setminus \n[g_{32}])$.

We now compute $\ind(A_{n,4}^1\setminus \n[g_{32}])$. Since $N(g_{13})\subseteq N(g_{24})$ in the graph $A_{n,4}^1\setminus \n[g_{32}]$ (see \Cref{subfig:An41}), $\ind(A_{n,4}^1\setminus \n[g_{32}])\simeq \ind((A_{n,4}^1\setminus \n[g_{32}])\setminus g_{24}) $. Let $A^{2}_{n,4}:=(A_{n,4}^1\setminus \n[g_{32}])\setminus g_{24}$. Then $\ind(A_{n,4}^1\setminus \n[g_{32}])\simeq \ind(A^{2}_{n,4})$. Observe that $[g_{34},g_{25};g_{13}]$ and $[g_{34},g_{15};g_{13}]$ are edge-invariant triplets in the graphs $A^{2}_{n,4}$ (see \Cref{subfig:An42}) and $A^{2}_{n,4} + \{g_{34}g_{25}\}$, respectively. Therefore, $\ind(A^{2}_{n,4}) \simeq \ind(A^{2}_{n,4} +\{g_{34}g_{25}, g_{34}g_{15} \})$. Let the resultant graph that we obtain be $A^{3}_{n,4}$ (see \Cref{subfig:An43}). Then $\ind(A^{2}_{n,4})\simeq\ind(A^{3}_{n,4})$.
Since $[g_{23},g_{34};g_{14}]$ is an edge-invariant triplet in the graph $A^{3}_{n,4}$ (see \Cref{subfig:An43}), we have $\ind(A^{3}_{n,4})\simeq \ind(A^{3}_{n,4}-\{g_{23}g_{34}\})$. Furthermore, $N(g_{23})\subseteq N(g_{14})$ in the graph $A^{3}_{n,4}-\{g_{23}g_{34}\}$ implies that  $\ind(A^3_{n,4})\simeq \ind((A^3_{n,4}-\{g_{23}g_{34}\})\setminus g_{14})$. Since the graph $(A^3_{n,4}-\{g_{23}g_{34}\})\setminus g_{14}$ consists of two connected components: one isomorphic to the complete graph $K_2$ and the other component is isomorphic to the graph $B_{n-2,4}$ (see \Cref{fig:Cn4}), we obtain the following 
\begin{equation}\label{eq: An43}
	\ind(A^{1}_{n,4}\setminus \n[g_{32}])\simeq \ind(A^2_{n,4})\simeq \ind(A^3_{n,4})\simeq \Sigma\ind(B_{n-2,4}).
\end{equation}
Now, the result follows from (\ref{eq:An41}), (\ref{eq:An42}) and (\ref{eq: An43}).
\end{proof}

\begin{claim}  For $n \geq 0$,\label{Bn4}
\begin{equation*} 
\ind(\widetilde{A}_{n,4}) \simeq \begin{cases}
	\mathbb{S}^1\vee\mathbb{S}^1 & \text{if  $n =0$,} \\
	\vee^{3}\mathbb{S}^2 & \text{if $n =1$,} \\
	\Sigma \ind(\widetilde{A}_{n-1,4}) \vee \Sigma^{2}\ind(A_{n-2,4})\vee \Sigma^{3}\ind(\widetilde{B}_{n-2,4})
	& \text{if $n \geq 2$.}
\end{cases}
\end{equation*}
\end{claim}  
\begin{proof}
Since $a_1$ is a simplicial vertex in $\widetilde{A}_{n,4}$ (see \Cref{fig:Bn4}) with $\n(a_1)=\{g_{11},g_{21}\}$, we get $\ind(\widetilde{A}_{n,4})\simeq \Sigma\ind(\widetilde{A}_{n,4}\setminus \n[g_{21}])\vee \Sigma\ind(\widetilde{A}_{n,4}\setminus \n[g_{11}])$. The graph $\widetilde{A}_{0,4}\setminus \n[g_{21}]$ consists of an isolated vertex $g_{12}$, and $\widetilde{A}_{0,4}\setminus \n[g_{11}]$ is isomorphic to the complete graph $K_3$. Therefore,  $\ind(\widetilde{A}_{0,4})\simeq \mathbb{S}^1\vee \mathbb{S}^1$.

 In the case of $\widetilde{A}_{1,4}$, the graph $\widetilde{A}_{1,4}\setminus \n[g_{21}]$ is isomorphic to $\widetilde{A}_{0,4}$. Hence, $\ind(\widetilde{A}_{1,4}\setminus \n[g_{21}])\simeq \mathbb{S}^1\vee \mathbb{S}^1$. Let $\widetilde{A}^1_{1,4}\coloneq\widetilde{A}_{1,4}\setminus \n[g_{11}]$. Then, $g_{31}$ is a simplicial vertex in the graph $\widetilde{A}^1_{1,4}$ with $\n(g_{31})=\{g_{22},g_{32}\}$, implying that $\ind(\widetilde{A}^1_{1,4})\simeq \Sigma\ind(\widetilde{A}^1_{1,4} \setminus \n[g_{22}])\vee \Sigma\ind(\widetilde{A}^1_{1,4}\setminus \n[g_{32}])$. Since $\n(g_{13})\subseteq \n(g_{24})$ in the graph $\widetilde{A}^1_{1,4}\setminus \n[g_{32}]$,  we get $\ind(\widetilde{A}^1_{1,4}\setminus \n[g_{32}])\simeq \ind((\widetilde{A}^1_{1,4}\setminus \n[g_{32}])\setminus g_{24})$. However, the graph $(\widetilde{A}^1_{1,4}\setminus \n[g_{32}])\setminus g_{24}$ is a path on $4$ vertices, implying that $\ind(\widetilde{A}^1_{1,4}\setminus  \n[g_{32}])$ is contractible. On the other hand, the graph $\widetilde{A}^1_{1,4}\setminus \n[g_{22}]$ is isomorphic to $A_{0,4}$. Therefore, $\ind(\widetilde{A}^1_{1,4})\simeq \mathbb{S}^1$, and it follows that $\ind(\widetilde{A}_{1,4})\simeq \vee^{3} \mathbb{S}^2$.
For $n\geq2$, the proof of the reduction follows similar lines as that of \Cref{An in pnp4}.
\end{proof}

\begin{claim}\label{Dn4} For $n\geq 0$,
\begin{equation*}
\ind(B_{n,4}) \simeq \begin{cases}
	\mathbb{S}^0\vee\mathbb{S}^0 & \text{for $n =0$,} \\
	\mathbb{S}^1 & \text{for $n =1$,} \\
	\Sigma \ind(A_{n-1,4}) \vee \Sigma^{2}\ind(\widetilde{A}_{n-2,4})\vee \Sigma^{2}\ind(\widetilde{B}_{n-2,4})
	& \text{for $n \geq 2$.}
\end{cases}
\end{equation*}
\end{claim}
\begin{proof}
Since $\n(g_{11})\subseteq \n(g_{31})$ in the graph $B_{0,4}$ (see \Cref{fig:Cn4}),  we get $\ind(B_{0,4})\simeq \ind(B_{0,4}\setminus g_{31})$. However, the graph  $B_{0,4}\setminus g_{31}$ is isomorphic to $K_3$, and therefore  $\ind(B_{0,4})\simeq \mathbb{S}^0\vee \mathbb{S}^0$.

For $n\geq 1$, the inclusion map $\ind(B_{n,4} \setminus \n[g_{21}]) \hookrightarrow \ind(B_{n,4} \setminus g_{21})$ is null homotopic, because $\ind(B_{n,4} \setminus \n[g_{21}])\ast\{b_1\}\subseteq \ind(B_{n,4} \setminus g_{21})$. Since the graph $B_{n,4}\setminus \n[g_{21}]$ is isomorphic to $A_{n-1,4}$, from \Cref{Link and Deletion}, we get
\begin{equation*}\label{Dn}
\begin{split}
	\ind(B_{n,4})\simeq  \ind(B_{n,4}\setminus g_{21})\vee \Sigma \ind(B_{n,4}\setminus \n[g_{21}]) & \simeq  \ind(B_{n,4}\setminus g_{21})\vee \Sigma \ind(A_{n-1,4}).
\end{split}
\end{equation*}

\begin{figure}[h!]
\centering
\begin{subfigure}[b]{0.3\textwidth}
	\begin{tikzpicture}[scale=0.3]
		\begin{scope}
			\foreach \x in {1,2,...,5}
			{
				\node[vertex] (1\x) at (2*\x,3*2) {};
			}
			\foreach \x in {2,...,5}
			{
				\node[vertex] (2\x) at (2*\x,2*2) {};
			}
			\foreach \x in {1,2,...,5}
			{
				\node[vertex] (3\x) at (2*\x,1*2) {};
			}

			\foreach \y [evaluate={\y as \x using int({\y+1})}] in {1,2,...,4} 
			{
				\draw[edge] (1\y) -- (1\x);
			}
			\foreach \y [evaluate={\y as \x using int({\y+1})}] in {2,3,4} 
			{
				\draw[edge] (2\y) -- (2\x);
			}
			\foreach \y [evaluate={\y as \x using int({\y+1})}] in {1,2,...,4} 
			{
				\draw[edge] (3\y) -- (3\x);
			}
			
			\foreach \z in {2,3,4,5}
			{
				\draw[edge] (1\z) -- (2\z);
			}
			\foreach \z in {2,3,4,5}
			{
				\draw[edge] (2\z) -- (3\z);
			}
			
			\foreach \z [evaluate={\znext=int({1+\z})}] in {2,4}
			{
				\draw[edge] (1\z) -- (2\znext);
			}
			\foreach \z [evaluate={\znext=int({\z-1})}] in {3,5}
			{
				\draw[edge] (1\z) -- (2\znext);
			}
			
			\foreach \z [evaluate={\znext=int({1+\z})}] in {3}
			{
				\draw[edge] (2\z) -- (3\znext);
			}
			\foreach \z [evaluate={\znext=int({\z-1})}] in {4}
			{
				\draw[edge] (2\z) -- (3\znext);
			}

			\draw[dotted, thick,shorten >=4pt, shorten <=4pt] (10,5) -- (12,5);
			\draw[dotted, thick,shorten >=4pt, shorten <=4pt] (10,3) -- (12,3);

			\node[label={[label distance=-3pt]90:\scriptsize{$g_{11}$}}] at (11) {};
			\foreach \x in {2,...,5}
			{
				\node[label={[label distance=-3pt]90:\scriptsize{$g_{1\x}$}}] at (1\x) {};
			}
			
			\node[label={[label distance=-8pt]60:\scriptsize{$g_{31}$}}] at (31) {};
			\node[label={[label distance=-8pt]-30:\scriptsize{$g_{22}$}}] at (22) {};
			\node[label={[label distance=-3pt]270:\scriptsize{$g_{32}$}}] at (32) {};
			\node[label={[label distance=-8pt]60:\scriptsize{$g_{23}$}}] at (23) {};

		\end{scope}
		\begin{scope}[shift={(2,0)}]
			\node[vertex] (c1) at (-2,4) {} ;
			
			\draw[edge] (c1) -- (11) ;
			\draw[edge] (c1) -- (31) ;
			
			\node[label={[label distance=-3pt]180:\scriptsize{$b_1$}}] at (c1) {};   
			
		\end{scope}

	\end{tikzpicture}
	\caption{$B_{n,4}^2$}
	\label{subfig:Cn4-2}
\end{subfigure}
\hspace{2mm}
\begin{subfigure}[b]{0.2\textwidth}
	\begin{tikzpicture}[scale=0.3]
		\begin{scope}
			\foreach \x in {2,...,5}
			{
				\node[vertex] (1\x) at (2*\x,3*2) {};
			}
			\foreach \x in {2,...,5}
			{
				\node[vertex] (2\x) at (2*\x,2*2) {};
			}
			\foreach \x in {2,...,5}
			{
				\node[vertex] (3\x) at (2*\x,1*2) {};
			}

			\foreach \y [evaluate={\y as \x using int({\y+1})}] in {2,...,4} 
			{
				\draw[edge] (1\y) -- (1\x);
			}
			\foreach \y [evaluate={\y as \x using int({\y+1})}] in {2,3,4} 
			{
				\draw[edge] (2\y) -- (2\x);
			}
			\foreach \y [evaluate={\y as \x using int({\y+1})}] in {2,...,4} 
			{
				\draw[edge] (3\y) -- (3\x);
			}
			
			\foreach \z in {2,3,4,5}
			{
				\draw[edge] (1\z) -- (2\z);
			}
			\foreach \z in {2,3,4,5}
			{
				\draw[edge] (2\z) -- (3\z);
			}
			
			\foreach \z [evaluate={\znext=int({1+\z})}] in {2,4}
			{
				\draw[edge] (1\z) -- (2\znext);
			}
			\foreach \z [evaluate={\znext=int({\z-1})}] in {3,5}
			{
				\draw[edge] (1\z) -- (2\znext);
			}
			
			\foreach \z [evaluate={\znext=int({1+\z})}] in {3}
			{
				\draw[edge] (2\z) -- (3\znext);
			}
			\foreach \z [evaluate={\znext=int({\z-1})}] in {4}
			{
				\draw[edge] (2\z) -- (3\znext);
			}

			\draw[dotted, thick,shorten >=4pt, shorten <=4pt] (10,5) -- (12,5);
			\draw[dotted, thick,shorten >=4pt, shorten <=4pt] (10,3) -- (12,3);

			\foreach \x in {2,...,5}
			{
				\node[label={[label distance=-3pt]90:\scriptsize{$g_{1\x}$}}] at (1\x) {};
			}

			\draw[edge] (12) to[bend right=45] (32);

			\node[label={[label distance=-8pt]-30:\scriptsize{$g_{22}$}}] at (22) {};
			\node[label={[label distance=-3pt]270:\scriptsize{$g_{32}$}}] at (32) {};
			\node[label={[label distance=-8pt]60:\scriptsize{$g_{23}$}}] at (23) {};
		\end{scope}
	\end{tikzpicture}
	\caption{$B_{n,4}^3$}
	\label{subfig:Cn4-3}
\end{subfigure}        
\hspace{1mm}
\begin{subfigure}[b]{0.2\textwidth}
	\begin{tikzpicture}[scale=0.3]
		\begin{scope}
			\foreach \x in {3,4,5}
			{
				\node[vertex] (1\x) at (2*\x,3*2) {};
			}
			\foreach \x in {2,4,5}
			{
				\node[vertex] (2\x) at (2*\x,2*2) {};
			}
			\foreach \x in {2,...,5}
			{
				\node[vertex] (3\x) at (2*\x,1*2) {};
			}

			\foreach \y [evaluate={\y as \x using int({\y+1})}] in {3,4} 
			{
				\draw[edge] (1\y) -- (1\x);
			}
			\foreach \y [evaluate={\y as \x using int({\y+1})}] in {4} 
			{
				\draw[edge] (2\y) -- (2\x);
			}
			\foreach \y [evaluate={\y as \x using int({\y+1})}] in {2,...,4} 
			{
				\draw[edge] (3\y) -- (3\x);
			}
			
			\foreach \z in {4,5}
			{
				\draw[edge] (1\z) -- (2\z);
			}
			\foreach \z in {2,4,5}
			{
				\draw[edge] (2\z) -- (3\z);
			}
			
			\foreach \z [evaluate={\znext=int({1+\z})}] in {4}
			{
				\draw[edge] (1\z) -- (2\znext);
			}
			\foreach \z [evaluate={\znext=int({\z-1})}] in {3,5}
			{
				\draw[edge] (1\z) -- (2\znext);
			}
			
			\foreach \z [evaluate={\znext=int({1+\z})}] in {}
			{
				\draw[edge] (2\z) -- (3\znext);
			}
			\foreach \z [evaluate={\znext=int({\z-1})}] in {4}
			{
				\draw[edge] (2\z) -- (3\znext);
			}

			\draw[dotted, thick,shorten >=4pt, shorten <=4pt] (10,5) -- (12,5);
			\draw[dotted, thick,shorten >=4pt, shorten <=4pt] (10,3) -- (12,3);

			\foreach \x in {3,...,5}
			{
				\node[label={[label distance=-3pt]90:\scriptsize{$g_{1\x}$}}] at (1\x) {};
			}

			\node[label={[label distance=-8pt]-30:\scriptsize{$g_{22}$}}] at (22) {};
			\node[label={[label distance=-8pt]60:\scriptsize{$g_{32}$}}] at (32) {};
			\node[label={[label distance=-3pt]270:\scriptsize{$g_{33}$}}] at (33) {};           
		\end{scope}
	\end{tikzpicture}
	\caption{$B_{n,4}^4$}
	\label{subfig:Cn4-4}
\end{subfigure}
\begin{subfigure}[b]{0.2\textwidth}
	\begin{tikzpicture}[scale=0.3]
		\begin{scope}
			\foreach \x in {3,4,5}
			{
				\node[vertex] (1\x) at (2*\x,3*2) {};
			}
			\foreach \x in {2,4,5}
			{
				\node[vertex] (2\x) at (2*\x,2*2) {};
			}
			\foreach \x in {2,...,5}
			{
				\node[vertex] (3\x) at (2*\x,1*2) {};
			}

			\foreach \y [evaluate={\y as \x using int({\y+1})}] in {4} 
			{
				\draw[edge] (1\y) -- (1\x);
			}
			\foreach \y [evaluate={\y as \x using int({\y+1})}] in {4} 
			{
				\draw[edge] (2\y) -- (2\x);
			}
			\foreach \y [evaluate={\y as \x using int({\y+1})}] in {2,...,4} 
			{
				\draw[edge] (3\y) -- (3\x);
			}
			
			\foreach \z in {4,5}
			{
				
				\draw[edge] (1\z) -- (2\z);
			}
			\foreach \z in {2,4,5}
			{
				\draw[edge] (2\z) -- (3\z);
			}
			
			\foreach \z [evaluate={\znext=int({1+\z})}] in {4}
			{
				\draw[edge] (1\z) -- (2\znext);
			}
			\foreach \z [evaluate={\znext=int({\z-1})}] in {3,5}
			{
				\draw[edge] (1\z) -- (2\znext);
			}
			
			\foreach \z [evaluate={\znext=int({1+\z})}] in {}
			{
				\draw[edge] (2\z) -- (3\znext);
			}
			\foreach \z [evaluate={\znext=int({\z-1})}] in {4}
			{
				\draw[edge] (2\z) -- (3\znext);
			}

			\draw[dotted, thick,shorten >=4pt, shorten <=4pt] (10,5) -- (12,5);
			\draw[dotted, thick,shorten >=4pt, shorten <=4pt] (10,3) -- (12,3);

			\foreach \x in {3,...,5}
			{
				\node[label={[label distance=-3pt]90:\scriptsize{$g_{1\x}$}}] at (1\x) {};
			}
			
			\draw[edge] (33) -- (14);
			
			\node[label={[label distance=-8pt]-30:\scriptsize{$g_{22}$}}] at (22) {};
			\node[label={[label distance=-8pt]60:\scriptsize{$g_{32}$}}] at (32) {};
			\node[label={[label distance=-3pt]270:\scriptsize{$g_{33}$}}] at (33) {};           
		\end{scope}
	\end{tikzpicture}
	\caption{$B_{n,4}^5$}
	\label{subfig:Cn4-5}
\end{subfigure}

\caption{}
\end{figure}

Let $B_{n,4}^1 \coloneq B_{n,4}\setminus g_{21}$. Since $[g_{31},g_{22};g_{11}]$ is an edge invariant triplet in the graph $B_{n,4}^1$,  $\ind(B_{n,4}^1)\simeq \ind(B_{n,4}^1-\{g_{31}g_{22}\})$. Let $ B_{n,4}^2\coloneq B_{n,4}^1-\{g_{31}g_{22}\}$. Then $\ind(B_{n,4}^1)\simeq \ind(B_{n,4}^2)$.
Since  $[g_{12},g_{32};b_1]$ is an edge-invariant triplet in the graph $B_{n,4}^2$, $\ind(B_{n, 4}^2) \simeq \ind(B_{n,4}^2+\{g_{12}g_{32}\})$. Furthermore,  $[g_{31},g_{32};g_{11}]$ is an edge-invariant triplet in the graph $B_{n,4}^2+\{g_{12}g_{32}\}$. Therefore, $\ind(B_{n,4}^2)\simeq \ind((B_{n,4}^2+\{g_{12}g_{32}\})-\{g_{31}g_{32}\})$. Observe that $\n(g_{31})\subseteq \n(g_{11})$ in $(B_{n,4}^2+\{g_{12}g_{32}\})-\{g_{31}g_{32}\}$. Hence $\ind(B_{n,4}^2)\simeq\ind(((B_{n,4}^2+\{g_{12}g_{32}\})-\{g_{31}g_{32}\})\setminus g_{11})$. The resultant graph we obtain has two connected components: one consists of the edge $g_{31}b_1$, and we denote the other component by $B_{n,4}^3$ (see \Cref{subfig:Cn4-3}). This implies that $\ind(B_{n,4}^2)\simeq \Sigma \ind(B_{n,4}^3)$.

Since  $\ind(B_{n,4}^3 \setminus \n[g_{12}])*\{g_{22}\} \subseteq \ind(B_{n,4}^3\setminus g_{12})$, the inclusion map  $\ind(B_{n,4}^3 \setminus \n[g_{12}]) \hookrightarrow \ind(B_{n,4}^3 \setminus g_{12})$ is null-homotopic. Therefore, by \Cref{Link and Deletion}, we get 
\begin{equation} 
\begin{split}\label{eq:D3n}
	\ind(B_{n,4}^3)\simeq  \ind(B_{n,4}^3\setminus g_{12})\vee \Sigma \ind(B_{n,4}^3\setminus \n[g_{12}]).
\end{split}
\end{equation}
Now, consider the case $n=1$. Since $g_{33}$ is an isolated vertex in $B_{1,4}^3\setminus \n[g_{12}]$, $\ind(B_{1,4}^3\setminus \n[g_{12}])$ is contractible. Observe that $\n(g_{32})\subseteq \n(g_{23})$ in the graph $B_{1,4}^3\setminus g_{12}$. Thus, using \Cref{Folding Lemma}, we find that $\ind(B_{1,4}^3\setminus g_{12})\simeq \ind(B_{1,4}^3\setminus \{g_{12},g_{23}\})$. Since the graph $B_{1,4}^3\setminus \{g_{12},g_{23}\}$ is a path on $4$- vertices, $\ind(B_{1,4}^3\setminus g_{12})$ is contractible. Therefore, $\ind(B_{1,4}^3)$ is contractible. This implies that $\ind(B_{1,4}^1)$ is contractible. Hence, $\ind(B_{1,4})\simeq \Sigma\ind(A_{0,4})\simeq \mathbb{S}^1$.

Now, consider $n\geq 2$. Observe that the graph $B_{n,4}^3 \setminus \n[g_{12}]$ is isomorphic to the graph $\widetilde{A}_{n-2,4}$. Thus, $\ind(B_{n,4}^3)\simeq  \ind(B_{n,4}^3\setminus g_{12})\vee \Sigma \ind(\widetilde{A}_{n-2,4}).$

Since $\n(g_{32})\subseteq \n(g_{23})$ in $B_{n,4}^3\setminus g_{12}$,  $\ind(B_{n,4}^3\setminus g_{12})\simeq \ind(B_{n,4}^3\setminus \{g_{12},g_{23}\})$. Let $B_{n,4}^4\coloneq B_{n,4}^3\setminus \{g_{12},g_{23}\}$. Then $\ind(B_{n,4}^3\setminus g_{12})\simeq \ind(B_{n,4}^4)$. Since $[g_{33},g_{14};g_{22}]$ and $[g_{13},g_{14}; g_{32}]$ are edge-invariant triplets in graphs $B_{n,4}^4$ and $B_{n,4}^4+\{g_{33}g_{14}\}$, respectively, \Cref{Edge deletion 1} provides us with $\ind(B_{n,4}^4)\simeq \ind((B_{n,4}^4+\{g_{33}g_{14}\})-\{g_{13}g_{14}\})$. Let  $B_{n,4}^5 \coloneq (B_{n,4}^4+\{g_{33}g_{14}\})-\{g_{13}g_{14}\}$. Then, $\ind(B_{n,4}^4)\simeq \ind(B_{n,4}^5)$. Since $\n(g_{13})\subseteq \n(g_{32}$) in the graph $B_{n,4}^5$,  $\ind(B_{n,4}^5)\simeq \ind(B_{n,4}^5\setminus g_{32})$. 
The graph $B_{n,4}^5\setminus g_{32}$ consists of two connected components: one component consists of the edge  $g_{13}g_{22}$, and the other component is isomorphic to $\widetilde{B}_{n-2,4}$. Therefore, $\ind(B_{n,4}^4)\simeq \ind(B_{n,4}^5)\simeq \Sigma \ind(\widetilde{B}_{n-2,4})$.
Hence, $\ind(B_{n,4}^3\setminus g_{12})\simeq \Sigma \ind(\widetilde{B}_{n-2,4})$. Thus, the result follows.
\end{proof}

\begin{claim}\label{Fn4} For $n\geq 0$,
\begin{equation*}
\ind(\widetilde{B}_{n,4}) \simeq \begin{cases}
	\mathbb{S}^1\vee \mathbb{S}^1  & \text{if $n =0$,} \\
	\bigvee^5 \mathbb{S}^2 & \text{if $n =1$,} \\
	\Sigma \ind(\widetilde{A}_{n-1,4}) \vee \Sigma^{2}\ind(A_{n-2,4})\vee \Sigma^{2}\ind(B_{n-2,4})
	& \text{if $n \geq 2$.}
\end{cases}
\end{equation*}
\end{claim}
\begin{proof}
Since $g_{32}$ is a simplicial vertex in $\widetilde{B}_{0,4}$ (see \Cref{fig:Dn4}) with $\n(g_{32})=\{g_{21},g_{22},g_{31}\}$, $ \ind(\widetilde{B}_{0,4})\simeq \Sigma \ind(\widetilde{B}_{0,4}\setminus\n[g_{21}])\vee \Sigma\ind(\widetilde{B}_{0,4}\setminus\n[g_{22}])\vee \Sigma\ind(\widetilde{B}_{0,4}\setminus\n[g_{31}]).$ 
Observe that both the graphs $\widetilde{B}_{0,4}\setminus\n[g_{22}]$ and  $\widetilde{B}_{0,4}\setminus\n[g_{31}]$ are paths on two vertices and the graph $\widetilde{B}_{0,4}\setminus\n[g_{21}]$ consists of an isolated vertex $g_{12}$. Hence $\ind(\widetilde{B}_{0,4})\simeq \mathbb{S}^1 \vee \Sp^1.$ 

Now, we compute $\ind(\widetilde{B}_{1,4})$.  The inclusion map 
$\ind(\widetilde{B}_{1,4}\setminus \n[g_{21}])\hookrightarrow \ind(\widetilde{B}_{1,4}\setminus g_{21})$ is null homotopic because $\ind(\widetilde{B}_{1,4}\setminus \n[g_{21}])\ast \{b_1\}$ is a subcomplex of $\ind(\widetilde{B}_{1,4}\setminus g_{21})$. Thus, from \Cref{Link and Deletion}, we get 
$\ind(\widetilde{B}_{1,4})\simeq \Sigma\ind(\widetilde{B}_{1,4}\setminus \n[g_{21}])\vee\ind(\widetilde{B}_{1,4}\setminus g_{21})$.
Observe that  $\widetilde{B}_{1,4}\setminus \n[g_{21}]$ is isomorphic to $\widetilde{A}_{0,4}$. Therefore, $\ind(\widetilde{A}_{0,4})\simeq \mathbb{S}^1\vee\mathbb{S}^1$ implies that $\ind(\widetilde{B}_{1,4})\simeq \mathbb{S}^2\vee\mathbb{S}^2\vee\ind(\widetilde{B}_{1,4}\setminus g_{21}).$
Let $\widetilde{B}_{1,4}^1:=\widetilde{B}_{1,4}\setminus g_{21}$. Then $g_{34}$ is a simplicial vertex in $\widetilde{B}_{1,4}^1$ with $\n(g_{34})=\{g_{23}, g_{33},g_{24}\}$. Using \Cref{Simplicial Vertex Lemma}, we have 
$\ind(\widetilde{B}_{1,4}^1)\simeq \Sigma\ind(\widetilde{B}_{1,4}^1\setminus \n[g_{23}])\vee\Sigma \ind(\widetilde{B}_{1,4}^1\setminus \n[g_{33}])\vee \Sigma \ind(\widetilde{B}_{1,4}^1\setminus \n[g_{24}]).$
Since $\n(g_{14})\subseteq \n(g_{22 })\cap \n(g_{12})$  in $\widetilde{B}_{1,4}^1\setminus \n[g_{33}]$,  $\ind(\widetilde{B}_{1,4}^1\setminus \n[g_{33}])\simeq \ind((\widetilde{B}_{1,4}^1\setminus \n[g_{33}])\setminus\{g_{12},g_{22}\})$. Observe that $(\widetilde{B}_{1,4}^1\setminus \n[g_{33}])\setminus\{g_{12},g_{22}\}$ is a disjoint union of two paths on $2$ and $3$ vertices. Hence $\ind(\widetilde{B}_{1,4}^1\setminus \n[g_{33}])\simeq \mathbb{S}^1$. 

Since the graph $\widetilde{B}_{1,4}^1\setminus \n[g_{23}]$ contains an isolated vertex $g_{14}$,  $\ind(\widetilde{B}_{1,4}^1\setminus \n[g_{23}])$ is contractible.

Let $\widetilde{B}_{1,4}^2\coloneq \widetilde{B}_{1,4}^1\setminus \n[g_{24}]$. Since $g_{13}$ is a simplicial vertex in  $\widetilde{B}_{1,4}^2$ with $\n(g_{13})=\{g_{12},g_{22}\}$, we have  $\ind(\widetilde{B}_{1,4}^2)\simeq \Sigma\ind(\widetilde{B}_{1,4}^2\setminus\n[g_{12}])\vee \Sigma \ind(\widetilde{B}_{1,4}^2\setminus\n[g_{22}])$. However, $\widetilde{B}_{1,4}^2\setminus\n[g_{22}]$ is an edge $b_{1}g_{11}$ and $\widetilde{B}_{1,4}^2\setminus\n[g_{12}]$
is a path on three vertices $b_1, g_{31}$ and $g_{32}$ and hence, $\ind(\widetilde{B}_{1,4}^2)\simeq \mathbb{S}^1\vee \mathbb{S}^1$. Consequently, $\ind(\widetilde{B}_{1,4})\simeq \vee^5\mathbb{S}^2$.
Further, for $n\geq 2$, the reduction is similar to the proof of \Cref{Dn4} which yields that
$
\ind(\widetilde{B}_{n,4}) \simeq \Sigma \ind(\widetilde{A}_{n-1,4}) \vee \Sigma^{2}\ind(A_{n-2,4})\vee \Sigma^{2}\ind(B_{n-2,4}).
$
%This completes the proof.
\end{proof}
\begin{theorem}\label{IntermediaryPn4}
For any graph $G_{n,4}\in\{\Gamma_{n,4},\widetilde{\Gamma}_{n,4},A_{n,4},\widetilde{A}_{n,4},B_{n,4},\widetilde{B}_{n,4}\}$, $n\geq 0$, the complex $\ind(G_{n,4})$ is homotopy equivalent to a wedge of spheres.
\end{theorem}
\begin{proof}
The proof follows from \Crefrange{Gamman4}{Fn4}, and the induction hypothesis.
\end{proof}

\subsection{\texorpdfstring{Homotopy type of $\M(P_n \times P_4)$}{Homotopy type of\M(Pn × P4)}}\label{PnP4homotopy}

In this section, we conclude the homotopy type of $\M(P_n \times P_4)$ along with the minimum and maximum dimensions of the spheres in the independence complexes of the graphs in $\{\Gamma_{n, 4}, \widetilde{\Gamma}_{n, 4}, A_{n, 4},$ $ \widetilde{A}_{n, 4}, B_{n, 4}, \widetilde{B}_{n, 4}\}$.

\begin{theorem}\label{pnxp4}
For $n\geq 1$,  $\M(P_n\times P_4)$ is homotopy equivalent to a wedge of spheres.  
\end{theorem}
\begin{proof}
Since $P_1\times P_4$ is a discrete graph, $\M(P_1\times P_4)$ is a void complex. For $n\geq 0$, we have relations $\M(P_{2n+2}\times P_4)\simeq\ind(\Gamma_{n,4})\ast \ind(\Gamma_{n,4})$ and $\M(P_{2n+3}\times P_4)\simeq\ind(\widetilde{\Gamma}_{n,4})\ast \ind(\widetilde{\Gamma}_{n,4})$. Now, the result follows from \Cref{IntermediaryPn4}.
\end{proof}

For a graph $G$ whose independence complex is homotopy equivalent to a wedge of spheres, let $\dimax{G}$ and $\dimin{G}$ denote the maximum and minimum dimension of spheres appearing in the homotopy type of $\ind(G)$.

\begin{proposition} \label{proposition:dimPnP4allgraphs}
For $n\geq 0$, the maximum and minimum dimension of the spheres in the homotopy type of the independence complexes of graphs in $\{\Gamma_{n, 4}, \widetilde{\Gamma}_{n, 4}, A_{n, 4}, \widetilde{A}_{n, 4}, B_{n, 4}, \widetilde{B}_{n, 4},\}$  are given in $\Cref{tab: max Pn4}$ and $\Cref{tab: min Pn4}$, respectively.
{\tiny \begin{table}[h!]
		\centering
		%\caption{$d_{\max}$ of independence complexes of graphs}
		\renewcommand{\arraystretch}{1.7}
		\setlength{\tabcolsep}{6pt} 
		\begin{tabular}{|c|c|c|c|c|c|c|}
			\hline
			\multirow{2}{*}{$n \equiv$} & \multicolumn{6}{c|}{$d_{\max}$ of independence complexes of graphs} \\ \cline{2-7}
			& $\Gamma_{n,4}$ & $\widetilde{\Gamma}_{n,4}$ & $A_{n,4}$ & $\widetilde{A}_{n,4}$ & $B_{n,4}$ & $\widetilde{B}_{n,4}$ \\ \hline
			0 (mod 7) & $10(\frac{n}{7})$ & $10(\frac{n}{7})(n\neq 0)$ & $10(\frac{n}{7})$ & $10(\frac{n}{7})+1$ & $10(\frac{n}{7})$ & $10(\frac{n}{7})+1$ \\ \hline
			1 (mod 7) & $10(\frac{n-1}{7})+1$ & $10(\frac{n-1}{7})+2$ & $10(\frac{n-1}{7})+2$ & $10(\frac{n-1}{7})+2$ & $10(\frac{n-1}{7})+1$ & $10(\frac{n-1}{7})+2$ \\ \hline
			2 (mod 7) & $10(\frac{n-2}{7})+3$ & $10(\frac{n-2}{7})+3$ & $10(\frac{n-2}{7})+3$ & $10(\frac{n-2}{7})+4$ & $10(\frac{n-2}{7})+3$ & $10(\frac{n-2}{7})+4$ \\ \hline
			3 (mod 7) & $10(\frac{n-3}{7})+4$ & $10(\frac{n-3}{7})+5$ & $10(\frac{n-3}{7})+4$ & $10(\frac{n-3}{7})+5$ & $10(\frac{n-3}{7})+4$ & $10(\frac{n-3}{7})+5$ \\ \hline
			4 (mod 7) & $10(\frac{n-4}{7})+5$ & $10(\frac{n-4}{7})+6$ & $10(\frac{n-4}{7})+6$ & $10(\frac{n-4}{7})+7$ & $10(\frac{n-4}{7})+6$ & $10(\frac{n-4}{7})+6$ \\ \hline
			5 (mod 7) & $10(\frac{n-5}{7})+7$ & $10(\frac{n-5}{7})+8$ & $10(\frac{n-5}{7})+7$ & $10(\frac{n-5}{7})+8$ & $10(\frac{n-5}{7})+7$ & $10(\frac{n-5}{7})+8$ \\ \hline
			6 (mod 7) & $10(\frac{n-6}{7})+8$ & $10(\frac{n-6}{7})+9$ & $10(\frac{n-6}{7})+9$ & $10(\frac{n-6}{7})+9$ & $10(\frac{n-6}{7})+9$ & $10(\frac{n-6}{7})+9$ \\ \hline
		\end{tabular}
		\caption{}
		\label{tab: max Pn4}
\end{table}}
%----------------------------------------------Table 4 Updated--------------------
{\footnotesize
	\begin{table}[h!]
		\centering
		\renewcommand{\arraystretch}{1.4}
		\setlength{\tabcolsep}{5pt} 
		\begin{tabular}{|c|c|c|}
			\hline
			\multirow{3}{*}{$n$} & \multicolumn{2}{c|}{$d_{\mathit{min}}$ of independence complexes of graphs} \\ \cline{2-3}
			& $\Gamma_{n,4}$, $A_{n,4}$, $B_{n,4}$ & $\widetilde{\Gamma}_{n,4}$, $\widetilde{A}_{n,4}$, $\widetilde{B}_{n,4}$ \\ \cline{2-3}
			& $n$ & $n+1$ \\ 
			\hline
		\end{tabular}
		\caption{}
		\label{tab: min Pn4}
\end{table}}

\end{proposition}
\begin{proof}
We prove the statement using induction on $n$. The base case for $n=0$ follows from \Crefrange{Gamman4}{Fn4}. Assume the statement is true for $k < n$. For $\Gamma_{n,4}$, we have
\begin{equation}{\label{eq: max-gamman4-eq1}}
\ind(\Gamma_{n,4})\simeq \Sigma \ind(A_{n-1,4}) \vee \Sigma^{2} \ind(\widetilde{\Gamma}_{n-2,4}),
\end{equation}
which implies that
\begin{equation*}
\dimin{\Gamma_{n,4}} = \min \set{1 + \dimin{A_{n-1,4}}, 2 + \dimin{\widetilde{\Gamma}_{n-2,4}}},
\end{equation*}
and 
\begin{equation}{\label{eq: max-gamman4-eq2}}
\dimax{\Gamma_{n,4}} = \max \set{1 + \dimax{A_{n-1,4}}, 2 + \dimax{\widetilde{\Gamma}_{n-2,4}}}.
\end{equation}
In particular, using the induction hypothesis, we have $\dimin{\Gamma_{n,4}} = \min\set{1+(n-1), 2+(n-1)} = n$. Further, for $n \equiv r \mod 7$ with $0 \leq r \leq 6$, \eqref{eq: max-gamman4-eq2} yields the following.
\begin{table}[h!]
	{\tiny
		\renewcommand{\arraystretch}{1.7} % extra vertical space
		\setlength{\tabcolsep}{6pt} % extra horizontal space
		\begin{tabular}{|c|c|c|c|}
			\hline
			$n \equiv$ & $1+\dimax{A_{n-1,4}}$ & $2+\dimax{\widetilde{\Gamma}_{n-2,4}}$ & $\dimax{\Gamma_{n,4}}$ \\
			\hline
			$0 \mod{7}$ & $1+10\big(\frac{n-1-6}{7}\big)+9$ & $2+10\big(\frac{n-2-5}{7}\big)+8$ & $10\big(\frac{n}{7}\big)$ \\
			\hline
			$1 \mod{7}$ & $1+10\big(\frac{n-1}{7}\big)$ & $2+10\big(\frac{n-2-6}{7}\big)+9$ & $10\big(\frac{n-1}{7}\big)+1$ \\
			\hline
			$2 \mod{7}$ & $1+10\big(\frac{n-1-1}{7}\big)+2$ & $2+10\big(\frac{n-2}{7}\big)$ & $10\big(\frac{n-2}{7}\big)+3$ \\
			\hline
			$3 \mod{7}$ & $1+10\big(\frac{n-1-2}{7}\big)+3$ & $2+10\big(\frac{n-2-1}{7}\big)+2$ & $10\big(\frac{n-3}{7}\big)+4$ \\
			\hline
			$4 \mod{7}$ & $1+10\big(\frac{n-1-3}{7}\big)+4$ & $2+10\big(\frac{n-2-2}{7}\big)+3$ & $10\big(\frac{n-4}{7}\big)+5$ \\
			\hline
			$5 \mod{7}$ & $1+10\big(\frac{n-1-4}{7}\big)+6$ & $2+10\big(\frac{n-2-3}{7}\big)+5$ & $10\big(\frac{n-5}{7}\big)+7$ \\
			\hline
			$6 \mod{7}$ & $1+10\big(\frac{n-1-5}{7}\big)+7$ & $2+10\big(\frac{n-2-4}{7}\big)+6$ & $10\big(\frac{n-6}{7}\big)+8$ \\
			\hline
	\end{tabular}}
	\caption{}
	\label{tab:Pn4minmax}
\end{table}

% \end{center}
By a similar argument, we can conclude the maximum and minimum dimensions of independence complexes of the graphs $\widetilde{\Gamma}_{n,4},A_{n,4},\widetilde{A}_{n,4},B_{n,4}$ and $\widetilde{B}_{n,4}$.
\end{proof}

\begin{rem}

The maximum and minimum dimension of spheres in the homotopy type of $\M(P_n \times P_4)$ as given in \Cref{proposition:dim_PnP4} can be concluded from $\Cref{proposition:dimPnP4allgraphs}$. 
\end{rem}

\section{\texorpdfstring{Matching complexes of $P_n \times P_5$}{Matching complexes of Pn x P5}}\label{section:PnP5}

In this section, we compute the homotopy type of the matching complex of the categorical product $P_n \times P_5$. Since $P_n$ is bipartite, $P_n \times P_5$ (and its line graph) is disconnected and contains exactly two connected components. When $n$ is even, these components are isomorphic; whereas, for odd $n$, they are non-isomorphic.

For $n \geq 0$, we denote the connected component of the line graph of $P_{2n+2}\times P_5$ by $\Gamma_{n,5}$ (see \Cref{fig:Gn}). Since the line graph  $L(P_{2n+2}\times P_5) $ is the same as $ \Gamma_{n,5} \sqcup \Gamma_{n,5}$, the matching complex
\begin{equation*}
    \M(P_{2n+2}\times P_5) = \ind(L(P_{2n+2}\times P_5)) \simeq \ind(\Gamma_{n,5}) * \ind(\Gamma_{n,5}).
\end{equation*}
Similarly, for $n\geq 0$, let $\Lambda_{n,5}$ and $\widetilde{\Gamma}_{n,5}$ (see \Cref{fig:Lambda-n,fig:G'n}) denote the two non-isomorphic connected components of the line graph of $P_{2n+3} \times P_5$. Therefore, we have
\begin{equation*}
    \M(P_{2n+3}\times P_5) = \ind(L(P_{2n+3}\times P_5)) \simeq \ind(\Lambda_{n,5}) * \ind(\widetilde{\Gamma}_{n,5}).
\end{equation*}

We now define the graphs $\Gamma_{n, 5}, \Lambda_{n, 5}$, and $\widetilde{\Gamma}_{n, 5}$. 

\begin{definition}
For $n\geq 0$, the graph $\Gamma_{n,5}$ comprises of the vertex set $$V(\Gamma_{n,5}) = \set{g_{ij} \given i=1,2,3,4 \text{ and }j = {1,2,\dots,2n{+}1}},$$ where the vertices are arranged in $4$ rows and $2n+1$ columns. The edge set of $\Gamma_{n,5}$ is given by 
\begin{align*}
    E(\Gamma_{n,5}) = &\set{g_{ij}g_{i+1\,j} \given i=1,2,3 \  \text{and} \ j=1,2,\dots,2n{+}1}  \\ & \cup \set{g_{ij}g_{i\,j+1} \given i=1,2,3,4 \ \text{and}  \ j=1,2,\dots,2n} \\
    & \cup \set{g_{ij}g_{i+1\,j+1}, g_{ij+1}g_{i+1\,j} \given i=1,2,3, \; j=1,2,\dots,2n, \text{ and } i+j\equiv 1 \mod{2}}.
\end{align*}
\end{definition}

\begin{figure}[ht]
	% \begin{subfigure}[]{0.48\textwidth}
		\centering
		\begin{tikzpicture}[scale=0.3]
			\begin{scope}
\foreach \x in {1,2,...,8}
            {
            \node[vertex] (1\x) at (3*\x,6) {};
            }
\foreach \x in {1,2,...,8}
            {
            \node[vertex] (2\x) at (3*\x,2*2) {};
            }
\foreach \x in {1,2,...,8}
            {
            \node[vertex] (3\x) at (3*\x,1*2) {};
            }
\foreach \x in {1,2,...,8}
            {
            \node[vertex] (4\x) at (3*\x,0*2) {};
            }
            
\foreach \y [evaluate={\y as \x using int({\y+1})}] in {1,2,...,4,6,7} 
            {
            \draw[edge] (1\y) to (1\x);
            \draw[edge] (2\y) -- (2\x);
            \draw[edge] (3\y) -- (3\x);
            \draw[edge] (4\y) -- (4\x);
            }
            
\foreach \z in {1,2,3,4,5,...,8}
            {\draw[edge] (1\z) -- (2\z);
            \draw[edge] (2\z) -- (3\z);
            \draw[edge] (3\z) -- (4\z);
            }
            
\foreach \z [evaluate={\znext=int({1+\z})}] in {2,4,7}
            {
            \draw[edge] (1\z) -- (2\znext);
            \draw[edge] (3\z) -- (4\znext);
            }
            \foreach \z [evaluate={\znext=int({\z-1})}] in {3,5,8}
            {
            \draw[edge] (1\z) -- (2\znext);
            \draw[edge] (3\z) -- (4\znext);
            }

\foreach \z [evaluate={\znext=int({1+\z})}] in {1,3,6}
            {
            \draw[edge] (2\z) -- (3\znext);
            }
            \foreach \z [evaluate={\znext=int({\z-1})}] in {2,4,7}
            {
            \draw[edge] (2\z) -- (3\znext);
            }

\draw[dotted, thick,shorten >=4pt, shorten <=4pt] (15.5,5) -- (17.5,5);
            \draw[dotted, thick,shorten >=4pt, shorten <=4pt] (15.5,3) -- (17.5,3);
            \draw[dotted, thick,shorten >=4pt, shorten <=4pt] (15.5,1) -- (17.5,1);
            
\draw[decoration={brace,raise=5pt,mirror},decorate, below]
            (41 |- 52,-0.8) -- node[below=6pt] {\scriptsize{$1$}}(43|- 0,-0.8);
            \draw[decoration={brace,raise=5pt,mirror},decorate]
            (43|- 0,-0.8) -- node[below=6pt] {\scriptsize{$2$}}(45|- 0,-0.8);
            \draw[decoration={brace,raise=5pt,mirror},decorate]
            (46|- 0,-0.8) -- node[below=6pt] {\scriptsize{$n$}}(48|- 0,-0.8);

        \foreach \x in {2,3,4,5}
        {
        \node[label={[label distance = -3pt]-90:\tiny{$g_{4\x}$}}] at (4\x) {};
        }
        \node[label={[label distance=-3pt]-90:\tiny{$g_{4\,{2n+1}}$}}] at (48) {};
        \node[label={[label distance=-3pt]-90:\tiny{$\cdots$}}] at (46) {};

        \foreach \x in {2,4}
        {
        \node[label={[label distance = -8pt]-30:\tiny{$g_{2\x}$}}] at (2\x) {};
        \node[label={[label distance = -8pt]30:\tiny{$g_{3\x}$}}] at (3\x) {};
        }

        \foreach \x in {3,5}
        {
        \node[label={[label distance = -8pt]30:\tiny{$g_{2\x}$}}] at (2\x) {};
        \node[label={[label distance = -8pt]-30:\tiny{$g_{3\x}$}}] at (3\x) {};
        }

\node[label={[label distance=-3pt]90:\scriptsize{{\tiny{$g_{11}$}}}}] at (11) {};
            \foreach \x in {2,...,5}
            {
            \node[label={[label distance=-3pt]90:\tiny{$g_{1\x}$}}] at (1\x) {};
            }
            \node[label={[label distance=-3pt]180:\tiny{$g_{21}$}}] at (21) {};
            \node[label={[label distance=-3pt]180:\tiny{$g_{31}$}}] at (31) {};
            \node[label={[label distance=-3pt]180:\tiny{$g_{41}$}}] at (41) {};;

            \node[label={[label distance=-3pt]90:\tiny{$g_{1\,{2n+1}}$}}] at (18) {};
            \node[label={[label distance=-3pt]90:\tiny{$\cdots$}}] at (16) {};
            \node[label={[label distance=-3pt]0:\tiny{$g_{2\,{2n+1}}$}}] at (28) {};
            \node[label={[label distance=-3pt]0:\tiny{$g_{3\,{2n+1}}$}}] at (38) {};

            \node[] at (9,-4) {(b) $\Gamma_{n,5}$, for $n\geq 1$};  
            
        \end{scope}

\begin{scope}[shift = {(-2,0)}]
            \foreach \x in {1,2,3,4}
            {\node[vertex] (x\x1) at (0,2*(4-\x) {};}

            \foreach \x [evaluate = \x as \y using int(\x+1)] in {1,2,3}
            {\draw[edge] (x\x1) -- (x\y1) ;}

            \foreach \y in {1,2,3,4}
            {
            \node[label={[label distance=-3pt]0:\tiny{$g_{\y 1}$}}] at (x\y1) {};
            }
            \node[] at (0,-4) {(a) $\Gamma_{0,5}$}; 
        \end{scope}

		\end{tikzpicture}
	\caption{$\Gamma_{n,5}$} \label{fig:Gn}	
	\end{figure}

		\begin{figure}
		% []{0.5\textwidth}
		\centering
		\begin{tikzpicture}[scale=0.3]
			\begin{scope}
				\foreach \x in {1,2,...,8}
				{
					\node[vertex] (1\x) at (3*\x,3*2) {};
				}
				\foreach \x in {1,2,...,8}
				{
					\node[vertex] (2\x) at (3*\x,2*2) {};
				}
				\foreach \x in {1,2,...,8}
				{
					\node[vertex] (3\x) at (3*\x,1*2) {};
				}
				\foreach \x in {1,2,...,8}
				{
					\node[vertex] (4\x) at (3*\x,0*2) {};
				}
				
				\foreach \z in {1,2,3,4}
				{
					\foreach \y [evaluate={\y as \x using {\y+1}}] in {1,2,...,4,6,7} 
					{
						\draw[edge] (\z\y) -- (\z\x);
					}
				}
				
				\foreach \z in {1,2,3,4,5,...,8}
				{\draw[edge] (1\z) -- (2\z);
					\draw[edge] (2\z) -- (3\z);
					\draw[edge] (3\z) -- (4\z);
				}
				
				\foreach \z [evaluate={\znext=int({1+\z})}] in {2,4,7}
				{
					\draw[edge] (1\z) -- (2\znext);
					\draw[edge] (3\z) -- (4\znext);
				}
				\foreach \z [evaluate={\znext=int({\z-1})}] in {3,5,8}
				{
					\draw[edge] (1\z) -- (2\znext);
					\draw[edge] (3\z) -- (4\znext);
				}
				
				\foreach \z [evaluate={\znext=int({1+\z})}] in {1,3,6}
				{
					\draw[edge] (2\z) -- (3\znext);
				}
				\foreach \z [evaluate={\znext=int({\z-1})}] in {2,4,7}
				{
					\draw[edge] (2\z) -- (3\znext);
				}
				\coordinate (c15) at (15);
				\coordinate (c16) at (16);

				\draw[dotted, thick,shorten >=4pt, shorten <=4pt] (15.5,5) -- (17.5,5);
				\draw[dotted, thick,shorten >=4pt, shorten <=4pt] (15.5,3) -- (17.5,3);
				\draw[dotted, thick,shorten >=4pt, shorten <=4pt] (15.5,1) -- (17.5,1);
				
				\draw[decoration={brace,raise=5pt,mirror},decorate, below]
				(41 |- 52,-0.8) -- node[below=6pt] {\scriptsize{$1$}}(43|- 0,-0.8);
				\draw[decoration={brace,raise=5pt,mirror},decorate]
				(43|- 0,-0.8) -- node[below=6pt] {\scriptsize{$2$}}(45|- 0,-0.8);
				\draw[decoration={brace,raise=5pt,mirror},decorate]
				(46|- 0,-0.8) -- node[below=6pt] {\scriptsize{$n$}}(48|- 0,-0.8);

				\foreach \x in {1,2,3,4,5}
				{
					\node[label={[label distance = -3pt]-90:\tiny{$g_{4\x}$}}] at (4\x) {};
				}
				\node[label={[label distance=-3pt]-90:\tiny{$g_{4\,{2n+1}}$}}] at (48) {};
				\node[label={[label distance=-3pt]-90:\tiny{$\cdots$}}] at (46) {};
				\node[label={[label distance=-3pt]0:}] at (28) {};
				\node[label={[label distance=-3pt]0:}] at (38) {};
				
				\foreach \x in {2,4}
				{
					\node[label={[label distance = -8pt]-30:\tiny{$g_{2\x}$}}] at (2\x) {};
					\node[label={[label distance = -8pt]30:\tiny{$g_{3\x}$}}] at (3\x) {};
				}
				\node[label={[label distance = -8pt]190:\tiny{$g_{21}$}}] at (21) {};
				\node[label={[label distance = -8pt]-30:\scriptsize{$g_{31}$}}] at (31) {};
				\foreach \x in {3,5}
				{
					\node[label={[label distance = -8pt]30:\tiny{$g_{2\x}$}}] at (2\x) {};
					\node[label={[label distance = -8pt]-30:\tiny{$g_{3\x}$}}] at (3\x) {};
				}
				
				\node[label={[label distance=-3pt]90:\tiny{$g_{11}$}}] at (11) {};
				\foreach \x in {2,...,5}
				{
					\node[label={[label distance=-3pt]90:\tiny{$g_{1\x}$}}] at (1\x) {};
				}
				
				\node[label={[label distance=-3pt]90:\tiny{$g_{1\,{2n+1}}$}}] at (18) {};
                \node[label={[label distance=-3pt]0:\tiny{$g_{2\,{2n+1}}$}}] at (28) {};
                \node[label={[label distance=-3pt]0:\tiny{$g_{3\,{2n+1}}$}}] at (38) {};
				\node[label={[label distance=-3pt]90:\scriptsize{$\cdots$}}] at (16) {};
			\end{scope}		
			
			\begin{scope}[shift={(2,0)}]
				\node[vertex] (l1) at (-2,6) {} ;
				\node[vertex] (l2) at (-2,4) {} ;
				\node[vertex] (l3) at (-2,2) {} ;
				\node[vertex] (l4) at (-2,0) {} ;

				\draw[edge] (l1) -- (l2) ;
				\draw[edge] (l2) -- (l3) ;
				\draw[edge] (l3) -- (l4) ;
				\draw[edge] (l1) -- (11) ;
				\draw[edge] (l2) -- (21) ;
				\draw[edge] (l3) -- (31) ;
				\draw[edge] (l4) -- (41) ;
				\draw[edge] (l1) -- (21) ;
				\draw[edge] (l2) -- (11) ;
				\draw[edge] (l4) -- (31) ;
				\draw[edge] (l3) -- (41) ;
				
				\node[label={[label distance=-3pt]180:\scriptsize{$l_{1}$}}] at (l1) {};
				\node[label={[label distance=-3pt]180:\scriptsize{$l_{2}$}}] at (l2) {};
				\node[label={[label distance=-3pt]180:\scriptsize{$l_{3}$}}] at (l3) {};
				\node[label={[label distance=-3pt]180:\scriptsize{$l_{4}$}}] at (l4) {};
				
				\node[] at (6,-4) {(b) $\Lambda_{n,5}$, for $n\geq 1$};   
			\end{scope}
			
			\begin{scope}[shift = {(-4.5,0)}]
				\foreach \x in {1,2,3,4}
				{\node[vertex] (x\x1) at (0,2*(4-\x) {};}
				
				\foreach \x [evaluate = \x as \y using int(\x+1)] in {1,2,3}
				{\draw[edge] (x\x1) -- (x\y1) ;}
				
				\node[vertex] (l1) at (-3,6) {} ;
				\node[vertex] (l2) at (-3,4) {} ;
				\node[vertex] (l3) at (-3,2) {} ;
				\node[vertex] (l4) at (-3,0) {} ;

				\draw[edge] (l1) -- (l2) ;
				\draw[edge] (l2) -- (l3) ;
				\draw[edge] (l3) -- (l4) ;
				\draw[edge] (l1) -- (x11) ;
				\draw[edge] (l2) -- (x21) ;
				\draw[edge] (l3) -- (x31) ;
				\draw[edge] (l4) -- (x41) ;
				\draw[edge] (l1) -- (x21) ;
				\draw[edge] (l2) -- (x11) ;
				\draw[edge] (l4) -- (x31) ;
				\draw[edge] (l3) -- (x41) ;
				
				\foreach \y in {1,2,3,4}
				{
					\node[label={[label distance=-3pt]0:\tiny{$g_{\y 1}$}}] at (x\y1) {};
				}
				
				\node[label={[label distance=-3pt]180:\scriptsize{$l_{1}$}}] at (l1) {};
				\node[label={[label distance=-3pt]180:\scriptsize{$l_{2}$}}] at (l2) {};
				\node[label={[label distance=-3pt]180:\scriptsize{$l_{3}$}}] at (l3) {};
				\node[label={[label distance=-3pt]180:\scriptsize{$l_{4}$}}] at (l4) {};
				
				\node[] at (-1.5,-4) {(a) $\Lambda_{0,5}$}; 
			\end{scope}
			
		\end{tikzpicture}
		\caption{$\Lambda_{n, 5}$} \label{fig:Lambda-n}	
	\end{figure}

\begin{definition}
For $n \geq 0$,  we define the graph $\Lambda_{n,5}$ as follows:
\begin{align*}
    V(\Lambda_{n,5}) = &V(\Gamma_{n,5}) \cup \{l_1,l_2,l_3,l_4\} , \\
    E(\Lambda_{n,5}) = & E(\Gamma_{n,5}) \cup\{l_1l_2,l_2l_3,l_3l_4,g_{11}l_1,g_{11}l_2,g_{21}l_1,g_{21}l_2,g_{31}l_3,g_{31}l_4,g_{41}l_3,g_{41}l_4 \}.
\end{align*}
\end{definition}

\begin{definition}
For $n \geq 0$, we define the graph $\widetilde{\Gamma}_{n, 5}$ as follows: 
\begin{equation*}
    \begin{split}
    V(\widetilde{\Gamma}_{n,5}) =  & \ V(\Gamma_{n,5})\cup \{g_{1\,2n+2},g_{2\,2n+2},g_{3\,2n+2},g_{4\,2n+2}\} \\
    E(\widetilde{\Gamma}_{n,5}) = & \ E(\Gamma_{n,5}) \cup \set{g_{i\,2n+1}g_{i\,2n+2} \given i=1,2,3,4}  \cup 
      \set{g_{i\,2n+2}g_{i+1\,2n+2} \given i = 1,2,3} \\
    & \cup \set{g_{2\,2n+1}g_{3\,2n+2}, g_{3\,2n+1}g_{2\,2n+2}}.
    \end{split}
\end{equation*}
\end{definition}

\begin{figure}[ht]
	\centering
	\begin{tikzpicture}[scale=0.3]
		\begin{scope}
			\foreach \x in {1,2,...,9}
			{
				\node[vertex] (1\x) at (3*\x,3*2) {};
			}
			\foreach \x in {1,2,...,9}
			{
				\node[vertex] (2\x) at (3*\x,2*2) {};
			}
			\foreach \x in {1,2,...,9}
			{
				\node[vertex] (3\x) at (3*\x,1*2) {};
			}
			\foreach \x in {1,2,...,9}
			{
				\node[vertex] (4\x) at (3*\x,0*2) {};
			}
			
			\foreach \y [evaluate={\y as \x using int({\y+1})}] in {1,2,...,4,6,7,8} 
			{
				\draw[edge] (1\y) -- (1\x);
				\draw[edge] (2\y) -- (2\x);
				\draw[edge] (3\y) -- (3\x);
				\draw[edge] (4\y) -- (4\x);
			}
			
			\foreach \z in {1,2,3,4,5,...,8,9}
			{\draw[edge] (1\z) -- (2\z);
				\draw[edge] (2\z) -- (3\z);
				\draw[edge] (3\z) -- (4\z);
			}
			
			\foreach \z [evaluate={\znext=int({1+\z})}] in {2,4,7}
			{
				\draw[edge] (1\z) -- (2\znext);
				\draw[edge] (3\z) -- (4\znext);
			}
			\foreach \z [evaluate={\znext=int({\z-1})}] in {3,5,8}
			{
				\draw[edge] (1\z) -- (2\znext);
				\draw[edge] (3\z) -- (4\znext);
			}
			
			\foreach \z [evaluate={\znext=int({1+\z})}] in {1,3,6,8}
			{
				\draw[edge] (2\z) -- (3\znext);
			}
			\foreach \z [evaluate={\znext=int({\z-1})}] in {2,4,7,9}
			{
				\draw[edge] (2\z) -- (3\znext);
			}
			\coordinate (c15) at (15);
			\coordinate (c16) at (16);

			\draw[dotted, thick,shorten >=4pt, shorten <=4pt] (15.5,5) -- (17.5,5);
			\draw[dotted, thick,shorten >=4pt, shorten <=4pt] (15.5,3) -- (17.5,3);
			\draw[dotted, thick,shorten >=4pt, shorten <=4pt] (15.5,1) -- (17.5,1);
			
			\draw[decoration={brace,raise=5pt,mirror},decorate, below]
			(41 |- 52,-0.8) -- node[below=6pt] {\scriptsize{$1$}}(43|- 0,-0.8);
			\draw[decoration={brace,raise=5pt,mirror},decorate]
			(43|- 0,-0.8) -- node[below=6pt] {\scriptsize{$2$}}(45|- 0,-0.8);
			\draw[decoration={brace,raise=5pt,mirror},decorate]
			(46|- 0,-0.8) -- node[below=6pt] {\scriptsize{$n$}}(48|- 0,-0.8);

			\foreach \x in {2,3,4,5}
			{
				\node[label={[label distance = -3pt]-90:\scriptsize{$g_{4\x}$}}] at (4\x) {};
			}

			\foreach \x in {2,4}
			{
				\node[label={[label distance = -8pt]-30:\scriptsize{$g_{2\x}$}}] at (2\x) {};
				\node[label={[label distance = -8pt]30:\scriptsize{$g_{3\x}$}}] at (3\x) {};
			}
			
			\foreach \x in {3,5}
			{
				\node[label={[label distance = -8pt]30:\scriptsize{$g_{2\x}$}}] at (2\x) {};
				\node[label={[label distance = -8pt]-30:\scriptsize{$g_{3\x}$}}] at (3\x) {};
			}

			\node[label={[label distance=-3pt]90:\scriptsize{$g_{11}$}}] at (11) {};
			\foreach \x in {2,...,5}
			{
				\node[label={[label distance=-3pt]90:\scriptsize{$g_{1\x}$}}] at (1\x) {};
			}
			\node[label={[label distance=-3pt]180:\scriptsize{$g_{21}$}}] at (21) {};
			\node[label={[label distance=-3pt]180:\scriptsize{$g_{31}$}}] at (31) {};
			\node[label={[label distance=-3pt]180:\scriptsize{$g_{41}$}}] at (41) {};;
			\node[label={[label distance=-3pt]90:\scriptsize{$g_{1\,{2n+2}}$}}] at (19) {};
			\node[label={[label distance=-3pt]90:\scriptsize{$\cdots\cdots$}}] at (16) {};

			\node[label={[label distance=-3pt]0:\scriptsize{$g_{2\,{2n+2}}$}}] at (29) {};
			\node[label={[label distance=-3pt]0:\scriptsize{$g_{3\,{2n+2}}$}}] at (39) {};
			\node[label={[label distance=-3pt]0:\scriptsize{$g_{4\,{2n+2}}$}}] at (49) {};
			
			\node[] at (14.5,-4) {(b) $\widetilde{\Gamma}_{n,5}$, for $n\geq 1$};   
		\end{scope}

		\begin{scope}[shift={(-5,0)}]
			\foreach \x in {1,2,3,4}
			{\node[vertex] (x\x1) at (-2,2*(4-\x) {};
				\node[vertex] (x\x2) at (1,2*(4-\x) {};
			}
			
			\foreach \x [evaluate = \x as \y using int(\x+1)] in {1,2,3}
			{\draw[edge] (x\x1) -- (x\y1) ;
				\draw[edge] (x\x2) -- (x\y2) ;
			}
			
			\draw[edge] (x11) -- (x12) ;
			\draw[edge] (x21) -- (x22) ;
			\draw[edge] (x21) -- (x32) ;
			\draw[edge] (x31) -- (x22) ;
			\draw[edge] (x31) -- (x32) ;
			\draw[edge] (x41) -- (x42) ;
			
			\foreach \y in {1,2,3,4}
			{
				\node[label={[label distance=-3pt]180:\scriptsize{$g_{\y 1}$}}] at (x\y1) {};
				\node[label={[label distance=-3pt]0:\scriptsize{$g_{\y 2}$}}] at (x\y2) {};
			}
			
			\node[] at (0,-4) {(a) $\widetilde{\Gamma}_{0,5}$}; 
		\end{scope}
		
	\end{tikzpicture}
	\caption{$\widetilde{\Gamma}_{n,5}$} \label{fig:G'n}	
\end{figure}

In the following section,  we compute the homotopy types of $\ind(\Gamma_{n, 5})$ and $\ind(\Lambda_{n, 5})$. The homotopy type of $\ind(\widetilde{\Gamma}_{n, 5})$ is computed in \Cref{subsection:GammaTilde}. Finally, we give the minimum and maximum dimension of spheres in the homotopy type of $\M(P_n \times P_5)$ in \Cref{subsection:PnP5dimensionbound}.

\subsection{\texorpdfstring{Independence complexes of  $\Gamma_{n,5}$ and $\Lambda_{n,5}$}{Independence complexes of Γn,5 and Λn,5}}{\label{subsection:Gamma and Lembda}}

In this section, we introduce $6$ new classes of graphs, namely $A_{n,5}, B_{n,5}, C_{n,5}, D_{n,5},$ $E_{n,5}$, and $F_{n,5}$ for $n \geq 0$. We show that the independence complexes of $\Gamma_{n,5}$ and $\Lambda_{n,5}$ can be expressed in terms of wedges of suspensions of independence complexes of these intermediary graphs. Using induction on $n$, we establish that the independence complexes of $\Gamma_{n,5}$, $\Lambda_{n,5}$, and these intermediary graphs are homotopy equivalent to a wedge of spheres.

We begin with the description of the intermediary graphs. 
% and compute the homotopy type of their independence complexes for $n=0$. 

\subsubsection{Definitions of intermediary graphs} \label{subsection:PnP5Basecase}

\noindent{(i)\bf{ Graph $A_{n,5}$.}}
For $n\geq 0$, we define the graph $A_{n,5}$ as follows:

$V(A_{n,5})= V(\Gamma_{n,5})\cup \{a_1,a_2\}$ and 
$E(A_{n,5})= E(\Gamma_{n,5})\cup \{a_{1}g_{11},a_{1}g_{21},a_{2}g_{31},a_{2}g_{41}\}$.

\begin{figure}[h!]
	\centering
	
	\begin{tikzpicture}[scale=0.3]
		\begin{scope}
			\foreach \x in {1,2,...,8}
			{
				\node[vertex] (1\x) at (3*\x,3*2) {};
			}
			\foreach \x in {1,2,...,8}
			{
				\node[vertex] (2\x) at (3*\x,2*2) {};
			}
			\foreach \x in {1,2,...,8}
			{
				\node[vertex] (3\x) at (3*\x,1*2) {};
			}
			\foreach \x in {1,2,...,8}
			{
				\node[vertex] (4\x) at (3*\x,0*2) {};
			}

			\foreach \z in {1,2,3,4}
			{
				\foreach \y [evaluate={\y as \x using {\y+1}}] in {1,2,...,4,6,7} 
				{
					\draw[edge] (\z\y) -- (\z\x);
				}
			}
			
			\foreach \z in {1,2,3,4,5,...,8}
			{\draw[edge] (1\z) -- (2\z);
				\draw[edge] (2\z) -- (3\z);
				\draw[edge] (3\z) -- (4\z);
			}
			
			\foreach \z [evaluate={\znext=int({1+\z})}] in {2,4,7}
			{
				\draw[edge] (1\z) -- (2\znext);
				\draw[edge] (3\z) -- (4\znext);
			}
			\foreach \z [evaluate={\znext=int({\z-1})}] in {3,5,8}
			{
				\draw[edge] (1\z) -- (2\znext);
				\draw[edge] (3\z) -- (4\znext);
			}
			
			\foreach \z [evaluate={\znext=int({1+\z})}] in {1,3,6}
			{
				\draw[edge] (2\z) -- (3\znext);
			}
			\foreach \z [evaluate={\znext=int({\z-1})}] in {2,4,7}
			{
				\draw[edge] (2\z) -- (3\znext);
			}
			\coordinate (c15) at (15);
			\coordinate (c16) at (16);

			\draw[dotted, thick,shorten >=4pt, shorten <=4pt] (15.5,5) -- (17.5,5);
			\draw[dotted, thick,shorten >=4pt, shorten <=4pt] (15.5,3) -- (17.5,3);
			\draw[dotted, thick,shorten >=4pt, shorten <=4pt] (15.5,1) -- (17.5,1);
			
			\draw[decoration={brace,raise=5pt,mirror},decorate, below]
			(41 |- 52,-0.8) -- node[below=6pt] {\scriptsize{$1$}}(43|- 0,-0.8);
			\draw[decoration={brace,raise=5pt,mirror},decorate]
			(43|- 0,-0.8) -- node[below=6pt] {\scriptsize{$2$}}(45|- 0,-0.8);
			\draw[decoration={brace,raise=5pt,mirror},decorate]
			(46|- 0,-0.8) -- node[below=6pt] {\scriptsize{$n$}}(48|- 0,-0.8);
		\end{scope}

		\foreach \x in {1,2,3,4,5}
		{
			\node[label={[label distance = -3pt]-90:\scriptsize{$g_{4\x}$}}] at (4\x) {};
		}
		\node[label={[label distance=-3pt]-90:\scriptsize{$g_{4\,{2n+1}}$}}] at (48) {};
		\node[label={[label distance=-3pt]-90:\scriptsize{$\cdots$}}] at (46) {};
		\node[label={[label distance=-3pt]0:\scriptsize{$g_{2\,{2n+1}}$}}] at (28) {};
		\node[label={[label distance=-3pt]0:\scriptsize{$g_{3\,{2n+1}}$}}] at (38) {};
		
		\foreach \x in {2,4}
		{
			\node[label={[label distance = -8pt]-30:\scriptsize{$g_{2\x}$}}] at (2\x) {};
			\node[label={[label distance = -8pt]30:\scriptsize{$g_{3\x}$}}] at (3\x) {};
		}
		
		\foreach \x in {3,5}
		{
			\node[label={[label distance = -8pt]30:\scriptsize{$g_{2\x}$}}] at (2\x) {};
			\node[label={[label distance = -8pt]-30:\scriptsize{$g_{3\x}$}}] at (3\x) {};
		}

		\begin{scope}[shift={(2,0)}]
			\node[vertex] (a1) at (-2,6) {} ;
			\node[vertex] (a2) at (-2,0) {} ;
			
			\draw[edge] (a1) -- (11) ;
			\draw[edge] (a1) -- (21) ;
			\draw[edge] (a2) -- (31) ;
			\draw[edge] (a2) -- (41) ;
			
			\node[label={[label distance=-3pt]180:\scriptsize{$a_{1}$}}] at (a1) {};   
			\node[label={[label distance=-3pt]180:\scriptsize{$a_{2}$}}] at (a2) {};
			
			\node[label={[label distance=-3pt]90:\scriptsize{$g_{11}$}}] at (11) {};
			\foreach \x in {2,...,5}
			{
				\node[label={[label distance=-3pt]90:\scriptsize{$g_{1\x}$}}] at (1\x) {};
			}
			\node[label={[label distance=-3pt]90:\scriptsize{$g_{1\,{2n+1}}$}}] at (18) {};
			\node[label={[label distance=-3pt]90:\scriptsize{$\cdots$}}] at (16) {};
			\foreach \x in {2,3}
			{
				\node[label={[label distance=-3pt]180:\scriptsize{$g_{\x1}$}}] at (\x1) {};
			}
			
			\node[] at (11.5,-4) {(b) $A_{n,5}$, for $n\geq 1$};
		\end{scope}
		\begin{scope}[shift = {(-6,0)}]
			\foreach \x in {1,2,3,4}
			{\node[vertex] (x\x1) at (1,2*(4-\x) {};}
			
			\foreach \x [evaluate = \x as \y using int(\x+1)] in {1,2,3}
			{\draw[edge] (x\x1) -- (x\y1) ;}
			
			\node[vertex] (a1) at (-2,6) {} ;
			\node[vertex] (a2) at (-2,0) {} ;
			
			\draw[edge] (a1) -- (x11) ;
			\draw[edge] (a1) -- (x21) ;
			\draw[edge] (a2) -- (x31) ;
			\draw[edge] (a2) -- (x41) ;
			
			\foreach \y in {1,2,3,4}
			{
				\node[label={[label distance=-3pt]0:\scriptsize{$g_{\y 1}$}}] at (x\y1) {};
			}
			\node[label={[label distance=-3pt]180:\scriptsize{$a_{1}$}}] at (a1) {};   
			\node[label={[label distance=-3pt]180:\scriptsize{$a_{2}$}}] at (a2) {};
			
			\node[] at (0,-4) {(a) $A_{0,5}$};
		\end{scope}
		
	\end{tikzpicture}
	\caption{$A_{n,5}$} \label{fig:An}	
\end{figure} 

%\begin{claim}{\label{homA base}}
%$\ind(A_{0,5}) \simeq \Sp^1 \vee \Sp^1 \vee \Sp^1$.
%\end{claim}
%\begin{proof}
%Since $a_1$ is a simplicial vertex in the graph $A_{0,5}$ with $N(a_1)=\{g_{11},g_{21}\}$, using \Cref{Simplicial Vertex Lemma} we get, 
%%\begin{equation*}
%    $\ind(A_{0,5}) \simeq \Sigma\ind(A_{0,5} \setminus \n[g_{11}]) \vee \Sigma\ind(A_{0,5} \setminus \n[g_{21}])$.
%%\end{equation*}
%However, $A_{0,5} \setminus \n[g_{11}]$ is isomorphic to the complete graph $K_3$, and $A_{0,5} \setminus \n[g_{21}]$ is isomorphic to $K_2$. Therefore, $\ind(A_{0,5} \setminus \n[g_{11}])\simeq \mathbb{S}^0\vee \mathbb{S}^0$ and $\ind(A_{0,5} \setminus \n[g_{21}])\simeq \mathbb{S}^0$. Hence, we have the stated conclusion.
%\end{proof}

\noindent{(ii)\bf{ Graph $B_{n,5}$.}}\label{B-1}
For $n\geq 0$, we define the graph $B_{n,5}$ as follows:
\begin{equation*}
    \begin{split}
V(B_{n,5})=& V(\Gamma_{n,5})\cup \{b_1,b_2, b_3,b_4,b_5,b_6\},\\ 
E(B_{n,5})=& E(\Gamma_{n,5})\cup \{b_{1}b_{2}, b_{1}b_{4},b_{1}b_{5},b_{2}b_{6},b_{3}b_{4},b_{4}b_{5},b_{5}b_{6}\}\\
& \cup \{b_{3}g_{11},b_{3}g_{21},b_{4}g_{11},b_{4}g_{21},b_{5}g_{31},b_{5}g_{41},b_{6}g_{31},b_{6}g_{41}\}.
\end{split}
\end{equation*}
% Furthermore, we define $B_{-1,5}$ as a graph on vertex set $\{b_1,b_2\}$, and edge set $\{b_1b_2\}$. Thus, $B_{-1,5} \cong K_2$. This graph particularly occurs while computing $\ind(\Gamma_{2,5})$ in \Cref{homG_{n,5}}.

\begin{figure}[h]
	\centering
	
	\begin{tikzpicture}[scale=0.3]
		\begin{scope}
			\foreach \x in {1,2,...,8}
			{
				\node[vertex] (1\x) at (3*\x,3*2) {};
			}
			\foreach \x in {1,2,...,8}
			{
				\node[vertex] (2\x) at (3*\x,2*2) {};
			}
			\foreach \x in {1,2,...,8}
			{
				\node[vertex] (3\x) at (3*\x,1*2) {};
			}
			\foreach \x in {1,2,...,8}
			{
				\node[vertex] (4\x) at (3*\x,0*2) {};
			}
			
			\foreach \z in {1,2,3,4}
			{
				\foreach \y [evaluate={\y as \x using {\y+1}}] in {1,2,...,4,6,7} 
				{
					\draw[edge] (\z\y) -- (\z\x);
				}
			}
			
			\foreach \z in {1,2,3,4,5,...,8}
			{\draw[edge] (1\z) -- (2\z);
				\draw[edge] (2\z) -- (3\z);
				\draw[edge] (3\z) -- (4\z);
			}
			
			\foreach \z [evaluate={\znext=int({1+\z})}] in {2,4,7}
			{
				\draw[edge] (1\z) -- (2\znext);
				\draw[edge] (3\z) -- (4\znext);
			}
			\foreach \z [evaluate={\znext=int({\z-1})}] in {3,5,8}
			{
				\draw[edge] (1\z) -- (2\znext);
				\draw[edge] (3\z) -- (4\znext);
			}
			
			\foreach \z [evaluate={\znext=int({1+\z})}] in {1,3,6}
			{
				\draw[edge] (2\z) -- (3\znext);
			}
			\foreach \z [evaluate={\znext=int({\z-1})}] in {2,4,7}
			{
				\draw[edge] (2\z) -- (3\znext);
			}
			\coordinate (c15) at (15);
			\coordinate (c16) at (16);

			\draw[dotted, thick,shorten >=4pt, shorten <=4pt] (15.5,5) -- (17.5,5);
			\draw[dotted, thick,shorten >=4pt, shorten <=4pt] (15.5,3) -- (17.5,3);
			\draw[dotted, thick,shorten >=4pt, shorten <=4pt] (15.5,1) -- (17.5,1);
			
			\draw[decoration={brace,raise=5pt,mirror},decorate, below]
			(41 |- 52,-0.8) -- node[below=6pt] {\scriptsize{$1$}}(43|- 0,-0.8);
			\draw[decoration={brace,raise=5pt,mirror},decorate]
			(43|- 0,-0.8) -- node[below=6pt] {\scriptsize{$2$}}(45|- 0,-0.8);
			\draw[decoration={brace,raise=5pt,mirror},decorate]
			(46|- 0,-0.8) -- node[below=6pt] {\scriptsize{$n$}}(48|- 0,-0.8);
		\end{scope}

		\foreach \x in {1,2,3,4,5}
		{
			\node[label={[label distance = -3pt]-90:\scriptsize{$g_{4\x}$}}] at (4\x) {};
		}
		\node[label={[label distance=-3pt]-90:\scriptsize{$g_{4\,{2n+1}}$}}] at (48) {};
		\node[label={[label distance=-3pt]-90:\scriptsize{$\cdots$}}] at (46) {};
		\node[label={[label distance=-3pt]0:\scriptsize{$g_{2\,{2n+1}}$}}] at (28) {};
		\node[label={[label distance=-3pt]0:\scriptsize{$g_{3\,{2n+1}}$}}] at (38) {};
		
		\foreach \x in {2,4}
		{
			\node[label={[label distance = -8pt]-30:\scriptsize{$g_{2\x}$}}] at (2\x) {};
			\node[label={[label distance = -8pt]30:\scriptsize{$g_{3\x}$}}] at (3\x) {};
		}
		
		\foreach \x in {1,3,5}
		{
			\node[label={[label distance = -8pt]30:\scriptsize{$g_{2\x}$}}] at (2\x) {};
			\node[label={[label distance = -8pt]-30:\scriptsize{$g_{3\x}$}}] at (3\x) {};
		}
		
		\begin{scope}[shift={(2,0)}]
			
			\node[vertex] (b1) at (-2,6) {} ;
			\node[vertex] (b2) at (-2,4) {} ;
			\node[vertex] (b3) at (-2,2) {} ;
			\node[vertex] (b4) at (-2,0) {} ;
			\node[vertex] (b31) at (-4,2) {} ;
			\node[vertex] (b41) at (-4,0) {} ;

			\draw[edge] (b1) -- (11) ;
			\draw[edge] (b1) -- (21) ;
			\draw[edge] (b1) -- (b2) ;
			\draw[edge] (b2) -- (b3) ;
			\draw[edge] (b2) -- (21) ;
			\draw[edge] (b2) -- (11) ;
			\draw[edge] (b3) -- (31) ;
			\draw[edge] (b3) -- (41) ;
			\draw[edge] (b3) -- (b4) ;
			\draw[edge] (b4) -- (41) ;
			\draw[edge] (b4) -- (31) ;
			
			\draw[edge] (b31) -- (b3) ;
			\draw[edge] (b31) -- (b2) ;
			\draw[edge] (b31) -- (b41) ;
			\draw[edge] (b41) -- (b4) ;

			\node[label={[label distance=-3pt]180:\scriptsize{$b_{3}$}}] at (b1) {};
			\node[label={[label distance=-3pt]180:\scriptsize{$b_{4}$}}] at (b2) {};
			\node[label={[label distance=-7pt]220:\scriptsize{$b_{5}$}}] at (b3) {};
			\node[label={[label distance=-3pt]270:\scriptsize{$b_{6}$}}] at (b4) {};
			\node[label={[label distance=-3pt]180:\scriptsize{$b_{1}$}}] at (b31) {};
			\node[label={[label distance=-3pt]180:\scriptsize{$b_{2}$}}] at (b41) {};
			
			\node[label={[label distance=-3pt]90:\scriptsize{$g_{11}$}}] at (11) {};
			\foreach \x in {2,...,5}
			{
				\node[label={[label distance=-3pt]90:\scriptsize{$g_{1\x}$}}] at (1\x) {};
			}
			\node[label={[label distance=-3pt]90:\scriptsize{$g_{1\,{2n+1}}$}}] at (18) {};
			\node[label={[label distance=-3pt]90:\scriptsize{$\cdots$}}] at (16) {};
			
			\node[] at (9,-4) {(b) $B_{n,5}$, for $n\geq 1$};  
			
		\end{scope}
		
		\begin{scope}[shift = {(-8,0)}]
			\foreach \x in {1,2,3,4}
			{\node[vertex] (x\x1) at (1,2*(4-\x) {};}
			
			\foreach \x [evaluate = \x as \y using int(\x+1)] in {1,2,3}
			{\draw[edge] (x\x1) -- (x\y1) ;}
			
			\node[vertex] (b1) at (-2,6) {} ;
			\node[vertex] (b2) at (-2,4) {} ;
			\node[vertex] (b3) at (-2,2) {} ;
			\node[vertex] (b4) at (-2,0) {} ;
			\node[vertex] (b31) at (-4,2) {} ;
			\node[vertex] (b41) at (-4,0) {} ;

			\draw[edge] (b1) -- (x11) ;
			\draw[edge] (b1) -- (x21) ;
			\draw[edge] (b1) -- (b2) ;
			\draw[edge] (b2) -- (b3) ;
			\draw[edge] (b2) -- (x21) ;
			\draw[edge] (b2) -- (x11) ;
			\draw[edge] (b3) -- (x31) ;
			\draw[edge] (b3) -- (x41) ;
			\draw[edge] (b3) -- (b4) ;
			\draw[edge] (b4) -- (x41) ;
			\draw[edge] (b4) -- (x31) ;
			
			\draw[edge] (b31) -- (b3) ;
			\draw[edge] (b31) -- (b2) ;
			\draw[edge] (b31) -- (b41) ;
			\draw[edge] (b41) -- (b4) ;

			\node[label={[label distance=-3pt]180:\scriptsize{$b_{3}$}}] at (b1) {};
			\node[label={[label distance=-3pt]180:\scriptsize{$b_{4}$}}] at (b2) {};
			\node[label={[label distance=-7pt]220:\scriptsize{$b_{5}$}}] at (b3) {};
			\node[label={[label distance=-3pt]270:\scriptsize{$b_{6}$}}] at (b4) {};
			\node[label={[label distance=-3pt]180:\scriptsize{$b_{1}$}}] at (b31) {};
			\node[label={[label distance=-3pt]180:\scriptsize{$b_{2}$}}] at (b41) {};
			
			\foreach \y in {1,2,3,4}
			{
				\node[label={[label distance=-3pt]0:\scriptsize{$g_{\y 1}$}}] at (x\y1) {};
			}
			
			\node[] at (-1,-4) {(a) $B_{0,5}$}; 
		\end{scope}
		
	\end{tikzpicture}
	\caption{$B_{n,5}$} \label{fig:Bn}	
\end{figure}

\noindent{(iii)\bf{ Graph $C_{n,5}$.}}
For $n\geq 0$, we define the graph $C_{n,5}$ as follows: 

$V(C_{n,5})=V(\Gamma_{n,5})\cup \{c_1\},$ and 
$E(C_{n,5})=E(\Gamma_{n,5})\cup \{c_{1}g_{31},c_{1}g_{41}\}$.

\begin{figure}[h]
	\centering
	
	\begin{tikzpicture}[scale=0.3]
		\begin{scope}
			\foreach \x in {1,2,...,8}
			{
				\node[vertex] (1\x) at (3*\x,3*2) {};
			}
			\foreach \x in {1,2,...,8}
			{
				\node[vertex] (2\x) at (3*\x,2*2) {};
			}
			\foreach \x in {1,2,...,8}
			{
				\node[vertex] (3\x) at (3*\x,1*2) {};
			}
			\foreach \x in {1,2,...,8}
			{
				\node[vertex] (4\x) at (3*\x,0*2) {};
			}
			
			\foreach \z in {1,2,3,4}
			{
				\foreach \y [evaluate={\y as \x using {\y+1}}] in {1,2,...,4,6,7} 
				{
					\draw[edge] (\z\y) -- (\z\x);
				}
			}
			
			\foreach \z in {1,2,3,4,5,...,8}
			{\draw[edge] (1\z) -- (2\z);
				\draw[edge] (2\z) -- (3\z);
				\draw[edge] (3\z) -- (4\z);
			}
			
			\foreach \z [evaluate={\znext=int({1+\z})}] in {2,4,7}
			{
				\draw[edge] (1\z) -- (2\znext);
				\draw[edge] (3\z) -- (4\znext);
			}
			\foreach \z [evaluate={\znext=int({\z-1})}] in {3,5,8}
			{
				\draw[edge] (1\z) -- (2\znext);
				\draw[edge] (3\z) -- (4\znext);
			}
			
			\foreach \z [evaluate={\znext=int({1+\z})}] in {1,3,6}
			{
				\draw[edge] (2\z) -- (3\znext);
			}
			\foreach \z [evaluate={\znext=int({\z-1})}] in {2,4,7}
			{
				\draw[edge] (2\z) -- (3\znext);
			}
			\coordinate (c15) at (15);
			\coordinate (c16) at (16);

			\draw[dotted, thick,shorten >=4pt, shorten <=4pt] (15.5,5) -- (17.5,5);
			\draw[dotted, thick,shorten >=4pt, shorten <=4pt] (15.5,3) -- (17.5,3);
			\draw[dotted, thick,shorten >=4pt, shorten <=4pt] (15.5,1) -- (17.5,1);
			
			\draw[decoration={brace,raise=5pt,mirror},decorate, below]
			(41 |- 52,-0.8) -- node[below=6pt] {\scriptsize{$1$}}(43|- 0,-0.8);
			\draw[decoration={brace,raise=5pt,mirror},decorate]
			(43|- 0,-0.8) -- node[below=6pt] {\scriptsize{$2$}}(45|- 0,-0.8);
			\draw[decoration={brace,raise=5pt,mirror},decorate]
			(46|- 0,-0.8) -- node[below=6pt] {\scriptsize{$n$}}(48|- 0,-0.8);

			\foreach \x in {1,2,3,4,5}
			{
				\node[label={[label distance = -3pt]-90:\scriptsize{$g_{4\x}$}}] at (4\x) {};
			}
			\node[label={[label distance=-3pt]-90:\scriptsize{$g_{4\,{2n+1}}$}}] at (48) {};
			\node[label={[label distance=-3pt]-90:\scriptsize{$\cdots$}}] at (46) {};
			\node[label={[label distance=-3pt]0:\scriptsize{$g_{2\,{2n+1}}$}}] at (28) {};
			\node[label={[label distance=-3pt]0:\scriptsize{$g_{3\,{2n+1}}$}}] at (38) {};
			
			\foreach \x in {2,4}
			{
				\node[label={[label distance = -8pt]-30:\scriptsize{$g_{2\x}$}}] at (2\x) {};
				\node[label={[label distance = -8pt]30:\scriptsize{$g_{3\x}$}}] at (3\x) {};
			}
			
			\foreach \x in {3,5}
			{
				\node[label={[label distance = -8pt]30:\scriptsize{$g_{2\x}$}}] at (2\x) {};
				\node[label={[label distance = -8pt]-30:\scriptsize{$g_{3\x}$}}] at (3\x) {};
			}
			
			\node[label={[label distance=-3pt]90:\scriptsize{$g_{11}$}}] at (11) {};
			\foreach \x in {2,...,5}
			{
				\node[label={[label distance=-3pt]90:\scriptsize{$g_{1\x}$}}] at (1\x) {};
			} 
			\node[label={[label distance=-3pt]90:\scriptsize{$g_{1\,{2n+1}}$}}] at (18) {};
			\node[label={[label distance=-3pt]90:\scriptsize{$\cdots$}}] at (16) {};
			\foreach \x in {2,3}
			{
				\node[label={[label distance=-3pt]180:\scriptsize{$g_{\x1}$}}] at (\x1) {};
			}
			\node[] at (11,-4) {(b) $C_{n,5}$, for $n\geq 1$};
		\end{scope}

		\begin{scope}[shift={(2,0)}]
			\node[vertex] (c2) at (-2,0) {} ;
			
			\draw[edge] (c2) -- (31) ;
			\draw[edge] (c2) -- (41) ;
			
			\node[label={[label distance=-3pt]270:\scriptsize{$c_{1}$}}] at (c2) {};
		\end{scope}

		\begin{scope}[shift = {(-5,0)}]
			\foreach \x in {1,2,3,4}
			{\node[vertex] (x\x1) at (1,2*(4-\x) {};}
			
			\foreach \x [evaluate = \x as \y using int(\x+1)] in {1,2,3}
			{\draw[edge] (x\x1) -- (x\y1) ;}
			
			\node[vertex] (c2) at (-2,0) {} ;
			
			\draw[edge] (c2) -- (x31) ;
			\draw[edge] (c2) -- (x41) ;
			
			\foreach \y in {1,2,3,4}
			{
				\node[label={[label distance=-3pt]0:\scriptsize{$g_{\y 1}$}}] at (x\y1) {};
			}
			\node[label={[label distance=-3pt]270:\scriptsize{$c_{1}$}}] at (c2) {};
			
			\node[] at (0,-4) {(a) $C_{0,5}$};
		\end{scope}
	\end{tikzpicture}
	\caption{$C_{n,5}$} \label{fig:Cn}	
\end{figure} 

\medskip

\noindent{(iv)\bf{ Graph $D_{n,5}$.}}
For $n\geq 0$, we define the graph $D_{n,5}$ as follows: 

$V(D_{n,5})=  V(\Gamma_{n,5})\cup \{d_{1},d_{2}\}$, and 
$E(D_{n,5})=  E(\Gamma_{n,5})\cup \{d_{1}d_{2},d_{1}g_{11},d_{1}g_{21},d_{2}g_{31},d_{2}g_{41}\}$.

\begin{figure}[h]
	\centering
	
	\begin{tikzpicture}[scale=0.3]
		\begin{scope}
			\foreach \x in {1,2,...,8}
			{
				\node[vertex] (1\x) at (3*\x,3*2) {};
			}
			\foreach \x in {1,2,...,8}
			{
				\node[vertex] (2\x) at (3*\x,2*2) {};
			}
			\foreach \x in {1,2,...,8}
			{
				\node[vertex] (3\x) at (3*\x,1*2) {};
			}
			\foreach \x in {1,2,...,8}
			{
				\node[vertex] (4\x) at (3*\x,0*2) {};
			}
			
			\foreach \z in {1,2,3,4}
			{
				\foreach \y [evaluate={\y as \x using {\y+1}}] in {1,2,...,4,6,7} 
				{
					\draw[edge] (\z\y) -- (\z\x);
				}
			}
			
			\foreach \z in {1,2,3,4,5,...,8}
			{\draw[edge] (1\z) -- (2\z);
				\draw[edge] (2\z) -- (3\z);
				\draw[edge] (3\z) -- (4\z);
			}
			
			\foreach \z [evaluate={\znext=int({1+\z})}] in {2,4,7}
			{
				\draw[edge] (1\z) -- (2\znext);
				\draw[edge] (3\z) -- (4\znext);
			}
			\foreach \z [evaluate={\znext=int({\z-1})}] in {3,5,8}
			{
				\draw[edge] (1\z) -- (2\znext);
				\draw[edge] (3\z) -- (4\znext);
			}
			
			\foreach \z [evaluate={\znext=int({1+\z})}] in {1,3,6}
			{
				\draw[edge] (2\z) -- (3\znext);
			}
			\foreach \z [evaluate={\znext=int({\z-1})}] in {2,4,7}
			{
				\draw[edge] (2\z) -- (3\znext);
			}
			\coordinate (c15) at (15);
			\coordinate (c16) at (16);

			\draw[dotted, thick,shorten >=4pt, shorten <=4pt] (15.5,5) -- (17.5,5);
			\draw[dotted, thick,shorten >=4pt, shorten <=4pt] (15.5,3) -- (17.5,3);
			\draw[dotted, thick,shorten >=4pt, shorten <=4pt] (15.5,1) -- (17.5,1);
			
			\draw[decoration={brace,raise=5pt,mirror},decorate, below]
			(41 |- 52,-0.8) -- node[below=6pt] {\scriptsize{$1$}}(43|- 0,-0.8);
			\draw[decoration={brace,raise=5pt,mirror},decorate]
			(43|- 0,-0.8) -- node[below=6pt] {\scriptsize{$2$}}(45|- 0,-0.8);
			\draw[decoration={brace,raise=5pt,mirror},decorate]
			(46|- 0,-0.8) -- node[below=6pt] {\scriptsize{$n$}}(48|- 0,-0.8);

			\foreach \x in {1,2,3,4,5}
			{
				\node[label={[label distance = -3pt]-90:\scriptsize{$g_{4\x}$}}] at (4\x) {};
			}
			\node[label={[label distance=-3pt]-90:\scriptsize{$g_{4\,{2n+1}}$}}] at (48) {};
			\node[label={[label distance=-3pt]-90:\scriptsize{$\cdots$}}] at (46) {};
			\node[label={[label distance=-3pt]0:\scriptsize{$g_{2\,{2n+1}}$}}] at (28) {};
			\node[label={[label distance=-3pt]0:\scriptsize{$g_{3\,{2n+1}}$}}] at (38) {};
			
			\foreach \x in {2,4}
			{
				\node[label={[label distance = -8pt]-30:\scriptsize{$g_{2\x}$}}] at (2\x) {};
				\node[label={[label distance = -8pt]30:\scriptsize{$g_{3\x}$}}] at (3\x) {};
			}
			
			\foreach \x in {1,3,5}
			{
				\node[label={[label distance = -8pt]30:\scriptsize{$g_{2\x}$}}] at (2\x) {};
				\node[label={[label distance = -8pt]-30:\scriptsize{$g_{3\x}$}}] at (3\x) {};
			}

			\node[label={[label distance=-3pt]90:\scriptsize{$g_{11}$}}] at (11) {};
			\foreach \x in {2,...,5}
			{
				\node[label={[label distance=-3pt]90:\scriptsize{$g_{1\x}$}}] at (1\x) {};
			}
			\node[label={[label distance=-3pt]90:\scriptsize{$g_{1\,{2n+1}}$}}] at (18) {};
			\node[label={[label distance=-3pt]90:\scriptsize{$\cdots$}}] at (16) {};
		\end{scope}		
		
		\begin{scope}[shift={(2,0)}]
			\node[vertex] (d1) at (-2,4) {} ;
			\node[vertex] (d2) at (-2,2) {} ;
			
			\draw[edge] (d1) -- (11) ;
			\draw[edge] (d1) -- (21) ;
			\draw[edge] (d1) -- (d2) ;
			\draw[edge] (d2) -- (31) ;
			\draw[edge] (d2) -- (41) ;
			
			\node[label={[label distance=-3pt]90:\scriptsize{$d_{1}$}}] at (d1) {};
			\node[label={[label distance=-3pt]270:\scriptsize{$d_{2}$}}] at (d2) {};
			
			\node[] at (11,-4) {(b) $D_{n,5}$, for $n\geq 1$};   
			
		\end{scope}
		
		\begin{scope}[shift = {(-5,0)}]
			\foreach \x in {1,2,3,4}
			{\node[vertex] (x\x1) at (1,2*(4-\x) {};}
			
			\foreach \x [evaluate = \x as \y using int(\x+1)] in {1,2,3}
			{\draw[edge] (x\x1) -- (x\y1) ;}
			
			\node[vertex] (d1) at (-2,4) {} ;
			\node[vertex] (d2) at (-2,2) {} ;
			
			\draw[edge] (d1) -- (x11) ;
			\draw[edge] (d1) -- (x21) ;
			\draw[edge] (d1) -- (d2) ;
			\draw[edge] (d2) -- (x31) ;
			\draw[edge] (d2) -- (x41) ;
			
			\foreach \y in {1,2,3,4}
			{
				\node[label={[label distance=-3pt]0:\scriptsize{$g_{\y 1}$}}] at (x\y1) {};
			}
			
			\node[label={[label distance=-3pt]90:\scriptsize{$d_{1}$}}] at (d1) {};
			\node[label={[label distance=-3pt]270:\scriptsize{$d_{2}$}}] at (d2) {};
			
			\node[] at (0,-4) {(a) $D_{0,5}$}; 
			
		\end{scope}
		
	\end{tikzpicture}
	\caption{$D_{n,5}$} \label{fig:Dn}	
\end{figure}

\noindent{(v)\bf{ Graph $E_{n,5}$.}}
For $n\geq 0$, we define the graph $E_{n,5}$ as follows: 

$V(E_{n,5})=V(\Gamma_{n,5})\cup \{e_{1}\}$ and 
$E(E_{n,5})=E(\Gamma_{n,5})\cup \{e_{1}g_{11},e_{1}g_{21},e_{1}g_{31},e_{1}g_{41}\}$.

\begin{figure}[h]
    \centering
    
    \begin{tikzpicture}[scale=0.3]
        \begin{scope}
\foreach \x in {1,2,...,8}
            {
            \node[vertex] (1\x) at (3*\x,3*2) {};
            }
\foreach \x in {1,2,...,8}
            {
            \node[vertex] (2\x) at (3*\x,2*2) {};
            }
\foreach \x in {1,2,...,8}
            {
            \node[vertex] (3\x) at (3*\x,1*2) {};
            }
\foreach \x in {1,2,...,8}
            {
            \node[vertex] (4\x) at (3*\x,0*2) {};
            }
            
\foreach \z in {1,2,3,4}
            {
            \foreach \y [evaluate={\y as \x using {\y+1}}] in {1,2,...,4,6,7} 
            {
            \draw[edge] (\z\y) -- (\z\x);
            }
            }
            
\foreach \z in {1,2,3,4,5,...,8}
            {\draw[edge] (1\z) -- (2\z);
            \draw[edge] (2\z) -- (3\z);
            \draw[edge] (3\z) -- (4\z);
            }
            
\foreach \z [evaluate={\znext=int({1+\z})}] in {2,4,7}
            {
            \draw[edge] (1\z) -- (2\znext);
            \draw[edge] (3\z) -- (4\znext);
            }
            \foreach \z [evaluate={\znext=int({\z-1})}] in {3,5,8}
            {
            \draw[edge] (1\z) -- (2\znext);
            \draw[edge] (3\z) -- (4\znext);
            }

\foreach \z [evaluate={\znext=int({1+\z})}] in {1,3,6}
            {
            \draw[edge] (2\z) -- (3\znext);
            }
            \foreach \z [evaluate={\znext=int({\z-1})}] in {2,4,7}
            {
            \draw[edge] (2\z) -- (3\znext);
            }
            \coordinate (c15) at (15);
            \coordinate (c16) at (16);

\draw[dotted, thick,shorten >=4pt, shorten <=4pt] (15.5,5) -- (17.5,5);
			\draw[dotted, thick,shorten >=4pt, shorten <=4pt] (15.5,3) -- (17.5,3);
			\draw[dotted, thick,shorten >=4pt, shorten <=4pt] (15.5,1) -- (17.5,1);
            
\draw[decoration={brace,raise=5pt,mirror},decorate, below]
            (41 |- 52,-0.8) -- node[below=6pt] {\scriptsize{$1$}}(43|- 0,-0.8);
            \draw[decoration={brace,raise=5pt,mirror},decorate]
            (43|- 0,-0.8) -- node[below=6pt] {\scriptsize{$2$}}(45|- 0,-0.8);
            \draw[decoration={brace,raise=5pt,mirror},decorate]
            (46|- 0,-0.8) -- node[below=6pt] {\scriptsize{$n$}}(48|- 0,-0.8);

        \foreach \x in {1,2,3,4,5}
        {
        \node[label={[label distance = -3pt]-90:\scriptsize{$g_{4\x}$}}] at (4\x) {};
        }
        \node[label={[label distance=-3pt]-90:\scriptsize{$g_{4\,{2n+1}}$}}] at (48) {};
        \node[label={[label distance=-3pt]-90:\scriptsize{$\cdots$}}] at (46) {};
        \node[label={[label distance=-3pt]0:\scriptsize{$g_{2\,{2n+1}}$}}] at (28) {};
        \node[label={[label distance=-3pt]0:\scriptsize{$g_{3\,{2n+1}}$}}] at (38) {};

        \foreach \x in {2,4}
        {
        \node[label={[label distance = -8pt]-30:\scriptsize{$g_{2\x}$}}] at (2\x) {};
        \node[label={[label distance = -8pt]30:\scriptsize{$g_{3\x}$}}] at (3\x) {};
        }

        \foreach \x in {1,3,5}
        {
        \node[label={[label distance = -8pt]30:\scriptsize{$g_{2\x}$}}] at (2\x) {};
        \node[label={[label distance = -8pt]-30:\scriptsize{$g_{3\x}$}}] at (3\x) {};
        }

\node[label={[label distance=-3pt]90:\scriptsize{$g_{11}$}}] at (11) {};
            \foreach \x in {2,...,5}
            {
            \node[label={[label distance=-3pt]90:\scriptsize{$g_{1\x}$}}] at (1\x) {};
            }
            \node[label={[label distance=-3pt]90:\scriptsize{$g_{1\,{2n+1}}$}}] at (18) {};
            \node[label={[label distance=-3pt]90:\scriptsize{$\cdots$}}] at (16) {};

            \node[] at (11,-4) {(b) $E_{n,5}$, for $n\geq 1$};   

        \end{scope}		

\begin{scope}[shift={(2,0)}]
            \node[vertex] (e1) at (-2,3) {} ;
            
            \draw[edge] (e1) -- (11) ;
            \draw[edge] (e1) -- (21) ;
            \draw[edge] (e1) -- (31) ;
            \draw[edge] (e1) -- (41) ;

            \node[label={[label distance=-3pt]180:\scriptsize{$e_{1}$}}] at (e1) {};
        \end{scope}

\begin{scope}[shift = {(-5,0)}]
            \foreach \x in {1,2,3,4}
            {\node[vertex] (x\x1) at (1,2*(4-\x) {};}

            \foreach \x [evaluate = \x as \y using int(\x+1)] in {1,2,3}
            {\draw[edge] (x\x1) -- (x\y1) ;}

            \node[vertex] (e1) at (-2,3) {} ;
            
            \draw[edge] (e1) -- (x11) ;
            \draw[edge] (e1) -- (x21) ;
            \draw[edge] (e1) -- (x31) ;
            \draw[edge] (e1) -- (x41) ;

            \node[label={[label distance=-3pt]180:\scriptsize{$e_{1}$}}] at (e1) {};
            \foreach \y in {1,2,3,4}
            {
            \node[label={[label distance=-3pt]0:\scriptsize{$g_{\y 1}$}}] at (x\y1) {};
            }

            \node[] at (0,-4) {(a) $E_{0,5}$}; 
        \end{scope}

    \end{tikzpicture}
    \caption{$E_{n,5}$} \label{fig:En}	
\end{figure}

\noindent{(vi)\bf{ Graph $F_{n,5}$.}}
For $n\geq 0$, we define the graph $F_{n,5}$ as follows:
\begin{equation*}
    \begin{split}
V(F_{n,5})=& V(\Gamma_{n,5})\cup \{f_1,f_2, f_3\},\\
E(F_{n,5})=& E(\Gamma_{n,5})\cup \{f_{1}f_{2}, f_{2}f_{3},f_{1}g_{11},f_{1}g_{21},f_{2}g_{31},f_{2}g_{41},f_{3}g_{31},f_{3}g_{41}\}.
\end{split}
\end{equation*}

\begin{figure}[h]
    \centering
    
    \begin{tikzpicture}[scale=0.3]
        \begin{scope}
\foreach \x in {1,2,...,8}
            {
            \node[vertex] (1\x) at (3*\x,3*2) {};
            }
\foreach \x in {1,2,...,8}
            {
            \node[vertex] (2\x) at (3*\x,2*2) {};
            }
\foreach \x in {1,2,...,8}
            {
            \node[vertex] (3\x) at (3*\x,1*2) {};
            }
\foreach \x in {1,2,...,8}
            {
            \node[vertex] (4\x) at (3*\x,0*2) {};
            }
            
\foreach \z in {1,2,3,4}
            {
            \foreach \y [evaluate={\y as \x using {\y+1}}] in {1,2,...,4,6,7} 
            {
            \draw[edge] (\z\y) -- (\z\x);
            }
            }
            
\foreach \z in {1,2,3,4,5,...,8}
            {\draw[edge] (1\z) -- (2\z);
            \draw[edge] (2\z) -- (3\z);
            \draw[edge] (3\z) -- (4\z);
            }
            
\foreach \z [evaluate={\znext=int({1+\z})}] in {2,4,7}
            {
            \draw[edge] (1\z) -- (2\znext);
            \draw[edge] (3\z) -- (4\znext);
            }
            \foreach \z [evaluate={\znext=int({\z-1})}] in {3,5,8}
            {
            \draw[edge] (1\z) -- (2\znext);
            \draw[edge] (3\z) -- (4\znext);
            }

\foreach \z [evaluate={\znext=int({1+\z})}] in {1,3,6}
            {
            \draw[edge] (2\z) -- (3\znext);
            }
            \foreach \z [evaluate={\znext=int({\z-1})}] in {2,4,7}
            {
            \draw[edge] (2\z) -- (3\znext);
            }
            \coordinate (c15) at (15);
            \coordinate (c16) at (16);

\draw[dotted, thick,shorten >=4pt, shorten <=4pt] (15.5,5) -- (17.5,5);
			\draw[dotted, thick,shorten >=4pt, shorten <=4pt] (15.5,3) -- (17.5,3);
			\draw[dotted, thick,shorten >=4pt, shorten <=4pt] (15.5,1) -- (17.5,1);
            
\draw[decoration={brace,raise=5pt,mirror},decorate, below]
            (41 |- 52,-0.8) -- node[below=6pt] {\scriptsize{$1$}}(43|- 0,-0.8);
            \draw[decoration={brace,raise=5pt,mirror},decorate]
            (43|- 0,-0.8) -- node[below=6pt] {\scriptsize{$2$}}(45|- 0,-0.8);
            \draw[decoration={brace,raise=5pt,mirror},decorate]
            (46|- 0,-0.8) -- node[below=6pt] {\scriptsize{$n$}}(48|- 0,-0.8);

        \foreach \x in {1,2,3,4,5}
        {
        \node[label={[label distance = -3pt]-90:\scriptsize{$g_{4\x}$}}] at (4\x) {};
        }
        \node[label={[label distance=-3pt]-90:\scriptsize{$g_{4\,{2n+1}}$}}] at (48) {};
        \node[label={[label distance=-3pt]-90:\scriptsize{$\cdots$}}] at (46) {};
        \node[label={[label distance=-3pt]0:\scriptsize{$g_{2\,{2n+1}}$}}] at (28) {};
        \node[label={[label distance=-3pt]0:\scriptsize{$g_{3\,{2n+1}}$}}] at (38) {};

        \foreach \x in {2,4}
        {
        \node[label={[label distance = -8pt]-30:\scriptsize{$g_{2\x}$}}] at (2\x) {};
        \node[label={[label distance = -8pt]30:\scriptsize{$g_{3\x}$}}] at (3\x) {};
        }

        \foreach \x in {1,3,5}
        {
        \node[label={[label distance = -8pt]30:\scriptsize{$g_{2\x}$}}] at (2\x) {};
        \node[label={[label distance = -8pt]-30:\scriptsize{$g_{3\x}$}}] at (3\x) {};
        }

\node[label={[label distance=-3pt]90:\scriptsize{$g_{11}$}}] at (11) {};
            \foreach \x in {2,...,5}
            {
            \node[label={[label distance=-3pt]90:\scriptsize{$g_{1\x}$}}] at (1\x) {};
            }

            \node[label={[label distance=-3pt]90:\scriptsize{$g_{1\,{2n+1}}$}}] at (18) {};
            \node[label={[label distance=-3pt]90:\scriptsize{$\cdots$}}] at (16) {};

            \node[] at (11,-4) {(b) $F_{n,5}$, for $n\geq 1$};   
        \end{scope}		

\begin{scope}[shift={(2,0)}]
            \node[vertex] (f2) at (-2,4) {} ;
            \node[vertex] (f3) at (-2,2) {} ;
            \node[vertex] (f4) at (-2,0) {} ;

            \draw[edge] (f2) -- (f3) ;
            \draw[edge] (f3) -- (f4) ;
            \draw[edge] (f2) -- (21) ;
            \draw[edge] (f3) -- (31) ;
            \draw[edge] (f4) -- (41) ;
            \draw[edge] (f2) -- (11) ;
            \draw[edge] (f4) -- (31) ;
            \draw[edge] (f3) -- (41) ;

            \node[label={[label distance=-3pt]180:\scriptsize{$f_{1}$}}] at (f2) {};
            \node[label={[label distance=-3pt]180:\scriptsize{$f_{2}$}}] at (f3) {};
            \node[label={[label distance=-3pt]180:\scriptsize{$f_{3}$}}] at (f4) {};
            
        \end{scope}
\begin{scope}[shift = {(-5,0)}]
            \foreach \x in {1,2,3,4}
            {\node[vertex] (x\x1) at (1,2*(4-\x) {};}

            \foreach \x [evaluate = \x as \y using int(\x+1)] in {1,2,3}
            {\draw[edge] (x\x1) -- (x\y1) ;}

            \node[vertex] (f2) at (-2,4) {} ;
            \node[vertex] (f3) at (-2,2) {} ;
            \node[vertex] (f4) at (-2,0) {} ;

            \draw[edge] (f2) -- (f3) ;
            \draw[edge] (f3) -- (f4) ;
            \draw[edge] (f2) -- (x21) ;
            \draw[edge] (f3) -- (x31) ;
            \draw[edge] (f4) -- (x41) ;
            \draw[edge] (f2) -- (x11) ;
            \draw[edge] (f4) -- (x31) ;
            \draw[edge] (f3) -- (x41) ;

            \node[label={[label distance=-3pt]180:\scriptsize{$f_{1}$}}] at (f2) {};
            \node[label={[label distance=-3pt]180:\scriptsize{$f_{2}$}}] at (f3) {};
            \node[label={[label distance=-3pt]180:\scriptsize{$f_{3}$}}] at (f4) {};
            
            \foreach \y in {1,2,3,4}
            {
            \node[label={[label distance=-3pt]0:\scriptsize{$g_{\y 1}$}}] at (x\y1) {};
            }

            \node[] at (0,-4) {(a) $F_{0,5}$}; 
 
        \end{scope}
        
    \end{tikzpicture}
    \caption{$F_{n,5}$} \label{fig:Fn}	
\end{figure}

\subsubsection{ Independence complexes  of $\Gamma_{n,5},\Lambda_{n,5}$ and  intermediary graphs defined in \Cref{subsection:PnP5Basecase}}\label{subsection:GeneralCasePnP5} 

In this section,  we  establish recursive relations among the independence complexes of the graphs in 
$\{\Gamma_{n,5},\Lambda_{n,5},A_{n,5},B_{n,5},C_{n,5},D_{n,5},E_{n,5}, $ $F_{n,5}\}$. 

\begin{claim}{\label{homG_{n,5}}}
    {\small
    \begin{equation*}
        \ind(\Gamma_{n,5}) \simeq \begin{cases}
        a \	point &   \text{if $n = 0$,} \\
             \bigvee^3\mathbb{S}^2 & \text{if $n = 1$,} \\
           \Sigma^3\ind(D_{n-2,5}) \vee\Sigma^3\ind(D_{n-2,5}) \vee \Sigma^3\ind(A_{n-2,5}) & \text{if $n \geq 2$.}
        \end{cases}
    \end{equation*}
    }
\end{claim}

    \begin{figure}[h!]
		\centering   
		\begin{subfigure}[b]{0.28\linewidth}
			\begin{tikzpicture}[scale=0.3]
				\begin{scope}
					% vertices in row 1
					\foreach \x in {1}
					{
						\node[vertex] (1\x) at (2*\x,3*2) {};
					}
				\foreach \x in {2,3}
						{
							\node[vertex] (1\x) at (3*\x,3*2) {};
					}
					% vertices in row 2
					\foreach \x in {3}
					{
						\node[vertex] (2\x) at (3*\x,2*2) {};
					}
					% vertices in row 3
					\foreach \x in {3}
					{
						\node[vertex] (3\x) at (3*\x,1*2) {};
					}
					% vertices in row 4
					\foreach \x in {2,3}
					{
						\node[vertex] (4\x) at (3*\x,0*2) {};
					}
					
					\node[vertex] (21) at (2,4) {};
					\node[vertex] (31) at (2,2) {};
					\node[vertex] (41) at (2,0) {};
					
					% Horizontal edges
					\foreach \y [evaluate={\y as \x using int({\y+1})}] in {1,2} 
					{
						\draw[edge] (1\y) -- (1\x);
					}
					
					% Horizontal edges
					\foreach \y [evaluate={\y as \x using int({\y+1})}] in {1,2} 
					{
						\draw[edge] (4\y) -- (4\x);
					}

					% Vertical edges               
					\foreach \z in {1,3}
					{
						\draw[edge] (1\z) -- (2\z);
					}
					\foreach \z in {1,3}
					{
						\draw[edge] (2\z) -- (3\z);}
					\foreach \z in {1,3}
					{
						\draw[edge] (3\z) -- (4\z);
					}
					
					% Extra edges
					
					% \draw[edge] (41) -- (12);
					\draw[edge] (42) -- (33);
					\draw[edge] (12) -- (23);

					\draw[dotted, thick,shorten >=4pt, shorten <=4pt] (9,5) -- (11,5);
					\draw[dotted, thick,shorten >=4pt, shorten <=4pt] (9,3) -- (11,3);
					\draw[dotted, thick,shorten >=4pt, shorten <=4pt] (9,1) -- (11,1);
					
					\foreach \x in {1,2,3}
					{
						\node[label={[label distance=-3pt]90:\footnotesize{$g_{1\x}$}}] at (1\x) {};
						\node[label={[label distance=-3pt]270:\footnotesize{$g_{4\x}$}}] at (4\x) {};
					}

					\node[label={[label distance=-3pt]180:\footnotesize{$g_{21}$}}] at (21) {};
					\node[label={[label distance=-3pt]180:\footnotesize{$g_{31}$}}] at (31) {};

					\node[label={[label distance=-3pt]180:\footnotesize{$g_{23}$}}] at (23) {};
					\node[label={[label distance=-3pt]180:\footnotesize{$g_{33}$}}] at (33) {};

				\end{scope}
			\end{tikzpicture}
			\caption{$\Gamma_{n,5}^1$}
			\label{subfig: gamma-n-15}
		\end{subfigure}
		\begin{subfigure}[b]{0.28\linewidth}
			\begin{tikzpicture}[scale=0.3]
				\begin{scope}
					% vertices in row 1
					\foreach \x in {1}
				{
					\node[vertex] (1\x) at (2*\x,3*2) {};
				}
				\foreach \x in {2,3}
				{
					\node[vertex] (1\x) at (3*\x,3*2) {};
				}
					% vertices in row 2
					\foreach \x in {3}
					{
						\node[vertex] (2\x) at (3*\x,2*2) {};
					}
					% vertices in row 3
					\foreach \x in {3}
					{
						\node[vertex] (3\x) at (3*\x,1*2) {};
					}
					% vertices in row 4
					\foreach \x in {2,3}
					{
						\node[vertex] (4\x) at (3*\x,0*2) {};
					}
					
					\node[vertex] (21) at (2,4) {};
					\node[vertex] (31) at (2,2) {};
					\node[vertex] (41) at (2,0) {};
					
					% Horizontal edges
					\foreach \y [evaluate={\y as \x using int({\y+1})}] in {1,2} 
					{
						\draw[edge] (1\y) -- (1\x);
					}
					
					% Horizontal edges
					\foreach \y [evaluate={\y as \x using int({\y+1})}] in {1,2} 
					{
						\draw[edge] (4\y) -- (4\x);
					}

					% Vertical edges               
					\foreach \z in {1,3}
					{
						\draw[edge] (1\z) -- (2\z);
					}
					\foreach \z in {1,3}
					{
						\draw[edge] (2\z) -- (3\z);}
					\foreach \z in {1,3}
					{
						\draw[edge] (3\z) -- (4\z);
					}
					
					% Extra edges
					
					\draw[edge] (41) -- (12);
					\draw[edge] (42) -- (33);
					\draw[edge] (12) -- (23);

				\draw[dotted, thick,shorten >=4pt, shorten <=4pt] (9,5) -- (11,5);
					\draw[dotted, thick,shorten >=4pt, shorten <=4pt] (9,3) -- (11,3);
					\draw[dotted, thick,shorten >=4pt, shorten <=4pt] (9,1) -- (11,1);

					\foreach \x in {1,2,3}
					{
						\node[label={[label distance=-3pt]90:\footnotesize{$g_{1\x}$}}] at (1\x) {};
						\node[label={[label distance=-3pt]270:\footnotesize{$g_{4\x}$}}] at (4\x) {};
					}

					\node[label={[label distance=-3pt]180:\footnotesize{$g_{21}$}}] at (21) {};
					\node[label={[label distance=-3pt]180:\footnotesize{$g_{31}$}}] at (31) {};
					
					\node[label={[label distance=-3pt]180:\footnotesize{$g_{23}$}}] at (23) {};
					\node[label={[label distance=-3pt]180:\footnotesize{$g_{33}$}}] at (33) {};

				\end{scope}
			\end{tikzpicture}
			\caption{$\Gamma_{n,5}^2$}
			\label{subfig: gamma-n-25}
		\end{subfigure}
		\begin{subfigure}[b]{0.28\linewidth}
			\begin{tikzpicture}[scale=0.3]
				\begin{scope}
					% vertices in row 1
					\foreach \x in {2,3}
					{
						\node[vertex] (1\x) at (3*\x,3*2) {};
					}
					% vertices in row 2
					\foreach \x in {3}
					{
						\node[vertex] (2\x) at (3*\x,2*2) {};
					}
					% vertices in row 3
					\foreach \x in {3}
					{
						\node[vertex] (3\x) at (3*\x,1*2) {};
					}
					% vertices in row 4
					\foreach \x in {2,3}
					{
						\node[vertex] (4\x) at (3*\x,0*2) {};
					}
					
					\node[vertex] (21) at (2,4) {};
					\node[vertex] (31) at (2,2) {};
					\node[vertex] (41) at (2,0) {};
					
					% Horizontal edges
					\foreach \y [evaluate={\y as \x using int({\y+1})}] in {2} 
					{
						\draw[edge] (1\y) -- (1\x);
					}
					% Horizontal edges
					\foreach \y [evaluate={\y as \x using int({\y+1})}] in {} 
					{
						\draw[edge] (2\y) -- (2\x);
					}
					% Horizontal edges
					\foreach \y [evaluate={\y as \x using int({\y+1})}] in {} 
					{
						\draw[edge] (3\y) -- (3\x);
					}
					% Horizontal edges
					\foreach \y [evaluate={\y as \x using int({\y+1})}] in {1,2} 
					{
						\draw[edge] (4\y) -- (4\x);
					}

					% Vertical edges               
					\foreach \z in {3}
					{
						\draw[edge] (1\z) -- (2\z);
					}
					\foreach \z in {1,3}
					{
						\draw[edge] (2\z) -- (3\z);}
					\foreach \z in {3}
					{
						\draw[edge] (3\z) -- (4\z);
					}
					
					% Extra edges
					
					\draw[edge] (41) -- (12);
					\draw[edge] (42) -- (33);
					\draw[edge] (12) -- (23);

				\draw[dotted, thick,shorten >=4pt, shorten <=4pt] (9,5) -- (11,5);
					\draw[dotted, thick,shorten >=4pt, shorten <=4pt] (9,3) -- (11,3);
					\draw[dotted, thick,shorten >=4pt, shorten <=4pt] (9,1) -- (11,1);
				
					\foreach \x in {1,2,3}
					{
						\node[label={[label distance=-3pt]90:\footnotesize{$g_{1\x}$}}] at (1\x) {};
						\node[label={[label distance=-3pt]270:\footnotesize{$g_{4\x}$}}] at (4\x) {};
					}

					\node[label={[label distance=-3pt]180:\footnotesize{$g_{21}$}}] at (21) {};
					\node[label={[label distance=-3pt]180:\footnotesize{$g_{31}$}}] at (31) {};
					
					\node[label={[label distance=-3pt]180:\footnotesize{$g_{23}$}}] at (23) {};
					\node[label={[label distance=-3pt]180:\footnotesize{$g_{33}$}}] at (33) {};

				\end{scope}
			\end{tikzpicture}
			\caption{$\Gamma_{n,5}^3$}
			\label{subfig: gamma-n-35}
		\end{subfigure}
	
		\begin{subfigure}[b]{0.28\linewidth}
			\begin{tikzpicture}[scale=0.3]
				\begin{scope}
					% vertices in row 1
					\foreach \x in {2,3}
					{
						\node[vertex] (1\x) at (3*\x,3*2) {};
					}
					% vertices in row 2
					\foreach \x in {3}
					{
						\node[vertex] (2\x) at (3*\x,2*2) {};
					}
					% vertices in row 3
					\foreach \x in {3}
					{
						\node[vertex] (3\x) at (3*\x,1*2) {};
					}
					% vertices in row 4
					\foreach \x in {2,3}
					{
						\node[vertex] (4\x) at (3*\x,0*2) {};
					}
					
					\node[vertex] (21) at (2,4) {};
					\node[vertex] (31) at (2,2) {};
					\node[vertex] (41) at (2,0) {};
					
					% Horizontal edges
					\foreach \y [evaluate={\y as \x using int({\y+1})}] in {2} 
					{
						\draw[edge] (1\y) -- (1\x);
					}
					% Horizontal edges
					\foreach \y [evaluate={\y as \x using int({\y+1})}] in {} 
					{
						\draw[edge] (2\y) -- (2\x);
					}
					% Horizontal edges
					\foreach \y [evaluate={\y as \x using int({\y+1})}] in {} 
					{
						\draw[edge] (3\y) -- (3\x);
					}
					% Horizontal edges
					\foreach \y [evaluate={\y as \x using int({\y+1})}] in {1} 
					{
						\draw[edge] (4\y) -- (4\x);
					}

					% Vertical edges               
					\foreach \z in {3}
					{\draw[edge] (1\z) -- (2\z);}
					\foreach \z in {1,3}
					{
						\draw[edge] (2\z) -- (3\z);}
					\foreach \z in {3}
					{
						\draw[edge] (3\z) -- (4\z);
					}
					
					% Extra edges
					
					\draw[edge] (41) -- (12);
					\draw[edge] (12) -- (23);
					\draw[edge] (13) to[bend right=30] (33);
					\draw[edge] (13) to[bend right=30] (43);
					\draw[edge] (23) to[bend right=30] (43);

					\draw[dotted, thick,shorten >=4pt, shorten <=4pt] (9,5) -- (11,5);
					\draw[dotted, thick,shorten >=4pt, shorten <=4pt] (9,3) -- (11,3);
					\draw[dotted, thick,shorten >=4pt, shorten <=4pt] (9,1) -- (11,1);
					
					\foreach \x in {1,2,3}
					{
						\node[label={[label distance=-3pt]90:\footnotesize{$g_{1\x}$}}] at (1\x) {};
						\node[label={[label distance=-3pt]270:\footnotesize{$g_{4\x}$}}] at (4\x) {};
					}

					\node[label={[label distance=-3pt]180:\footnotesize{$g_{21}$}}] at (21) {};
					\node[label={[label distance=-3pt]180:\footnotesize{$g_{31}$}}] at (31) {};

				\end{scope}
			\end{tikzpicture}
			\caption{$\Gamma_{n,5}^4$}
			\label{subfig: gamma-n-45}
		\end{subfigure}
			\begin{subfigure}[b]{0.28\linewidth}
			\begin{tikzpicture}[scale=0.3]
				\begin{scope}
					% vertices in row 1
					
						\foreach \x in {3,4,5}
					{
						\node[vertex] (1\x) at (3*\x,3*2) {};
					}
					\foreach \x in {3,4,5}
					{
						\node[vertex] (1\x) at (3*\x,3*2) {};
					}
					% vertices in row 2
					\foreach \x in {3,4,5}
					{
						\node[vertex] (2\x) at (3*\x,2*2) {};
					}
					% vertices in row 3
					\foreach \x in {3,4,5}
					{
						\node[vertex] (3\x) at (3*\x,1*2) {};
					}
					% vertices in row 4
					\foreach \x in {3,4,5}
					{
						\node[vertex] (4\x) at (3*\x,0*2) {};
					}
					
					% Horizontal edges
					\foreach \y [evaluate={\y as \x using int({\y+1})}] in {3,4} 
					{
						\draw[edge] (1\y) -- (1\x);
						\draw[edge] (2\y) -- (2\x);
						\draw[edge] (3\y) -- (3\x);
						\draw[edge] (4\y) -- (4\x);
					}
					
					% Vertical edges               
					\foreach \z in {3,4,5}
					{\draw[edge] (1\z) -- (2\z);
						\draw[edge] (2\z) -- (3\z);
						\draw[edge] (3\z) -- (4\z);
					}
					
					\draw[edge] (13) to [bend right=45] (33);
					\draw[edge] (13) to [bend right=45] (43);
					\draw[edge] (23) to [bend right=45] (43);

					%cross-vacant-cross
					\foreach \z [evaluate={\znext=int({1+\z})}] in {4}
					{
						\draw[edge] (1\z) -- (2\znext);
						\draw[edge] (3\z) -- (4\znext);
					}
					\foreach \z [evaluate={\znext=int({\z-1})}] in {5}
					{
						\draw[edge] (1\z) -- (2\znext);
						\draw[edge] (3\z) -- (4\znext);
					}
					
					%vacant-cross-vacant
					\foreach \z [evaluate={\znext=int({1+\z})}] in {3}
					{
						\draw[edge] (2\z) -- (3\znext);
					}
					\foreach \z [evaluate={\znext=int({\z-1})}] in {4}
					{
						\draw[edge] (2\z) -- (3\znext);
					}

					% dotted (and so on)
					\draw[dotted, thick,shorten >=4pt, shorten <=4pt] (15,5) -- (17,5);
				\draw[dotted, thick,shorten >=4pt, shorten <=4pt] (15,3) -- (17,3);
				\draw[dotted, thick,shorten >=4pt, shorten <=4pt] (15,1) -- (17,1);

					% vertex labels
					\foreach \x in {3,...,5}
					{
						\node[label={[label distance=-3pt]90:\footnotesize{$g_{1\x}$}}] at (1\x) {};
						\node[label={[label distance=-3pt]270:\footnotesize{$g_{4\x}$}}] at (4\x) {};
					}
					\node[label={[label distance=-8pt]45:\footnotesize{$g_{23}$}}] at (23) {};
					\node[label={[label distance=-8pt]-45:\footnotesize{$g_{33}$}}] at (33) {};
					
				\end{scope}
			\end{tikzpicture}
			\caption{$\Gamma_{n,5}^5$}
			\label{subfig: gamma-n-55}
		\end{subfigure}
			\begin{subfigure}[b]{0.28\linewidth}
			\begin{tikzpicture}[scale=0.3]
				\begin{scope}
					% vertices in row 1
					\foreach \x in {3,4,5}
					{
						\node[vertex] (1\x) at (3*\x,3*2) {};
					}
					% vertices in row 2
					\foreach \x in {4,5}
					{
						\node[vertex] (2\x) at (3*\x,2*2) {};
					}
					% vertices in row 3
					\foreach \x in {3,4,5}
					{
						\node[vertex] (3\x) at (3*\x,1*2) {};
					}
					% vertices in row 4
					\foreach \x in {3,4,5}
					{
						\node[vertex] (4\x) at (3*\x,0*2) {};
					}
					
					% Horizontal edges
					\foreach \y [evaluate={\y as \x using int({\y+1})}] in {3,4} 
					{
						\draw[edge] (1\y) -- (1\x);
						\draw[edge] (3\y) -- (3\x);
						\draw[edge] (4\y) -- (4\x);
					}
					\foreach \y [evaluate={\y as \x using int({\y+1})}] in {4} 
					{
						\draw[edge] (2\y) -- (2\x);
					}
					\draw[edge] (33) -- (24);
					% Vertical edges               
					\foreach \z in {3,4,5}
					{
						\draw[edge] (3\z) -- (4\z);
					}
					\foreach \z in {4,5}
					{\draw[edge] (2\z) -- (3\z);
						\draw[edge] (1\z) -- (2\z);
					}
					
					\draw[edge] (13) to [bend right=45] (33);
					\draw[edge] (13) to [bend right=45] (43);

					%cross-vacant-cross
					\foreach \z [evaluate={\znext=int({1+\z})}] in {4}
					{
						\draw[edge] (1\z) -- (2\znext);
						\draw[edge] (3\z) -- (4\znext);
					}
					\foreach \z [evaluate={\znext=int({\z-1})}] in {5}
					{
						\draw[edge] (1\z) -- (2\znext);
						\draw[edge] (3\z) -- (4\znext);
					}

					% dotted (and so on)
					\draw[dotted, thick,shorten >=4pt, shorten <=4pt] (15,5) -- (17,5);
				\draw[dotted, thick,shorten >=4pt, shorten <=4pt] (15,3) -- (17,3);
				\draw[dotted, thick,shorten >=4pt, shorten <=4pt] (15,1) -- (17,1);

					% vertex labels
					\foreach \x in {3,...,5}
					{
						\node[label={[label distance=-3pt]90:\footnotesize{$g_{1\x}$}}] at (1\x) {};
						\node[label={[label distance=-3pt]270:\footnotesize{$g_{4\x}$}}] at (4\x) {};
					}
					\node[label={[label distance=0pt]90:\footnotesize{$g_{33}$}}] at (33) {};
					\node[label={[label distance=-8pt]-30:\footnotesize{$g_{24}$}}] at (24) {};
					\node[label={[label distance=-8pt]30:\footnotesize{$g_{34}$}}] at (34) {};

				\end{scope}
			\end{tikzpicture}
			\caption{$\Gamma_{n,5}^6$}
			\label{subfig: gamma-n-65}
		\end{subfigure}
	
		\begin{subfigure}[b]{0.28\linewidth}
			\begin{tikzpicture}[scale=0.3]
				\begin{scope}
					% vertices in row 1
					\foreach \x in {3,4,5}
					{
						\node[vertex] (1\x) at (3*\x,3*2) {};
					}
					% vertices in row 2
					\foreach \x in {4,5}
					{
						\node[vertex] (2\x) at (3*\x,2*2) {};
					}
					% vertices in row 3
					\foreach \x in {3,4,5}
					{
						\node[vertex] (3\x) at (3*\x,1*2) {};
					}
					% vertices in row 4
					\foreach \x in {3,4,5}
					{
						\node[vertex] (4\x) at (3*\x,0*2) {};
					}
					
					% Horizontal edges
					\foreach \y [evaluate={\y as \x using int({\y+1})}] in {3,4} 
					{
						
						\draw[edge] (3\y) -- (3\x);
						\draw[edge] (4\y) -- (4\x);
					}
					\foreach \y [evaluate={\y as \x using int({\y+1})}] in {4} 
					{
						\draw[edge] (2\y) -- (2\x);
						\draw[edge] (1\y) -- (1\x);
					}
					\draw[edge] (33) -- (24);
					% Vertical edges               
					\foreach \z in {3,4,5}
					{
						\draw[edge] (3\z) -- (4\z);
					}
					\foreach \z in {5}
					{\draw[edge] (2\z) -- (3\z);
						\draw[edge] (1\z) -- (2\z);
					}
					\foreach \z in {4}
					{\draw[edge] (3\z) -- (2\z);}
					
					\draw[edge] (13) to [bend right=45] (33);
					\draw[edge] (13) to [bend right=45] (43);
					\draw[edge] (14) to [bend right=45] (34);
					\draw[edge] (14) to [bend right=45] (44);
					\draw[edge] (24) to [bend right=45] (44);

					%cross-vacant-cross
					\foreach \z [evaluate={\znext=int({1+\z})}] in {4}
					{
						\draw[edge] (1\z) -- (2\znext);
						\draw[edge] (3\z) -- (4\znext);
					}
					\foreach \z [evaluate={\znext=int({\z-1})}] in {5}
					{
						\draw[edge] (1\z) -- (2\znext);
						\draw[edge] (3\z) -- (4\znext);
					}
					
					%vacant-cross-vacant
					% \foreach \z [evaluate={\znext=int({1+\z})}] in {5}
					%{
						%\draw[edge] (2\z) -- (3\znext);
						%}
					%\foreach \z [evaluate={\znext=int({\z-1})}] in {4,6}
					%{
						%\draw[edge] (2\z) -- (3\znext);
						%}

					% dotted (and so on)
					\draw[dotted, thick,shorten >=4pt, shorten <=4pt] (15,5) -- (17,5);
				\draw[dotted, thick,shorten >=4pt, shorten <=4pt] (15,3) -- (17,3);
				\draw[dotted, thick,shorten >=4pt, shorten <=4pt] (15,1) -- (17,1);
					% % under braces
					% \draw[decoration={brace,raise=5pt,mirror},decorate]
					% (43) -- node[below=6pt] {$\widetilde{G}_n$}(48);
					% \draw[decoration={brace,raise=5pt,mirror},decorate]
					% (43) -- node[below=6pt] {$2$}(45);
					% \draw[decoration={brace,raise=5pt,mirror},decorate]
					% (46) -- node[below=6pt] {$n$}(48);
					
					% vertex labels
					\foreach \x in {3,...,5}
					{
						\node[label={[label distance=-3pt]90:\footnotesize{$g_{1\x}$}}] at (1\x) {};
						\node[label={[label distance=-3pt]270:\footnotesize{$g_{4\x}$}}] at (4\x) {};
					}
					\node[label={[label distance=0pt]90:\footnotesize{$g_{33}$}}] at (33) {};
					\node[label={[label distance=-8pt]-30:\footnotesize{$g_{24}$}}] at (24) {};
					\node[label={[label distance=-8pt]30:\footnotesize{$g_{34}$}}] at (34) {};
					
					% \node[label={45:\footnotesize{$g_{1\,{2n+1}}$}}] at (18) {};
					% \node[label={90:\footnotesize{$g_{1\,{2n-1}}$}}] at (16) {};
					% \node[label={90:\footnotesize{$g_{1\,{2n}}$}}] at (17) {};
					
					% \node[label={0:\footnotesize{$g_{2\,{2n+1}}$}}] at (28) {};
					% \node[label={0:\footnotesize{$g_{3\,{2n+1}}$}}] at (38) {};
					% \node[label={0:\footnotesize{$g_{4\,{2n+1}}$}}] at (48) {};
				\end{scope}
			\end{tikzpicture}
			\caption{$\Gamma_{n,5}^7$}
			\label{subfig: gamma-n-75}
		\end{subfigure}
		\begin{subfigure}[b]{0.28\linewidth}
			\begin{tikzpicture}[scale=0.3]
				\begin{scope}
					% vertices in row 1
					\foreach \x in {3,4,5}
					{
						\node[vertex] (1\x) at (3*\x,3*2) {};
					}
					% vertices in row 2
					\foreach \x in {5}
					{
						\node[vertex] (2\x) at (3*\x,2*2) {};
					}
					% vertices in row 3
					\foreach \x in {3,4,5}
					{
						\node[vertex] (3\x) at (3*\x,1*2) {};
					}
					% vertices in row 4
					\foreach \x in {3,4,5}
					{
						\node[vertex] (4\x) at (3*\x,0*2) {};
					}
					
					% Horizontal edges
					\foreach \y [evaluate={\y as \x using int({\y+1})}] in {3,4} 
					{
						
						\draw[edge] (3\y) -- (3\x);
						\draw[edge] (4\y) -- (4\x);
					}
					\foreach \y [evaluate={\y as \x using int({\y+1})}] in {4} 
					{
						%\draw[edge] (2\y) -- (2\x);
						\draw[edge] (1\y) -- (1\x);
					}
					
					% Vertical edges               
					\foreach \z in {3,4,5}
					{
						\draw[edge] (3\z) -- (4\z);
					}
					\foreach \z in {5}
					{\draw[edge] (2\z) -- (3\z);
						\draw[edge] (1\z) -- (2\z);
					}

					\draw[edge] (13) to [bend right=45] (33);
					\draw[edge] (13) to [bend right=45] (43);
					\draw[edge] (14) to [bend right=45] (34);
					\draw[edge] (14) to [bend right=45] (44);
					%\draw[edge] (24) to [bend right=45] (44);

					%cross-vacant-cross
					\foreach \z [evaluate={\znext=int({1+\z})}] in {4}
					{
						\draw[edge] (1\z) -- (2\znext);
						\draw[edge] (3\z) -- (4\znext);
					}
					\foreach \z [evaluate={\znext=int({\z-1})}] in {5}
					{
						%\draw[edge] (1\z) -- (2\znext);
						\draw[edge] (3\z) -- (4\znext);
					}
					
					%vacant-cross-vacant
					% \foreach \z [evaluate={\znext=int({1+\z})}] in {5}
					%{
						%\draw[edge] (2\z) -- (3\znext);
						%}
					%\foreach \z [evaluate={\znext=int({\z-1})}] in {4,6}
					%{
						%\draw[edge] (2\z) -- (3\znext);
						%}

					% dotted (and so on)
					\draw[dotted, thick,shorten >=4pt, shorten <=4pt] (15,5) -- (17,5);
				\draw[dotted, thick,shorten >=4pt, shorten <=4pt] (15,3) -- (17,3);
				\draw[dotted, thick,shorten >=4pt, shorten <=4pt] (15,1) -- (17,1);
					% % under braces
					% \draw[decoration={brace,raise=5pt,mirror},decorate]
					% (43) -- node[below=6pt] {$\widetilde{G}_n$}(48);
					% \draw[decoration={brace,raise=5pt,mirror},decorate]
					% (43) -- node[below=6pt] {$2$}(45);
					% \draw[decoration={brace,raise=5pt,mirror},decorate]
					% (46) -- node[below=6pt] {$n$}(48);
					
					% vertex labels
					\foreach \x in {3,...,5}
					{
						\node[label={[label distance=-3pt]90:\footnotesize{$g_{1\x}$}}] at (1\x) {};
						\node[label={[label distance=-3pt]270:\footnotesize{$g_{4\x}$}}] at (4\x) {};
					}
					\node[label={[label distance=0pt]90:\footnotesize{$g_{33}$}}] at (33) {};
					% \node[label={[label distance=-8pt]-30:\footnotesize{$g_{24}$}}] at (24) {};
					\node[label={[label distance=-8pt]30:\footnotesize{$g_{34}$}}] at (34) {};
					
					% \node[label={45:\footnotesize{$g_{1\,{2n+1}}$}}] at (18) {};
					% \node[label={90:\footnotesize{$g_{1\,{2n-1}}$}}] at (16) {};
					% \node[label={90:\footnotesize{$g_{1\,{2n}}$}}] at (17) {};
					
					% \node[label={0:\footnotesize{$g_{2\,{2n+1}}$}}] at (28) {};
					% \node[label={0:\footnotesize{$g_{3\,{2n+1}}$}}] at (38) {};
					% \node[label={0:\footnotesize{$g_{4\,{2n+1}}$}}] at (48) {};
				\end{scope}
			\end{tikzpicture}
			\caption{$\Gamma_{n,5}^8$}
			\label{subfig: gamma-n-85}
		\end{subfigure}
			\begin{subfigure}[b]{0.28\linewidth}
			\begin{tikzpicture}[scale=0.3]
				\begin{scope}
					% vertices in row 1
					\foreach \x in {4,5,6,7}
					{
						\node[vertex] (1\x) at (3*\x,3*2) {};
					}
					% vertices in row 2
					\foreach \x in {5,6,7}
					{
						\node[vertex] (2\x) at (3*\x,2*2) {};
					}
					% vertices in row 3
					\foreach \x in {5,6,7}
					{
						\node[vertex] (3\x) at (3*\x,1*2) {};
					}
					% vertices in row 4
					\foreach \x in {4,5,6,7}
					{
						\node[vertex] (4\x) at (3*\x,0*2) {};
					}
					
					% Horizontal edges
					\foreach \y [evaluate={\y as \x using int({\y+1})}] in {5,6} 
					{
						\draw[edge] (1\y) -- (1\x);
						\draw[edge] (2\y) -- (2\x);
						\draw[edge] (3\y) -- (3\x);
						\draw[edge] (4\y) -- (4\x);
					}
					\draw[edge] (14) -- (15);
					\draw[edge] (44) -- (45);
					\draw[edge] (44) -- (35);
					\draw[edge] (14) -- (25);

					% Vertical edges               
					\foreach \z in {5,6,7}
					{\draw[edge] (1\z) -- (2\z);
						\draw[edge] (2\z) -- (3\z);
						\draw[edge] (3\z) -- (4\z);
					}
					
					\draw[edge] (14) to [bend right=45] (44);
					%\draw[edge] (13) to [bend right=45] (43);
					%\draw[edge] (23) to [bend right=45] (43);

					%cross-vacant-cross
					\foreach \z [evaluate={\znext=int({1+\z})}] in {6}
					{
						\draw[edge] (1\z) -- (2\znext);
						\draw[edge] (3\z) -- (4\znext);
					}
					\foreach \z [evaluate={\znext=int({\z-1})}] in {7}
					{
						\draw[edge] (1\z) -- (2\znext);
						\draw[edge] (3\z) -- (4\znext);
					}
					
					%vacant-cross-vacant
					\foreach \z [evaluate={\znext=int({1+\z})}] in {5}
					{
						\draw[edge] (2\z) -- (3\znext);
					}
					\foreach \z [evaluate={\znext=int({\z-1})}] in {6}
					{
						\draw[edge] (2\z) -- (3\znext);
					}

					% dotted (and so on)
					\draw[dotted, thick,shorten >=4pt, shorten <=4pt] (23,5) -- (21,5);
					\draw[dotted, thick,shorten >=4pt, shorten <=4pt] (23,3) -- (21,3);
					\draw[dotted, thick,shorten >=4pt, shorten <=4pt] (23,1) -- (21,1);

					% vertex labels
					\foreach \x in {5,...,7}
					{
						\node[label={[label distance=-3pt]90:\footnotesize{$g_{1\x}$}}] at (1\x) {};
						\node[label={[label distance=-3pt]270:\footnotesize{$g_{4\x}$}}] at (4\x) {};
					}
					%\node[label={[label distance=-8pt]45:\footnotesize{$g_{23}$}}] at (23) {};
					%\node[label={[label distance=-8pt]-45:\footnotesize{$g_{33}$}}] at (33) {};
					
				\end{scope}
			\end{tikzpicture}
			\caption{$\Gamma_{n,5}^8\setminus \n[g_{33}]$}
			
			\label{subfig: gamma-n-5g33-nbd}
		\end{subfigure}
		\caption{} \label{fig:G-tilda}	
	\end{figure}

\begin{proof} Since $\Gamma_{0,5}$ is a path with 4-vertices, by \Cref{Ind(path)}, we have $\ind(\Gamma_{0,5})$ is homotopy equivalent to a point. For the remaining part of the proof, we assume that $n\geq 1$.

Since $N_{}(g_{11}) \subseteq \on{}{g_{22}}$ and $N_{}(g_{41}) \subseteq \on{}{g_{32}}$ in $\Gamma_{n,5}$ and $\Gamma_{n,5} \setminus g_{22}$, respectively (see \Cref{fig:Gn}), using \Cref{Folding Lemma}, we get that $\ind(\Gamma_{n,5}) \simeq \ind(\Gamma_{n,5} \setminus \{g_{22},g_{32}\})$. Let $\Gamma_{n,5}^1 \coloneq \Gamma_{n,5} \setminus \{g_{22},g_{32}\}$. 
Since  $[g_{41},g_{12};g_{21}]$ forms an edge-invariant triplet in $\Gamma_{n,5}^1$ (see \Cref{subfig: gamma-n-15}), from \Cref{Edge deletion 1}, we have $\ind(\Gamma_{n,5}^1) \simeq \ind(\Gamma_{n,5}^1{+}\{g_{41}g_{12}\})$. Let $ \Gamma_{n,5}^2 \eqcolon \Gamma_{n,5}^1{+}\{g_{41}g_{12}\}$.
Similarly, $[g_{31},g_{41};g_{11}]$ forms an edge-invariant triplet in $\Gamma_{n,5}^2$ (see \Cref{subfig: gamma-n-25}), which allows us to delete the edge $g_{31}g_{41}$ from $\Gamma_{n,5}^2$ without altering the homotopy type of $\ind(\Gamma_{n,5}^2)$. Now, since $N(g_{31}) \subseteq N(g_{11})$ in $\Gamma_{n,5}^2 - \{g_{31}g_{41}\}$,  using \Cref{Folding Lemma}, we get that $\ind(\Gamma_{n, 5}^2-\{g_{31}g_{41}\}) \simeq \ind((\Gamma_{n, 5}^2 - \{g_{31}g_{41}\}) \setminus g_{11})$. Let 
$\Gamma_{n,5}^3 \eqcolon (\Gamma_{n, 5}^2 - \{g_{31}g_{41}\}) \setminus g_{11}$ (see \Cref{subfig: gamma-n-35}). 
Observe that $[g_{13},g_{33};g_{41}]$  is an edge-invariant triplet in 
$\Gamma_{n, 5}^3$, $[g_{13},g_{43};g_{41}]$  is an edge-invariant triplet in 
$\Gamma_{n, 5}^3 + \{g_{13}g_{33}\}$, and  $[g_{23},g_{43};g_{41}]$  is an edge-invariant triplet in 
$\Gamma_{n, 5}^3 + \{g_{13} g_{33}, g_{13}g_{43} \}$. Hence $\ind(\Gamma_{n, 5}^3) \simeq \ind(\Gamma_{n, 5}^3 + \{g_{13} g_{33}, g_{13}g_{43}, g_{23}g_{43}\})$. Furthermore, $[g_{42},g_{33};g_{12}]$ is an edge-invariant triplet in $\Gamma_{n, 5}^3 + \{g_{13} g_{33}, g_{13}g_{43}, g_{23}g_{43} \}$, and $[g_{42},g_{43};g_{12}]$ is an edge-invariant triplet in $\Gamma_{n, 5}^3 + \{g_{13} g_{33}, g_{13}g_{43}, g_{23}g_{43} \}-\{g_{33}g_{42}\}$. Let $\Gamma_{n,5}^4 \coloneq \Gamma_{n, 5}^3 + \{g_{13} g_{33}, g_{13}g_{43}, g_{23}g_{43}\}-\{g_{33}g_{42},g_{42}g_{43}\}$ (see \Cref{subfig: gamma-n-45}).    Then, we have $\ind(\Gamma_{n,5}^3) \simeq \ind(\Gamma_{n,5}^4)$. 
In $\Gamma_{n,5}^4$, since $N(g_{42}) \subseteq N(g_{12})$, we have $\ind(\Gamma_{n,5}^4) \simeq \ind(\Gamma_{n,5}^4 \setminus g_{12})$. The  graph, $\Gamma_{n,5}^4 \setminus g_{12}$, consists of three connected components, namely $g_{21}g_{31}, g_{41}g_{42}$ and the third component denoted by $\Gamma_{n,5}^5$ depicted in \Cref{subfig: gamma-n-55}. From \Cref{prop: join}, it follows that 
\begin{equation*}
    \ind(\Gamma_{n,5}) \simeq \Sigma^{2}\ind(\Gamma_{n,5}^5).
\end{equation*}

In particular, for $\Gamma_{1,5}$, the graph $\Gamma^5_{1,5}$ is isomorphic to the complete graph $K_4$. Therefore, $\ind(\Gamma_{1,5})\simeq \Sigma^{2}\ind(K_4)\simeq \mathbb{S}^2\vee \mathbb{S}^2 \vee \mathbb{S}^2$.

Now, let $n \geq 2$. Note that, $\Gamma_{n,5}^5\setminus\cn{}{g_{23}}$ is isomorphic to $A_{n-2,5}$ (see \Cref{fig:An}) and $\ind(\Gamma_{n,5}^5\setminus\cn{}{g_{23}})* \{g_{33}\}$ is a subcomplex of $\ind(\Gamma_{n,5}^5\setminus g_{23})$, which implies that $\ind(\Gamma_{n,5}^5\setminus\cn{}{g_{23}})$ is contractible in $\ind(\Gamma_{n,5}^5\setminus g_{23})$, that is, $\ind(\Gamma_{n,5}^5\setminus\cn{}{g_{23}}) \hookrightarrow \ind(\Gamma_{n,5}^5\setminus g_{23})$ is null homotopic.
Hence from \Cref{Link and Deletion}, we have
\begin{equation}{\label{eq: Gn lk del}}
    \ind(\Gamma_{n,5}) \simeq \Sigma^2\left(\ind(\Gamma_{n,5}^5\setminus g_{23}) \vee \Sigma\ind(A_{n-2,5}) \right).
\end{equation}
 We now compute the homotopy type of $\ind(\Gamma_{n,5}^5\setminus g_{23})$. Let $\Gamma_{n,5}^6:=\Gamma_{n,5}^5\setminus g_{23}$ depicted in \Cref{subfig: gamma-n-65}. Since $[g_{14},g_{34};g_{43}]$, $[g_{24},g_{44};g_{13}]$, and $[g_{14},g_{44};g_{33}]$ are edge-invariant triplets in $\Gamma_{n,5}^6$, $\Gamma_{n,5}^6 + \{g_{14}g_{34}\}$, and $\Gamma_{n,5}^6 + \{g_{14}g_{34}, g_{24}g_{44}\}$ respectively, using \Cref{Edge deletion 1}, we have $\ind(\Gamma_{n,5}^6) \simeq \ind(\Gamma_{n,5}^6+\{g_{14}g_{34}, g_{24}g_{44}$, $g_{14}g_{44}\})$. Moreover, $[g_{14},g_{24};g_{43}]$ is an edge-invariant triplet in $\Gamma_{n,5}^6+\{g_{14}g_{34}, g_{24}g_{44}$, $ g_{14}g_{44}\}$ and  $[g_{13},g_{14};g_{24}]$ is an edge-invariant triplet in $\Gamma_{n,5}^6+\{g_{14}g_{34}, g_{24}g_{44}$, $g_{14}g_{44}\}-\{g_{14}g_{24}\}$. Let $\Gamma_{n,5}^7 \coloneq \Gamma_{n,5}^6+\{g_{14}g_{34}, g_{24}g_{44}$, $g_{14}g_{44}\}-\{g_{14}g_{24},g_{13}g_{14}\}$ 
 (see \Cref{subfig: gamma-n-75}). Then, we have $\ind(\Gamma_{n,5}^6) \simeq \ind(\Gamma_{n,5}^7)$. 

 Observe that $N(g_{14})\subseteq N(g_{24})$ in the graph $\Gamma_{n,5}^7$. Let $\Gamma_{n,5}^8:=\Gamma_{n,5}^7\setminus g_{24}$. Therefore, $\ind(\Gamma_{n,5}^6) \simeq\ind(\Gamma_{n,5}^7) \simeq \ind(\Gamma_{n,5}^8)$. Furthermore, $g_{13}$ is a simplicial vertex in $\Gamma_{n,5}^8$ with $N(g_{13})=\{g_{33},g_{43}\}$, as shown in \Cref{subfig: gamma-n-85}. Using \Cref{Simplicial Vertex Lemma}, we have 
 
 \begin{equation}{\label{eq: Gn 75}}
 \ind(\Gamma_{n,5}^8)\simeq\Sigma\ind(\Gamma_{n,5}^8\setminus \n[g_{33}])\vee\Sigma\ind(\Gamma_{n,5}^8\setminus \n[g_{43}])
 \end{equation}
However, both the graphs $\Gamma_{n,5}^8\setminus \n[g_{33}]$ (as shown in \Cref{subfig: gamma-n-5g33-nbd}) and $\Gamma_{n,5}^8\setminus \n[g_{43}]$ are isomorphic to $D_{n-2,5}$. Hence, using \Cref{eq: Gn 75} we get $\ind(\Gamma_{n,5}^6)\simeq\ind(\Gamma_{n,5}^8) \simeq\Sigma\ind(D_{n-2,5})\vee\Sigma\ind(D_{n-2,5})$. Therefore, from \Cref{eq: Gn lk del} we have $$\ind(\Gamma_{n,5}) \simeq \Sigma^3\ind(D_{n-2,5}) \vee\Sigma^3\ind(D_{n-2,5}) \vee \Sigma^3\ind(A_{n-2,5}).$$
This completes the proof.
\end{proof}

\begin{claim}{\label{homLambdan}}
{\small
    \begin{equation*}
        \ind(\Lambda_{n,5}) \simeq \begin{cases}
        \vee^{7}\mathbb{S}^1 &   \text{if $n = 0$,} \\\left({\bigvee}^{2} \Sigma^2\ind(A_{n-1,5})\right) \vee  \left({\bigvee}^{2} \Sigma^2\ind(C_{n-1,5})\right) \vee  \left({\bigvee}^{2} \Sigma^2\ind(F_{n-1,5})\right) \vee \Sigma^2\ind(D_{n-1,5}) & \text{if $n \geq 1$.}
        \end{cases}
    \end{equation*}
    }

\end{claim}
\begin{figure}[h!]
	\centering
	\begin{subfigure}[b]{0.3\textwidth}
		\begin{tikzpicture}[scale=0.3]
			\begin{scope}
				\foreach \x in {2,...,5}
				{
					\node[vertex] (1\x) at (2*\x,3*2) {};
				}
				\foreach \x in {2,3,4,5}
				{
					\node[vertex] (2\x) at (2*\x,2*2) {};
				}
				\foreach \x in {1,2,...,5}
				{
					\node[vertex] (3\x) at (2*\x,1*2) {};
				}
				\foreach \x in {1,2,3,...,5}
				{
					\node[vertex] (4\x) at (2*\x,0*2) {};
				}

				\foreach \y [evaluate={\y as \x using {\y+1}}] in {2,3,4} 
				{
					\draw[edge] (1\y) -- (1\x);
				}
				\foreach \y [evaluate={\y as \x using {\y+1}}] in {2,3,4} 
				{
					\draw[edge] (2\y) -- (2\x);
				}
				\foreach \y [evaluate={\y as \x using {\y+1}}] in {1,2,...,4} 
				{
					\draw[edge] (3\y) -- (3\x);
				}
				\foreach \y [evaluate={\y as \x using {\y+1}}] in {1,2,3,...,4} 
				{
					\draw[edge] (4\y) -- (4\x);
				}
				
				\foreach \z in {2,3,4,5}
				{
					\draw[edge] (1\z) -- (2\z);
				}
				\foreach \z in {2,3,4,5}
				{
					\draw[edge] (2\z) -- (3\z);
				}
				\foreach \z in {1,2,3,4,5}
				{
					\draw[edge] (3\z) -- (4\z);
				}
				
				\foreach \z [evaluate={\znext=int({1+\z})}] in {2,4}
				{
					\draw[edge] (1\z) -- (2\znext);
				}
				\foreach \z [evaluate={\znext=int({1+\z})}] in {2,4}
				{
					\draw[edge] (3\z) -- (4\znext);
				}
				\foreach \z [evaluate={\znext=int({\z-1})}] in {3,5}
				{
					\draw[edge] (1\z) -- (2\znext);
				}
				\foreach \z [evaluate={\znext=int({\z-1})}] in {3,5}
				{
					\draw[edge] (3\z) -- (4\znext);
				}
				
				\foreach \z [evaluate={\znext=int({1+\z})}] in {3}
				{
					\draw[edge] (2\z) -- (3\znext);
				}
				\foreach \z [evaluate={\znext=int({\z-1})}] in {2,4}
				{
					\draw[edge] (2\z) -- (3\znext);
				}
				
				\draw[dotted, thick,shorten >=4pt, shorten <=4pt] (10,5) -- (12,5);
				\draw[dotted, thick,shorten >=4pt, shorten <=4pt] (10,3) -- (12,3);
				\draw[dotted, thick,shorten >=4pt, shorten <=4pt] (10,1) -- (12,1);
				
				\foreach \x in {2,...,5}
				{
					\node[label={[label distance=-3pt]90:\scriptsize{$g_{1\x}$}}] at (1\x) {};
					\node[label={[label distance=-3pt]-90:\scriptsize{$g_{4\x}$}}] at (4\x) {};
				}
				\node[label={[label distance=-3pt]-90:\scriptsize{$g_{41}$}}] at (41) {};
				\node[label={[label distance=-5pt]180:\scriptsize{$g_{31}$}}] at (31) {};
				\node[label={[label distance=-5pt]180:\scriptsize{$g_{22}$}}] at (22) {};
			\end{scope}
			\begin{scope}[shift={(2,0)}]
				
				\node[vertex] (l4) at (-2,0) {} ;

				\draw[edge] (l4) -- (41) ;
				\draw[edge] (l4) -- (31) ;
				
				\node[label={[label distance=-3pt]180:\small{$l_{4}$}}] at (l4) {};

			\end{scope}

		\end{tikzpicture}
		\caption{$\Lambda_{n,5}\setminus \n[l_2]$}
		\label{subfig: Lambda-n-l2}
	\end{subfigure}
	\begin{subfigure}[b]{0.3\textwidth}
		\begin{tikzpicture}[scale=0.3]
			\begin{scope}
				\foreach \x in {3,...,5}
				{
					\node[vertex] (1\x) at (2*\x,3*2) {};
				}
				\foreach \x in {2,3,4,5}
				{
					\node[vertex] (2\x) at (2*\x,2*2) {};
				}
				\foreach \x in {1,2,...,5}
				{
					\node[vertex] (3\x) at (2*\x,1*2) {};
				}
				\foreach \x in {1,2,3,...,5}
				{
					\node[vertex] (4\x) at (2*\x,0*2) {};
				}

				\foreach \y [evaluate={\y as \x using {\y+1}}] in {3,4} 
				{
					\draw[edge] (1\y) -- (1\x);
				}
				\foreach \y [evaluate={\y as \x using {\y+1}}] in {2,3,4} 
				{
					\draw[edge] (2\y) -- (2\x);
				}
				\foreach \y [evaluate={\y as \x using {\y+1}}] in {1,2,...,4} 
				{
					\draw[edge] (3\y) -- (3\x);
				}
				\foreach \y [evaluate={\y as \x using {\y+1}}] in {1,2,3,...,4} 
				{
					\draw[edge] (4\y) -- (4\x);
				}
				
				\foreach \z in {3,4,5}
				{
					\draw[edge] (1\z) -- (2\z);
				}
				\foreach \z in {2,3,4,5}
				{
					\draw[edge] (2\z) -- (3\z);
				}
				\foreach \z in {1,2,3,4,5}
				{
					\draw[edge] (3\z) -- (4\z);
				}
				
				\foreach \z [evaluate={\znext=int({1+\z})}] in {4}
				{
					\draw[edge] (1\z) -- (2\znext);
				}
				\foreach \z [evaluate={\znext=int({1+\z})}] in {2,4}
				{
					\draw[edge] (3\z) -- (4\znext);
				}
				\foreach \z [evaluate={\znext=int({\z-1})}] in {3,5}
				{
					\draw[edge] (1\z) -- (2\znext);
				}
				\foreach \z [evaluate={\znext=int({\z-1})}] in {3,5}
				{
					\draw[edge] (3\z) -- (4\znext);
				}
				
				\foreach \z [evaluate={\znext=int({1+\z})}] in {3}
				{
					\draw[edge] (2\z) -- (3\znext);
				}
				\foreach \z [evaluate={\znext=int({\z-1})}] in {2,4}
				{
					\draw[edge] (2\z) -- (3\znext);
				}
				
				\draw[dotted, thick,shorten >=4pt, shorten <=4pt] (10,5) -- (12,5);
				\draw[dotted, thick,shorten >=4pt, shorten <=4pt] (10,3) -- (12,3);
				\draw[dotted, thick,shorten >=4pt, shorten <=4pt] (10,1) -- (12,1);
				
				\foreach \x in {3,...,5}
				{
					\node[label={[label distance=-3pt]90:\scriptsize{$g_{1\x}$}}] at (1\x) {};
				}
				\foreach \x in {1,...,5}
				{
					\node[label={[label distance=-3pt]-90:\scriptsize{$g_{4\x}$}}] at (4\x) {};
				}
				\node[label={[label distance=-5pt]180:\scriptsize{$g_{22}$}}] at (22) {};
				\node[label={[label distance=-3pt]90:\scriptsize{$g_{31}$}}] at (31) {};
				
			\end{scope}
			
			\begin{scope}[shift={(2,0)}]
				
				\node[vertex] (l3) at (-2,2) {} ;
				\node[vertex] (l4) at (-2,0) {} ;

				\draw[edge] (l3) -- (l4) ;
				
				\draw[edge] (l3) -- (31) ;
				\draw[edge] (l4) -- (41) ;
				
				\draw[edge] (l4) -- (31) ;
				\draw[edge] (l3) -- (41) ;

				\node[label={[label distance=-3pt]180:\small{$l_{3}$}}] at (l3) {};
				\node[label={[label distance=-3pt]180:\small{$l_{4}$}}] at (l4) {};

			\end{scope}
			
		\end{tikzpicture}
		\caption{$\Lambda_{n,5}^1$}
		\label{subfig: Lambda-n-1}
	\end{subfigure}
	\begin{subfigure}[b]{0.3\textwidth}
		\begin{tikzpicture}[scale=0.3]
			\begin{scope}
				\foreach \x in {2,...,5}
				{
					\node[vertex] (1\x) at (2*\x,3*2) {};
				}
				\foreach \x in {3,4,5}
				{
					\node[vertex] (2\x) at (2*\x,2*2) {};
				}
				\foreach \x in {3,...,5}
				{
					\node[vertex] (3\x) at (2*\x,1*2) {};
				}
				\foreach \x in {1,2,3,...,5}
				{
					\node[vertex] (4\x) at (2*\x,0*2) {};
				}

				\foreach \y [evaluate={\y as \x using {\y+1}}] in {2,3,4} 
				{
					\draw[edge] (1\y) -- (1\x);
				}
				\foreach \y [evaluate={\y as \x using {\y+1}}] in {3,4} 
				{
					\draw[edge] (2\y) -- (2\x);
				}
				\foreach \y [evaluate={\y as \x using {\y+1}}] in {3,4} 
				{
					\draw[edge] (3\y) -- (3\x);
				}
				\foreach \y [evaluate={\y as \x using {\y+1}}] in {1,2,3,...,4} 
				{
					\draw[edge] (4\y) -- (4\x);
				}
				
				\foreach \z in {3,4,5}
				{
					\draw[edge] (1\z) -- (2\z);
				}
				\foreach \z in {3,4,5}
				{
					\draw[edge] (2\z) -- (3\z);
				}
				\foreach \z in {3,4,5}
				{
					\draw[edge] (3\z) -- (4\z);
				}
				
				\foreach \z [evaluate={\znext=int({1+\z})}] in {2,4}
				{
					\draw[edge] (1\z) -- (2\znext);
				}
				\foreach \z [evaluate={\znext=int({1+\z})}] in {4}
				{
					\draw[edge] (3\z) -- (4\znext);
				}
				\foreach \z [evaluate={\znext=int({\z-1})}] in {5}
				{
					\draw[edge] (1\z) -- (2\znext);
				}
				\foreach \z [evaluate={\znext=int({\z-1})}] in {3,5}
				{
					\draw[edge] (3\z) -- (4\znext);
				}
				
				\foreach \z [evaluate={\znext=int({1+\z})}] in {3}
				{
					\draw[edge] (2\z) -- (3\znext);
				}
				\foreach \z [evaluate={\znext=int({\z-1})}] in {4}
				{
					\draw[edge] (2\z) -- (3\znext);
				}
				
				\draw[dotted, thick,shorten >=4pt, shorten <=4pt] (10,5) -- (12,5);
				\draw[dotted, thick,shorten >=4pt, shorten <=4pt] (10,3) -- (12,3);
				\draw[dotted, thick,shorten >=4pt, shorten <=4pt] (10,1) -- (12,1);
				
				\foreach \x in {2,...,5}
				{
					\node[label={[label distance=-3pt]90:\scriptsize{$g_{1\x}$}}] at (1\x) {};
				}
				\foreach \x in {1,2,...,5}
				{
					\node[label={[label distance=-3pt]-90:\scriptsize{$g_{4\x}$}}] at (4\x) {};
				}
				\node[label={[label distance=-5pt]180:\scriptsize{$g_{23}$}}] at (23) {};
				\node[label={[label distance=-5pt]180:\scriptsize{$g_{33}$}}] at (33) {};
			\end{scope}
			
			\begin{scope}[shift={(2,0)}]
				
				\node[vertex] (l3) at (-2,2) {} ;
				\node[vertex] (l4) at (-2,0) {} ;

				\draw[edge] (l3) -- (l4) ;
				\draw[edge] (l4) -- (41) ;
				\draw[edge] (l3) -- (41) ;

				\node[label={[label distance=-3pt]180:\small{$l_{3}$}}] at (l3) {};
				\node[label={[label distance=-3pt]180:\small{$l_{4}$}}] at (l4) {};

			\end{scope}
			
		\end{tikzpicture}
		\caption{$\Lambda_{n,5}^2$}
		\label{subfig: Lambda-n-2}
	\end{subfigure}
	
	\caption{}
\end{figure} 
\begin{proof}
	The proof follows primarily from the application of \Cref{Simplicial Vertex Lemma}. Since $l_1$ is a simplicial vertex and $N(l_1)=\{g_{11},g_{21},l_2\}$ in $\Lambda_{n,5}$ (see \Cref{fig:Lambda-n}), we have
	\begin{equation}{\label{eq: En-e1-sim}}
		\ind(\Lambda_{n,5}) \simeq \Sigma \ind(\Lambda_{n,5} \setminus \n[g_{11}]) \vee \Sigma \ind(\Lambda_{n,5} \setminus \n[g_{21}]) \vee \Sigma \ind(\Lambda_{n,5} \setminus \n[l_{2}]).
	\end{equation}
	First, let $n = 0$. Observe that the graphs  $\Lambda_{0,5} \setminus \n[g_{21}]$ and $\Lambda_{0,5} \setminus \n[l_2]$ are isomorphic to the complete graph $K_3$ and the graph $\Lambda_{0,5} \setminus \n[g_{11}]$ is isomorphic to the complete graph $K_4$. Therefore, 
	$\ind (\Lambda_{0,5})\simeq\Sigma\ind(K_3) \vee \Sigma\ind(K_3) \vee \Sigma\ind(K_4)
	\simeq \vee^{7}\mathbb{S}^1.$

Now, assume that $n \geq 1$. 	Note that $\Lambda_{n,5} \setminus \n[l_{2}]$ is isomorphic to $C_{n+1,5}^5$ (see \Cref{subfig: C-n-5} in \Cref{homCn}). Therefore, it follows from \eqref{eq:Cn5-sim-vert} that
	\begin{equation}{\label{eq: En-e2}}
		\ind(\Lambda_{n,5} \setminus \n[l_{2}]) \simeq \Sigma \ind(A_{n-1,5}) \vee \Sigma\ind(F_{n-1,5}).
	\end{equation}
	Let $\Lambda_{n,5}^1 := \Lambda_{n,5} \setminus \n[g_{11}]$ (see \Cref{subfig: Lambda-n-1}) and  $\Lambda_{n,5}^2 := \Lambda_{n,5} \setminus \n[g_{21}]$ (see \Cref{subfig: Lambda-n-2}). Since $l_4$ is a simplicial vertex in $\Lambda_{n,5}^1$ with $N(l_4)=\{ l_3,g_{31},g_{41}\}$, using \Cref{Simplicial Vertex Lemma}, we have
	\begin{equation*}
		\begin{split}
			\ind(\Lambda_{n,5}^1) &\simeq \Sigma \ind(\Lambda_{n,5}^1 \setminus \n[l_3]) \vee \Sigma \ind(\Lambda_{n,5}^1 \setminus \n[g_{31}]) \vee \Sigma \ind(\Lambda_{n,5}^1 \setminus \n[g_{41}]).
		\end{split}
	\end{equation*}
	Observe that the graphs $\Lambda_{n,5}^1 \setminus \n[l_3]$, $\Lambda_{n,5}^1 \setminus \n[g_{31}]$, and $\Lambda_{n,5}^1 \setminus \n[g_{41}]$ are isomorphic to the graphs $F_{n-1,5}$,  $C_{n-1,5}$ and $D_{n-1,5}$, respectively (see  \Cref{fig:Fn}, \ref{fig:Cn} and \ref{fig:Dn} respectively). Thus, we have \begin{equation}{\label{En-e21-sim-e4}}
		\begin{split}
			\ind(\Lambda_{n,5}^1) &\simeq \Sigma \ind(F_{n-1,5}) \vee \Sigma \ind(C_{n-1,5}) \vee \Sigma \ind(D_{n-1,5}).
		\end{split}
	\end{equation}
	
	Further, for the simplicial vertex $l_3$ with $N(l_3)=\{l_4,g_{41}\}$ in $\Lambda_{n,5}^2$,  we have
	\begin{equation}{\label{En-e22-sim-e3}}
		\begin{split}
			\ind(\Lambda_{n,5}^2) &\simeq \Sigma \ind(\Lambda_{n,5}^2 \setminus \n[l_4]) \vee \Sigma \ind(\Lambda_{n,5}^2 \setminus \n[g_{41}]) \simeq \Sigma \ind(A_{n-1,5}) \vee \Sigma \ind(C_{n-1,5}),
		\end{split}
	\end{equation}
	wherein the graphs $\Lambda_{n,5}^2 \setminus \n[l_4]$ and $\Lambda_{n,5}^2 \setminus \n[g_{41}]$ are isomorphic to graphs $A_{n-1,5}$ and $C_{n-1,5}$, respectively.
	Combining \eqref{eq: En-e1-sim}, \eqref{eq: En-e2}, \eqref{En-e21-sim-e4}, and \eqref{En-e22-sim-e3}, we get the desired result.
\end{proof}

\begin{figure}[h!]
    \centering
    \begin{subfigure}[b]{0.4\textwidth}
    \begin{tikzpicture}[scale=0.3]
            \begin{scope}
\foreach \x in {3,4,5}
            {
            \node[vertex] (1\x) at (3*\x,3*2) {};
            }
\foreach \x in {2,3,4,5}
            {
            \node[vertex] (2\x) at (3*\x,2*2) {};
            }
\foreach \x in {1,2,...,5}
            {
            \node[vertex] (3\x) at (3*\x,1*2) {};
            }
\foreach \x in {1,2,3,...,5}
            {
            \node[vertex] (4\x) at (3*\x,0*2) {};
            }

            \foreach \y [evaluate={\y as \x using {\y+1}}] in {3,4} 
            {
            \draw[edge] (1\y) -- (1\x);
            }
            \foreach \y [evaluate={\y as \x using {\y+1}}] in {2,3,4} 
            {
            \draw[edge] (2\y) -- (2\x);
            }
            \foreach \y [evaluate={\y as \x using {\y+1}}] in {1,2,...,4} 
            {
            \draw[edge] (3\y) -- (3\x);
            }
            \foreach \y [evaluate={\y as \x using {\y+1}}] in {1,2,3,...,4} 
            {
            \draw[edge] (4\y) -- (4\x);
            }
            
\foreach \z in {3,4,5}
            {\draw[edge] (1\z) -- (2\z);}
            \foreach \z in {2,3,4,5}
            {
            \draw[edge] (2\z) -- (3\z);
            }
            \foreach \z in {1,2,3,4,5}
            {
            \draw[edge] (3\z) -- (4\z);
            }
            
\foreach \z [evaluate={\znext=int({1+\z})}] in {4}
            {
            \draw[edge] (1\z) -- (2\znext);
            }
            \foreach \z [evaluate={\znext=int({1+\z})}] in {2,4}
            {
            \draw[edge] (3\z) -- (4\znext);
            }
            \foreach \z [evaluate={\znext=int({\z-1})}] in {3,5}
            {
            \draw[edge] (1\z) -- (2\znext);
            }
            \foreach \z [evaluate={\znext=int({\z-1})}] in {3,5}
            {
            \draw[edge] (3\z) -- (4\znext);
            }

\foreach \z [evaluate={\znext=int({1+\z})}] in {3}
            {
            \draw[edge] (2\z) -- (3\znext);
            }
            \foreach \z [evaluate={\znext=int({\z-1})}] in {2,4}
            {
            \draw[edge] (2\z) -- (3\znext);
            }
            \coordinate (c15) at (15);
            \coordinate (c16) at (16);

\draw[dotted, thick,shorten >=4pt, shorten <=4pt] (15,5) -- (17,5);
            \draw[dotted, thick,shorten >=4pt, shorten <=4pt] (15,3) -- (17,3);
            \draw[dotted, thick,shorten >=4pt, shorten <=4pt] (15,1) -- (17,1);
            
        \end{scope}		

\begin{scope}[shift={(2,0)}]
            \node[vertex] (a2) at (-2,0) {} ;

            \draw[edge] (a2) -- (31) ;
            \draw[edge] (a2) -- (41) ;

            \node[label={[label distance=-3pt]180:\scriptsize{$a_{2}$}}] at (a2) {};

            \foreach \x in {3,...,5}
            {
            \node[label={[label distance=-3pt]90:\scriptsize{$g_{1\x}$}}] at (1\x) {};
            }
            \foreach \x in {1,2,...,5}
            {
            \node[label={[label distance=-3pt]-90:\scriptsize{$g_{4\x}$}}] at (4\x) {};
            }
            \node[label={[label distance=-3pt]180:\scriptsize{$g_{22}$}}] at (22) {};

           \foreach \x in {3}
            {
            \node[label={[label distance=-3pt]180:\scriptsize{$g_{\x1}$}}] at (\x1) {};
            }
        \end{scope}
    \end{tikzpicture}
    \caption{$A_{n,5} \setminus \n[g_{11}]$}
    \end{subfigure}
    \begin{subfigure}[b]{0.4\textwidth}
    \begin{tikzpicture}[scale=0.3]
            \begin{scope}
\foreach \x in {2,3,4,5}
            {
            \node[vertex] (1\x) at (3*\x,3*2) {};
            }
\foreach \x in {3,4,5}
            {
            \node[vertex] (2\x) at (3*\x,2*2) {};
            }
\foreach \x in {3,...,5}
            {
            \node[vertex] (3\x) at (3*\x,1*2) {};
            }
\foreach \x in {1,2,3,...,5}
            {
            \node[vertex] (4\x) at (3*\x,0*2) {};
            }

            \foreach \y [evaluate={\y as \x using {\y+1}}] in {2,3,4} 
            {
            \draw[edge] (1\y) -- (1\x);
            }
            \foreach \y [evaluate={\y as \x using {\y+1}}] in {3,4} 
            {
            \draw[edge] (2\y) -- (2\x);
            }
            \foreach \y [evaluate={\y as \x using {\y+1}}] in {3,...,4} 
            {
            \draw[edge] (3\y) -- (3\x);
            }
            \foreach \y [evaluate={\y as \x using {\y+1}}] in {1,2,3,...,4} 
            {
            \draw[edge] (4\y) -- (4\x);
            }
            
\foreach \z in {3,4,5}
            {\draw[edge] (1\z) -- (2\z);}
            \foreach \z in {3,4,5}
            {
            \draw[edge] (2\z) -- (3\z);
            }
            \foreach \z in {3,4,5}
            {
            \draw[edge] (3\z) -- (4\z);
            }
            
\foreach \z [evaluate={\znext=int({1+\z})}] in {2,4}
            {
            \draw[edge] (1\z) -- (2\znext);
            }
            \foreach \z [evaluate={\znext=int({1+\z})}] in {4}
            {
            \draw[edge] (3\z) -- (4\znext);
            }
            \foreach \z [evaluate={\znext=int({\z-1})}] in {5}
            {
            \draw[edge] (1\z) -- (2\znext);
            }
            \foreach \z [evaluate={\znext=int({\z-1})}] in {3,5}
            {
            \draw[edge] (3\z) -- (4\znext);
            }

\foreach \z [evaluate={\znext=int({1+\z})}] in {3}
            {
            \draw[edge] (2\z) -- (3\znext);
            }
            \foreach \z [evaluate={\znext=int({\z-1})}] in {4}
            {
            \draw[edge] (2\z) -- (3\znext);
            }
            \coordinate (c15) at (15);
            \coordinate (c16) at (16);

\draw[dotted, thick,shorten >=4pt, shorten <=4pt] (15,5) -- (17,5);
            \draw[dotted, thick,shorten >=4pt, shorten <=4pt] (15,3) -- (17,3);
            \draw[dotted, thick,shorten >=4pt, shorten <=4pt] (15,1) -- (17,1);
            
        \end{scope}		

\begin{scope}[shift={(2,0)}]
            
            \node[vertex] (a2) at (-2,0) {} ;

            \draw[edge] (a2) -- (41) ;

            \node[label={[label distance=-3pt]180:\scriptsize{$a_{2}$}}] at (a2) {};

            \foreach \x in {2,...,5}
            {
            \node[label={[label distance=-3pt]90:\scriptsize{$g_{1\x}$}}] at (1\x) {};
            }

            \foreach \x in {1,2,...,5}
            {
            \node[label={[label distance=-3pt]-90:\scriptsize{$g_{4\x}$}}] at (4\x) {};
            }
            \node[label={[label distance=-5pt]180:\scriptsize{$g_{23}$}}] at (23) {};
            \node[label={[label distance=-5pt]180:\scriptsize{$g_{33}$}}] at (33) {};

        \end{scope}
    \end{tikzpicture}
    \caption{$A_{n,5} \setminus \n[g_{21}]$}
    
    \end{subfigure}
    \caption{}
    \label{fig:An5minus}
\end{figure}

\begin{claim}{\label{homA_{n,5}}}
	\begin{equation*}\ind(A_{n,5}) \simeq  \begin{cases}
			 \vee^3\Sp^1   & \ \ \text{if $n = 0$,} \\
			\Sigma^2\ind(C_{n-1,5}) \vee \Sigma^2\ind(C_{n-1,5}) \vee \Sigma^2\ind(D_{n-1,5}) & \ \ \text{if $n \geq 1$.} \\
		\end{cases}
	\end{equation*}
	
\end{claim}

\begin{proof}

Since $a_1$ is a simplicial vertex in $A_{n,5}$ (see \Cref{fig:An}) with $N(a_1)=\{g_{11},g_{21}\}$, using \Cref{Simplicial Vertex Lemma}, we have   
    \begin{equation*}
        \ind(A_{n,5}) \simeq \Sigma\ind(A_{n,5} \setminus \n[g_{11}]) \vee \Sigma\ind(A_{n,5} \setminus \n[g_{21}]).
    \end{equation*}
    First, let $n = 0$.  Note that $A_{0,5} \setminus \n[g_{11}]$ is isomorphic to the complete graph $K_3$, and $A_{0,5} \setminus \n[g_{21}]$ is isomorphic to $K_2$. Therefore, $\ind(A_{0,5} \setminus \n[g_{11}])\simeq \mathbb{S}^0\vee \mathbb{S}^0$ and $\ind(A_{0,5} \setminus \n[g_{21}])\simeq \mathbb{S}^0$. Hence, we have the stated conclusion for $n=0$.
    
Now, assume that $n \geq 1$.  Since $a_2$ is a simplicial vertex in $A_{n,5} \setminus \n[g_{11}]$ with $N(a_2)=\{g_{31},g_{41}\}$, and a simplicial vertex in $A_{n,5} \setminus \n[g_{21}]$ with $N(a_2)=\{g_{41}\}$. Therefore, \Cref{Simplicial Vertex Lemma} provides us with
\begin{align*}
    \ind(A_{n,5}) \simeq \Sigma^2\ind(A_{n,5}\setminus\n[g_{11},g_{31}]) &\vee \Sigma^2\ind(A_{n,5}\setminus\n[g_{11},g_{41}]) \vee \Sigma^2\ind(A_{n,5}\setminus\n[g_{21},g_{41}]).
\end{align*}  
Note that the graphs $A_{n,5} \setminus \n[g_{11},g_{31}]$ and $A_{n,5} \setminus \n[g_{21},g_{41}]$ are isomorphic to $C_{n-1,5}$, whereas the graph $A_{n,5} \setminus \n[g_{11},g_{41}]$ is isomorphic to $D_{n-1,5}$ (refer to \Cref{fig:An5minus,,fig:Cn,,fig:Dn}). Therefore,
   $ \ind(A_{n,5}) \simeq \Sigma^2\ind(C_{n-1,5}) \vee \Sigma^2\ind(D_{n-1,5}) \vee \Sigma^2\ind(C_{n-1,5})$.
\end{proof}

\begin{claim}{\label{homBn}}

\begin{equation*}
	\ind(B_{n,5}) \simeq \begin{cases}
		\Sp^1 \vee \Sp^2  & \ \ \text{if \ $n = 0$}, \\
	\Sigma\ind(C_{n,5}) \vee \Sigma^2\ind(\Lambda_{n-1,5}) & \ \ \text{if \ $n \geq 1$}. 
	\end{cases}
	\end{equation*}
\end{claim}
\begin{figure}[h]
    \centering
    \begin{subfigure}[b]{0.3\linewidth}
        \begin{tikzpicture}[scale=0.3]
        \begin{scope}
\foreach \x in {1,2,3}
            {
            \node[vertex] (1\x) at (2*\x,3*2) {};
            }
\foreach \x in {1,2,3}
            {
            \node[vertex] (2\x) at (2*\x,2*2) {};
            }
\foreach \x in {1,2,3}
            {
            \node[vertex] (3\x) at (2*\x,1*2) {};
            }
\foreach \x in {1,2,3}
            {
            \node[vertex] (4\x) at (2*\x,0*2) {};
            }

            \foreach \y [evaluate={\y as \x using {\y+1}}] in {1,2} 
            {
            \draw[edge] (1\y) -- (1\x);
            }
            \foreach \y [evaluate={\y as \x using {\y+1}}] in {1,2} 
            {
            \draw[edge] (2\y) -- (2\x);
            }
            \foreach \y [evaluate={\y as \x using {\y+1}}] in {1,2} 
            {
            \draw[edge] (3\y) -- (3\x);
            }
            \foreach \y [evaluate={\y as \x using {\y+1}}] in {1,2} 
            {
            \draw[edge] (4\y) -- (4\x);
            }
            
\foreach \z in {1,2,3}
            {\draw[edge] (1\z) -- (2\z);}
            \foreach \z in {1,2,3}
            {
            \draw[edge] (2\z) -- (3\z);
            }
            \foreach \z in {1,2,3}
            {
            \draw[edge] (3\z) -- (4\z);
            }
            
\foreach \z [evaluate={\znext=int({1+\z})}] in {2}
            {
            \draw[edge] (1\z) -- (2\znext);
            \draw[edge] (3\z) -- (4\znext);
            }
            \foreach \z [evaluate={\znext=int({\z-1})}] in {3}
            {
            \draw[edge] (1\z) -- (2\znext);
            \draw[edge] (3\z) -- (4\znext);
            }

\foreach \z [evaluate={\znext=int({1+\z})}] in {1}
            {
            \draw[edge] (2\z) -- (3\znext);
            }
            \foreach \z [evaluate={\znext=int({\z-1})}] in {2}
            {
            \draw[edge] (2\z) -- (3\znext);
            }
            \coordinate (c15) at (15);
            \coordinate (c16) at (16);

\draw[dotted, thick,shorten >=4pt, shorten <=4pt] (6,5) -- (8,5);
            \draw[dotted, thick,shorten >=4pt, shorten <=4pt] (6,3) -- (8,3);
            \draw[dotted, thick,shorten >=4pt, shorten <=4pt] (6,1) -- (8,1);

        \end{scope}		

\begin{scope}[shift={(2,0)}]
            
            \node[vertex] (b1) at (-2,6) {} ;
            \node[vertex] (b2) at (-2,4) {} ;
            \node[vertex] (b4) at (-2,0) {} ;
            \node[vertex] (b31) at (-4,2) {} ;
            \node[vertex] (b41) at (-4,0) {} ;

            \draw[edge] (b1) -- (11) ;
            \draw[edge] (b1) -- (21) ;
            \draw[edge] (b1) -- (b2) ;
            \draw[edge] (b2) -- (21) ;
            \draw[edge] (b2) -- (11) ;
            \draw[edge] (b4) -- (41) ;
            \draw[edge] (b4) -- (31) ;

            \draw[edge] (b31) -- (b41) ;
            \draw[edge] (b41) -- (b4) ;
            \draw[edge] (b31) -- (b2) ;

            \draw[edge] (41) -- (b2) ;

            \node[label={[label distance=-3pt]180:\footnotesize{$b_{3}$}}] at (b1) {};
            \node[label={[label distance=-3pt]180:\footnotesize{$b_{4}$}}] at (b2) {};
            \node[label={[label distance=-3pt]270:\footnotesize{$b_{6}$}}] at (b4) {};
            \node[label={[label distance=-3pt]180:\footnotesize{$b_{1}$}}] at (b31) {};
            \node[label={[label distance=-3pt]180:\footnotesize{$b_{2}$}}] at (b41) {};

            \node[label={[label distance=-8pt]30:\footnotesize{$g_{21}$}}] at (21) {};
            \node[label={[label distance=-8pt]-30:\footnotesize{$g_{31}$}}] at (31) {};
            \foreach \x in {1,2,3}
            {
            \node[label={[label distance=-3pt]90:\footnotesize{$g_{1\x}$}}] at (1\x) {};
            \node[label={[label distance=-3pt]-90:\footnotesize{$g_{4\x}$}}] at (4\x) {};
            }

        \end{scope}
        
    \end{tikzpicture}
    \caption{$B_{n,5}^1$}
    \label{subfig: B-n-1}
    \end{subfigure}
    \begin{subfigure}[b]{0.3\linewidth}
        \begin{tikzpicture}[scale=0.3]
        \begin{scope}
\foreach \x in {1,2,3}
            {
            \node[vertex] (1\x) at (2*\x,3*2) {};
            }
\foreach \x in {1,2,3}
            {
            \node[vertex] (2\x) at (2*\x,2*2) {};
            }
\foreach \x in {1,2,3}
            {
            \node[vertex] (3\x) at (2*\x,1*2) {};
            }
\foreach \x in {1,2,3}
            {
            \node[vertex] (4\x) at (2*\x,0*2) {};
            }

            \foreach \y [evaluate={\y as \x using {\y+1}}] in {1,2} 
            {
            \draw[edge] (1\y) -- (1\x);
            }
            \foreach \y [evaluate={\y as \x using {\y+1}}] in {1,2} 
            {
            \draw[edge] (2\y) -- (2\x);
            }
            \foreach \y [evaluate={\y as \x using {\y+1}}] in {1,2} 
            {
            \draw[edge] (3\y) -- (3\x);
            }
            \foreach \y [evaluate={\y as \x using {\y+1}}] in {1,2} 
            {
            \draw[edge] (4\y) -- (4\x);
            }
            
\foreach \z in {1,2,3}
            {\draw[edge] (1\z) -- (2\z);}
            \foreach \z in {1,2,3}
            {
            \draw[edge] (2\z) -- (3\z);
            }
            \foreach \z in {1,2,3}
            {
            \draw[edge] (3\z) -- (4\z);
            }
            
\foreach \z [evaluate={\znext=int({1+\z})}] in {2}
            {
            \draw[edge] (1\z) -- (2\znext);
            \draw[edge] (3\z) -- (4\znext);
            }
            \foreach \z [evaluate={\znext=int({\z-1})}] in {3}
            {
            \draw[edge] (1\z) -- (2\znext);
            \draw[edge] (3\z) -- (4\znext);
            }

\foreach \z [evaluate={\znext=int({1+\z})}] in {1}
            {
            \draw[edge] (2\z) -- (3\znext);
            }
            \foreach \z [evaluate={\znext=int({\z-1})}] in {2}
            {
            \draw[edge] (2\z) -- (3\znext);
            }
            \coordinate (c15) at (15);
            \coordinate (c16) at (16);

\draw[dotted, thick,shorten >=4pt, shorten <=4pt] (6,5) -- (8,5);
            \draw[dotted, thick,shorten >=4pt, shorten <=4pt] (6,3) -- (8,3);
            \draw[dotted, thick,shorten >=4pt, shorten <=4pt] (6,1) -- (8,1);

        \end{scope}		

\begin{scope}[shift={(2,0)}]
            
            \node[vertex] (b1) at (-2,6) {} ;
            \node[vertex] (b2) at (-2,4) {} ;
            \node[vertex] (b4) at (-2,0) {} ;
            \node[vertex] (b31) at (-4,2) {} ;
            \node[vertex] (b41) at (-4,0) {} ;

            \draw[edge] (b1) -- (11) ;
            \draw[edge] (b1) -- (21) ;
            \draw[edge] (b1) -- (b2) ;
            \draw[edge] (b2) -- (21) ;
            \draw[edge] (b2) -- (11) ;
            \draw[edge] (b4) -- (41) ;
            \draw[edge] (b4) -- (31) ;

            \draw[edge] (b31) -- (b41) ;
            \draw[edge] (b41) -- (b4) ;

            \draw[edge] (41) -- (b2) ;
            \draw[edge] (31) -- (b2) ;

            \node[label={[label distance=-3pt]180:\footnotesize{$b_{3}$}}] at (b1) {};
            \node[label={[label distance=-3pt]180:\footnotesize{$b_{4}$}}] at (b2) {};
            \node[label={[label distance=-3pt]270:\footnotesize{$b_{6}$}}] at (b4) {};
            \node[label={[label distance=-3pt]180:\footnotesize{$b_{1}$}}] at (b31) {};
            \node[label={[label distance=-3pt]180:\footnotesize{$b_{2}$}}] at (b41) {};

            \node[label={[label distance=-8pt]30:\footnotesize{$g_{21}$}}] at (21) {};
            \node[label={[label distance=-8pt]-30:\footnotesize{$g_{31}$}}] at (31) {};
            \foreach \x in {1,2,3}
            {
            \node[label={[label distance=-3pt]90:\footnotesize{$g_{1\x}$}}] at (1\x) {};
            \node[label={[label distance=-3pt]-90:\footnotesize{$g_{4\x}$}}] at (4\x) {};
            }

        \end{scope}
        
    \end{tikzpicture}
    \caption{$B_{n,5}^2$}
    \label{subfig: B-n-2}
    \end{subfigure}
    \begin{subfigure}[b]{0.3\linewidth}
        \begin{tikzpicture}[scale=0.3]
               \begin{scope}
\foreach \x in {1,2,3}
            {
            \node[vertex] (1\x) at (2*\x,3*2) {};
            }
\foreach \x in {1,2,3}
            {
            \node[vertex] (2\x) at (2*\x,2*2) {};
            }
\foreach \x in {1,2,3}
            {
            \node[vertex] (3\x) at (2*\x,1*2) {};
            }
\foreach \x in {1,2,3}
            {
            \node[vertex] (4\x) at (2*\x,0*2) {};
            }

            \foreach \y [evaluate={\y as \x using {\y+1}}] in {1,2} 
            {
            \draw[edge] (1\y) -- (1\x);
            }
            \foreach \y [evaluate={\y as \x using {\y+1}}] in {1,2} 
            {
            \draw[edge] (2\y) -- (2\x);
            }
            \foreach \y [evaluate={\y as \x using {\y+1}}] in {1,2} 
            {
            \draw[edge] (3\y) -- (3\x);
            }
            \foreach \y [evaluate={\y as \x using {\y+1}}] in {1,2} 
            {
            \draw[edge] (4\y) -- (4\x);
            }
            
\foreach \z in {1,2,3}
            {\draw[edge] (1\z) -- (2\z);}
            \foreach \z in {1,2,3}
            {
            \draw[edge] (2\z) -- (3\z);
            }
            \foreach \z in {1,2,3}
            {
            \draw[edge] (3\z) -- (4\z);
            }
            
\foreach \z [evaluate={\znext=int({1+\z})}] in {2}
            {
            \draw[edge] (1\z) -- (2\znext);
            \draw[edge] (3\z) -- (4\znext);
            }
            \foreach \z [evaluate={\znext=int({\z-1})}] in {3}
            {
            \draw[edge] (1\z) -- (2\znext);
            \draw[edge] (3\z) -- (4\znext);
            }

\foreach \z [evaluate={\znext=int({1+\z})}] in {1}
            {
            \draw[edge] (2\z) -- (3\znext);
            }
            \foreach \z [evaluate={\znext=int({\z-1})}] in {2}
            {
            \draw[edge] (2\z) -- (3\znext);
            }
            \coordinate (c15) at (15);
            \coordinate (c16) at (16);

\draw[dotted, thick,shorten >=4pt, shorten <=4pt] (6,5) -- (8,5);
            \draw[dotted, thick,shorten >=4pt, shorten <=4pt] (6,3) -- (8,3);
            \draw[dotted, thick,shorten >=4pt, shorten <=4pt] (6,1) -- (8,1);

        \end{scope}		

\begin{scope}[shift={(2,0)}]
            
            \node[vertex] (b1) at (-2,6) {} ;
            \node[vertex] (b2) at (-2,4) {} ;
            \node[vertex] (b31) at (-4,2) {} ;
            \node[vertex] (b41) at (-4,0) {} ;

            \draw[edge] (b1) -- (11) ;
            \draw[edge] (b1) -- (21) ;
            \draw[edge] (b1) -- (b2) ;
            \draw[edge] (b2) -- (21) ;
            \draw[edge] (b2) -- (11) ;

            \draw[edge] (b31) -- (b41) ;

            \draw[edge] (41) -- (b2) ;
            \draw[edge] (31) -- (b2) ;

            \node[label={[label distance=-3pt]180:\footnotesize{$b_{3}$}}] at (b1) {};
            \node[label={[label distance=-3pt]180:\footnotesize{$b_{4}$}}] at (b2) {};
            \node[label={[label distance=-3pt]180:\footnotesize{$b_{1}$}}] at (b31) {};
            \node[label={[label distance=-3pt]180:\footnotesize{$b_{2}$}}] at (b41) {};

            \node[label={[label distance=-8pt]30:\footnotesize{$g_{21}$}}] at (21) {};
            \node[label={[label distance=-8pt]-30:\footnotesize{$g_{31}$}}] at (31) {};
            \foreach \x in {1,2,3}
            {
            \node[label={[label distance=-3pt]90:\footnotesize{$g_{1\x}$}}] at (1\x) {};
            \node[label={[label distance=-3pt]-90:\footnotesize{$g_{4\x}$}}] at (4\x) {};
            }

        \end{scope}
        
    \end{tikzpicture}
    \caption{$B_{n,5}^3 \sqcup \{b_1b_2\}$}
    \label{subfig: B-n-3}
    \end{subfigure}
    \caption{}
\end{figure}
 \begin{proof}
 	 First, let $ n =0$. Since $g_{41}$ is a simplicial vertex in $B_{0,5}$ with $N(g_{41})=\{g_{31},b_5,b_{6}\}$,  (see \Cref{fig:Bn}), using \Cref{Simplicial Vertex Lemma}, we get $
 	\ind(B_{0,5}) \simeq \Sigma\ind(B_{0,5}\setminus \n[g_{31}])\vee \Sigma\ind(B_{0,5} \setminus \n[b_{6}])\vee \Sigma\ind(B_{0,5}\setminus \n[b_5])$.
 	Now, $b_2$ is an isolated vertex in the graph $B_{0,5}\setminus \n[b_5]$. Therefore,  $\ind(B_{0,5}\setminus \n[b_5])$ is contractible. Since $\n(b_2)\subseteq \n(b_4)$ in the graph $B_{0,5}\setminus \n[g_{31}]$, using \Cref{Edge deletion 1} we get $\ind(B_{0,5}\setminus \n[g_{31}])\simeq \ind(B_{0,5}\setminus (\n[g_{31}]\cup \{b_4\}))$. Since the graph $B_{0,5}\setminus (\n[g_{31}]\cup \{b_4\})$ consists of two disjoint edges, we have $\ind(B_{0,5}\setminus \n[g_{31}])\simeq\ind(K_2)\ast\ind(K_2)$, that is, $\ind(B_{0,5}\setminus \n[g_{31}])\simeq \mathbb{S}^1$. Similarly, $\n(b_1)\subseteq\n(b_3)\cap\n(g_{11})\cap \n(g_{21})$ in the graph  $B_{0,5} \setminus \n[b_{6}]$, which implies that $\ind(B_{0,5} \setminus \n[b_{6}])\simeq \ind(B_{0,5} \setminus (\n[b_{6}]\cup\{b_3,g_{11},g_{21}\}))$. Since the graph $B_{0,5} \setminus (\n[b_{6}]\cup\{b_3,g_{11},g_{21}\})$ is isomorphic to the complete graph $K_2$, we have $\ind(B_{0,5} \setminus \n[b_{6}])\simeq \mathbb{S}^0$. Therefore, using the fact that $\ind(B_{0,5}\setminus \n[g_{31}])\simeq \mathbb{S}^1$ and $\ind(B_{0,5} \setminus \n[b_{6}])\simeq\mathbb{S}^0$, we get $\ind(B_{0,5}) \simeq \Sp^1 \vee \Sp^2$.

 	Now, assume $n \geq 1$.
 	  Since $N(b_2) \subseteq N(b_5)$ in the graph $B_{n,5}$ (see \Cref{fig:Bn}),  using \Cref{Folding Lemma}, we get $\ind(B_{n,5}) \simeq \ind(B_{n,5} \setminus b_5)$. Observe that $[b_4,g_{41};b_2]$ is an edge-invariant triplet in $B_{n,5}\setminus b_5$. Therefore, by \Cref{Edge deletion 1}, we have $\ind(B_{n,5}) \simeq \ind(B_{n,5}^1)$, where $B_{n,5}^1 \coloneq (B_{n,5}\setminus b_{5}) {+} \{b_4g_{41}\}$ (see \Cref{subfig: B-n-1}). Furthermore, $[b_{4},g_{31};b_2]$ is an edge-invariant triplet in $B_{n,5}^1$, and $[b_1,b_4;b_6]$ is an  edge-invariant triplet in  $B_{n,5}^1+ \{b_4g_{31}\}$. Let $B_{n,5}^2 \coloneqq (B_{n,5}^1{+}\{b_4g_{31}\})- \{b_1b_4\}$ (see \Cref{subfig: B-n-2}). Then, $\ind(B_{n,5}) \simeq \ind(B_{n,5}^2)$. Since  $N(b_1) \subseteq N(b_6)$ in $B_{n,5}^2$, using  \Cref{Folding Lemma}, we have  $\ind(B_{n,5}^2) \simeq \ind(B_{n,5}^2 \setminus b_6)$. Note that the graph $B_{n,5}^2 \setminus b_6$ comprises two connected components: an edge $b_1b_2$ and another component that we denote by $B_{n,5}^3$ (see \Cref{subfig: B-n-3}). This implies that $\ind(B_{n,5}) \simeq \ind(B_{n,5}^2)\simeq \Sigma \ind(B_{n,5}^3)$. Observe that $\ind(B_{n,5}^3 \setminus \n[b_4])*\{b_3\}$ is a subcomplex of $\ind(B_{n,5}^3 \setminus b_4)$, which implies that the inclusion map $\ind(B_{n,5}^3 \setminus \n[b_4]) \hookrightarrow \ind(B_{n,5}^3 \setminus b_4)$ is null homotopic. Therefore, by \Cref{Link and Deletion}, we get
    \begin{equation*}
        \ind(B_{n,5}^3) \simeq \ind(B_{n,5}^3 \setminus b_4) \vee \Sigma\ind(B_{n,5}^3 \setminus \n[b_4]).
    \end{equation*}
    Note that, the graph $B_{n,5}^3 \setminus \n[b_4]$ is isomorphic to the graph $\Lambda_{n-1,5}$, and the graph $B_{n,5}^3 \setminus b_4$ is isomorphic to $C_{n,5}$ (refer to \Cref{fig:Lambda-n} and \Cref{fig:Cn}). Further, since $\ind(B_{n,5}) \simeq \Sigma \ind(B_{n,5}^3)$, we have
    $\ind(B_{n,5}) \simeq \Sigma\ind(C_{n,5}) \vee \Sigma^2\ind(\Lambda_{n-1,5}).$
\end{proof}

\begin{claim}{\label{homCn}} 
    \begin{equation*}
    \ind(C_{n,5}) \simeq \begin{cases} 
    	\mathbb{S}^1  &  \text{if $n = 0$,} \\
    		\vee^3\mathbb{S}^2
              & \text{if $n = 1$,} \\
            \Sigma^{3}\ind(A_{n-2,5}) \bigvee \Sigma^2\ind(E_{n-1,5})\bigvee \Sigma^{3}\ind(F_{n-2,5}) & \text{if $n \geq 2$.}
            \end{cases}
    \end{equation*}
\end{claim}

	\begin{figure}[h!]
		\centering
		\begin{subfigure}[b]{0.22\textwidth}
			\begin{tikzpicture}[scale=0.3]
				\begin{scope}
					\foreach \x in {1,2,...,5}
					{
						\node[vertex] (1\x) at (2.5*\x,3*2) {};
					}
					\foreach \x in {1,3,4,5}
					{
						\node[vertex] (2\x) at (2.5*\x,2*2) {};
					}
					\foreach \x in {2,...,5}
					{
						\node[vertex] (3\x) at (2.5*\x,1*2) {};
					}
					\foreach \x in {3,...,5}
					{
						\node[vertex] (4\x) at (2.5*\x,0*2) {};
					}

					\foreach \y [evaluate={\y as \x using {\y+1}}] in {1,2,3,4} 
					{
						\draw[edge] (1\y) -- (1\x);
					}
					\foreach \y [evaluate={\y as \x using {\y+1}}] in {3,4} 
					{
						\draw[edge] (2\y) -- (2\x);
					}
					\foreach \y [evaluate={\y as \x using {\y+1}}] in {2,...,4} 
					{
						\draw[edge] (3\y) -- (3\x);
					}
					\foreach \y [evaluate={\y as \x using {\y+1}}] in {3,...,4} 
					{
						\draw[edge] (4\y) -- (4\x);
					}
					
					\foreach \z in {1,3,4,5}
					{
						\draw[edge] (1\z) -- (2\z);
					}
					\foreach \z in {3,4,5}
					{
						\draw[edge] (2\z) -- (3\z);
					}
					\foreach \z in {3,4,5}
					{
						\draw[edge] (3\z) -- (4\z);
					}
					
					\foreach \z [evaluate={\znext=int({1+\z})}] in {2,4}
					{
						\draw[edge] (1\z) -- (2\znext);
					}
					\foreach \z [evaluate={\znext=int({1+\z})}] in {2,4}
					{
						\draw[edge] (3\z) -- (4\znext);
					}
					\foreach \z [evaluate={\znext=int({\z-1})}] in {5}
					{
						\draw[edge] (1\z) -- (2\znext);
					}
					\foreach \z [evaluate={\znext=int({\z-1})}] in {5}
					{
						\draw[edge] (3\z) -- (4\znext);
					}
					
					\foreach \z [evaluate={\znext=int({1+\z})}] in {1,3}
					{
						\draw[edge] (2\z) -- (3\znext);
					}
					\foreach \z [evaluate={\znext=int({\z-1})}] in {4}
					{
						\draw[edge] (2\z) -- (3\znext);
					}

					\draw[dotted, thick,shorten >=4pt, shorten <=4pt] (12.2,5) -- (14.2,5);
					\draw[dotted, thick,shorten >=4pt, shorten <=4pt] (12.2,3) -- (14.2,3);
					\draw[dotted, thick,shorten >=4pt, shorten <=4pt] (12.2,1) -- (14.2,1);

					\foreach \x in {1,2,...,5}
					{
						\node[label={[label distance=-3pt]90:\scriptsize{$g_{1\x}$}}] at (1\x) {};
					} 
					\foreach \x in {3,...,5}
					{
						\node[label={[label distance=-3pt]-90:\scriptsize{$g_{4\x}$}}] at (4\x) {};
					} 
					\node[label={[label distance=-3pt]180:\scriptsize{$g_{32}$}}] at (32) {};
					\node[label={[label distance=-3pt]-90:\scriptsize{$g_{21}$}}] at (21) {};
					\node[label={[label distance=-7pt]210:\scriptsize{$g_{23}$}}] at (23) {};
				\end{scope}		
				
			\end{tikzpicture}
			\caption{$C_{n,5}^2$}
			\label{subfig: C-n-2}
		\end{subfigure}
		\begin{subfigure}[b]{0.22\linewidth}
			\begin{tikzpicture}[scale=0.3]
				\begin{scope}
					\foreach \x in {1,2,...,5}
					{
						\node[vertex] (1\x) at (2.5*\x,3*2) {};
					}
					\foreach \x in {1,3,4,5}
					{
						\node[vertex] (2\x) at (2.5*\x,2*2) {};
					}
					\foreach \x in {2,...,5}
					{
						\node[vertex] (3\x) at (2.5*\x,1*2) {};
					}
					\foreach \x in {3,...,5}
					{
						\node[vertex] (4\x) at (2.5*\x,0*2) {};
					}

					\foreach \y [evaluate={\y as \x using {\y+1}}] in {1,3,4} 
					{
						\draw[edge] (1\y) -- (1\x);
					}
					\foreach \y [evaluate={\y as \x using {\y+1}}] in {3,4} 
					{
						\draw[edge] (2\y) -- (2\x);
					}
					\foreach \y [evaluate={\y as \x using {\y+1}}] in {2,...,4} 
					{
						\draw[edge] (3\y) -- (3\x);
					}
					\foreach \y [evaluate={\y as \x using {\y+1}}] in {3,...,4} 
					{
						\draw[edge] (4\y) -- (4\x);
					}
					
					\foreach \z in {1,3,4,5}
					{
						\draw[edge] (1\z) -- (2\z);
					}
					\foreach \z in {3,4,5}
					{
						\draw[edge] (2\z) -- (3\z);
					}
					\foreach \z in {3,4,5}
					{
						\draw[edge] (3\z) -- (4\z);
					}
					
					\foreach \z [evaluate={\znext=int({1+\z})}] in {4}
					{
						\draw[edge] (1\z) -- (2\znext);
					}
					\foreach \z [evaluate={\znext=int({1+\z})}] in {2,4}
					{
						\draw[edge] (3\z) -- (4\znext);
					}
					\foreach \z [evaluate={\znext=int({\z-1})}] in {5}
					{
						\draw[edge] (1\z) -- (2\znext);
					}
					\foreach \z [evaluate={\znext=int({\z-1})}] in {5}
					{
						\draw[edge] (3\z) -- (4\znext);
					}
					
					\foreach \z [evaluate={\znext=int({1+\z})}] in {1,3}
					{
						\draw[edge] (2\z) -- (3\znext);
					}
					\foreach \z [evaluate={\znext=int({\z-1})}] in {4}
					{
						\draw[edge] (2\z) -- (3\znext);
					}
				%	\coordinate (c15) at (15);
				%	\coordinate (c16) at (16);

						\draw[dotted, thick,shorten >=4pt, shorten <=4pt] (12.2,5) -- (14.2,5);
					\draw[dotted, thick,shorten >=4pt, shorten <=4pt] (12.2,3) -- (14.2,3);
					\draw[dotted, thick,shorten >=4pt, shorten <=4pt] (12.2,1) -- (14.2,1);
					
					\node[label={[label distance=-3pt]90:\scriptsize{$g_{11}$}}] at (11) {};
					\foreach \x in {2,...,5}
					{
						\node[label={[label distance=-3pt]90:\scriptsize{$g_{1\x}$}}] at (1\x) {};}
					\foreach \x in {3,...,5}
					{
						\node[label={[label distance=-3pt]270:\scriptsize{$g_{4\x}$}}] at (4\x) {};
					} 
					\node[label={[label distance=-3pt]180:\scriptsize{$g_{32}$}}] at (32) {};
					\node[label={[label distance=-3pt]-90:\scriptsize{$g_{21}$}}] at (21) {};
					\node[label={[label distance=-7pt]10:\scriptsize{$g_{23}$}}] at (23) {};

					\draw[edge] (13) --(32) {};
					\draw[edge] (23) --(32) {};

				\end{scope}

			\end{tikzpicture}
			\caption{$C_{n,5}^3$}
			\label{subfig: C-n-3}
		\end{subfigure}
		\begin{subfigure}[b]{0.22\textwidth}
			\begin{tikzpicture}[scale=0.3]
				\begin{scope}
					\foreach \x in {1,2,...,5}
					{
						\node[vertex] (1\x) at (2.5*\x,3*2) {};
					}
					\foreach \x in {3,4,5}
					{
						\node[vertex] (2\x) at (2.5*\x,2*2) {};
					}
					\foreach \x in {3,...,5}
					{
						\node[vertex] (3\x) at (2.5*\x,1*2) {};
					}
					\foreach \x in {2,3,...,5}
					{
						\node[vertex] (4\x) at (2.5*\x,0*2) {};
					}

					\foreach \y [evaluate={\y as \x using {\y+1}}] in {1,2,3,4} 
					{
						\draw[edge] (1\y) -- (1\x);
					}
					\foreach \y [evaluate={\y as \x using {\y+1}}] in {3,4} 
					{
						\draw[edge] (2\y) -- (2\x);
					}
					\foreach \y [evaluate={\y as \x using {\y+1}}] in {3,...,4} 
					{
						\draw[edge] (3\y) -- (3\x);
					}
					\foreach \y [evaluate={\y as \x using {\y+1}}] in {3,...,4} 
					{
						\draw[edge] (4\y) -- (4\x);
					}
					
					\foreach \z in {3,4,5}
					{
						\draw[edge] (1\z) -- (2\z);
					}
					\foreach \z in {3,4,5}
					{
						\draw[edge] (2\z) -- (3\z);
					}
					\foreach \z in {3,4,5}
					{
						\draw[edge] (3\z) -- (4\z);
					}
					
					\foreach \z [evaluate={\znext=int({1+\z})}] in {2,4}
					{
						\draw[edge] (1\z) -- (2\znext);
					}
					\foreach \z [evaluate={\znext=int({1+\z})}] in {4}
					{
						\draw[edge] (3\z) -- (4\znext);
					}
					\foreach \z [evaluate={\znext=int({\z-1})}] in {5}
					{
						\draw[edge] (1\z) -- (2\znext);
					}
					\foreach \z [evaluate={\znext=int({\z-1})}] in {5}
					{
						\draw[edge] (3\z) -- (4\znext);
					}
					
					\foreach \z [evaluate={\znext=int({1+\z})}] in {3}
					{
						\draw[edge] (2\z) -- (3\znext);
					}
					\foreach \z [evaluate={\znext=int({\z-1})}] in {4}
					{
						\draw[edge] (2\z) -- (3\znext);
					}

					\draw[dotted, thick,shorten >=4pt, shorten <=4pt] (12.2,5) -- (14.2,5);
					\draw[dotted, thick,shorten >=4pt, shorten <=4pt] (12.2,3) -- (14.2,3);
					\draw[dotted, thick,shorten >=4pt, shorten <=4pt] (12.2,1) -- (14.2,1);

					\draw[edge] (42) -- (33);
					\draw[edge] (42) -- (43);
					
					\foreach \x in {1,2,...,5}
					{
						\node[label={[label distance=-3pt]90:\scriptsize{$g_{1\x}$}}] at (1\x) {};
					} 
					\foreach \x in {2,3,...,5}
					{
						\node[label={[label distance=-3pt]-90:\scriptsize{$g_{4\x}$}}] at (4\x) {};
					}
					\node[label={[label distance=-5pt]180:\scriptsize{$g_{23}$}}] at (23) {};
					\node[label={[label distance=-5pt]180:\scriptsize{$g_{33}$}}] at (33) {};
				\end{scope}		
				
			\end{tikzpicture}
			\caption{$C_{n,5}^4$}
			\label{subfig: C-n-4}
		\end{subfigure}
		\begin{subfigure}[b]{0.22\linewidth}
			\begin{tikzpicture}[scale=0.3]
				\begin{scope}
					\foreach \x in {4,5}
					{
						\node[vertex] (1\x) at (2.5*\x,3*2) {};
					}
					\foreach \x in {4,5}
					{
						\node[vertex] (2\x) at (2.5*\x,2*2) {};
					}
					\foreach \x in {3,...,5}
					{
						\node[vertex] (3\x) at (2.5*\x,1*2) {};
					}
					\foreach \x in {2,3,...,5}
					{
						\node[vertex] (4\x) at (2.5*\x,0*2) {};
					}

					\foreach \y [evaluate={\y as \x using {\y+1}}] in {4} 
					{
						\draw[edge] (1\y) -- (1\x);
					}
					\foreach \y [evaluate={\y as \x using {\y+1}}] in {4} 
					{
						\draw[edge] (2\y) -- (2\x);
					}
					\foreach \y [evaluate={\y as \x using {\y+1}}] in {3,...,4} 
					{
						\draw[edge] (3\y) -- (3\x);
					}
					\foreach \y [evaluate={\y as \x using {\y+1}}] in {2,3,...,4} 
					{
						\draw[edge] (4\y) -- (4\x);
					}
					
					\foreach \z in {4,5}
					{\draw[edge] (1\z) -- (2\z);}
					\foreach \z in {4,5}
					{
						\draw[edge] (2\z) -- (3\z);
					}
					\foreach \z in {3,4,5}
					{
						\draw[edge] (3\z) -- (4\z);
					}
					
					\foreach \z [evaluate={\znext=int({1+\z})}] in {4}
					{
						\draw[edge] (1\z) -- (2\znext);
					}
					\foreach \z [evaluate={\znext=int({1+\z})}] in {4}
					{
						\draw[edge] (3\z) -- (4\znext);
					}
					\foreach \z [evaluate={\znext=int({\z-1})}] in {5}
					{
						\draw[edge] (1\z) -- (2\znext);
					}
					\foreach \z [evaluate={\znext=int({\z-1})}] in {3,5}
					{
						\draw[edge] (3\z) -- (4\znext);
					}
					
					\foreach \z [evaluate={\znext=int({1+\z})}] in {}
					{
						\draw[edge] (2\z) -- (3\znext);
					}
					\foreach \z [evaluate={\znext=int({\z-1})}] in {4}
					{
						\draw[edge] (2\z) -- (3\znext);
					}
				%	\coordinate (c15) at (15);
				%	\coordinate (c16) at (16);

					\draw[dotted, thick,shorten >=4pt, shorten <=4pt] (12.2,5) -- (14.2,5);
					\draw[dotted, thick,shorten >=4pt, shorten <=4pt] (12.2,3) -- (14.2,3);
					\draw[dotted, thick,shorten >=4pt, shorten <=4pt] (12.2,1) -- (14.2,1);

					\foreach \x in {4,5}
					{
						\node[label={[label distance=-3pt]90:\scriptsize{$g_{1\x}$}}] at (1\x) {};
					} 
					\foreach \x in {2,3,4,5}
					{
						\node[label={[label distance=-3pt]-90:\scriptsize{$g_{4\x}$}}] at (4\x) {};
					} 
					\node[label={[label distance=-3pt]180:\scriptsize{$g_{33}$}}] at (33) {};
					\node[label={[label distance=-5pt]180:\scriptsize{$g_{24}$}}] at (24) {};
				\end{scope}

			\end{tikzpicture}
			\caption{$C_{n,5}^5$}
			\label{subfig: C-n-5}
		\end{subfigure}
		\caption{}
		\label{fig:Cn-inter}
	\end{figure}

\begin{proof}

	Let $n = 0$.  Since $N(g_{11})\subseteq N(g_{31})$ in the graph $C_{0,5}$ (see \Cref{fig:Cn}), using \Cref{Folding Lemma}, we conclude that $\ind(C_{0,5})\simeq\ind(C_{0,5}\setminus g_{31})$. However, the graph $C_{0,5}\setminus g_{31}$ consists of exactly two disjoint edges. Therefore, $\ind(C_{0,5})\simeq \mathbb{S}^1.$

   Now, assume $n \geq 1$. Since $N(g_{11}) \subseteq N(g_{22})$ in the graph $C_{n,5}$,  by \Cref{Folding Lemma}, we have $\ind(C_{n,5}) \simeq \ind(C_{n,5} \setminus g_{22})$. We denote the graph $C_{n,5} \setminus g_{22}$ by $C_{n,5}^1$. Since $c_1$ is a simplicial vertex and $N(c_1) = \{g_{31}, g_{41}\}$ in $C_{n,5}^1$, using \Cref{Simplicial Vertex Lemma}, we get that 
    \begin{equation}{\label{eq: Cn sim-vertex}}
         \ind(C_{n,5}^1) \simeq \Sigma \ind(C_{n,5}^1 \setminus \n[g_{31}]) \vee \Sigma \ind(C_{n,5}^1 \setminus \n[g_{41}]).
    \end{equation}
    Let $C_{n,5}^2 \coloneqq C_{n,5}^1 \setminus \n[g_{41}]$ (see \Cref{subfig: C-n-2}). Observe that $[g_{13},g_{32};g_{11}]$ and $[g_{23},g_{32};g_{11}]$ are edge-invariant triplets in $C_{n,5}^2$ and $C_{n,5}^2{+}\{g_{13}g_{32}\}$, respectively. Therefore, using \Cref{Edge deletion 1}, $\ind(C_{n,5}^2)\simeq \ind(C_{n,5}^2{+}\{g_{13}g_{32},$ $g_{23}g_{32}\})$. Analogously, $[g_{12},g_{13};g_{21}]$ and $[g_{12},g_{23};g_{21}]$ form edge invariant triplets in $C_{n,5}^2+\{g_{13}g_{32},g_{32}g_{23}\}$ and $C_{n,5}^2+\{g_{13}g_{32},g_{32}g_{23}\}-\{g_{12}g_{13}\}$, respectively. Thus, we have $\ind(C_{n,5}^2+\{g_{13}g_{32},g_{32}g_{23}\}) \simeq \ind(C_{n,5}^2+\{g_{13}g_{32},g_{32}g_{23}\}-\{g_{12}g_{13},g_{12}g_{23}\})$. We denote the  graph $C_{n,5}^2+\{g_{13}g_{32},g_{32}g_{23}\}-\{g_{12}g_{13},g_{12}g_{23}\}$ by $C_{n,5}^3$, depicted in \Cref{subfig: C-n-3}. Then, $\ind(C_{n,5}^2)\simeq\ind(C_{n,5}^3)$. Furthermore, $N(g_{12}) \subseteq N(g_{21})$ in $C_{n,5}^3$. Therefore,  by \Cref{Folding Lemma}, we have $\ind(C_{n,5}^3) \simeq \ind(C_{n,5}^3 \setminus g_{21})$. Observe that the graph $C_{n,5}^3 \setminus g_{21}$ consists of two connected components, one of which is an edge $g_{11}g_{12}$ and the other connected component is isomorphic to $E_{n-1,5}$ (see \Cref{fig:En}). Thus, we have
    \begin{equation}{\label{eq: Cn simeq En-1}}
        \ind(C_{n,5}^2)\simeq \Sigma \ind(E_{n-1,5}).
    \end{equation}
    On the other hand, let $C_{n,5}^4 \coloneqq C_{n,5}^1 \setminus \n[g_{31}]$ (see \Cref{subfig: C-n-4}). Since $g_{11}$ is a simplicial vertex in  $C_{n,5}^4$ with a unique neighbour $g_{12}$, by \Cref{Simplicial Vertex Lemma}, we have $\ind(C_{n,5}^4) \simeq \Sigma \ind(C_{n,5}^4 \setminus \n[g_{12}])$. Let  $C_{n,5}^5 \coloneq  C_{n,5}^4 \setminus \n[g_{12}] $ (see \Cref{subfig: C-n-5}).  Now, for $n=1$, we have $C_{1,5}^5 \simeq K_3$. Therefore, from  \eqref{eq: Cn sim-vertex} and \eqref{eq: Cn simeq En-1}, we get that $\ind(C_{1,5}) \simeq \Sigma^2\ind(E_{0,5}) \vee \Sigma^2\ind(K_3)$.
      Since $\ind(E_{0,5})\simeq \Sp^0$ (using \Cref{homEn}) and $\ind(K_3) \simeq \Sp^0 \vee \Sp^0$, we have $\ind(C_{1,5}) \simeq \vee^3 \Sp^2$. 
    
    Let  $n\geq 2$. Since $g_{42}$ is a simplicial vertex and $N(g_{42})=\{g_{33},g_{43}\}$ in $C_{n,5}^5$,  \Cref{Simplicial Vertex Lemma} provides us 
    \begin{equation*}
        \ind(C_{n,5}^5) \simeq \Sigma \ind(C_{n,5}^5 \setminus \n[g_{33}]) \vee \Sigma\ind(C_{n,5}^5 \setminus \n[g_{43}]).
    \end{equation*}
    However, the graph $C_{n,5}^5 \setminus \n[g_{33}]$ is isomorphic to the graph $A_{n-2,5}$ and the graph $C_{n,5}^5 \setminus \n[g_{43}]$ is isomorphic to $F_{n-2,5}$ (refer to \Cref{fig:An} and \Cref{fig:Fn}). Therefore, we have
    \begin{equation}{\label{eq:Cn5-sim-vert}}
        \ind(C_{n,5}^4) \simeq \Sigma \ind(C_{n,5}^5) \simeq \Sigma^2\ind(A_{n-2,5}) \vee \Sigma^2 \ind(F_{n-2,5}),
    \end{equation}
which, along with \eqref{eq: Cn sim-vertex} and \eqref{eq: Cn simeq En-1} gives the required homotopy equivalence.
\end{proof}

\begin{claim}{\label{homDn}} 
    \begin{equation*}
        \ind(D_{n,5}) \simeq \begin{cases}
        	 \vee^2\mathbb{S}^1 & \text{if $n =0$,} \\
              \vee^3\mathbb{S}^2 & \text{if $n =1$,} \\
            \Sigma^3\ind(\Lambda_{n-2,5}) \vee \Sigma^2\ind(E_{n-1,5}) \vee \Sigma^2 \ind(E_{n-1,5}) & \text{if $n \geq 2$.}
        \end{cases}
    \end{equation*}
\end{claim}
\begin{proof} 
	
First, let $n=0$. 	Since  $g_{11}$ is a simplicial vertex in $D_{0,5}$ (see \Cref{fig:Dn}) with $N(g_{11})=\{g_{21},d_1\}$, we get
	$\ind (D_{0,5})\simeq\Sigma\ind(D_{0,5} \setminus \n[g_{21}]) \vee \Sigma\ind(D_{0,5} \setminus \n[d_1]).$
	However, both the graphs $D_{0,5} \setminus \n[g_{21}]$, and $D_{0,5} \setminus \n[d_1]$ are isomorphic to the complete graph $K_2$. Therefore, both  $\ind(D_{0,5} \setminus \n[g_{21}])$ and  $\ind(D_{0,5} \setminus \n[d_1])$ are homotopy equivalent to $\mathbb{S}^0$. Hence the result.
	
	Now, assume $n \geq 1$.  Observe that  $[g_{31},g_{12};d_{1}]$ and $[g_{21},g_{42};d_{2}]$ are edge-invariant triplets in $D_{n,5}$ (see \Cref{fig:Dn}) and $D_{n,5}{+}\{g_{31}g_{12}\}$, respectively. Hence by  using \Cref{Edge deletion 1}, we have $\ind(D_{n,5})\simeq \ind(D_{n,5}{+}\{g_{31}g_{12},g_{21}g_{42}\})$. Furthermore, $[d_{1},g_{21};g_{41}]$ and $[d_{2},g_{31};g_{11}]$ are edge-invariant triplets in $D_{n,5}{+}\{g_{31}g_{12},g_{21}g_{42}\}$ and $(D_{n,5}{+}\{g_{31}g_{12},g_{21}g_{42}\})- \{d_1g_{21}\}$, respectively. Let $D_{n,5}^1$ be the graph that we obtain after removing the edges $d_1g_{21}$ and $d_2g_{31}$ from $D_{n,5}{+}\{g_{31}g_{12},g_{21}g_{42}\}$ (see \Cref{subfig: D-n-1}). Then, it follows from \Cref{Edge deletion 1} that $\ind(D_{n,5})\simeq\ind(D_{n,5}^1)$.

Now, we proceed with $D_{n,5}^1$. Observe that $[g_{21},g_{41};d_{1}]$, $[g_{12},g_{41};d_{1}]$, and $[d_{2},g_{41};g_{11}]$ are edge-invariant triplets in the graphs $D_{n,5}^1$, $D_{n,5}^1+\{g_{21}g_{41}\}$, and $D_{n,5}^1+\{g_{21}g_{41}, g_{12}g_{41}\}$, respectively. Let $D_{n,5}^2\coloneq(D_{n,5}^1+\{g_{21}g_{41}, g_{41}g_{12}\})-\{d_2g_{41}\}$ (see \Cref{subfig: D-n-2}). Then \Cref{Edge deletion 1} implies that $\ind(D_{n,5}^1)\simeq \ind(D_{n,5}^2)$. Note that $N(d_2)\subseteq N(g_{11})$ in the resultant graph $D_{n,5}^2$. Therefore, from \Cref{Folding Lemma}, we have $\ind(D_{n,5}^2)\simeq \ind(D_{n,5}^2\setminus g_{11})$. Since $[g_{12},g_{33};g_{21}]$, and $[g_{12},g_{43};g_{21}]$ are edge-invariant triplets in the graphs $D_{n,5}^2\setminus g_{11}$ and $(D_{n,5}^2\setminus g_{11})+\{g_{12}g_{33}\}$, respectively, we get $\ind(D_{n,5}^2\setminus g_{11})\simeq \ind((D_{n,5}^2\setminus g_{11})+\{g_{12}g_{33},g_{12}g_{43}\})$. Let us take $D_{n,5}^3\coloneq (D_{n,5}^2\setminus g_{11})+\{g_{12}g_{33},g_{12}g_{43}\}$ (see \Cref{subfig: D-n-3}). Then, $\ind(D_{n,5}^2)\simeq \ind(D_{n,5}^3)$. Furthermore, $[g_{31},g_{12};g_{42}]$ and $[g_{12},g_{41};g_{32}]$ are edge-invariant triplets in the graphs $D_{n,5}^3$ and $D_{n,5}^3-\{g_{31}g_{12}\}$, respectively. Hence   $\ind(D_{n,5}^3)\simeq \ind(D_{n,5}^3-\{g_{31}g_{12},g_{12}g_{41}\})$. Let us denote the graph $D_{n,5}^3-\{g_{31}g_{12},g_{12}g_{41}\}$ by $D_{n,5}^4$ (see \Cref{subfig: D-n-4}). Then, $\ind(D_{n,5}^3)\simeq\ind(D_{n,5}^4)$.

    \begin{figure}[h]
    \centering
    \begin{subfigure}[b]{0.22\textwidth}
    \begin{tikzpicture}[scale=0.3]
        \begin{scope}
\foreach \x in {1,2,3}
            {
            \node[vertex] (1\x) at (2*\x,3*2) {};
            }
\foreach \x in {1,2,3}
            {
            \node[vertex] (2\x) at (2*\x,2*2) {};
            }
\foreach \x in {1,2,3}
            {
            \node[vertex] (3\x) at (2*\x,1*2) {};
            }
\foreach \x in {1,2,3}
            {
            \node[vertex] (4\x) at (2*\x,0*2) {};
            }
            
\foreach \z in {1,2,3,4}
            {
            \foreach \y [evaluate={\y as \x using {\y+1}}] in {1,2} 
            {
            \draw[edge] (\z\y) -- (\z\x);
            }
            }
            \draw[edge] (12) -- (13) ;
            
\foreach \z in {1,2,3}
            {\draw[edge] (1\z) -- (2\z);
            \draw[edge] (2\z) -- (3\z);
            \draw[edge] (3\z) -- (4\z);
            }
            
\foreach \z [evaluate={\znext=int({1+\z})}] in {2}
            {
            \draw[edge] (1\z) -- (2\znext);
            \draw[edge] (3\z) -- (4\znext);
            }
            \foreach \z [evaluate={\znext=int({\z-1})}] in {3}
            {
            \draw[edge] (1\z) -- (2\znext);
            \draw[edge] (3\z) -- (4\znext);
            }

\foreach \z [evaluate={\znext=int({1+\z})}] in {1}
            {
            \draw[edge] (2\z) -- (3\znext);
            }
            \foreach \z [evaluate={\znext=int({\z-1})}] in {2}
            {
            \draw[edge] (2\z) -- (3\znext);
            }

\draw[dotted, thick,shorten >=4pt, shorten <=4pt] (6,5) -- (8,5);
            \draw[dotted, thick,shorten >=4pt, shorten <=4pt] (6,3) -- (8,3);
            \draw[dotted, thick,shorten >=4pt, shorten <=4pt] (6,1) -- (8,1);
            
\end{scope}		

\begin{scope}[shift={(2,0)}]
            \node[vertex] (d1) at (-2,4) {} ;
            \node[vertex] (d2) at (-2,2) {} ;

            \draw[edge] (d1) -- (11) ;
\draw[edge] (d1) -- (d2) ;
\draw[edge] (d2) -- (41) ;

            Edges after preprocessing
           
            \draw[edge] (31) -- (12) ;
            \draw[edge] (21) -- (42) ;

        \end{scope}
\foreach \x in {1,2,3}
            {
            \node[label={[label distance=-3pt]90:\scriptsize{$g_{1\x}$}}] at (1\x) {};
            } 
            \node[label={[label distance=-7pt]180:\scriptsize{$g_{21}$}}] at (21) {};
            \node[label={[label distance=-7pt]180:\scriptsize{$g_{31}$}}] at (31) {};
            \foreach \x in {1,2,3}
            {
            \node[label={[label distance=-3pt]-90:\scriptsize{$g_{4\x}$}}] at (4\x) {};
            } 
            \node[label={[label distance=-3pt]180:\scriptsize{$d_{1}$}}] at (d1) {};
            \node[label={[label distance=-3pt]180:\scriptsize{$d_{2}$}}] at (d2) {};

\end{tikzpicture}
    \caption{$D_{n,5}^1$}
    \label{subfig: D-n-1}
    \end{subfigure}
    \begin{subfigure}[b]{0.22\textwidth}
    \begin{tikzpicture}[scale=0.3]
        \begin{scope}
\foreach \x in {1,2,3}
            {
            \node[vertex] (1\x) at (2*\x,3*2) {};
            }
\foreach \x in {1,2,3}
            {
            \node[vertex] (2\x) at (2*\x,2*2) {};
            }
\foreach \x in {1,2,3}
            {
            \node[vertex] (3\x) at (2*\x,1*2) {};
            }
\foreach \x in {1,2,3}
            {
            \node[vertex] (4\x) at (2*\x,0*2) {};
            }
            
\foreach \z in {1,2,3,4}
            {
            \foreach \y [evaluate={\y as \x using {\y+1}}] in {1,2} 
            {
            \draw[edge] (\z\y) -- (\z\x);
            }
            }
            \draw[edge] (12) -- (13) ;
            
\foreach \z in {1,2,3}
            {\draw[edge] (1\z) -- (2\z);
            \draw[edge] (2\z) -- (3\z);
            \draw[edge] (3\z) -- (4\z);
            }
            
\foreach \z [evaluate={\znext=int({1+\z})}] in {2}
            {
            \draw[edge] (1\z) -- (2\znext);
            \draw[edge] (3\z) -- (4\znext);
            }
            \foreach \z [evaluate={\znext=int({\z-1})}] in {3}
            {
            \draw[edge] (1\z) -- (2\znext);
            \draw[edge] (3\z) -- (4\znext);
            }

\foreach \z [evaluate={\znext=int({1+\z})}] in {1}
            {
            \draw[edge] (2\z) -- (3\znext);
            }
            \foreach \z [evaluate={\znext=int({\z-1})}] in {2}
            {
            \draw[edge] (2\z) -- (3\znext);
            }

\draw[dotted, thick,shorten >=4pt, shorten <=4pt] (6,5) -- (8,5);
            \draw[dotted, thick,shorten >=4pt, shorten <=4pt] (6,3) -- (8,3);
            \draw[dotted, thick,shorten >=4pt, shorten <=4pt] (6,1) -- (8,1);
            
\end{scope}		

\begin{scope}[shift={(2,0)}]
            \node[vertex] (d1) at (-2,4) {} ;
            \node[vertex] (d2) at (-2,2) {} ;

            \draw[edge] (d1) -- (11) ;
\draw[edge] (d1) -- (d2) ;

            Edges after preprocessing
           
            \draw[edge] (31) -- (12) ;
            \draw[edge] (21) -- (42) ;
            \draw[edge] (21) to[bend right=30] (41) ;
            \draw[edge] (12) -- (41) ;

        \end{scope}
\foreach \x in {1,2,3}
            {
            \node[label={[label distance=-3pt]90:\scriptsize{$g_{1\x}$}}] at (1\x) {};
            } 
            \node[label={[label distance=-7pt]180:\scriptsize{$g_{21}$}}] at (21) {};
            \node[label={[label distance=-7pt]180:\scriptsize{$g_{31}$}}] at (31) {};
            \foreach \x in {1,2,3}
            {
            \node[label={[label distance=-3pt]-90:\scriptsize{$g_{4\x}$}}] at (4\x) {};
            } 
            \node[label={[label distance=-3pt]180:\scriptsize{$d_{1}$}}] at (d1) {};
            \node[label={[label distance=-3pt]180:\scriptsize{$d_{2}$}}] at (d2) {};

            \node[label={[label distance=-3pt]0:\scriptsize{$g_{33}$}}] at (33) {};

\end{tikzpicture}
    \caption{$D_{n,5}^2$}
    \label{subfig: D-n-2}
    \end{subfigure}
    \begin{subfigure}[b]{0.22\textwidth}
    \begin{tikzpicture}[scale=0.3]
        \begin{scope}
\foreach \x in {2,3}
            {
            \node[vertex] (1\x) at (2*\x,3*2) {};
            }
\foreach \x in {1,2,3}
            {
            \node[vertex] (2\x) at (2*\x,2*2) {};
            }
\foreach \x in {1,2,3}
            {
            \node[vertex] (3\x) at (2*\x,1*2) {};
            }
\foreach \x in {1,2,3}
            {
            \node[vertex] (4\x) at (2*\x,0*2) {};
            }
            
\foreach \z in {2,3,4}
            {
            \foreach \y [evaluate={\y as \x using {\y+1}}] in {1,2} 
            {
            \draw[edge] (\z\y) -- (\z\x);
            }
            }
            \draw[edge] (12) -- (13) ;
            
\foreach \z in {2,3}
            {\draw[edge] (1\z) -- (2\z);
            \draw[edge] (2\z) -- (3\z);
            \draw[edge] (3\z) -- (4\z);
            }
            \draw[edge] (21) -- (31);
            \draw[edge] (31) -- (41);
            
\foreach \z [evaluate={\znext=int({1+\z})}] in {2}
            {
            \draw[edge] (1\z) -- (2\znext);
            \draw[edge] (3\z) -- (4\znext);
            }
            \foreach \z [evaluate={\znext=int({\z-1})}] in {3}
            {
            \draw[edge] (1\z) -- (2\znext);
            \draw[edge] (3\z) -- (4\znext);
            }

\foreach \z [evaluate={\znext=int({1+\z})}] in {1}
            {
            \draw[edge] (2\z) -- (3\znext);
            }
            \foreach \z [evaluate={\znext=int({\z-1})}] in {2}
            {
            \draw[edge] (2\z) -- (3\znext);
            }

\draw[dotted, thick,shorten >=4pt, shorten <=4pt] (6,5) -- (8,5);
            \draw[dotted, thick,shorten >=4pt, shorten <=4pt] (6,3) -- (8,3);
            \draw[dotted, thick,shorten >=4pt, shorten <=4pt] (6,1) -- (8,1);
            
\end{scope}		

\begin{scope}[shift={(2,0)}]
            \node[vertex] (d1) at (-2,4) {} ;
            \node[vertex] (d2) at (-2,2) {} ;

\draw[edge] (d1) -- (d2) ;

\draw[edge] (21) to[bend right=45] (41) ;
            \draw[edge] (21) -- (42) ;
            \draw[edge] (12) -- (33) ;
            \draw[edge] (12) -- (43) ;
            \draw[edge] (31) -- (12) ;
            \draw[edge] (12) -- (41) ;

        \end{scope}
\foreach \x in {2,3}
            {
            \node[label={[label distance=-3pt]90:\scriptsize{$g_{1\x}$}}] at (1\x) {};
            } 
            \foreach \x in {1,2,3}
            {
            \node[label={[label distance=-3pt]-90:\scriptsize{$g_{4\x}$}}] at (4\x) {};
            } 
            \node[label={[label distance=-5pt]90:\scriptsize{$g_{21}$}}] at (21) {};
            \node[label={[label distance=-8pt]45:\scriptsize{$g_{32}$}}] at (32) {};

            \node[label={[label distance=-3pt]180:\scriptsize{$d_{1}$}}] at (d1) {};
            \node[label={[label distance=-3pt]180:\scriptsize{$d_{2}$}}] at (d2) {};
            \node[label={[label distance=-7pt]180:\scriptsize{$g_{31}$}}] at (31) {};

\end{tikzpicture}
    \caption{$D_{n,5}^3$}
    \label{subfig: D-n-3}
    \end{subfigure}
     \begin{subfigure}[b]{0.22\textwidth}
    \begin{tikzpicture}[scale=0.3]
        \begin{scope}
\foreach \x in {2,3}
            {
            \node[vertex] (1\x) at (2*\x,3*2) {};
            }
\foreach \x in {1,2,3}
            {
            \node[vertex] (2\x) at (2*\x,2*2) {};
            }
\foreach \x in {1,2,3}
            {
            \node[vertex] (3\x) at (2*\x,1*2) {};
            }
\foreach \x in {1,2,3}
            {
            \node[vertex] (4\x) at (2*\x,0*2) {};
            }
            
\foreach \z in {2,3,4}
            {
            \foreach \y [evaluate={\y as \x using {\y+1}}] in {1,2} 
            {
            \draw[edge] (\z\y) -- (\z\x);
            }
            }
            \draw[edge] (12) -- (13) ;
            
\foreach \z in {2,3}
            {\draw[edge] (1\z) -- (2\z);
            \draw[edge] (2\z) -- (3\z);
            \draw[edge] (3\z) -- (4\z);
            }
            \draw[edge] (21) -- (31);
            \draw[edge] (31) -- (41);
            
\foreach \z [evaluate={\znext=int({1+\z})}] in {2}
            {
            \draw[edge] (1\z) -- (2\znext);
            \draw[edge] (3\z) -- (4\znext);
            }
            \foreach \z [evaluate={\znext=int({\z-1})}] in {3}
            {
            \draw[edge] (1\z) -- (2\znext);
            \draw[edge] (3\z) -- (4\znext);
            }

\foreach \z [evaluate={\znext=int({1+\z})}] in {1}
            {
            \draw[edge] (2\z) -- (3\znext);
            }
            \foreach \z [evaluate={\znext=int({\z-1})}] in {2}
            {
            \draw[edge] (2\z) -- (3\znext);
            }

\draw[dotted, thick,shorten >=4pt, shorten <=4pt] (6,5) -- (8,5);
            \draw[dotted, thick,shorten >=4pt, shorten <=4pt] (6,3) -- (8,3);
            \draw[dotted, thick,shorten >=4pt, shorten <=4pt] (6,1) -- (8,1);
            
\end{scope}		

\begin{scope}[shift={(2,0)}]
            \node[vertex] (d1) at (-2.5,4) {} ;
            \node[vertex] (d2) at (-2.5,2) {} ;

\draw[edge] (d1) -- (d2) ;

\draw[edge] (21) to[bend right=45] (41) ;
            \draw[edge] (21) -- (42) ;
            \draw[edge] (12) -- (33) ;
            \draw[edge] (12) -- (43) ;

        \end{scope}
\foreach \x in {2,3}
            {
            \node[label={[label distance=-3pt]90:\scriptsize{$g_{1\x}$}}] at (1\x) {};
            } 
            \foreach \x in {1,2,3}
            {
            \node[label={[label distance=-3pt]-90:\scriptsize{$g_{4\x}$}}] at (4\x) {};
            } 
            \node[label={[label distance=-5pt]180:\scriptsize{$g_{21}$}}] at (21) {};
            \node[label={[label distance=-8pt]45:\scriptsize{$g_{32}$}}] at (32) {};
            
            \node[label={[label distance=-5pt]180:\scriptsize{$d_{1}$}}] at (d1) {};
            \node[label={[label distance=-5pt]180:\scriptsize{$d_{2}$}}] at (d2) {};
            \node[label={[label distance=-6pt]94:\scriptsize{$g_{22}$}}] at (22) {};
            \node[label={[label distance=-5pt]360:\scriptsize{$g_{33}$}}] at (33) {};

\end{tikzpicture}
    \caption{$D_{n,5}^4$}
    \label{subfig: D-n-4}
     \end{subfigure}
        \begin{subfigure}[b]{0.22\textwidth}
    \begin{tikzpicture}[scale=0.3]
        \begin{scope}
\foreach \x in {2,3}
            {
            \node[vertex] (1\x) at (2*\x,3*2) {};
            }
\foreach \x in {2,3}
            {
            \node[vertex] (2\x) at (2*\x,2*2) {};
            }
\foreach \x in {1,3}
            {
            \node[vertex] (3\x) at (2*\x,1*2) {};
            }
\foreach \x in {1,2,3}
            {
            \node[vertex] (4\x) at (2*\x,0*2) {};
            }
            
\foreach \y [evaluate={\y as \x using {\y+1}}] in {2} 
            {
            \draw[edge] (2\y) -- (2\x);
            }
            \foreach \y [evaluate={\y as \x using {\y+1}}] in {} 
            {
            \draw[edge] (3\y) -- (3\x);
            }
            \foreach \y [evaluate={\y as \x using {\y+1}}] in {1,2} 
            {
            \draw[edge] (4\y) -- (4\x);
            }
            \draw[edge] (12) -- (13) ;
            
\foreach \z in {2,3}
            {
            \draw[edge] (1\z) -- (2\z);
            }
            \foreach \z in {3}
            {
            \draw[edge] (2\z) -- (3\z);
            }
            \foreach \z in {3}
            {
            \draw[edge] (3\z) -- (4\z);
            }
      
            \draw[edge] (31) -- (41);
            
\foreach \z [evaluate={\znext=int({1+\z})}] in {2}
            {
            \draw[edge] (1\z) -- (2\znext);
            }
            \foreach \z [evaluate={\znext=int({\z-1})}] in {3}
            {
            \draw[edge] (1\z) -- (2\znext);
            \draw[edge] (3\z) -- (4\znext);
            }

\foreach \z [evaluate={\znext=int({1+\z})}] in {}
            {
            \draw[edge] (2\z) -- (3\znext);
            }
            \foreach \z [evaluate={\znext=int({\z-1})}] in {2}
            {
            \draw[edge] (2\z) -- (3\znext);
            }

\draw[dotted, thick,shorten >=4pt, shorten <=4pt] (6,5) -- (8,5);
            \draw[dotted, thick,shorten >=4pt, shorten <=4pt] (6,3) -- (8,3);
            \draw[dotted, thick,shorten >=4pt, shorten <=4pt] (6,1) -- (8,1);
            
\end{scope}		

\begin{scope}[shift={(2,0)}]
            \node[vertex] (d1) at (-2,4) {} ;
            \node[vertex] (d2) at (-2,2) {} ;

\draw[edge] (d1) -- (d2) ;

            \draw[edge] (12) -- (33) ;
            \draw[edge] (12) -- (43) ;
            \draw[edge] (22) -- (33) ;
            \draw[edge] (22) -- (43) ;

        \end{scope}
\foreach \x in {2,3}
            {
            \node[label={[label distance=-3pt]90:\scriptsize{$g_{1\x}$}}] at (1\x) {};
            } 
            \foreach \x in {1,2,3}
            {
            \node[label={[label distance=-3pt]-90:\scriptsize{$g_{4\x}$}}] at (4\x) {};
            }
            \node[label={[label distance=-4pt]-180:\scriptsize{$g_{22}$}}] at (22) {};
            \node[label={[label distance=-4pt]0:\scriptsize{$g_{23}$}}] at (23) {};
            \node[label={[label distance=-4pt]0:\scriptsize{$g_{33}$}}] at (33) {};
            \node[label={[label distance=-4pt]90:\scriptsize{$g_{31}$}}] at (31) {};
            \node[label={[label distance=-3pt]180:\scriptsize{$d_{1}$}}] at (d1) {};
            \node[label={[label distance=-3pt]180:\scriptsize{$d_{2}$}}] at (d2) {};

\end{tikzpicture}
    \caption{$D_{n,5}^5$}
    \label{subfig: D-n-5}
    \end{subfigure}
    \begin{subfigure}[b]{0.22\textwidth}
    \begin{tikzpicture}[scale=0.3]
        \begin{scope}
\foreach \x in {2,3,4}
            {
            \node[vertex] (1\x) at (2*\x,3*2) {};
            }
\foreach \x in {2,3,4}
            {
            \node[vertex] (2\x) at (2*\x,2*2) {};
            }
\foreach \x in {1,3,4}
            {
            \node[vertex] (3\x) at (2*\x,1*2) {};
            }
\foreach \x in {1,3,4}
            {
            \node[vertex] (4\x) at (2*\x,0*2) {};
            }
            
\foreach \z in {2}
            {
            \foreach \y [evaluate={\y as \x using {\y+1}}] in {2} 
            {
            \draw[edge] (\z\y) -- (\z\x);
            }
            }
            \draw[edge] (12) -- (13) ;
            
            \foreach \z in {1,2,3,4}
            {\draw[edge] (\z3) -- (\z4);
            
            }
            \draw[edge] (33) -- (24) ;
            \draw[edge] (23) -- (34) ;
\foreach \z in {3,4}
            {\draw[edge] (1\z) -- (2\z);
            \draw[edge] (2\z) -- (3\z);
            \draw[edge] (3\z) -- (4\z);
            }
            \draw[edge] (12) -- (22);
\draw[edge] (31) -- (41);
            
\foreach \z [evaluate={\znext=int({1+\z})}] in {2}
            {
            \draw[edge] (1\z) -- (2\znext);
}
            \foreach \z [evaluate={\znext=int({\z-1})}] in {3}
            {
            \draw[edge] (1\z) -- (2\znext);
}

\draw[dotted, thick,shorten >=4pt, shorten <=4pt] (8,5) -- (10,5);
            \draw[dotted, thick,shorten >=4pt, shorten <=4pt] (8,3) -- (10,3);
            \draw[dotted, thick,shorten >=4pt, shorten <=4pt] (8,1) -- (10,1);
            
\end{scope}		

\begin{scope}[shift={(2,0)}]
            \node[vertex] (d1) at (-2,4) {} ;
            \node[vertex] (d2) at (-2,2) {} ;

\draw[edge] (d1) -- (d2) ;

\draw[edge] (12) -- (33) ;
            \draw[edge] (12) -- (43) ;

            \draw[edge] (22) -- (33) ;
            \draw[edge] (22) -- (43) ;

        \end{scope}
\foreach \x in {2,3}
            {
            \node[label={[label distance=-3pt]90:\scriptsize{$g_{1\x}$}}] at (1\x) {};
            } 
            \foreach \x in {3}
            {
            \node[label={[label distance=-7pt]-90:\scriptsize{$g_{4\x}$}}] at (4\x) {};
            } 
             \node[label={[label distance=-3pt]90:\scriptsize{$g_{14}$}}] at (14) {};
            \node[label={[label distance=-3pt]270:\scriptsize{$g_{44}$}}] at (44) {};
            \node[label={[label distance=-3pt]180:\scriptsize{$d_{1}$}}] at (d1) {};
            \node[label={[label distance=-3pt]180:\scriptsize{$d_{2}$}}] at (d2) {};
            \node[label={[label distance=-3pt]90:\scriptsize{$g_{31}$}}] at (31) {};
            \node[label={[label distance=-5pt]180:\scriptsize{$g_{22}$}}] at (22) {};
            \node[label={[label distance=-3pt]-90:\scriptsize{$g_{41}$}}] at (41) {};

\end{tikzpicture}
    \caption{$D_{n,5}^6$}
    \label{subfig: D-n-6}
    \end{subfigure}
    \hfill
    \caption{}
    
    \end{figure} 
Observe that $\ind(D_{n,5}^4 \setminus \n[g_{21}])*\{g_{41}\} \subseteq\ind(D_{n,5}^4 \setminus g_{21})$, {\it i.e.,} $\ind(D_{n,5}^4 \setminus \n[g_{21}])$ is contractible in $\ind(D_{n,5}^4 \setminus g_{21})$. Therefore, \Cref{Link and Deletion} implies that
    \begin{equation}{\label{eq: D'n}}
    \begin{split}
        \ind(D_{n,5}^4) &\simeq \ind(D_{n,5}^4 \setminus g_{21}) \vee \Sigma\ind(D_{n,5}^4 \setminus \n[g_{21}]) 
       \simeq \ind(D_{n,5}^4 \setminus g_{21}) \vee \Sigma^2 \ind(E_{n-1,5}),
    \end{split}
    \end{equation}
    where the second homotopy equivalence holds because the graph $D_{n,5}^4 \setminus \n[g_{21}]$ has two connected components; one component is an edge $d_1d_2$ and the other component is isomorphic to the graph $E_{n-1,5}$ (see \Cref{fig:En}).
    
    Now, consider the graph $D_{n,5}^4 \setminus g_{21}$. Since $N(g_{41})\subseteq N(g_{32})$ in $D_{n,5}^4 \setminus g_{21}$, we have  $\ind(D_{n,5}^4 \setminus g_{21}) \simeq \ind(D_{n,5}^4 \setminus \{g_{21}, g_{32}\})$. Furthermore, $[g_{22},g_{33};g_{41}]$ and $[g_{22},g_{43};g_{41}]$ are edge-invariant triplets in the graphs $D_{n,5}^4 \setminus \{g_{21},g_{32}\}$ and $(D_{n,5}^4 \setminus \{g_{21},g_{32}\})+\{g_{22}g_{33}\}$, respectively. Therefore,  we have $\ind( D_{n,5}^4 \setminus g_{21})\simeq \ind((D_{n,5}^4 \setminus \{g_{21},g_{32}\})+\{g_{22}g_{33},g_{22}g_{43}\})$. Let $D_{n,5}^5$ (see \Cref{subfig: D-n-5}) denote the graph $(D_{n,5}^4 \setminus \{g_{21},g_{32}\})+\{g_{22}g_{33},g_{22}g_{43}\}$. Then, $\ind(D_{n,5}^4 \setminus g_{21})\simeq \ind(D_{n,5}^5)$.

    Since $[g_{22},g_{31};g_{42}]$ is an edge-invariant triplet in the graph $D_{n,5}^5$, we obtain that $\ind(D_{n,5}^5)\simeq \ind(D_{n,5}^5- \{g_{22}g_{31}\})$. Moreover, $N(g_{31})\subseteq N(g_{42})$ in the graph $D_{n,5}^5- \{g_{22}g_{31}\}$ and therefore   $\ind(D_{n,5}^5- \{g_{22}g_{31}\}) \simeq \ind((D_{n,5}^5- \{g_{22}g_{31}\}) \setminus g_{42})$.  Let $D_{n,5}^6$ denote the graph $(D_{n,5}^5- \{g_{22}g_{31}\}) \setminus g_{42}$. Then, $\ind( D_{n,5}^4 \setminus g_{21})\simeq \ind(D_{n,5}^5) \simeq \ind(D_{n,5}^6)$.  
    
Suppose $n=1$. Since $N(g_{13})\subseteq N(g_{33})$ in the graph $D_{1,5}^6$,  we get $\ind(D_{1,5}^6)\simeq \ind(D_{1,5}^6\setminus\{g_{33}\})$. Furthermore, $N(g_{43})\subseteq N(g_{13})\cap N(g_{23})$ in the graph $D_{1,5}^6\setminus\{g_{33}\}$. Therefore, a repeated application of \Cref{Folding Lemma} provides us with $\ind(D_{1,5}^6)\simeq \ind(D_{1,5}^6\setminus\{g_{33},g_{13},g_{23}\})$. Since the graph $D_{1,5}^6\setminus\{g_{33},g_{13},g_{23}\}$ is isomorphic to the disjoint union of $K_3$ and two edges, we have, $\ind (D_{1,5}^6)\simeq \mathbb{S}^2\vee \mathbb{S}^2$. Hence $\ind(D_{1,5}^4 \setminus g_{21})\simeq \ind(D_{1,5}^6) \simeq \mathbb{S}^2\vee \mathbb{S}^2$. Moreover, from \Cref{homEn}, it follows that $\Sigma^2 \ind(E_{0,5})\simeq \mathbb{S}^2$. Hence, $\ind(D_{1,5})\simeq \mathbb{S}^2\vee \mathbb{S}^2\vee \mathbb{S}^2$.

Now, consider $n\geq 2$. Since $\ind(D_{n,5}^6 \setminus \n[g_{12}])*\{g_{22}\} \subseteq \ind(D_{n,5}^6 \setminus g_{12})$, from \Cref{Link and Deletion}, we have 
    \begin{equation}{\label{eq: D''n}}
    \begin{split}
        \ind(D_{n,5}^6) \simeq  \ind(D_{n,5}^6 \setminus g_{12}) \vee \Sigma\ind(D_{n,5}^6 \setminus \n[g_{12}]).
    \end{split}  
    \end{equation}
   However, the graph $D_{n,5}^6 \setminus g_{12}$ consists of three connected components, two copies of $K_2$, and the third component is isomorphic to $E_{n-1,5}$. Similarly, the graph $D_{n,5}^6 \setminus \n[g_{12}]$ has three connected components, two copies of $K_2$, and the third component is isomorphic to $\Lambda_{n-2,5}$. Therefore, from \eqref{eq: D''n}, we get $\ind(D_{n,5}^6)\simeq \Sigma^2\ind(E_{n-1,5}) \vee \Sigma^3\ind(\Lambda_{n-2,5})$. Hence, from \eqref{eq: D'n}, we get that $\ind(D_{n,5}) \simeq \Sigma^2\ind(E_{n-1,5}) \vee \Sigma^3\ind(\Lambda_{n-2,5}) \vee \Sigma^2 \ind(E_{n-1,5})$.
\end{proof}

\begin{figure}[h]
	\centering
	\begin{subfigure}[b]{0.28\textwidth}
		\begin{tikzpicture}[scale=0.3]
			\begin{scope}
				\foreach \x in {1,2,3}
				{
					\node[vertex] (1\x) at (3*\x,3*2) {};
				}
				\foreach \x in {1,2,3}
				{
					\node[vertex] (2\x) at (3*\x,2*2) {};
				}
				\foreach \x in {1,2,3}
				{
					\node[vertex] (3\x) at (3*\x,1*2) {};
				}
				\foreach \x in {1,2,3}
				{
					\node[vertex] (4\x) at (3*\x,0*2) {};
				}

				\foreach \y [evaluate={\y as \x using {\y+1}}] in {1,2} 
				{
					\draw[edge] (1\y) -- (1\x);
				}
				\foreach \y [evaluate={\y as \x using {\y+1}}] in {1,2} 
				{
					\draw[edge] (2\y) -- (2\x);
				}
				\foreach \y [evaluate={\y as \x using {\y+1}}] in {1,2} 
				{
					\draw[edge] (3\y) -- (3\x);
				}
				\foreach \y [evaluate={\y as \x using {\y+1}}] in {1,2} 
				{
					\draw[edge] (4\y) -- (4\x);
				}
				
				\foreach \z in {1,2,3}
				{
					\draw[edge] (1\z) -- (2\z);
				}
				\foreach \z in {1,2,3}
				{
					\draw[edge] (2\z) -- (3\z);
				}
				\foreach \z in {1,2,3}
				{
					\draw[edge] (3\z) -- (4\z);
				}
				
				\foreach \z [evaluate={\znext=int({1+\z})}] in {2}
				{
					\draw[edge] (1\z) -- (2\znext);
				}
				\foreach \z [evaluate={\znext=int({1+\z})}] in {2}
				{
					\draw[edge] (3\z) -- (4\znext);
				}
				\foreach \z [evaluate={\znext=int({\z-1})}] in {3}
				{
					\draw[edge] (1\z) -- (2\znext);
				}
				\foreach \z [evaluate={\znext=int({\z-1})}] in {3}
				{
					\draw[edge] (3\z) -- (4\znext);
				}
				
				\foreach \z [evaluate={\znext=int({1+\z})}] in {}
				{
					\draw[edge] (2\z) -- (3\znext);
				}
				\foreach \z [evaluate={\znext=int({\z-1})}] in {}
				{
					\draw[edge] (2\z) -- (3\znext);
				}
				\draw[dotted, thick,shorten >=4pt, shorten <=4pt] (9,5) -- (11,5);
				\draw[dotted, thick,shorten >=4pt, shorten <=4pt] (9,3) -- (11,3);
				\draw[dotted, thick,shorten >=4pt, shorten <=4pt] (9,1) -- (11,1);

				\node[label={[label distance=-3pt]90:\scriptsize{$g_{11}$}}] at (11) {};
				\foreach \x in {2,3}
				{
					\node[label={[label distance=-3pt]90:\scriptsize{$g_{1\x}$}}] at (1\x) {};
				}
				\foreach \x in {1,2,3}
				{
					\node[label={[label distance=-3pt]-90:\scriptsize{$g_{4\x}$}}] at (4\x) {};
				}
				
				\node[label={[label distance=-7pt]10:\scriptsize{$g_{32}$}}] at (32) {};
				\node[label={[label distance=-7pt]10:\scriptsize{$g_{21}$}}] at (21) {};
				\node[label={[label distance=-7pt]10:\scriptsize{$g_{31}$}}] at (31) {};
				
			\end{scope}		
			
			\begin{scope}[shift={(2,0)}]
				\node[vertex] (e1) at (-1.5,3) {} ;
				
				\draw[edge] (e1) -- (11) ;
				\draw[edge] (e1) -- (21) ;
				\draw[edge] (e1) -- (31) ;
				\draw[edge] (e1) -- (41) ;
				
				\node[label={[label distance=-3pt]180:\scriptsize{$e_{1}$}}] at (e1) {};
			\end{scope}
		\end{tikzpicture}
		\caption{$E_{n,5}^1$}
		\label{subfig: E-n-1}

	\end{subfigure}
	\begin{subfigure}[b]{0.28\textwidth}
		\begin{tikzpicture}[scale=0.3]
			\begin{scope}
				\foreach \x in {1,2,3,4}
				{
					\node[vertex] (1\x) at (3*\x,3*2) {};
				}
				\foreach \x in {1,2,3,4}
				{
					\node[vertex] (2\x) at (3*\x,2*2) {};
				}
				\foreach \x in {1,2,3,4}
				{
					\node[vertex] (3\x) at (3*\x,1*2) {};
				}
				\foreach \x in {1,2,3,4}
				{
					\node[vertex] (4\x) at (3*\x,0*2) {};
				}

				\foreach \y [evaluate={\y as \x using {\y+1}}] in {1,2,3} 
				{
					\draw[edge] (1\y) -- (1\x);
				}
				\foreach \y [evaluate={\y as \x using {\y+1}}] in {1,2,3} 
				{
					\draw[edge] (2\y) -- (2\x);
				}
				\foreach \y [evaluate={\y as \x using {\y+1}}] in {1,2,3} 
				{
					\draw[edge] (3\y) -- (3\x);
				}
				\foreach \y [evaluate={\y as \x using {\y+1}}] in {1,2,3} 
				{
					\draw[edge] (4\y) -- (4\x);
				}
				
				\foreach \z in {1,2,3,4}
				{
					\draw[edge] (1\z) -- (2\z);
				}
				\foreach \z in {1,2,3,4}
				{
					\draw[edge] (2\z) -- (3\z);
				}
				\foreach \z in {1,2,3,4}
				{
					\draw[edge] (3\z) -- (4\z);
				}
				
				\foreach \z [evaluate={\znext=int({1+\z})}] in {2}
				{
					\draw[edge] (1\z) -- (2\znext);
				}
				\foreach \z [evaluate={\znext=int({1+\z})}] in {2}
				{
					\draw[edge] (3\z) -- (4\znext);
				}
				\foreach \z [evaluate={\znext=int({\z-1})}] in {3}
				{
					\draw[edge] (1\z) -- (2\znext);
				}
				\foreach \z [evaluate={\znext=int({\z-1})}] in {3}
				{
					\draw[edge] (3\z) -- (4\znext);
				}
				\draw[edge] (23) -- (34);
				\draw[edge] (33) -- (24);

				\draw[edge] (11) -- (32);
				\draw[edge] (41) -- (32);

				\draw[dotted, thick,shorten >=4pt, shorten <=4pt] (12,5) -- (14,5);
			\draw[dotted, thick,shorten >=4pt, shorten <=4pt] (12,3) -- (14,3);
			\draw[dotted, thick,shorten >=4pt, shorten <=4pt] (12,1) -- (14,1);

				\node[label={[label distance=-3pt]90:\scriptsize{$g_{11}$}}] at (11) {};
				\foreach \x in {2,3}
				{
					\node[label={[label distance=-3pt]90:\scriptsize{$g_{1\x}$}}] at (1\x) {};
				}
				\foreach \x in {1,2,3}
				{
					\node[label={[label distance=-3pt]-90:\scriptsize{$g_{4\x}$}}] at (4\x) {};
				}
				
				\node[label={[label distance=-7pt]10:\scriptsize{$g_{32}$}}] at (32) {};
				\node[label={[label distance=-7pt]30:\scriptsize{$g_{21}$}}] at (21) {};
			\end{scope}		
			
			\begin{scope}[shift={(2,0)}]
				\node[vertex] (e1) at (-1.5,3) {} ;
				
				\draw[edge] (e1) -- (11) ;
				\draw[edge] (e1) -- (21) ;
				\draw[edge] (e1) -- (31) ;
				\draw[edge] (e1) -- (41) ;
				
				\node[label={[label distance=-3pt]180:\scriptsize{$e_{1}$}}] at (e1) {};
			\end{scope}
			
\end{tikzpicture}
		\caption{$E_{n,5}^2$}
		\label{subfig: E-n-2}
	\end{subfigure}
	\begin{subfigure}[b]{0.28\textwidth}
		\begin{tikzpicture}[scale=0.3]
			\begin{scope}
				\foreach \x in {1,2,3}
				{
					\node[vertex] (1\x) at (3*\x,3*2) {};
				}
				\foreach \x in {1,2,3}
				{
					\node[vertex] (2\x) at (3*\x,2*2) {};
				}
				\foreach \x in {1,3}
				{
					\node[vertex] (3\x) at (3*\x,1*2) {};
				}
				\foreach \x in {1,2,3}
				{
					\node[vertex] (4\x) at (3*\x,0*2) {};
				}

				\foreach \y [evaluate={\y as \x using {\y+1}}] in {1,2} 
				{
					\draw[edge] (1\y) -- (1\x);
				}
				\foreach \y [evaluate={\y as \x using {\y+1}}] in {1,2} 
				{
					\draw[edge] (2\y) -- (2\x);
				}
				\foreach \y [evaluate={\y as \x using {\y+1}}] in {} 
				{
					\draw[edge] (3\y) -- (3\x);
				}
				\foreach \y [evaluate={\y as \x using {\y+1}}] in {1,2} 
				{
					\draw[edge] (4\y) -- (4\x);
				}
				
				\foreach \z in {1,2,3}
				{
					\draw[edge] (1\z) -- (2\z);
				}
				\foreach \z in {1,3}
				{
					\draw[edge] (2\z) -- (3\z);
				}
				\foreach \z in {1,3}
				{
					\draw[edge] (3\z) -- (4\z);
				}
				
				\foreach \z [evaluate={\znext=int({1+\z})}] in {2}
				{
					\draw[edge] (1\z) -- (2\znext);
				}
				\foreach \z [evaluate={\znext=int({1+\z})}] in {}
				{
					\draw[edge] (3\z) -- (4\znext);
				}
				\foreach \z [evaluate={\znext=int({\z-1})}] in {3}
				{
					\draw[edge] (1\z) -- (2\znext);
				}
				\foreach \z [evaluate={\znext=int({\z-1})}] in {3}
				{
					\draw[edge] (3\z) -- (4\znext);
				}
				
				\draw[dotted, thick,shorten >=4pt, shorten <=4pt] (9,5) -- (11,5);
			\draw[dotted, thick,shorten >=4pt, shorten <=4pt] (9,3) -- (11,3);
			\draw[dotted, thick,shorten >=4pt, shorten <=4pt] (9,1) -- (11,1);

				\node[label={[label distance=-3pt]90:\scriptsize{$g_{11}$}}] at (11) {};
				\foreach \x in {2,3}
				{
					\node[label={[label distance=-3pt]90:\scriptsize{$g_{1\x}$}}] at (1\x) {};
				}
				\foreach \x in {1,2,3}
				{
					\node[label={[label distance=-3pt]-90:\scriptsize{$g_{4\x}$}}] at (4\x) {};
				}
				\node[label={[label distance=-7pt]30:\scriptsize{$g_{21}$}}] at (21) {};
				\node[label={[label distance=-7pt]-10:\scriptsize{$g_{22}$}}] at (22) {};
				\node[label={[label distance=-7pt]30:\scriptsize{$g_{31}$}}] at (31) {};
			\end{scope}		
			\begin{scope}[shift={(2,0)}]
				\node[vertex] (e1) at (-1.5,3) {} ;
				
				\draw[edge] (e1) -- (11) ;
				\draw[edge] (e1) -- (21) ;
				\draw[edge] (e1) -- (31) ;
				\draw[edge] (e1) -- (41) ;
				
				\node[label={[label distance=-3pt]180:\scriptsize{$e_{1}$}}] at (e1) {};
			\end{scope}
		\end{tikzpicture}
		\caption{$E_{n,5}^3$}
		\label{subfig: E-n-3}

	\end{subfigure}

	\begin{subfigure}[b]{0.4\textwidth}
		\hspace{1 cm} \begin{tikzpicture}[scale=0.3]
			\begin{scope}
				\foreach \x in {1,2,3}
				{
					\node[vertex] (1\x) at (3*\x,3*2) {};
				}
				\foreach \x in {1,2,3}
				{
					\node[vertex] (2\x) at (3*\x,2*2) {};
				}
				\foreach \x in {1,3}
				{
					\node[vertex] (3\x) at (3*\x,1*2) {};
				}
				\foreach \x in {1,2,3}
				{
					\node[vertex] (4\x) at (3*\x,0*2) {};
				}

				\foreach \y [evaluate={\y as \x using {\y+1}}] in {1,2} 
				{
					\draw[edge] (1\y) -- (1\x);
				}
				\foreach \y [evaluate={\y as \x using {\y+1}}] in {1,2} 
				{
					\draw[edge] (2\y) -- (2\x);
				}
				\foreach \y [evaluate={\y as \x using {\y+1}}] in {} 
				{
					\draw[edge] (3\y) -- (3\x);
				}
				\foreach \y [evaluate={\y as \x using {\y+1}}] in {1,2} 
				{
					\draw[edge] (4\y) -- (4\x);
				}
				
				\foreach \z in {1,2,3}
				{
					\draw[edge] (1\z) -- (2\z);
				}
				\foreach \z in {1,3}
				{
					\draw[edge] (2\z) -- (3\z);
				}
				\foreach \z in {1,3}
				{
					\draw[edge] (3\z) -- (4\z);
				}
				
				\foreach \z [evaluate={\znext=int({1+\z})}] in {2}
				{
					\draw[edge] (1\z) -- (2\znext);
				}
				\foreach \z [evaluate={\znext=int({1+\z})}] in {}
				{
					\draw[edge] (3\z) -- (4\znext);
				}
				\foreach \z [evaluate={\znext=int({\z-1})}] in {3}
				{
					\draw[edge] (1\z) -- (2\znext);
				}
				\foreach \z [evaluate={\znext=int({\z-1})}] in {3}
				{
					\draw[edge] (3\z) -- (4\znext);
				}

				\draw[edge] (22) -- (11);
				\draw[edge] (22) -- (41);

				\draw[dotted, thick,shorten >=4pt, shorten <=4pt] (9,5) -- (11,5);
			\draw[dotted, thick,shorten >=4pt, shorten <=4pt] (9,3) -- (11,3);
			\draw[dotted, thick,shorten >=4pt, shorten <=4pt] (9,1) -- (11,1);

				\node[label={[label distance=-3pt]90:\scriptsize{$g_{11}$}}] at (11) {};
				\foreach \x in {2,3}
				{
					\node[label={[label distance=-3pt]90:\scriptsize{$g_{1\x}$}}] at (1\x) {};
				}
				\foreach \x in {1,2,3}
				{
					\node[label={[label distance=-3pt]-90:\scriptsize{$g_{4\x}$}}] at (4\x) {};
				}
				
				\node[label={[label distance=-8pt]-10:\scriptsize{$g_{22}$}}] at (22) {};
				\node[label={[label distance=-8pt]-30:\scriptsize{$g_{21}$}}] at (21) {};
				
			\end{scope}		
			
			\begin{scope}[shift={(2,0)}]
				\node[vertex] (e1) at (-1.5,3) {} ;
				
				\draw[edge] (e1) -- (11) ;
				\draw[edge] (e1) -- (21) ;
				\draw[edge] (e1) -- (31) ;
				\draw[edge] (e1) -- (41) ;
				
				\node[label={[label distance=-3pt]180:\scriptsize{$e_{1}$}}] at (e1) {};
			\end{scope}

		\end{tikzpicture}
		\caption{$E_{n,5}^4$}
		\label{subfig: E-n-4}
	\end{subfigure}
	\begin{subfigure}[b]{0.45\textwidth}
		 \hspace{1 cm} \begin{tikzpicture}[scale=0.3]
			\begin{scope}
				\foreach \x in {1,2,3}
				{
					\node[vertex] (1\x) at (3*\x,3*2) {};
				}
				\foreach \x in {1,3}
				{
					\node[vertex] (2\x) at (3*\x,2*2) {};
				}
				\foreach \x in {1,3}
				{
					\node[vertex] (3\x) at (3*\x,1*2) {};
				}
				\foreach \x in {1,2,3}
				{
					\node[vertex] (4\x) at (3*\x,0*2) {};
				}

				\foreach \y [evaluate={\y as \x using {\y+1}}] in {1,2} 
				{
					\draw[edge] (1\y) -- (1\x);
				}
				\foreach \y [evaluate={\y as \x using {\y+1}}] in {} 
				{
					\draw[edge] (2\y) -- (2\x);
				}
				\foreach \y [evaluate={\y as \x using {\y+1}}] in {} 
				{
					\draw[edge] (3\y) -- (3\x);
				}
				\foreach \y [evaluate={\y as \x using {\y+1}}] in {1,2} 
				{
					\draw[edge] (4\y) -- (4\x);
				}
				
				\foreach \z in {1,3}
				{
					\draw[edge] (1\z) -- (2\z);
				}
				\foreach \z in {1,3}
				{
					\draw[edge] (2\z) -- (3\z);
				}
				\foreach \z in {1,3}
				{
					\draw[edge] (3\z) -- (4\z);
				}
				
				\foreach \z [evaluate={\znext=int({1+\z})}] in {2}
				{
					\draw[edge] (1\z) -- (2\znext);
				}
				\foreach \z [evaluate={\znext=int({1+\z})}] in {}
				{
					\draw[edge] (3\z) -- (4\znext);
				}
				\foreach \z [evaluate={\znext=int({\z-1})}] in {}
				{
					\draw[edge] (1\z) -- (2\znext);
				}
				\foreach \z [evaluate={\znext=int({\z-1})}] in {3}
				{
					\draw[edge] (3\z) -- (4\znext);
				}

				\draw[dotted, thick,shorten >=4pt, shorten <=4pt] (9,5) -- (11,5);
			\draw[dotted, thick,shorten >=4pt, shorten <=4pt] (9,3) -- (11,3);
			\draw[dotted, thick,shorten >=4pt, shorten <=4pt] (9,1) -- (11,1);

				\node[label={[label distance=-3pt]90:\scriptsize{$g_{11}$}}] at (11) {};
				\foreach \x in {2,3}
				{
					\node[label={[label distance=-3pt]90:\scriptsize{$g_{1\x}$}}] at (1\x) {};
				}
				\foreach \x in {1,2,3}
				{
					\node[label={[label distance=-3pt]-90:\scriptsize{$g_{4\x}$}}] at (4\x) {};
				}
				\node[label={[label distance=-5pt]0:\scriptsize{$g_{21}$}}] at (21) {};
				\node[label={[label distance=-5pt]0:\scriptsize{$g_{31}$}}] at (31) {};

			\end{scope}		
			
			\begin{scope}[shift={(2,0)}]
				\node[vertex] (e1) at (-1.5,3) {} ;
				
				\draw[edge] (e1) -- (11) ;
				\draw[edge] (e1) -- (21) ;
				\draw[edge] (e1) -- (31) ;
				\draw[edge] (e1) -- (41) ;
				
				\node[label={[label distance=-3pt]180:\scriptsize{$e_{1}$}}] at (e1) {};
			\end{scope}

		\end{tikzpicture}
		\caption{$E_{n,5}^5$}
		\label{subfig: E-n-5}
	\end{subfigure}
	\caption{}
\end{figure}
\begin{claim}{\label{homEn}} 
    {\small
    \begin{equation*}
      \ind(E_{n,5}) \simeq \begin{cases}
      	\mathbb{S}^0 &  \text{if $n = 0$,} \\
              \bigvee^{10} \mathbb{S}^2 & \text{if $n = 1$,} \\
            \left({\bigvee}^{2} \Sigma^3\ind(A_{n-2,5})\right) \vee \left({\bigvee}^{2} \Sigma^3\ind(F_{n-2,5})\right) \vee \Sigma\ind(A_{n-1,5}) \vee \ind(\Gamma_{n,5})
            & \text{if $n \geq 2$.}
        \end{cases}
    \end{equation*}}
\end{claim}
\begin{proof}
	First, let $n = 0$. Since $N(g_{11})\subseteq N(g_{31})$ in the graph $E_{0,5}$, by \Cref{Folding Lemma}, we get $\ind (E_{0,5})\simeq \ind (E_{0,5}\setminus g_{31}) $. Furthermore, $N(g_{41})\subseteq N(g_{11})$ in $E_{n,5} \setminus g_{31}$ implies $\ind (E_{0,5})\simeq \ind (E_{0,5}\setminus \{g_{11},g_{31}\})$. Note that the graph $E_{0,5}\setminus \{g_{11},g_{31}\}$ is a path on three vertices. Therefore, $ \ind (E_{0,5})\simeq \mathbb{S}^0 $.

    Now, assume that $n \geq 1$. Observe that $[g_{21},g_{32};g_{41}]$ and $[g_{22},g_{31};g_{11}]$ are edge-invariant triplets in the graphs $E_{n,5}$ and $E_{n,5} - \{g_{21}g_{32}\}$, respectively. Thus, we have $\ind(E_{n,5}) \simeq \ind(E_{n,5} - \{g_{21}g_{32},g_{22}g_{31}\})$. Let $E_{n,5}^1 \coloneq E_{n,5} - \{g_{21}g_{32},g_{22}g_{31}\}$ (see \Cref{subfig: E-n-1}). Since $[g_{11},g_{32};g_{41}]$ and $[g_{41},g_{32};g_{21}]$ are edge-invariant triplets in $E_{n,5}^1$ and $E_{n,5}^1{+}\{g_{11}g_{32}\}$, respectively,  we have $\ind(E_{n,5}^1) \simeq \ind(E_{n,5}^1+\{g_{11}g_{32},g_{41}g_{32}\})$. Let $E_{n,5}^2 \coloneq E_{n,5}^1{+}\{g_{11}g_{32},g_{41}g_{32}\}$ (see \Cref{subfig: E-n-2}). Note that $\ind(E_{n,5}^2 \setminus \n[g_{32}])*\{g_{42}\}$ is a subcomplex of $\ind(E_{n,5}^2\setminus g_{32})$. This implies that the inclusion map $\ind(E_{n,5}^2 \setminus \n[g_{32}]) \hookrightarrow \ind(E_{n,5}^2\setminus g_{32})$ is null homotopic, and by \Cref{Link and Deletion}, we have
    \begin{equation}{\label{En first}}
        \ind(E_{n,5})\simeq \ind(E_{n,5}^2) \simeq \ind(E_{n,5}^2\setminus g_{32}) \vee \Sigma \ind(E_{n,5}^2 \setminus \n[g_{32}]).
    \end{equation}
    The graph $E_{n,5}^2 \setminus \n[g_{32}]$ has two connected components; one component consists of the edge $e_1g_{21}$, and the other component is isomorphic to $C_{n,5}^5$ (see \Cref{subfig: C-n-5}). Therefore,
    \begin{equation}{\label{En2-g32}}
        \ind(E_{n,5}^2 \setminus \n[g_{32}]) \simeq \Sigma\ind(C_{n,5}^5).
    \end{equation}
    Let $E_{n,5}^3 \coloneq E_{n,5}^2 \setminus g_{32}$ (see \Cref{subfig: E-n-3}). Using \Cref{Edge deletion 1}, for the edge-invariant triplets $[g_{22},g_{41};g_{11}]$ and $[g_{22},g_{11};g_{31}]$ in $E_{n,5}^3$ and $E_{n,5}^3{+}\{g_{22}g_{41}\}$, respectively, we get $\ind(E_{n,5}^3) \simeq \ind(E_{n,5}^3{+}\{g_{22}g_{11},g_{22}g_{41}\})$. Let $E_{n,5}^4 \coloneq E_{n,5}^3{+}\{g_{22}g_{11},g_{22}g_{41}\}$ (see \Cref{subfig: E-n-4}). Then $\ind(E_{n,5}^3)\simeq \ind(E_{n,5}^4)$. Furthermore,  $\ind(E_{n,5}^4 \setminus \n[g_{22}])*\{g_{12}\}$ is a subcomplex of $\ind(E_{n,5}^4 \setminus g_{22})$. Therefore, using \Cref{Link and Deletion}, we have
    \begin{equation*}
        \ind(E_{n,5}^3) \simeq\ind(E_{n,5}^4) \simeq \ind(E_{n,5}^4 \setminus g_{22}) \vee \Sigma \ind(E_{n,5}^4\setminus \n[g_{22}]).
    \end{equation*}
    The graph $E_{n,5}^4\setminus \n[g_{22}]$ has two connected components; one component consists of the edge $e_1g_{31}$, and the other component is a graph isomorphic to $C_{n,5}^5$ (see \Cref{subfig: C-n-5}). Hence
    \begin{equation}{\label{En4-link-del}}
        \ind(E_{n,5}^3) \simeq \ind(E_{n,5}^4 \setminus g_{22}) \vee \Sigma^2 \ind(C_{n,5}^5).
    \end{equation} 
    Finally, let $E_{n,5}^5 \coloneq E_{n,5}^4 \setminus g_{22}$. Since $\ind(E_{n,5}^5 \setminus \n[e_1])*\{g_{21}\} \subseteq \ind(E_{n,5}^5 \setminus e_1)$, the inclusion map  $\ind(E_{n,5}^5 \setminus \n[e_1])*\{g_{21}\} \hookrightarrow \ind(E_{n,5}^5 \setminus e_1)$ is null homotopic. Therefore, by \Cref{Link and Deletion}, we have
    \begin{equation*}
    \begin{split}
        \ind(E^5_{n,5}) &\simeq \ind(E^5_{n,5} \setminus e_{1}) \vee \Sigma\ind(E^5_{n,5} \setminus \n[e_{1}]).
    \end{split}
    \end{equation*}
    Further, observe that the graph $E^5_{n,5} \setminus e_1$ is isomorphic to $\Gamma_{n,5}\setminus\{g_{22},g_{32}\}$ (see \Cref{fig:Gn}, and \Cref{subfig: E-n-5}), and the graph  $E_{n,5}^5 \setminus \n[e_1]$ is isomorphic to $A_{n-1,5}$ (see \Cref{fig:An}, and \Cref{subfig: E-n-5}). Therefore, $\ind(E^5_{n,5} \setminus e_1) \simeq \ind(\Gamma_{n,5}\setminus\{g_{22},g_{32}\})$ and $\ind(E_{n,5}^5 \setminus \n[e_1]) \simeq \ind(A_{n-1,5})$. Since $N(g_{11}) \subseteq N(g_{22})$ and $N(g_{41}) \subseteq N(g_{32})$ in $\Gamma_{n, 5}$ and $\Gamma_{n, 5} \setminus g_{22}$, respectively (see \Cref{fig:Gn}), we have that $\ind(\Gamma_{n, 5}) \simeq \ind(\Gamma_{n, 5} \setminus \{g_{22}, g_{33}\})$. Therefore,  
    
    \begin{equation}{\label{En5-link-del}}
        \ind(E^5_{n,5}) \simeq \ind(\Gamma_{n,5}) \vee \Sigma\ind(A_{n-1,5}).
    \end{equation}
    Combining  \eqref{En4-link-del} and \eqref{En5-link-del}, we get 
\begin{equation}{\label{En3-final}}
        \ind(E_{n,5}^3) \simeq \ind(\Gamma_{n,5}) \vee \Sigma\ind(A_{n-1,5}) \vee \Sigma^2 \ind(C_{n,5}^5).
    \end{equation} 
   
    Thus, combining the homotopy equivalences in \eqref{En first},\eqref{En2-g32} and \eqref{En3-final}, we get  
    \begin{equation}{\label{En pre}}
        \ind(E_{n,5}) \simeq ( \vee^2 \Sigma^2\ind(C_{n,5}^5)) \vee \Sigma\ind(A_{n-1,5}) \vee \ind(\Gamma_{n,5}).
    \end{equation}
    It follows from \Cref{homA_{n,5},,homG_{n,5},,homCn} that $\ind(A_{0,5})\simeq \mathbb{S}^1\vee\mathbb{S}^1\vee\mathbb{S}^1$, $\ind(\Gamma_{1,5})\simeq \mathbb{S}^2\vee\mathbb{S}^2\vee\mathbb{S}^2 $ and $\ind(C_{1,5}^5)\simeq \mathbb{S}^0\vee\mathbb{S}^0$ (because $C_{1,5}^5 \cong K_3$). Therefore, $\ind(E_{1,5})\simeq \bigvee^{10} \mathbb{S}^2.$ Furthermore, for $n\geq 2$, we have $\ind(C_{n,5}^5) \simeq  \Sigma(\ind(A_{n-2,5}) \vee  \ind(F_{n-2,5}))$ from \eqref{eq:Cn5-sim-vert}. Hence, from \Cref{En pre},
    %\begin{equation*}
        $\ind(E_{n,5}) \simeq \left({\bigvee}^2 \Sigma^3\ind(A_{n-2,5})\right) \vee  \left({\bigvee}^2\Sigma^3\ind(F_{n-2,5})\right) \vee \Sigma\ind(A_{n-1,5}) \vee \ind(\Gamma_{n,5}).$
    %\end{equation*}     
\end{proof}

\begin{claim}{\label{homFn}}
	\begin{equation*}
   \ind(F_{n,5}) \simeq  \begin{cases} 
   	\vee^{4}\Sp^1 \hspace{7.3 cm}  \text{if $n = 0$}, \\
   	\Sigma\ind(B_{n-1,5}) \vee \left({\bigvee}^2 
            \Sigma^2\ind(C_{n-1,5})\right) \vee 
            \Sigma^2\ind(D_{n-1,5})  \  \text{if $n \geq 1$}.
            \end{cases}
\end{equation*}

\end{claim}

\begin{figure}[h]
    \centering
    \begin{subfigure}[b]{0.3\textwidth}
        \begin{tikzpicture}[scale=0.3]
        \begin{scope}
\foreach \x in {1,2,...,5}
            {
            \node[vertex] (1\x) at (2*\x,3*2) {};
            }
\foreach \x in {1,2,...,5}
            {
            \node[vertex] (2\x) at (2*\x,2*2) {};
            }
\foreach \x in {2,...,5}
            {
            \node[vertex] (3\x) at (2*\x,1*2) {};
            }
\foreach \x in {2,...,5}
            {
            \node[vertex] (4\x) at (2*\x,0*2) {};
            }

            \foreach \y [evaluate={\y as \x using {\y+1}}] in {1,2,...,4} 
            {
            \draw[edge] (1\y) -- (1\x);
            }
            \foreach \y [evaluate={\y as \x using {\y+1}}] in {1,2,...,4} 
            {
            \draw[edge] (2\y) -- (2\x);
            }
            \foreach \y [evaluate={\y as \x using {\y+1}}] in {2,...,4} 
            {
            \draw[edge] (3\y) -- (3\x);
            }
            \foreach \y [evaluate={\y as \x using {\y+1}}] in {2,...,4} 
            {
            \draw[edge] (4\y) -- (4\x);
            }

\foreach \z in {1,2,3,4,5}
            {
            \draw[edge] (1\z) -- (2\z);
            }
            \foreach \z in {2,3,4,5}
            {
            \draw[edge] (2\z) -- (3\z);
            }
            \foreach \z in {2,3,4,5}
            {
            \draw[edge] (3\z) -- (4\z);
            }
            
\foreach \z [evaluate={\znext=int({1+\z})}] in {2,4}
            {
            \draw[edge] (1\z) -- (2\znext);
            }
            \foreach \z [evaluate={\znext=int({1+\z})}] in {2,4}
            {
            \draw[edge] (3\z) -- (4\znext);
            }
            \foreach \z [evaluate={\znext=int({\z-1})}] in {3,5}
            {
            \draw[edge] (1\z) -- (2\znext);
            }
            \foreach \z [evaluate={\znext=int({\z-1})}] in {3,5}
            {
            \draw[edge] (3\z) -- (4\znext);
            }

\foreach \z [evaluate={\znext=int({1+\z})}] in {1,3}
            {
            \draw[edge] (2\z) -- (3\znext);
            }
            \foreach \z [evaluate={\znext=int({\z-1})}] in {4}
            {
            \draw[edge] (2\z) -- (3\znext);
            }

\draw[dotted, thick,shorten >=4pt, shorten <=4pt] (10,5) -- (12,5);
            \draw[dotted, thick,shorten >=4pt, shorten <=4pt] (10,3) -- (12,3);
            \draw[dotted, thick,shorten >=4pt, shorten <=4pt] (10,1) -- (12,1);

\node[label={[label distance=-3pt]90:\scriptsize{$g_{11}$}}] at (11) {};
            \node[label={[label distance=-3pt]-90:\scriptsize{$g_{21}$}}] at (21) {};
            \foreach \x in {2,...,5}
            {
            \node[label={[label distance=-3pt]90:\scriptsize{$g_{1\x}$}}] at (1\x) {};
            \node[label={[label distance=-3pt]-90:\scriptsize{$g_{4\x}$}}] at (4\x) {};
            }
           \node[label={[label distance=-5pt]180:\scriptsize{$g_{32}$}}] at (32) {};
        \end{scope}

    \end{tikzpicture}
    \caption{$F_{n,5} \setminus \n[f_2]$}
    \label{subfig: Fn-l2}
    \end{subfigure}
    \begin{subfigure}[b]{0.3\textwidth}
        \begin{tikzpicture}[scale=0.3]
        \begin{scope}
\foreach \x in {1,2,...,5}
            {
            \node[vertex] (1\x) at (2*\x,3*2) {};
            }
\foreach \x in {3,...,5}
            {
            \node[vertex] (2\x) at (2*\x,2*2) {};
            }
\foreach \x in {3,...,5}
            {
            \node[vertex] (3\x) at (2*\x,1*2) {};
            }
\foreach \x in {2,...,5}
            {
            \node[vertex] (4\x) at (2*\x,0*2) {};
            }

            \foreach \y [evaluate={\y as \x using {\y+1}}] in {1,2,...,4} 
            {
            \draw[edge] (1\y) -- (1\x);
            }
            \foreach \y [evaluate={\y as \x using {\y+1}}] in {3,...,4} 
            {
            \draw[edge] (2\y) -- (2\x);
            }
            \foreach \y [evaluate={\y as \x using {\y+1}}] in {3,...,4} 
            {
            \draw[edge] (3\y) -- (3\x);
            }
            \foreach \y [evaluate={\y as \x using {\y+1}}] in {2,...,4} 
            {
            \draw[edge] (4\y) -- (4\x);
            }

\foreach \z in {3,4,5}
            {
            \draw[edge] (1\z) -- (2\z);
            }
            \foreach \z in {3,4,5}
            {
            \draw[edge] (2\z) -- (3\z);
            }
            \foreach \z in {3,4,5}
            {
            \draw[edge] (3\z) -- (4\z);
            }
            
\foreach \z [evaluate={\znext=int({1+\z})}] in {2,4}
            {
            \draw[edge] (1\z) -- (2\znext);
            }
            \foreach \z [evaluate={\znext=int({1+\z})}] in {4}
            {
            \draw[edge] (3\z) -- (4\znext);
            }
            \foreach \z [evaluate={\znext=int({\z-1})}] in {5}
            {
            \draw[edge] (1\z) -- (2\znext);
            }
            \foreach \z [evaluate={\znext=int({\z-1})}] in {3,5}
            {
            \draw[edge] (3\z) -- (4\znext);
            }

\foreach \z [evaluate={\znext=int({1+\z})}] in {3}
            {
            \draw[edge] (2\z) -- (3\znext);
            }
            \foreach \z [evaluate={\znext=int({\z-1})}] in {4}
            {
            \draw[edge] (2\z) -- (3\znext);
            }

\draw[dotted, thick,shorten >=4pt, shorten <=4pt] (10,5) -- (12,5);
            \draw[dotted, thick,shorten >=4pt, shorten <=4pt] (10,3) -- (12,3);
            \draw[dotted, thick,shorten >=4pt, shorten <=4pt] (10,1) -- (12,1);

\node[label={[label distance=-3pt]90:\scriptsize{$g_{11}$}}] at (11) {};
            \foreach \x in {2,...,5}
            {
            \node[label={[label distance=-3pt]90:\scriptsize{$g_{1\x}$}}] at (1\x) {};
            }
            \foreach \x in {2,...,5}
            {
            \node[label={[label distance=-3pt]-90:\scriptsize{$g_{4\x}$}}] at (4\x) {};
            }

        \end{scope}

\begin{scope}[shift={(2,0)}]
            \node[vertex] (f2) at (-2,4) {} ;
            \draw[edge] (f2) -- (11) ;

            \node[label={[label distance=-3pt]180:\scriptsize{$f_{1}$}}] at (f2) {};
           
            \node[label={[label distance=-5pt]180:\scriptsize{$g_{23}$}}] at (23) {};
            \node[label={[label distance=-5pt]180:\scriptsize{$g_{33}$}}] at (33) {};
        \end{scope}

    \end{tikzpicture}
    \caption{$F_{n,5} \setminus \n[g_{31}]$}
    \label{subfig: Fn-g31}
    \end{subfigure}
    \begin{subfigure}[b]{0.3\textwidth}
        \begin{tikzpicture}[scale=0.3]
        \begin{scope}
\foreach \x in {1,2,...,5}
            {
            \node[vertex] (1\x) at (2*\x,3*2) {};
            }
\foreach \x in {1,2,...,5}
            {
            \node[vertex] (2\x) at (2*\x,2*2) {};
            }
\foreach \x in {2,...,5}
            {
            \node[vertex] (3\x) at (2*\x,1*2) {};
            }
\foreach \x in {3,...,5}
            {
            \node[vertex] (4\x) at (2*\x,0*2) {};
            }

            \foreach \y [evaluate={\y as \x using {\y+1}}] in {1,2,...,4} 
            {
            \draw[edge] (1\y) -- (1\x);
            }
            \foreach \y [evaluate={\y as \x using {\y+1}}] in {1,2,...,4} 
            {
            \draw[edge] (2\y) -- (2\x);
            }
            \foreach \y [evaluate={\y as \x using {\y+1}}] in {2,...,4} 
            {
            \draw[edge] (3\y) -- (3\x);
            }
            \foreach \y [evaluate={\y as \x using {\y+1}}] in {3,...,4} 
            {
            \draw[edge] (4\y) -- (4\x);
            }

\foreach \z in {1,2,3,4,5}
            {
            \draw[edge] (1\z) -- (2\z);
            }
            \foreach \z in {2,3,4,5}
            {
            \draw[edge] (2\z) -- (3\z);
            }
            \foreach \z in {3,4,5}
            {
            \draw[edge] (3\z) -- (4\z);
            }
            
\foreach \z [evaluate={\znext=int({1+\z})}] in {2,4}
            {
            \draw[edge] (1\z) -- (2\znext);
            }
            \foreach \z [evaluate={\znext=int({1+\z})}] in {2,4}
            {
            \draw[edge] (3\z) -- (4\znext);
            }
            \foreach \z [evaluate={\znext=int({\z-1})}] in {3,5}
            {
            \draw[edge] (1\z) -- (2\znext);
            }
            \foreach \z [evaluate={\znext=int({\z-1})}] in {5}
            {
            \draw[edge] (3\z) -- (4\znext);
            }

\foreach \z [evaluate={\znext=int({1+\z})}] in {1,3}
            {
            \draw[edge] (2\z) -- (3\znext);
            }
            \foreach \z [evaluate={\znext=int({\z-1})}] in {4}
            {
            \draw[edge] (2\z) -- (3\znext);
            }

\draw[dotted, thick,shorten >=4pt, shorten <=4pt] (10,5) -- (12,5);
            \draw[dotted, thick,shorten >=4pt, shorten <=4pt] (10,3) -- (12,3);
            \draw[dotted, thick,shorten >=4pt, shorten <=4pt] (10,1) -- (12,1);

\node[label={[label distance=-3pt]90:\scriptsize{$g_{11}$}}] at (11) {};
            \node[label={[label distance=-3pt]-90:\scriptsize{$g_{21}$}}] at (21) {};
            \foreach \x in {2,...,5}
            {
            \node[label={[label distance=-3pt]90:\scriptsize{$g_{1\x}$}}] at (1\x) {};
            }
            \foreach \x in {3,...,5}
            {
            \node[label={[label distance=-3pt]-90:\scriptsize{$g_{4\x}$}}] at (4\x) {};
            }
            \node[label={[label distance=-5pt]180:\scriptsize{$g_{32}$}}] at (32) {};

        \end{scope}

\begin{scope}[shift={(2,0)}]
            \node[vertex] (f2) at (-2,4) {} ;
  
            \draw[edge] (f2) -- (21) ;
            \draw[edge] (f2) -- (11) ;

            \node[label={[label distance=-3pt]180:\scriptsize{$f_{1}$}}] at (f2) {};

        \end{scope}

    \end{tikzpicture}
    \caption{$F_{n,5} \setminus \n[g_{41}]$}
    \label{subfig: Fn-g41}
    \end{subfigure}
    
    \caption{}	
\end{figure}
 \begin{proof}

 	First, let $n = 0$.  Since   $g_{11}$ is a simplicial vertex in $F_{0,5}$  with $N(g_{11})=\{g_{21},f_{1}\}$, (see \Cref{fig:Fn}) using \Cref{Simplicial Vertex Lemma} we get 
 	$\ind (F_{0,5})\simeq\Sigma\ind(F_{0,5} \setminus \n[g_{21}]) \vee \Sigma\ind(F_{0,5} \setminus \n[f_1]).$
 	Since both the graphs $F_{0,5} \setminus \n[g_{21}]$ and $F_{0,5} \setminus \n[f_1]$ are isomorphic to the complete graph $K_3$, the result follows in this case.

  Now, assume that $n \geq 1$.   Since $f_3$ is a simplicial vertex and $N(f_3)=\{f_2,g_{31},g_{41}\}$ in the graph $F_{n,5}$ (see \Cref{fig:Fn}), from \Cref{Simplicial Vertex Lemma}, we get
    \begin{equation}{\label{Ln-sim-l3}}
    \begin{split}
        \ind(F_{n,5}) & \simeq \Sigma\ind(F_{n,5} \setminus \n[f_2]) \vee \Sigma\ind(F_{n,5} \setminus \n[g_{31}]) \vee \Sigma\ind(F_{n,5} \setminus \n[g_{41}]).
    \end{split}
    \end{equation}
    Observe that the graph $F_{n,5}\setminus \n[f_2]$ is isomorphic to $ B_{n-1,5}$ (see \Cref{fig:Bn} and \Cref{subfig: Fn-l2}). Hence, $\ind(F_{n,5}\setminus \n[f_2])\simeq \ind(B_{n-1,5})$. 
    
   Since  $N(f_1) = \{g_{11}\}$  in $F_{n,5} \setminus \n[g_{31}]$ (see \Cref{subfig: Fn-g31}),   using \Cref{Simplicial Vertex Lemma}, we have $\ind(F_{n,5} \setminus \n[g_{31}]) \simeq \Sigma \ind(F_{n,5} \setminus \n[g_{31},g_{11}])$. However, the graph $F_{n,5} \setminus \n[g_{31},g_{11}]$ is isomorphic to $ C_{n-1,5}$ and  therefore, $\ind(F_{n,5} \setminus \n[g_{31}]) \simeq \Sigma \ind(C_{n-1,5})$. 
    
    Finally, consider the graph $F_{n,5} \setminus \n[g_{41}]$. Since $f_1$ is a simplicial vertex in $F_{n,5} \setminus \n[g_{41}]$, with $N(f_1)=\{g_{11},g_{21} \}$, from \Cref{Simplicial Vertex Lemma}, we have 
    \begin{equation}{\label{Ln-sim-l1}}
    \begin{split}
        \ind(F_{n,5} \setminus \n[g_{41}]) \simeq \Sigma\ind(F_{n,5} \setminus \n[g_{41},g_{11}]) \vee \Sigma \ind(F_{n,5} \setminus \n[g_{41},g_{21}]).
    \end{split}
    \end{equation}

    Note that, the graphs $F_{n,5} \setminus \n[g_{41},g_{11}] $ and $F_{n,5} \setminus \n[g_{41},g_{21}] $ are isomorphic to  $D_{n-1,5}$ and $ C_{n-1,5}$, respectively (refer to \Cref{fig:Cn,,fig:Dn}). Therefore, $\ind(F_{n,5} \setminus \n[g_{41}])\simeq \Sigma \ind(D_{n-1,5})\vee\Sigma \ind(C_{n-1,5})$. Now, the result follows from \eqref{Ln-sim-l3}.
\end{proof}

\begin{theorem} \label{Theorem:PnP5firstcase}
 For any graph $G_{n,5}\in \set{\Gamma_{n,5},\Lambda_{n,5}, A_{n,5}, B_{n,5}, C_{n,5}, D_{n,5}, E_{n,5}, F_{n,5}:  n\geq 0 }$, $\ind(G_{n,5})$ is homotopy equivalent to a wedge of spheres.
\end{theorem}
\begin{proof}
 For any graph $G_{0, 5} \in \set{\Gamma_{0,5},\Lambda_{0,5}, A_{0,5}, B_{0,5}, C_{0,5}, D_{0,5}, E_{0,5}, F_{0,5} }$, $\ind(G_{0, 5})$ is homotopy equivalent to a wedge of spheres from \Crefrange{homG_{n,5}}{homFn}.   Now using induction on $n$, and recursive relations established in \Crefrange{homG_{n,5}}{homFn}, we conclude that, for any graph $G_{n, 5} \in \set{ \Gamma_{n,5},\Lambda_{n,5}, A_{n,5}, B_{n,5}, C_{n,5}, D_{n,5}, E_{n,5}, F_{n,5} : n\geq 0}$, $\ind(G_{n, 5})$ is homotopy equivalent to a wedge of spheres. 
\end{proof}
\subsection{Independence complex of \texorpdfstring{$\widetilde{\Gamma}_{n,5}$}{Ind(Lambda n, for n >=0}} \label{subsection:GammaTilde}

Recall that for $n\geq 0$, the line graph of $P_{2n+3}\times P_5$ comprises of two connected components $\Lambda_{n,5}$ and $\widetilde{\Gamma}_{n,5}$ (see \Cref{fig:Lambda-n,,fig:G'n}). The computation of $\ind(\widetilde{\Gamma}_{n,5})$ follows a similar approach as that of $\ind(\Gamma_{n,5})$. The intermediary graphs occurring during the computation of $\ind(\widetilde{\Gamma}_{n,5})$ resemble those defined in \Cref{subsection:PnP5Basecase}, with an inclusion of certain additional vertices and edges.
For each $H_{n,5} \in \{\Gamma_{n,5},\Lambda_{n,5},A_{n,5},B_{n,5},C_{n,5},D_{n,5},E_{n,5},F_{n,5}\}$, where $n\in\mathbb{N}\cup\{0\}$, we define
\begin{equation*}
    \begin{split}
    V(\widetilde{H}_{n,5}) =  & V(H_{n,5})\cup \{g_{1\,2n+2},g_{2\,2n+2},g_{3\,2n+2},g_{4\,2n+2}\}, \\
    E(\widetilde{H}_{n,5}) = & E(H_{n,5}) \cup \set{g_{i\,2n+1}g_{i\,2n+2} : i=1,2,3,4} 
    \cup \set{g_{i\,2n+2}g_{i+1\,2n+2} : i = 1,2,3} \\
    & \cup \set{g_{2\,2n+1}g_{3\,2n+2}, g_{3\,2n+1}g_{2\,2n+2}}.
    \end{split}
\end{equation*}

\noindent In this section, we compute the independence complexes of graphs $\widetilde{\Gamma}_{n,5},\widetilde{\Lambda}_{n,5}, \widetilde{A}_{n,5}, \widetilde{B}_{n,5},$ $\widetilde{C}_{n,5}$, $\widetilde{D}_{n,5}, \widetilde{E}_{n,5}$, and $\widetilde{F}_{n,5}$. In computing the independence complexes of these graphs, as discussed in \Cref{subsection:GeneralCasePnP5}, our approach focuses on performing computations primarily on one end (the leftmost part) of the graph while keeping the other end (the rightmost part) fixed. Notably, for $n\geq 1$, observe that the left ends of the graphs $H_{n,5}$ and $\widetilde{H}_{n,5}$ coincide, where $H_{n,5}\in \{\Gamma_{n,5},\Lambda_{n,5}, A_{n, 5}, B_{n,5}, C_{n, 5}, D_{n, 5}, E_{n, 5}, F_{n, 5}\}$. This observation suggests a method for computing the independence complexes of graphs $\widetilde{\Gamma}_{n,5},\widetilde{\Lambda}_{n,5}, \widetilde{A}_{n,5}, \widetilde{B}_{n,5}, \widetilde{C}_{n,5}, \widetilde{D}_{n,5}, \widetilde{E}_{n,5}$, and $\widetilde{F}_{n,5}$ using a similar strategy to that outlined in \Cref{subsection:GeneralCasePnP5}. Furthermore, for $n\geq 1$, these intermediary graphs adhere to a similar recurrence relation as those described in the preceding section.  Since the base graphs of these intermediary graphs are completely different from the graphs in \Cref{subsection:PnP5Basecase}, we provide the complete computations for base cases. Also, we give the explicit recurrence relation for each graph and higher values of $n$.

\begin{figure}
    \centering
    \begin{subfigure}[b]{0.23\textwidth}
        \begin{tikzpicture}[scale=0.28]
    \begin{scope}
\foreach \x in {1,2,...,4}
        {
        \node[vertex] (1\x) at (2*\x,3*2) {};
        }
\foreach \x in {1,2,...,4}
        {
        \node[vertex] (2\x) at (2*\x,2*2) {};
        }
\foreach \x in {1,2,...,4}
        {
        \node[vertex] (3\x) at (2*\x,1*2) {};
        }
\foreach \x in {1,2,...,4}
        {
        \node[vertex] (4\x) at (2*\x,0*2) {};
        }
        
\foreach \y [evaluate={\y as \x using int({\y+1})}] in {1,2,...,3} 
        {
        \draw[edge] (1\y) -- (1\x);
        \draw[edge] (2\y) -- (2\x);
        \draw[edge] (3\y) -- (3\x);
        \draw[edge] (4\y) -- (4\x);
        }
        
\foreach \z in {1,2,3,4}
        {\draw[edge] (1\z) -- (2\z);
        \draw[edge] (2\z) -- (3\z);
        \draw[edge] (3\z) -- (4\z);
        }
        
\foreach \z [evaluate={\znext=int({1+\z})}] in {2}
        {
        \draw[edge] (1\z) -- (2\znext);
        \draw[edge] (3\z) -- (4\znext);
        }
        \foreach \z [evaluate={\znext=int({\z-1})}] in {3}
        {
        \draw[edge] (1\z) -- (2\znext);
        \draw[edge] (3\z) -- (4\znext);
        }

\foreach \z [evaluate={\znext=int({1+\z})}] in {1,3}
        {
        \draw[edge] (2\z) -- (3\znext);
        }
        \foreach \z [evaluate={\znext=int({\z-1})}] in {2,4}
        {
        \draw[edge] (2\z) -- (3\znext);
        }
        \coordinate (c15) at (15);
        \coordinate (c16) at (16);

\node[label={[label distance=-3pt]90:\scriptsize{$g_{11}$}}] at (11) {};
        \foreach \x in {2,...,4}
        {
        \node[label={[label distance=-3pt]90:\scriptsize{$g_{1\x}$}}] at (1\x) {};
        }

        \node[label={[label distance=-3pt]180:\scriptsize{$g_{21}$}}] at (21) {};
        \node[label={[label distance=-3pt]180:\scriptsize{$g_{31}$}}] at (31) {};
        \node[label={[label distance=-3pt]180:\scriptsize{$g_{41}$}}] at (41) {};

        \node[label={[label distance=-3pt]0:\scriptsize{$g_{24}$}}] at (24) {};
        \node[label={[label distance=-3pt]0:\scriptsize{$g_{34}$}}] at (34) {};
        \node[label={[label distance=-3pt]0:\scriptsize{$g_{44}$}}] at (44) {};
    \end{scope}
\end{tikzpicture}
\caption{$\widetilde{\Gamma}_{1,5}$} 
     \end{subfigure}
    \begin{subfigure}[b]{0.23\textwidth}
        \begin{tikzpicture}[scale=0.28]
    \begin{scope}
\foreach \x in {1,2,...,4}
        {
        \node[vertex] (1\x) at (2*\x,3*2) {};
        }
\foreach \x in {1,2,...,4}
        {
        \node[vertex] (2\x) at (2*\x,2*2) {};
        }
\foreach \x in {1,2,...,4}
        {
        \node[vertex] (3\x) at (2*\x,1*2) {};
        }
\foreach \x in {1,2,...,4}
        {
        \node[vertex] (4\x) at (2*\x,0*2) {};
        }
        
\foreach \y [evaluate={\y as \x using int({\y+1})}] in {1,2,...,3} 
        {
        \draw[edge] (1\y) -- (1\x);
        \draw[edge] (2\y) -- (2\x);
        \draw[edge] (3\y) -- (3\x);
        \draw[edge] (4\y) -- (4\x);
        }
        
\foreach \z in {1,2,3,4}
        {\draw[edge] (1\z) -- (2\z);
        \draw[edge] (2\z) -- (3\z);
        \draw[edge] (3\z) -- (4\z);
        }
        
\foreach \z [evaluate={\znext=int({1+\z})}] in {2}
        {
        \draw[edge] (1\z) -- (2\znext);
        \draw[edge] (3\z) -- (4\znext);
        }
        \foreach \z [evaluate={\znext=int({\z-1})}] in {3}
        {
        \draw[edge] (1\z) -- (2\znext);
        \draw[edge] (3\z) -- (4\znext);
        }

\foreach \z [evaluate={\znext=int({1+\z})}] in {1,3}
        {
        \draw[edge] (2\z) -- (3\znext);
        }
        \foreach \z [evaluate={\znext=int({\z-1})}] in {2,4}
        {
        \draw[edge] (2\z) -- (3\znext);
        }
        \coordinate (c15) at (15);
        \coordinate (c16) at (16);

\node[label={[label distance=-3pt]90:\scriptsize{$g_{11}$}}] at (11) {};
        \foreach \x in {2,...,4}
        {
        \node[label={[label distance=-3pt]90:\scriptsize{$g_{1\x}$}}] at (1\x) {};
        }

        \node[label={[label distance=-3pt]0:\scriptsize{$g_{24}$}}] at (24) {};
        \node[label={[label distance=-3pt]0:\scriptsize{$g_{34}$}}] at (34) {};
        \node[label={[label distance=-3pt]0:\scriptsize{$g_{44}$}}] at (44) {};
    \end{scope}

     \begin{scope}[shift={(2,0)}]
            \node[vertex] (l1) at (-2,6) {} ;
            \node[vertex] (l2) at (-2,4) {} ;
            \node[vertex] (l3) at (-2,2) {} ;
            \node[vertex] (l4) at (-2,0) {} ;

            \draw[edge] (l1) -- (l2) ;
            \draw[edge] (l2) -- (l3) ;
            \draw[edge] (l3) -- (l4) ;
            \draw[edge] (l1) -- (11) ;
            \draw[edge] (l2) -- (21) ;
            \draw[edge] (l3) -- (31) ;
            \draw[edge] (l4) -- (41) ;
            \draw[edge] (l1) -- (21) ;
            \draw[edge] (l2) -- (11) ;
            \draw[edge] (l4) -- (31) ;
            \draw[edge] (l3) -- (41) ;

            \node[label={[label distance=-3pt]180:\scriptsize{$l_{1}$}}] at (l1) {};
            \node[label={[label distance=-3pt]180:\scriptsize{$l_{2}$}}] at (l2) {};
            \node[label={[label distance=-3pt]180:\scriptsize{$l_{3}$}}] at (l3) {};
            \node[label={[label distance=-3pt]180:\scriptsize{$l_{4}$}}] at (l4) {};

        \end{scope}

\end{tikzpicture}
\caption{$\widetilde{\Lambda}_{1,5}$} 
     \end{subfigure}
    \begin{subfigure}[b]{0.23\textwidth}
        \begin{tikzpicture}[scale=0.28]
    \begin{scope}
\foreach \x in {1,2,...,4}
        {
        \node[vertex] (1\x) at (2*\x,3*2) {};
        }
\foreach \x in {1,2,...,4}
        {
        \node[vertex] (2\x) at (2*\x,2*2) {};
        }
\foreach \x in {1,2,...,4}
        {
        \node[vertex] (3\x) at (2*\x,1*2) {};
        }
\foreach \x in {1,2,...,4}
        {
        \node[vertex] (4\x) at (2*\x,0*2) {};
        }
        
\foreach \y [evaluate={\y as \x using int({\y+1})}] in {1,2,...,3} 
        {
        \draw[edge] (1\y) -- (1\x);
        \draw[edge] (2\y) -- (2\x);
        \draw[edge] (3\y) -- (3\x);
        \draw[edge] (4\y) -- (4\x);
        }
        
\foreach \z in {1,2,3,4}
        {\draw[edge] (1\z) -- (2\z);
        \draw[edge] (2\z) -- (3\z);
        \draw[edge] (3\z) -- (4\z);
        }
        
\foreach \z [evaluate={\znext=int({1+\z})}] in {2}
        {
        \draw[edge] (1\z) -- (2\znext);
        \draw[edge] (3\z) -- (4\znext);
        }
        \foreach \z [evaluate={\znext=int({\z-1})}] in {3}
        {
        \draw[edge] (1\z) -- (2\znext);
        \draw[edge] (3\z) -- (4\znext);
        }

\foreach \z [evaluate={\znext=int({1+\z})}] in {1,3}
        {
        \draw[edge] (2\z) -- (3\znext);
        }
        \foreach \z [evaluate={\znext=int({\z-1})}] in {2,4}
        {
        \draw[edge] (2\z) -- (3\znext);
        }
        \coordinate (c15) at (15);
        \coordinate (c16) at (16);

\node[label={[label distance=-3pt]90:\scriptsize{$g_{11}$}}] at (11) {};
        \foreach \x in {2,...,4}
        {
        \node[label={[label distance=-3pt]90:\scriptsize{$g_{1\x}$}}] at (1\x) {};
        }

        \node[label={[label distance=-3pt]0:\scriptsize{$g_{24}$}}] at (24) {};
        \node[label={[label distance=-3pt]0:\scriptsize{$g_{34}$}}] at (34) {};
        \node[label={[label distance=-3pt]0:\scriptsize{$g_{44}$}}] at (44) {};
    \end{scope}

\begin{scope}[shift={(2,0)}]
            \node[vertex] (a1) at (-2,6) {} ;
            \node[vertex] (a2) at (-2,0) {} ;

            \draw[edge] (a1) -- (11) ;
            \draw[edge] (a1) -- (21) ;
            \draw[edge] (a2) -- (31) ;
            \draw[edge] (a2) -- (41) ;

            \node[label={[label distance=-3pt]180:\scriptsize{$a_{1}$}}] at (a1) {};   
            \node[label={[label distance=-3pt]180:\scriptsize{$a_{2}$}}] at (a2) {};
        \end{scope}

\end{tikzpicture}
\caption{$\widetilde{A}_{1,5}$} 
     \end{subfigure}
    \begin{subfigure}[b]{0.23\textwidth}
        \begin{tikzpicture}[scale=0.28]
    \begin{scope}
\foreach \x in {1,2,...,4}
        {
        \node[vertex] (1\x) at (2*\x,3*2) {};
        }
\foreach \x in {1,2,...,4}
        {
        \node[vertex] (2\x) at (2*\x,2*2) {};
        }
\foreach \x in {1,2,...,4}
        {
        \node[vertex] (3\x) at (2*\x,1*2) {};
        }
\foreach \x in {1,2,...,4}
        {
        \node[vertex] (4\x) at (2*\x,0*2) {};
        }
        
\foreach \y [evaluate={\y as \x using int({\y+1})}] in {1,2,...,3} 
        {
        \draw[edge] (1\y) -- (1\x);
        \draw[edge] (2\y) -- (2\x);
        \draw[edge] (3\y) -- (3\x);
        \draw[edge] (4\y) -- (4\x);
        }
        
\foreach \z in {1,2,3,4}
        {\draw[edge] (1\z) -- (2\z);
        \draw[edge] (2\z) -- (3\z);
        \draw[edge] (3\z) -- (4\z);
        }
        
\foreach \z [evaluate={\znext=int({1+\z})}] in {2}
        {
        \draw[edge] (1\z) -- (2\znext);
        \draw[edge] (3\z) -- (4\znext);
        }
        \foreach \z [evaluate={\znext=int({\z-1})}] in {3}
        {
        \draw[edge] (1\z) -- (2\znext);
        \draw[edge] (3\z) -- (4\znext);
        }

\foreach \z [evaluate={\znext=int({1+\z})}] in {1,3}
        {
        \draw[edge] (2\z) -- (3\znext);
        }
        \foreach \z [evaluate={\znext=int({\z-1})}] in {2,4}
        {
        \draw[edge] (2\z) -- (3\znext);
        }
        \coordinate (c15) at (15);
        \coordinate (c16) at (16);

\node[label={[label distance=-3pt]90:\scriptsize{$g_{11}$}}] at (11) {};
        \foreach \x in {2,...,4}
        {
        \node[label={[label distance=-3pt]90:\scriptsize{$g_{1\x}$}}] at (1\x) {};
        }

        \node[label={[label distance=-3pt]0:\scriptsize{$g_{24}$}}] at (24) {};
        \node[label={[label distance=-3pt]0:\scriptsize{$g_{34}$}}] at (34) {};
        \node[label={[label distance=-3pt]0:\scriptsize{$g_{44}$}}] at (44) {};
    \end{scope}

\begin{scope}[shift={(2,0)}]
            
            \node[vertex] (b1) at (-2,6) {} ;
            \node[vertex] (b2) at (-2,4) {} ;
            \node[vertex] (b3) at (-2,2) {} ;
            \node[vertex] (b4) at (-2,0) {} ;
            \node[vertex] (b31) at (-4,2) {} ;
            \node[vertex] (b41) at (-4,0) {} ;

            \draw[edge] (b1) -- (11) ;
            \draw[edge] (b1) -- (21) ;
            \draw[edge] (b1) -- (b2) ;
            \draw[edge] (b2) -- (b3) ;
            \draw[edge] (b2) -- (21) ;
            \draw[edge] (b2) -- (11) ;
            \draw[edge] (b3) -- (31) ;
            \draw[edge] (b3) -- (41) ;
            \draw[edge] (b3) -- (b4) ;
            \draw[edge] (b4) -- (41) ;
            \draw[edge] (b4) -- (31) ;

            \draw[edge] (b31) -- (b3) ;
            \draw[edge] (b31) -- (b2) ;
            \draw[edge] (b31) -- (b41) ;
            \draw[edge] (b41) -- (b4) ;

            \node[label={[label distance=-3pt]180:\scriptsize{$b_{3}$}}] at (b1) {};
            \node[label={[label distance=-3pt]180:\scriptsize{$b_{4}$}}] at (b2) {};
            \node[label={[label distance=-9pt]45:\scriptsize{$b_{5}$}}] at (b3) {};
            \node[label={[label distance=-9pt]135:\scriptsize{$b_{6}$}}] at (b4) {};
            \node[label={[label distance=-3pt]180:\scriptsize{$b_{1}$}}] at (b31) {};
            \node[label={[label distance=-3pt]180:\scriptsize{$b_{2}$}}] at (b41) {};
        \end{scope}

\end{tikzpicture}
\caption{$\widetilde{B}_{1,5}$} 
     \end{subfigure}
     
    \begin{subfigure}[b]{0.23\textwidth}
        \begin{tikzpicture}[scale=0.28]
    \begin{scope}
\foreach \x in {1,2,...,4}
        {
        \node[vertex] (1\x) at (2*\x,3*2) {};
        }
\foreach \x in {1,2,...,4}
        {
        \node[vertex] (2\x) at (2*\x,2*2) {};
        }
\foreach \x in {1,2,...,4}
        {
        \node[vertex] (3\x) at (2*\x,1*2) {};
        }
\foreach \x in {1,2,...,4}
        {
        \node[vertex] (4\x) at (2*\x,0*2) {};
        }
        
\foreach \y [evaluate={\y as \x using int({\y+1})}] in {1,2,...,3} 
        {
        \draw[edge] (1\y) -- (1\x);
        \draw[edge] (2\y) -- (2\x);
        \draw[edge] (3\y) -- (3\x);
        \draw[edge] (4\y) -- (4\x);
        }
        
\foreach \z in {1,2,3,4}
        {\draw[edge] (1\z) -- (2\z);
        \draw[edge] (2\z) -- (3\z);
        \draw[edge] (3\z) -- (4\z);
        }
        
\foreach \z [evaluate={\znext=int({1+\z})}] in {2}
        {
        \draw[edge] (1\z) -- (2\znext);
        \draw[edge] (3\z) -- (4\znext);
        }
        \foreach \z [evaluate={\znext=int({\z-1})}] in {3}
        {
        \draw[edge] (1\z) -- (2\znext);
        \draw[edge] (3\z) -- (4\znext);
        }

\foreach \z [evaluate={\znext=int({1+\z})}] in {1,3}
        {
        \draw[edge] (2\z) -- (3\znext);
        }
        \foreach \z [evaluate={\znext=int({\z-1})}] in {2,4}
        {
        \draw[edge] (2\z) -- (3\znext);
        }
        \coordinate (c15) at (15);
        \coordinate (c16) at (16);

\node[label={[label distance=-3pt]90:\scriptsize{$g_{11}$}}] at (11) {};
        \foreach \x in {2,...,4}
        {
        \node[label={[label distance=-3pt]90:\scriptsize{$g_{1\x}$}}] at (1\x) {};
        }

        \node[label={[label distance=-3pt]0:\scriptsize{$g_{24}$}}] at (24) {};
        \node[label={[label distance=-3pt]0:\scriptsize{$g_{34}$}}] at (34) {};
        \node[label={[label distance=-3pt]0:\scriptsize{$g_{44}$}}] at (44) {};
    \end{scope}

    \begin{scope}[shift={(2,0)}]
            \node[vertex] (c2) at (-2,0) {} ;

            \draw[edge] (c2) -- (31) ;
            \draw[edge] (c2) -- (41) ;

            \node[label={[label distance=-3pt]180:\scriptsize{$c_{1}$}}] at (c2) {};
    \end{scope}

\end{tikzpicture}
\caption{$\widetilde{C}_{1,5}$} 
\label{tildeC15}     \end{subfigure}
    \begin{subfigure}[b]{0.23\textwidth}
        \begin{tikzpicture}[scale=0.28]
    \begin{scope}
\foreach \x in {1,2,...,4}
        {
        \node[vertex] (1\x) at (2*\x,3*2) {};
        }
\foreach \x in {1,2,...,4}
        {
        \node[vertex] (2\x) at (2*\x,2*2) {};
        }
\foreach \x in {1,2,...,4}
        {
        \node[vertex] (3\x) at (2*\x,1*2) {};
        }
\foreach \x in {1,2,...,4}
        {
        \node[vertex] (4\x) at (2*\x,0*2) {};
        }
        
\foreach \y [evaluate={\y as \x using int({\y+1})}] in {1,2,...,3} 
        {
        \draw[edge] (1\y) -- (1\x);
        \draw[edge] (2\y) -- (2\x);
        \draw[edge] (3\y) -- (3\x);
        \draw[edge] (4\y) -- (4\x);
        }
        
\foreach \z in {1,2,3,4}
        {\draw[edge] (1\z) -- (2\z);
        \draw[edge] (2\z) -- (3\z);
        \draw[edge] (3\z) -- (4\z);
        }
        
\foreach \z [evaluate={\znext=int({1+\z})}] in {2}
        {
        \draw[edge] (1\z) -- (2\znext);
        \draw[edge] (3\z) -- (4\znext);
        }
        \foreach \z [evaluate={\znext=int({\z-1})}] in {3}
        {
        \draw[edge] (1\z) -- (2\znext);
        \draw[edge] (3\z) -- (4\znext);
        }

\foreach \z [evaluate={\znext=int({1+\z})}] in {1,3}
        {
        \draw[edge] (2\z) -- (3\znext);
        }
        \foreach \z [evaluate={\znext=int({\z-1})}] in {2,4}
        {
        \draw[edge] (2\z) -- (3\znext);
        }
        \coordinate (c15) at (15);
        \coordinate (c16) at (16);

\node[label={[label distance=-3pt]90:\scriptsize{$g_{11}$}}] at (11) {};
        \foreach \x in {2,...,4}
        {
        \node[label={[label distance=-3pt]90:\scriptsize{$g_{1\x}$}}] at (1\x) {};
        }

        \node[label={[label distance=-3pt]0:\scriptsize{$g_{24}$}}] at (24) {};
        \node[label={[label distance=-3pt]0:\scriptsize{$g_{34}$}}] at (34) {};
        \node[label={[label distance=-3pt]0:\scriptsize{$g_{44}$}}] at (44) {};
    \end{scope}

    \begin{scope}[shift={(2,0)}]
            \node[vertex] (d1) at (-2,4) {} ;
            \node[vertex] (d2) at (-2,2) {} ;

            \draw[edge] (d1) -- (11) ;
            \draw[edge] (d1) -- (21) ;
            \draw[edge] (d1) -- (d2) ;
            \draw[edge] (d2) -- (31) ;
            \draw[edge] (d2) -- (41) ;

            \node[label={[label distance=-3pt]180:\scriptsize{$d_{1}$}}] at (d1) {};
            \node[label={[label distance=-3pt]180:\scriptsize{$d_{2}$}}] at (d2) {}; 

        \end{scope}

\end{tikzpicture}
\caption{$\widetilde{D}_{1,5}$} 
     \end{subfigure}
    \begin{subfigure}[b]{0.23\textwidth}
        \begin{tikzpicture}[scale=0.28]
    \begin{scope}
\foreach \x in {1,2,...,4}
        {
        \node[vertex] (1\x) at (2*\x,3*2) {};
        }
\foreach \x in {1,2,...,4}
        {
        \node[vertex] (2\x) at (2*\x,2*2) {};
        }
\foreach \x in {1,2,...,4}
        {
        \node[vertex] (3\x) at (2*\x,1*2) {};
        }
\foreach \x in {1,2,...,4}
        {
        \node[vertex] (4\x) at (2*\x,0*2) {};
        }
        
\foreach \y [evaluate={\y as \x using int({\y+1})}] in {1,2,...,3} 
        {
        \draw[edge] (1\y) -- (1\x);
        \draw[edge] (2\y) -- (2\x);
        \draw[edge] (3\y) -- (3\x);
        \draw[edge] (4\y) -- (4\x);
        }
        
\foreach \z in {1,2,3,4}
        {\draw[edge] (1\z) -- (2\z);
        \draw[edge] (2\z) -- (3\z);
        \draw[edge] (3\z) -- (4\z);
        }
        
\foreach \z [evaluate={\znext=int({1+\z})}] in {2}
        {
        \draw[edge] (1\z) -- (2\znext);
        \draw[edge] (3\z) -- (4\znext);
        }
        \foreach \z [evaluate={\znext=int({\z-1})}] in {3}
        {
        \draw[edge] (1\z) -- (2\znext);
        \draw[edge] (3\z) -- (4\znext);
        }

\foreach \z [evaluate={\znext=int({1+\z})}] in {1,3}
        {
        \draw[edge] (2\z) -- (3\znext);
        }
        \foreach \z [evaluate={\znext=int({\z-1})}] in {2,4}
        {
        \draw[edge] (2\z) -- (3\znext);
        }
        \coordinate (c15) at (15);
        \coordinate (c16) at (16);

\node[label={[label distance=-3pt]90:\scriptsize{$g_{11}$}}] at (11) {};
        \foreach \x in {2,...,4}
        {
        \node[label={[label distance=-3pt]90:\scriptsize{$g_{1\x}$}}] at (1\x) {};
        }

        \node[label={[label distance=-3pt]0:\scriptsize{$g_{24}$}}] at (24) {};
        \node[label={[label distance=-3pt]0:\scriptsize{$g_{34}$}}] at (34) {};
        \node[label={[label distance=-3pt]0:\scriptsize{$g_{44}$}}] at (44) {};
    \end{scope}

\begin{scope}[shift={(2,0)}]
            \node[vertex] (e1) at (-2,3) {} ;
            
            \draw[edge] (e1) -- (11) ;
            \draw[edge] (e1) -- (21) ;
            \draw[edge] (e1) -- (31) ;
            \draw[edge] (e1) -- (41) ;

            \node[label={[label distance=-3pt]180:\scriptsize{$e_{1}$}}] at (e1) {};
        \end{scope}

\end{tikzpicture}
\caption{$\widetilde{E}_{1,5}$} 
     \end{subfigure}
    \begin{subfigure}[b]{0.23\textwidth}
        \begin{tikzpicture}[scale=0.28]
    \begin{scope}
\foreach \x in {1,2,...,4}
        {
        \node[vertex] (1\x) at (2*\x,3*2) {};
        }
\foreach \x in {1,2,...,4}
        {
        \node[vertex] (2\x) at (2*\x,2*2) {};
        }
\foreach \x in {1,2,...,4}
        {
        \node[vertex] (3\x) at (2*\x,1*2) {};
        }
\foreach \x in {1,2,...,4}
        {
        \node[vertex] (4\x) at (2*\x,0*2) {};
        }
        
\foreach \y [evaluate={\y as \x using int({\y+1})}] in {1,2,...,3} 
        {
        \draw[edge] (1\y) -- (1\x);
        \draw[edge] (2\y) -- (2\x);
        \draw[edge] (3\y) -- (3\x);
        \draw[edge] (4\y) -- (4\x);
        }
        
\foreach \z in {1,2,3,4}
        {\draw[edge] (1\z) -- (2\z);
        \draw[edge] (2\z) -- (3\z);
        \draw[edge] (3\z) -- (4\z);
        }
        
\foreach \z [evaluate={\znext=int({1+\z})}] in {2}
        {
        \draw[edge] (1\z) -- (2\znext);
        \draw[edge] (3\z) -- (4\znext);
        }
        \foreach \z [evaluate={\znext=int({\z-1})}] in {3}
        {
        \draw[edge] (1\z) -- (2\znext);
        \draw[edge] (3\z) -- (4\znext);
        }

\foreach \z [evaluate={\znext=int({1+\z})}] in {1,3}
        {
        \draw[edge] (2\z) -- (3\znext);
        }
        \foreach \z [evaluate={\znext=int({\z-1})}] in {2,4}
        {
        \draw[edge] (2\z) -- (3\znext);
        }
        \coordinate (c15) at (15);
        \coordinate (c16) at (16);

\node[label={[label distance=-3pt]90:\scriptsize{$g_{11}$}}] at (11) {};
        \foreach \x in {2,...,4}
        {
        \node[label={[label distance=-3pt]90:\scriptsize{$g_{1\x}$}}] at (1\x) {};
        }

        \node[label={[label distance=-3pt]0:\scriptsize{$g_{24}$}}] at (24) {};
        \node[label={[label distance=-3pt]0:\scriptsize{$g_{34}$}}] at (34) {};
        \node[label={[label distance=-3pt]0:\scriptsize{$g_{44}$}}] at (44) {};
    \end{scope}

\begin{scope}[shift={(2,0)}]
            \node[vertex] (f2) at (-2,4) {} ;
            \node[vertex] (f3) at (-2,2) {} ;
            \node[vertex] (f4) at (-2,0) {} ;

            \draw[edge] (f2) -- (f3) ;
            \draw[edge] (f3) -- (f4) ;
            \draw[edge] (f2) -- (21) ;
            \draw[edge] (f3) -- (31) ;
            \draw[edge] (f4) -- (41) ;
            \draw[edge] (f2) -- (11) ;
            \draw[edge] (f4) -- (31) ;
            \draw[edge] (f3) -- (41) ;

            \node[label={[label distance=-3pt]180:\scriptsize{$f_{1}$}}] at (f2) {};
            \node[label={[label distance=-3pt]180:\scriptsize{$f_{2}$}}] at (f3) {};
            \node[label={[label distance=-3pt]180:\scriptsize{$f_{3}$}}] at (f4) {};
            
        \end{scope}

\end{tikzpicture}
\caption{$\widetilde{F}_{1,5}$} 
     \end{subfigure}
    \caption{}
    \label{fig:tilda-base-graphs-Pn5}
\end{figure}

\begin{itemize}[leftmargin=*]

\item $\ind(\widetilde{\Gamma}_{n, 5})$:  

Since $N(g_{11})\subseteq N(g_{22})$ and $N(g_{41})\subseteq N(g_{32})$ in the graphs $\Tilde{\Gamma}_{0,5}$, and $\Tilde{\Gamma}_{0,5}\setminus g_{22}$, using \Cref{Folding Lemma}, we get $\ind(\Tilde{\Gamma}_{0,5})\simeq \ind(\Tilde{\Gamma}_{0,5}\setminus\{g_{22},g_{32}\})$. The graph $\Tilde{\Gamma}_{0,5}\setminus\{g_{22},g_{32}\}$ is a path with 6 vertices, and hence, using \Cref{Ind(path)}, we get $\ind(\Tilde{\Gamma}_{0,5})\simeq \mathbb{S}^1$. 
    
Now, consider $n\geq 1$. Application of the same procedure as that used in \Cref{homG_{n,5}} provides us with a graph $\Tilde{\Gamma}^{5}_{n,5}$ such that $\ind(\Tilde{\Gamma}_{n,5}) \simeq \Sigma^{2}\ind(\Tilde{\Gamma}^{5}_{n,5})$, where the graph $\Tilde{\Gamma}^{5}_{n,5}$ (see \Cref{subfig: gammatilde-n-5}) is obtained from $\Gamma_{n,5}^5$ (see \Cref{subfig: gamma-n-55}, in the proof of \Cref{homG_{n,5}}). 

\begin{figure}[h!]
\centering
\begin{subfigure}[b]{0.25\textwidth}
    \begin{tikzpicture}[scale=0.3]
			\begin{scope}
            % vertices in row 1
            \foreach \x in {3,...,7}
            {
            \node[vertex] (1\x) at (2*\x,3*2) {};
            }
            % vertices in row 2
            \foreach \x in {3,...,7}
            {
            \node[vertex] (2\x) at (2*\x,2*2) {};
            }
            % vertices in row 3
            \foreach \x in {3,...,7}
            {
            \node[vertex] (3\x) at (2*\x,1*2) {};
            }
            % vertices in row 4
            \foreach \x in {3,...,7}
            {
            \node[vertex] (4\x) at (2*\x,0*2) {};
            }
            
            % Horizontal edges
            \foreach \y [evaluate={\y as \x using int({\y+1})}] in {3,...,4,6} 
            {
            \draw[edge] (1\y) -- (1\x);
            \draw[edge] (2\y) -- (2\x);
            \draw[edge] (3\y) -- (3\x);
            \draw[edge] (4\y) -- (4\x);
            }
            
            % Vertical edges               
            \foreach \z in {3,4,5,...,7}
            {\draw[edge] (1\z) -- (2\z);
            \draw[edge] (2\z) -- (3\z);
            \draw[edge] (3\z) -- (4\z);
            }
            
            %cross-vacant-cross
            \foreach \z [evaluate={\znext=int({1+\z})}] in {4}
            {
            \draw[edge] (1\z) -- (2\znext);
            \draw[edge] (3\z) -- (4\znext);
            }
            \foreach \z [evaluate={\znext=int({\z-1})}] in {5}
            {
            \draw[edge] (1\z) -- (2\znext);
            \draw[edge] (3\z) -- (4\znext);
            }

            %vacant-cross-vacant
            \foreach \z [evaluate={\znext=int({1+\z})}] in {3,6}
            {
            \draw[edge] (2\z) -- (3\znext);
            }
            \foreach \z [evaluate={\znext=int({\z-1})}] in {4,7}
            {
            \draw[edge] (2\z) -- (3\znext);
            }
            \coordinate (c15) at (15);
            \coordinate (c16) at (16);

            % dotted (and so on)
            \draw[dotted, thick,shorten >=4pt, shorten <=4pt] (10,5) -- (12,5);
            \draw[dotted, thick,shorten >=4pt, shorten <=4pt] (10,3) -- (12,3);
            \draw[dotted, thick,shorten >=4pt, shorten <=4pt] (10,1) -- (12,1);
            
            % under braces
            \draw[edge] (13) to [bend right=45] (33);
            \draw[edge] (13) to [bend right=45] (43);
            \draw[edge] (23) to [bend right=45] (43);

            % vertex labels
            
            \foreach \x in {3,...,5}
            {
            \node[label={[label distance=-3pt]90:\footnotesize{$g_{1\x}$}}] at (1\x) {};
            \node[label={[label distance=-3pt]270:\footnotesize{$g_{4\x}$}}] at (4\x) {};
        }
        \node[label={[label distance=-8pt]45:\footnotesize{$g_{23}$}}] at (23) {};
        \node[label={[label distance=-8pt]-45:\footnotesize{$g_{33}$}}] at (33) {};

            % \node[label={[label distance=-3pt]90:\footnotesize{$g_{1\,{2n}}$}}] at (17) {};

        \end{scope}
    \end{tikzpicture}
\caption{$\Tilde{\Gamma}^{5}_{n,5}$}
\label{subfig: gammatilde-n-5}
\end{subfigure}
 \begin{subfigure}[b]{0.25\linewidth}
    \begin{tikzpicture}[scale=0.3]
        \begin{scope}
        % vertices in row 1
        \foreach \x in {3,4,5,6,7}
        {
        \node[vertex] (1\x) at (2*\x,3*2) {};
        }
        % vertices in row 2
        \foreach \x in {4,5,6,7}
        {
        \node[vertex] (2\x) at (2*\x,2*2) {};
        }
        % vertices in row 3
        \foreach \x in {3,4,5,6,7}
        {
        \node[vertex] (3\x) at (2*\x,1*2) {};
        }
        % vertices in row 4
        \foreach \x in {3,4,5,6,7}
        {
        \node[vertex] (4\x) at (2*\x,0*2) {};
        }
        
        % Horizontal edges
        \foreach \y [evaluate={\y as \x using int({\y+1})}] in {3,4} 
        {
        \draw[edge] (1\y) -- (1\x);
        \draw[edge] (3\y) -- (3\x);
        \draw[edge] (4\y) -- (4\x);
        }
        \foreach \y [evaluate={\y as \x using int({\y+1})}] in {4} 
        {
        \draw[edge] (2\y) -- (2\x);
        }
        \draw[edge] (16) -- (17);
        \draw[edge] (26) -- (27);
        \draw[edge] (36) -- (37);
        \draw[edge] (46) -- (47);
        
        % Vertical edges               
        \foreach \z in {3,4,5,6,7}
        {
        \draw[edge] (3\z) -- (4\z);
        }
        \foreach \z in {4,5,6,7}
        {\draw[edge] (2\z) -- (3\z);
        \draw[edge] (1\z) -- (2\z);
        }
        
        \draw[edge] (13) to [bend right=45] (33);
        \draw[edge] (13) to [bend right=45] (43);

        %cross-vacant-cross
        \foreach \z [evaluate={\znext=int({1+\z})}] in {4}
        {
        \draw[edge] (1\z) -- (2\znext);
        \draw[edge] (3\z) -- (4\znext);
        }
        \foreach \z [evaluate={\znext=int({\z-1})}] in {5}
        {
        \draw[edge] (1\z) -- (2\znext);
        \draw[edge] (3\z) -- (4\znext);
        }

       % vacant-cross-vacant
        \foreach \z [evaluate={\znext=int({1+\z})}] in {6}
        {
        \draw[edge] (2\z) -- (3\znext);
        }
       \foreach \z [evaluate={\znext=int({\z-1})}] in {7}
        {
        \draw[edge] (2\z) -- (3\znext);
        }
\draw[edge] (33) -- (24);

        % dotted (and so on)
        \draw[dotted, thick,shorten >=4pt, shorten <=4pt] (10,5) -- (12,5);
        \draw[dotted, thick,shorten >=4pt, shorten <=4pt] (10,3) -- (12,3);
        \draw[dotted, thick,shorten >=4pt, shorten <=4pt] (10,1) -- (12,1);

        % vertex labels
        \foreach \x in {3,...,5}
        {
        \node[label={[label distance=-3pt]90:\footnotesize{$g_{1\x}$}}] at (1\x) {};
        \node[label={[label distance=-3pt]270:\footnotesize{$g_{4\x}$}}] at (4\x) {};
        }
        \node[label={[label distance=0pt]90:\footnotesize{$g_{33}$}}] at (33) {};
        \node[label={[label distance=-8pt]-30:\footnotesize{$g_{24}$}}] at (24) {};
          \node[label={[label distance=-8pt]30:\footnotesize{$g_{34}$}}] at (34) {};

    \end{scope}
    \end{tikzpicture}
    \caption{$\Tilde{\Gamma}_{n,5}^6$}
    \label{subfig: gamma-n-6'5}
    \end{subfigure}
    %\hspace{1cm}
      \begin{subfigure}[b]{0.25\linewidth}
    \begin{tikzpicture}[scale=0.3]
        \begin{scope}
        % vertices in row 1
        \foreach \x in {3,4,5,6,7}
        {
        \node[vertex] (1\x) at (2*\x,3*2) {};
        }
        % vertices in row 2
        \foreach \x in {5,6,7}
        {
        \node[vertex] (2\x) at (2*\x,2*2) {};
        }
        % vertices in row 3
        \foreach \x in {3,4,5,6,7}
        {
        \node[vertex] (3\x) at (2*\x,1*2) {};
        }
        % vertices in row 4
        \foreach \x in {3,4,5,6,7}
        {
        \node[vertex] (4\x) at (2*\x,0*2) {};
        }
        
        % Horizontal edges
        \foreach \y [evaluate={\y as \x using int({\y+1})}] in {3,4} 
        {
        \draw[edge] (3\y) -- (3\x);
        \draw[edge] (4\y) -- (4\x);
        }
        \draw[edge] (14) -- (15);
       
        \draw[edge] (16) -- (17);
        \draw[edge] (26) -- (27);
        \draw[edge] (36) -- (37);
        \draw[edge] (46) -- (47);
        
        % Vertical edges               
        \foreach \z in {3,4,5,6,7}
        {
        \draw[edge] (3\z) -- (4\z);
        }
        \foreach \z in {5,6,7}
        {\draw[edge] (2\z) -- (3\z);
        \draw[edge] (1\z) -- (2\z);
        }
        
        \draw[edge] (13) to [bend right=45] (33);
        \draw[edge] (13) to [bend right=45] (43);
         \draw[edge] (14) to [bend right=45] (44);
          \draw[edge] (14) to [bend right=45] (34);

        %cross-vacant-cross
        \foreach \z [evaluate={\znext=int({1+\z})}] in {4}
        {
        \draw[edge] (1\z) -- (2\znext);
        \draw[edge] (3\z) -- (4\znext);
        }
        \foreach \z [evaluate={\znext=int({\z-1})}] in {5}
        {
        \draw[edge] (3\z) -- (4\znext);
        }

       % vacant-cross-vacant
        \foreach \z [evaluate={\znext=int({1+\z})}] in {6}
        {
        \draw[edge] (2\z) -- (3\znext);
        }
       \foreach \z [evaluate={\znext=int({\z-1})}] in {7}
        {
        \draw[edge] (2\z) -- (3\znext);
        }

        % dotted (and so on)
        \draw[dotted, thick,shorten >=4pt, shorten <=4pt] (10,5) -- (12,5);
        \draw[dotted, thick,shorten >=4pt, shorten <=4pt] (10,3) -- (12,3);
        \draw[dotted, thick,shorten >=4pt, shorten <=4pt] (10,1) -- (12,1);

        % vertex labels
        \foreach \x in {3,...,5}
        {
        \node[label={[label distance=-3pt]90:\footnotesize{$g_{1\x}$}}] at (1\x) {};
        \node[label={[label distance=-3pt]270:\footnotesize{$g_{4\x}$}}] at (4\x) {};
        }
        \node[label={[label distance=0pt]90:\footnotesize{$g_{33}$}}] at (33) {};
        \node[label={[label distance=-8pt]-30:\footnotesize{$g_{24}$}}] at (24) {};
          \node[label={[label distance=-8pt]30:\footnotesize{$g_{34}$}}] at (34) {};

        % \node[label={45:\footnotesize{$g_{1\,{2n+1}}$}}] at (18) {};
        % \node[label={90:\footnotesize{$g_{1\,{2n-1}}$}}] at (16) {};
        % \node[label={90:\footnotesize{$g_{1\,{2n}}$}}] at (17) {};
        
        % \node[label={0:\footnotesize{$g_{2\,{2n+1}}$}}] at (28) {};
        % \node[label={0:\footnotesize{$g_{3\,{2n+1}}$}}] at (38) {};
        % \node[label={0:\footnotesize{$g_{4\,{2n+1}}$}}] at (48) {};
    \end{scope}
    \end{tikzpicture}
    \caption{$\Tilde{\Gamma}_{n,5}^8$}
    \label{subfig: gamma-n-8'5}
    \end{subfigure}
\caption{}
\end{figure} 
If $n =1 $, then $\n(g_{14})\subseteq \n(g_{23})$ and $\n(g_{44})\subseteq \n(g_{33})$ in the graphs $\Tilde{\Gamma}^{5}_{1,5}$ and $\Tilde{\Gamma}^{5}_{1,5}\setminus g_{23}$, respectively. Thus, from \Cref{Folding Lemma}, $\ind(\Tilde{\Gamma}^{5}_{1,5})\simeq \ind(\Tilde{\Gamma}^{5}_{1,5}\setminus \{g_{23},g_{33}\})$. We can observe that the resultant graph is isomorphic to a cycle on $6$ vertices. Now  \Cref{Ind(Cycle)} implies that $\ind(\mathcal{C}_6) \simeq \Sp^1 \vee \Sp^1$ and therefore $\ind(\Tilde{\Gamma}_{1, 5}) \simeq \Sp^3 \vee \Sp^3$. 
    
Now, let  $n\geq 2$ and proceed with the similar methods as in \Cref{homG_{n,5}}. We obtain the following relations:
\begin{itemize}

    \item $\ind(\Tilde{\Gamma}_{n,5}) \simeq \Sigma^2\left(\ind(\Tilde{\Gamma}_{n,5}^5\setminus g_{23}) \vee \Sigma\ind(\Tilde{A}_{n-2,5}) \right),$
    
    \item $\ind(\Tilde{\Gamma}_{n,5}^6) \simeq\ind(\Tilde{\Gamma}_{n,5}^8 )$, where $\Tilde{\Gamma}_{n,5}^6 \coloneq \Tilde{\Gamma}_{n,5}^5\setminus g_{23}$, and $ \Tilde{\Gamma}_{n,5}^8$ is as in \Cref{subfig: gamma-n-8'5}
    
    \item $\ind(\Tilde{\Gamma}_{n,5}^8) \simeq \Sigma \ind(\Tilde{\Gamma}_{n,5}^8 \setminus \n[g_{33}])\vee \Sigma \ind(\Tilde{\Gamma}_{n,5}^8\setminus \n[g_{43}]).$
\end{itemize}
 By combining all these relations, we find that 
 $$\ind(\Tilde{\Gamma}_{n,5})\simeq \Sigma^3(\ind(\Tilde{A}_{n-2,5}))\bigvee \Sigma^3(\ind(\Tilde{\Gamma}_{n,5}^8 \setminus \n[g_{33}])) \bigvee \Sigma^3(\ind(\Tilde{\Gamma}_{n,5}^8\setminus \n[g_{43}])).$$
However, both the graphs $\Tilde{\Gamma}_{n,5}^8 \setminus \n[g_{33}]$ and $\Tilde{\Gamma}_{n,5}^8 \setminus \n[g_{43}]$ are isomorphic to $\Tilde{D}_{n-2,5}$. 
 Therefore, we have the following:

    \begin{equation}\label{Gamman5tilde}
    \small
        \ind(\Tilde{\Gamma}_{n,5}) \simeq \begin{cases}
         \mathbb{S}^1 & \text{for $n = 0$,} \\
             \mathbb{S}^3\vee \mathbb{S}^3 & \text{for $n = 1$,} \\
            \Sigma^{3}\ind(\Tilde{A}_{n-2,5}) \vee \Sigma^{3} \ind(\Tilde{D}_{n-2,5}) \vee \Sigma^{3} \ind(\Tilde{D}_{n-2,5}) & \text{for $n \geq 2$.}
        \end{cases}
    \end{equation}

\item   $\ind(\widetilde{\Lambda}_{n, 5})$:  
Observe that the graphs $\widetilde{\Lambda}_{n,5}$ and $\Gamma_{n+1,5}$ are isomorphic. Thus, $\ind(\widetilde{\Lambda}_{n,5})\simeq \ind(\Gamma_{n+1,5})$. Hence, from \Cref{homG_{n,5}}, we get the following:

{\small
    \begin{equation}
        \ind(\widetilde{\Lambda}_{n,5}) \simeq \begin{cases}
              \bigvee^3\mathbb{S}^2 & \text{if $n = 0$,} \\
           \Sigma^3\ind(D_{n-1,5}) \vee \Sigma^3\ind(D_{n-1,5}) \vee \Sigma^3\ind(A_{n-1,5}) & \text{if $n \geq 1$.}
        \end{cases}
    \end{equation}
    }

\item $\ind(\widetilde{A}_{n, 5})$: 

Since $N(g_{12})\subseteq N(g_{21})$ and $N(g_{42})\subseteq N(g_{31})$ in the graphs $\widetilde{A}_{0,5}$ and $\widetilde{A}_{0,5}\setminus g_{21}$, respectively, using \Cref{Folding Lemma}, we get $\ind(\widetilde{A}_{0,5})\simeq \ind(\widetilde{A}_{0,5}\setminus\{g_{21},g_{31}\})$. The graph $\widetilde{A}_{0,5}\setminus\{g_{21},g_{31}\}$ is a path on 8 vertices, and hence $\ind(\widetilde{A}_{0,5})\simeq \mathbb{S}^2$ by \Cref{Ind(path)}. For $n\geq 1$, we use the same procedures as in \Cref{homA_{n,5}} to obtain the independence complex of $\widetilde{A}_{n,5}$. Therefore, we have the following:

 \begin{equation}\label{An'}
        \ind(\widetilde{A}_{n,5}) \simeq \begin{cases}
            \mathbb{S}^2 & \text{for $n = 0$,} \\
            \bigvee^2 \Sigma^2\ind(\widetilde{C}_{n-1,5}) \bigvee \Sigma^2\ind(\widetilde{D}_{n-1,5})  & \text{for $n \geq 1$.}
        \end{cases}
    \end{equation}

\item $\ind(\widetilde{B}_{n, 5})$: 

 Since $N(g_{12})\subseteq N(g_{21})$, $N(g_{42})\subseteq N(g_{31})$ and $N(b_2)\subseteq N(b_5)$ in $\widetilde{B}_{0,5}$, $\widetilde{B}_{0,5}\setminus g_{21}$, and $\widetilde{B}_{0,5}\setminus\{g_{21},g_{31}\}$, respectively, using \Cref{Folding Lemma}, we get $\ind(\widetilde{B}_{0,5})\simeq \ind(\widetilde{B}_{0,5}\setminus\{g_{21},g_{31},b_5\})$. Note that $b_3$ is a simplicial vertex and $N(b_3)=\{g_{11},b_4\}$ in the graph $\widetilde{B}_{0,5}\setminus\{g_{21},g_{31},b_5\}$. Therefore, $\ind(\widetilde{B}_{0,5}\setminus\{g_{21},g_{31},b_5\})\simeq \Sigma \ind((\widetilde{B}_{0,5}\setminus\{g_{21},g_{31},b_5\})\setminus \n[g_{11}])\vee \Sigma \ind((\widetilde{B}_{0,5}\setminus\{g_{21},g_{31},b_5\})\setminus \n[{b_4}])$. However, both the graphs  $(\widetilde{B}_{0,5}\setminus\{g_{21},g_{31},b_5\})\setminus \n[g_{11}]$ and $(\widetilde{B}_{0,5}\setminus\{g_{21},g_{31},b_5\})\setminus \n[b_4]$ are isomorphic to $P_7$. Therefore, \Cref{Ind(path)}, implies that $\ind(\widetilde{B}_{0,5})$ is contractible. For $n\geq 1$, $\ind(\widetilde{B}_{n,5})$, as mentioned below, can be computed using the same argument as in \Cref{homBn}. 

\begin{equation}
        \ind(\widetilde{B}_{n,5}) \simeq \begin{cases}
              \ast & \text{for $n = 0$,} \\
            \Sigma\ind(\widetilde{C}_{n,5}) \bigvee\Sigma^2\ind(\widetilde{\Lambda}_{n-1, 5}) & \text{for $n \geq 1$.}
        \end{cases}
    \end{equation}

\begin{figure}
    \centering
    \begin{subfigure}[b]{0.22\linewidth}
        \begin{tikzpicture}[scale=0.3]
        \begin{scope}
\foreach \x in {4,5,6,7}
            {
            \node[vertex] (1\x) at (2*\x,3*2) {};
            }
\foreach \x in {4,5,6,7}
            {
            \node[vertex] (2\x) at (2*\x,2*2) {};
            }
\foreach \x in {3,...,5,6,7}
            {
            \node[vertex] (3\x) at (2*\x,1*2) {};
            }
\foreach \x in {2,3,...,5,6,7}
            {
            \node[vertex] (4\x) at (2*\x,0*2) {};
            }

            \foreach \y [evaluate={\y as \x using {\y+1}}] in {4,6} 
            {
            \draw[edge] (1\y) -- (1\x);
            }
            \foreach \y [evaluate={\y as \x using {\y+1}}] in {4,6} 
            {
            \draw[edge] (2\y) -- (2\x);
            }
            \foreach \y [evaluate={\y as \x using {\y+1}}] in {3,...,4,6} 
            {
            \draw[edge] (3\y) -- (3\x);
            }
            \foreach \y [evaluate={\y as \x using {\y+1}}] in {2,3,...,4,6} 
            {
            \draw[edge] (4\y) -- (4\x);
            }
            
\foreach \z in {4,5,6,7}
            {\draw[edge] (1\z) -- (2\z);}
            \foreach \z in {4,5,6,7}
            {
            \draw[edge] (2\z) -- (3\z);
            }
            \foreach \z in {3,4,5,6,7}
            {
            \draw[edge] (3\z) -- (4\z);
            }
            
\foreach \z [evaluate={\znext=int({1+\z})}] in {4}
            {
            \draw[edge] (1\z) -- (2\znext);
            }
            \foreach \z [evaluate={\znext=int({1+\z})}] in {4}
            {
            \draw[edge] (3\z) -- (4\znext);
            }
            \foreach \z [evaluate={\znext=int({\z-1})}] in {5}
            {
            \draw[edge] (1\z) -- (2\znext);
            }
            \foreach \z [evaluate={\znext=int({\z-1})}] in {3,5}
            {
            \draw[edge] (3\z) -- (4\znext);
            }

\foreach \z [evaluate={\znext=int({1+\z})}] in {6}
            {
            \draw[edge] (2\z) -- (3\znext);
            }
            \foreach \z [evaluate={\znext=int({\z-1})}] in {4,7}
            {
            \draw[edge] (2\z) -- (3\znext);
            }
            \coordinate (c15) at (15);
            \coordinate (c16) at (16);

\draw[dotted, thick,shorten >=4pt, shorten <=4pt] (10,5) -- (12,5);
            \draw[dotted, thick,shorten >=4pt, shorten <=4pt] (10,3) -- (12,3);
            \draw[dotted, thick,shorten >=4pt, shorten <=4pt] (10,1) -- (12,1);

            \foreach \x in {4,5}
            {
            \node[label={[label distance=-3pt]90:\scriptsize{$g_{1\x}$}}] at (1\x) {};
            } 
            \foreach \x in {2,3,4,5}
            {
            \node[label={[label distance=-3pt]-90:\scriptsize{$g_{4\x}$}}] at (4\x) {};
            } 
            \node[label={[label distance=-3pt]180:\scriptsize{$g_{33}$}}] at (33) {};
            
        \end{scope}		
    \end{tikzpicture}
    \caption{$\widetilde{C}_{n,5}^5$}
    \label{subfig: C-n-5-tilde}
    \end{subfigure}
    \hspace{0.2\textwidth}
    \begin{subfigure}[b]{0.22\textwidth}
    \begin{tikzpicture}[scale=0.3]
        \begin{scope}
\foreach \x in {2,3,4,5}
            {
            \node[vertex] (1\x) at (2*\x,3*2) {};
            }
\foreach \x in {2,3,4,5}
            {
            \node[vertex] (2\x) at (2*\x,2*2) {};
            }
\foreach \x in {1,3,4,5}
            {
            \node[vertex] (3\x) at (2*\x,1*2) {};
            }
\foreach \x in {1,3,4,5}
            {
            \node[vertex] (4\x) at (2*\x,0*2) {};
            }
            
\foreach \y [evaluate={\y as \x using {\y+1}}] in {2,4} 
            {
            \draw[edge] (2\y) -- (2\x);
            }
            \foreach \y [evaluate={\y as \x using {\y+1}}] in {4} 
            {
            \draw[edge] (3\y) -- (3\x);
            }
            % \foreach \y [evaluate={\y as \x using {\y+1}}] in {1,2,4} 
            % {
            % \draw[edge] (4\y) -- (4\x);
            % }
            \draw[edge] (12) -- (13) ;
            \draw[edge] (14) -- (15) ;
            
\foreach \z in {2,3,4,5}
            {
            \draw[edge] (1\z) -- (2\z);
            }
            \foreach \z in {3,4,5}
            {
            \draw[edge] (2\z) -- (3\z);
            }
            \foreach \z in {3,4,5}
            {
            \draw[edge] (3\z) -- (4\z);
            }
      
            \draw[edge] (31) -- (41);
            
\foreach \z [evaluate={\znext=int({1+\z})}] in {2}
            {
            \draw[edge] (1\z) -- (2\znext);
            }
            \foreach \z [evaluate={\znext=int({\z-1})}] in {3}
            {
            \draw[edge] (1\z) -- (2\znext);
            % \draw[edge] (3\z) -- (4\znext);
            }

\foreach \z [evaluate={\znext=int({1+\z})}] in {4}
            {
            \draw[edge] (2\z) -- (3\znext);
            }
            \foreach \z [evaluate={\znext=int({\z-1})}] in {5}
            {
            \draw[edge] (2\z) -- (3\znext);
            }

\draw[edge] (44) -- (45);

\draw[dotted, thick,shorten >=4pt, shorten <=4pt] (6,5) -- (8,5);
            \draw[dotted, thick,shorten >=4pt, shorten <=4pt] (6,3) -- (8,3);
            \draw[dotted, thick,shorten >=4pt, shorten <=4pt] (6,1) -- (8,1);
            
\end{scope}		

\begin{scope}[shift={(2,0)}]
            \node[vertex] (d1) at (-2,4) {} ;
            \node[vertex] (d2) at (-2,2) {} ;

\draw[edge] (d1) -- (d2) ;

            \draw[edge] (12) -- (33) ;
            \draw[edge] (12) -- (43) ;
            \draw[edge] (22) -- (33) ;
            \draw[edge] (22) -- (43) ;

        \end{scope}
\foreach \x in {2,3}
            {
            \node[label={[label distance=-3pt]90:\scriptsize{$g_{1\x}$}}] at (1\x) {};
            } 
            \foreach \x in {1,3}
            {
            \node[label={[label distance=-3pt]-90:\scriptsize{$g_{4\x}$}}] at (4\x) {};
            }
            \node[label={[label distance=-4pt]-180:\scriptsize{$g_{22}$}}] at (22) {};
            % \node[label={[label distance=-4pt]0:\scriptsize{$g_{23}$}}] at (23) {};
            % \node[label={[label distance=-4pt]0:\scriptsize{$g_{33}$}}] at (33) {};
            \node[label={[label distance=-4pt]90:\scriptsize{$g_{31}$}}] at (31) {};
            \node[label={[label distance=-3pt]180:\scriptsize{$d_{1}$}}] at (d1) {};
            \node[label={[label distance=-3pt]180:\scriptsize{$d_{2}$}}] at (d2) {};

\end{tikzpicture}
    \caption{$\widetilde{D}_{n,5}^6$}
    \label{subfig: D-n-5-5-tilde}
    \end{subfigure}
    \caption{}
    \label{fig:Cn-5-5tilde-Dn-5-6-tilde}
\end{figure}

\item $\ind(\widetilde{C}_{n, 5})$:  

Observe that $N(g_{12})\subseteq N(g_{21})$, and  $N(g_{42})\subseteq N(g_{31})$ in $\widetilde{C}_{0,5}$, $\widetilde{C}_{0,5}\setminus g_{21}$, respectively. From \Cref{Folding Lemma}, we get $\ind(\widetilde{C}_{0,5})\simeq \ind(\widetilde{C}_{0,5}\setminus\{g_{21},g_{31}\})$. The graph $\widetilde{C}_{0,5}\setminus\{g_{21},g_{31}\}$ is a path on 7 vertices and  hence  $\ind(\widetilde{C}_{0,5})$ is contractible.

Since $\n(g_{11})\subseteq \n(g_{22})$, $\n(g_{14})\subseteq \n(g_{23})$ and $\n(g_{44})\subseteq \n(g_{33})$ in the graphs $\widetilde{C}_{1,5}$ (see \Cref{tildeC15}), $\widetilde{C}_{1, 5}\setminus g_{22}$, and $\widetilde{C}_{1, 5}\setminus \{g_{22},g_{23}\}$, respectively, we get $\ind(\widetilde{C}_{1,5})\simeq \ind(\widetilde{C}_{1, 5}\setminus\{g_{22},g_{23},g_{33}\}).$ Let $\widetilde{C}_{1, 5}^1 = \widetilde{C}_{1, 5}\setminus\{g_{22},g_{23},g_{33}\} $.  Observe that $c_1$ is a simplicial vertex in $\widetilde{C}_{1, 5}^1$ with $\n(c_1)=\{g_{31},g_{41}\}$. Thus, using \Cref{Simplicial Vertex Lemma}, 
we get $\ind(\widetilde{C}_{1, 5})\simeq \Sigma\ind(\widetilde{C}_{1, 5}^1 \setminus\n[g_{31}]) \vee \Sigma\ind(\widetilde{C}_{1, 5}^1 \setminus\n[g_{41}])$. Since the graph $\widetilde{C}_{1, 5}^1\setminus\n[g_{31}]$ is isomorphic to $P_9$, and $\widetilde{C}_{1, 5}^1\setminus\n[g_{41}]$ is isomorphic to the cycle graph $\mathcal{C}_{10}$, using \Cref{Ind(Cycle),,Ind(path)}, we get
$\ind(\widetilde{C}_{1, 5}) \simeq \mathbb{S}^3\vee\mathbb{S}^3$.
 For $n\geq 2$, we use the same argument as in \Cref{homCn} to get recurrence relations for $\ind(\widetilde{C}_{n,5})$. The homotopy type of $\ind(\widetilde{C}_{n,5})$ is given by the following:

\begin{equation}
        \ind(\widetilde{C}_{n,5}) \simeq \begin{cases} 
        * & \text{for $n = 0$,} \\
             \mathbb{S}^3\vee\mathbb{S}^3& \text{for $n = 1$,} \\          
             \Sigma^{3}\ind(\widetilde{A}_{n-2,5}) \bigvee \Sigma^2\ind(\widetilde{E}_{n-1,5}) \bigvee \Sigma^{3}\ind(\widetilde{F}_{n-2,5}) & \text{for $n \geq 2$.}
        \end{cases}
    \end{equation}

\item $\ind(\widetilde{D}_{n, 5})$: 

Consider the graph $\widetilde{D}_{0,5}$. Since $N(g_{12})\subseteq N(g_{21})$, and  $N(g_{42})\subseteq N(g_{31})$ in graphs $\widetilde{D}_{0,5}$, and $\widetilde{D}_{0,5}\setminus g_{21}$, respectively, using \Cref{Folding Lemma}, we obtain $\ind(\widetilde{D}_{0,5})\simeq\ind(\widetilde{D}_{0,5}\setminus\{g_{21},g_{31}\})$. The graph $\widetilde{D}_{0,5}\setminus\{g_{21},g_{31}\}$ is a cycle on 8 vertices. Therefore, using \Cref{Ind(Cycle)}, we get that $\ind(\widetilde{D}_{0,5})\simeq \mathbb{S}^2$. For $n\geq 1$, we proceed in the same way as in \Cref{homDn}. We obtain the following relations:
\begin{itemize}
    \item $\ind(\widetilde{D}_{n,5})\simeq \ind(\widetilde{D}_{n,5}^4)\simeq \ind(\widetilde{D}_{n,5}^4 \setminus g_{21}) \vee \Sigma^2 \ind(\widetilde{E}_{n-1,5})$,
        \item $\ind(\widetilde{D}_{n,5}^4\setminus g_{21})\simeq \ind(\widetilde{D}_{n,5}^6),$
        \item $\ind(\widetilde{D}_{n,5}^6) \simeq  \ind(\widetilde{D}_{n,5}^6 \setminus g_{12}) \vee \Sigma\ind(\widetilde{D}_{n,5}^6 \setminus \n[g_{12}])$.
\end{itemize}
By combining all the above relations, we get
$\ind(\widetilde{D}_{n,5})\simeq \ind(\widetilde{D}_{n,5}^6 \setminus g_{12}) \vee \Sigma\ind(\widetilde{D}_{n,5}^6 \setminus \n[g_{12}])\vee  \Sigma^2 \ind(\widetilde{E}_{n-1,5})$.
The graph $\widetilde{D}_{n,5}^6$ (see \Cref{subfig: D-n-5-5-tilde}) is obtained from the graph $D_{n,5}^6$ (see \Cref{subfig: D-n-6}). In the case of $n=1$, we find that the graph $\widetilde{D}_{1,5}^6 \setminus \n[g_{12}]$ consists of three components: two edges and one with a path on $4$ vertices. Thus, using \Cref{Ind(path)}, we get $\ind(\widetilde{D}_{1,5}^6 \setminus \n[g_{12}])$ is contractible. On the other hand, the graph $\widetilde{D}_{1,5}^6 \setminus g_{12}$ is isomorphic to a three-component graph, where two are edges and the remaining one is a graph isomorphic to $\widetilde{E}_{0,5}$. Therefore, using \eqref{En5tilde}, we get $\ind(\widetilde{D}_{1,5}^6 \setminus g_{12})\simeq \mathbb{S}^3$. Hence, we get  $\ind(\widetilde{D}_{1,5})\simeq \mathbb{S}^3\vee\mathbb{S}^3$. For $n\geq 2$, the independence complexes of $\widetilde{D}_{n,5}$ follow the same pattern as those in \Cref{homDn}. Thus, we have the following:

   \begin{equation}
        \ind(\widetilde{D}_{n,5}) \simeq \begin{cases}
         \mathbb{S}^2 & \text{for $n =0$,} \\
              \mathbb{S}^3\vee \mathbb{S}^3 & \text{for $n =1$,} \\
            \Sigma^3\ind(\widetilde{\Lambda}_{n-2,5}) \bigvee \Sigma^2\ind(\widetilde{E}_{n-1,5}) \bigvee \Sigma^2 \ind(\widetilde{E}_{n-1,5}) & \text{for $n \geq 2$.}
        \end{cases}
    \end{equation}

\item  $\ind(\widetilde{E}_{n, 5})$: 

We find that $N(g_{12})\subseteq N(g_{21})$, and  $N(g_{42})\subseteq N(g_{31})$ in graphs $\widetilde{E}_{0,5}$, and $\widetilde{E}_{0,5}\setminus g_{21}$, respectively. Using \Cref{Folding Lemma}, we get $\ind(\widetilde{E}_{0,5})\simeq \ind(\widetilde{E}_{0,5}\setminus\{g_{21},g_{31}\})$. Observe that the graph $\widetilde{E}_{0,5}\setminus\{g_{21},g_{31}\}$ is a cycle on $7$ vertices. Thus, 
$\ind(\widetilde{E}_{0,5})\simeq \mathbb{S}^1$.

For $n\geq 1$, we proceed in the same way as in \Cref{homEn}. Using \eqref{En pre}, we get $\ind(\widetilde{E}_{n,5}) \simeq \Sigma^2 \left(\vee^2 \ind(\widetilde{C}_{n,5}^5)\right) \vee \Sigma\ind(\widetilde{A}_{n-1,5}) \vee \ind(\widetilde{\Gamma}_{n,5}),$ where the graph $\widetilde{C}_{n,5}^5$ (see \Cref{subfig: C-n-5-tilde}) is obtained from the graph $C_{n,5}^5$ (see \Cref{subfig: C-n-5}).
In case of $n=1$,  $\n(g_{44})\subseteq \n(g_{33})$ in $\widetilde{C}_{1,5}^5$ and therefore  $\ind(\widetilde{C}_{1,5}^5)\simeq \ind(\widetilde{C}_{1,5}^5\setminus g_{33})$. Since the graph $\widetilde{C}_{1,5}^5\setminus g_{33}$  is a path on 
$6$ vertices, $\ind(\widetilde{C}_{1,5}^5)\simeq \mathbb{S}^1$.
Now $\ind(\widetilde{A}_{0,5})\simeq \mathbb{S}^2$ and $\ind(\widetilde{\Gamma}_{1,5})\simeq \mathbb{S}^3\vee \mathbb{S}^3$ imply that $\ind(\widetilde{E}_{1,5}) \simeq \bigvee^5{\mathbb{S}^3}$. For $n\geq 2$, we use the same arguments as those in \Cref{homCn} to get  $\ind(\widetilde{C}_{n,5}^5) \simeq \Sigma\ind(\widetilde{A}_{n-2,5}) \vee \Sigma\ind(\widetilde{F}_{n-2,5})$. Thus, combining all of these, we have the following:

{\small
\begin{equation}\label{En5tilde}
        \ind(\widetilde{E}_{n,5}) \simeq \begin{cases}
        \mathbb{S}^1 & \text{for $n = 0$,} \\
              \bigvee^5\mathbb{S}^3 & \text{for $n = 1$,} \\
            \left({\bigvee}^2\Sigma^3\ind(\widetilde{A}_{n-2,5})\right) \vee \left({\bigvee}^2\Sigma^3 \ind(\widetilde{F}_{n-2,5})\right) \vee \Sigma\ind(\widetilde{A}_{n-1,5}) \vee \ind(\widetilde{\Gamma}_{n,5})
            & \text{for $n \geq 2$.}
        \end{cases}
    \end{equation}}
\item $\ind(\widetilde{F}_{n, 5})$:

Since $N(g_{12})\subseteq N(g_{21})$, and  $N(g_{42})\subseteq N(g_{31})$ in graphs $\widetilde{F}_{0,5}$, and $\widetilde{F}_{0,5}\setminus g_{21}$, respectively, from \Cref{Folding Lemma},  $\ind(\widetilde{F}_{0,5})\simeq \ind(\widetilde{F}_{0,5}\setminus\{g_{21},g_{31}\})$. Since $f_3$ is a simplicial vertex in $\widetilde{F}_{0,5}\setminus\{g_{21},g_{31}\}$ with $N(f_3)=\{g_{41},f_2\}$, $\ind(\widetilde{F}_{0,5})\simeq \Sigma \ind((\widetilde{F}_{0,5}\setminus\{g_{21},g_{31}\})\setminus\cn{}{g_{41}})\vee \Sigma \ind((\widetilde{F}_{0,5}\setminus\{g_{21},g_{31}\})\setminus\cn{}{f_2}) $. Both the graphs $(\widetilde{F}_{0,5}\setminus\{g_{21},g_{31}\})\setminus\cn{}{g_{41}}$ and $(\widetilde{F}_{0,5}\setminus\{g_{21},g_{31}\})\setminus\cn{}{f_2}$ are isomorphic to $P_5$. Hence,  $\ind(\widetilde{F}_{0,5})\simeq \mathbb{S}^2\vee\mathbb{S}^2$. For $n\geq 1$, the relation for $\ind(\widetilde{F}_{n,5})$ follows from a similar argument as that of \Cref{homFn}. Thus, 
    
{\small
\begin{equation}\label{lambda_{n,5}'}
        \ind(\widetilde{F}_{n,5}) \simeq \begin{cases}
              \mathbb{S}^2\vee\mathbb{S}^2 & \text{for $n =0$,} \\
            \Sigma\ind(\widetilde{B}_{n-1,5}) \bigvee 
            \Sigma^2\ind(\widetilde{C}_{n-1,5}) \bigvee 
            \Sigma^2\ind(\widetilde{C}_{n-1,5}) \bigvee 
            \Sigma^2\ind(\widetilde{D}_{n-1,5})
            & \text{for $n \geq 1$.}
        \end{cases}
    \end{equation}}
    \end{itemize}
  
 \begin{theorem}\label{PnxP5}
     For any graph $\widetilde{G}_{n,5}\in \set{\widetilde{\Gamma}_{n,5}, \widetilde{\Lambda}_{n,5}, \widetilde{A}_{n,5}, \widetilde{B}_{n,5}, \widetilde{C}_{n,5}, \widetilde{D}_{n,5}, \widetilde{E}_{n,5}, \widetilde{F}_{n,5} : n\geq 0 }$, $\ind(\widetilde{G}_{n,5})$ is homotopy equivalent to a wedge of spheres.
 \end{theorem}
 \begin{proof}
     The result follows from \eqref{Gamman5tilde} to \eqref{lambda_{n,5}'}, \Cref{Theorem:PnP5firstcase} and induction on $n$.
 \end{proof}

\subsection{\texorpdfstring{Homotopy type of $\M(P_n \times P_5)$}{Homotopy type of\M(Pn × P5)}}\label{subsection:PnP5dimensionbound}
In this section, we conclude the homotopy type of $\M(P_n \times P_5)$. We also determine the maximum and minimum dimensions of the spheres occurring in the homotopy type of the independence complexes of the graphs in 
$\{\Gamma_{n,5},\Lambda_{n,5},A_{n,5},B_{n,5}$, $C_{n,5},D_{n,5},E_{n,5}, F_{n,5} \}$.

\begin{theorem}\label{pnxp5}
    For $n\geq 2$,  $\M(P_{n}\times P_5)$ is homotopy equivalent to a wedge of spheres.
\end{theorem}
\begin{proof}
  \textbf{Case 1:} If $n$ is even, i.e., $n=2k$ for some positive integer $k$, then  $P_n\times P_5$ comprises of two isomorphic connected components whose line graphs are isomorphic to $\Gamma_{k-1,5}$. Thus, $$\M(P_n\times P_5)\simeq \ind(\Gamma_{k-1,5})\ast \ind(\Gamma_{k-1,5}).$$ Now, the result follows from \Cref{Theorem:PnP5firstcase}.

  \noindent\textbf{Case 2:} If $n$ is odd, i.e., $n=2k+1$ for some positive integer $k$, then  $P_n\times P_5$ comprises of two non-isomorphic connected components with corresponding line graphs isomorphic to $\Lambda_{k-1,5}$ and $\widetilde{\Gamma}_{k-1,5}$. Thus, $$\M(P_n\times P_5)\simeq \ind(\Lambda_{k-1,5})\ast \ind(\widetilde{\Gamma}_{k-1,5}).$$
 The result follows from \Cref{Theorem:PnP5firstcase,,PnxP5}.
  \end{proof}

\begin{proposition}{\label{propn: dim-min-max-Pn5}}
    For $n\geq 0$, the maximum and minimum dimension of the spheres in the homotopy type of the independence complexes of graphs in $\{\Gamma_{n, 5}, \Lambda_{n, 5}, A_{n, 5}, B_{n, 5}, C_{n, 5}, $  $D_{n, 5}, E_{n, 5}, F_{n, 5}\}$  are given in $\Cref{tab: max Pn5}$ and $\Cref{tab: min Pn5}$, respectively.

  \begin{table}[h!]
  	{\tiny
  		\setlength{\tabcolsep}{5pt} 
  		\renewcommand{\arraystretch}{1.6} 
  		\centering
  		\begin{tabular}{|c|c|c|c|c|c|c|}
  			\hline
  			\multirow{2}{*}{$n \equiv$} & \multicolumn{6}{c|}{$d_{\mathit{max}}$ of independence complexes of graphs} \\
  			\cline{2-7}
  			{} & $\Gamma_{n,5}$ & $\Lambda_{n,5}$ & $A_{n,5}, F_{n,5}$ & $B_{n,5}$ & $C_{n,5}, D_{n,5}$ & $E_{n,5}$ \\
  			\hline
  			$0 \mod{4}$ & $7\big(\frac{n}{4}\big) \ (n \neq 0)$ & $7\big(\frac{n}{4}\big)+1$ & $7\big(\frac{n}{4}\big)+1$ & $7\big(\frac{n}{4}\big)+2$ & $7\big(\frac{n}{4}\big)+1$ & $7\big(\frac{n}{4}\big)$ \\
  			\hline
  			$1 \mod{4}$ & $7\big(\frac{n-1}{4}\big)+2$ & $7\big(\frac{n-1}{4}\big)+3$ & $7\big(\frac{n-1}{4}\big)+3$ & $7\big(\frac{n-1}{4}\big)+3$ & $7\big(\frac{n-1}{4}\big)+2$ & $7\big(\frac{n-1}{4}\big)+2$ \\
  			\hline
  			$2 \mod{4}$ & $7\big(\frac{n-2}{4}\big)+4$ & $7\big(\frac{n-2}{4}\big)+5$ & $7\big(\frac{n-2}{4}\big)+4$ & $7\big(\frac{n-2}{4}\big)+5$ & $7\big(\frac{n-2}{4}\big)+4$ & $7\big(\frac{n-2}{4}\big)+4$ \\
  			\hline
  			$3 \mod{4}$ & $7\big(\frac{n-3}{4}\big)+6$ & $7\big(\frac{n-3}{4}\big)+6$ & $7\big(\frac{n-3}{4}\big)+6$ & $7\big(\frac{n+1}{4}\big)$ & $7\big(\frac{n-3}{4}\big)+6$ & $7\big(\frac{n-3}{4}\big)+6$ \\
  			\hline
  	\end{tabular}}
  	\caption{}
  	\label{tab: max Pn5}
  \end{table}

  %\vspace{0.5 cm}
  %-------------------------------Table 7 Updated------------------
  
  \begin{table}[h!]
  	{\tiny
  		\setlength{\tabcolsep}{5pt} 
  		\renewcommand{\arraystretch}{1.6} 
  		\centering
  		\begin{tabular}{|c|c|c|c|c|c|c|}
  			\hline
  			\multirow{2}{*}{\centering $n \equiv$} & \multicolumn{6}{c|}{$d_{\mathit{min}}$ of independence complexes of graphs} \\
  			\cline{2-7}
  			& $\Gamma_{n,5}$ & $\Lambda_{n,5}, A_{n,5}$ &  $B_{n,5}$ & $C_{n,5}$ & $D_{n,5}, F_{n,5}$ & $E_{n,5}$ \\
  			\hline
  			$0 \mod{5}$ & $8\big(\frac{n}{5}\big)+1 \ (n \neq 0)$ & $8\big(\frac{n}{5}\big)+1$ &  $8\big(\frac{n}{5}\big)+1$ & $8\big(\frac{n}{5}\big)+1$ & $8\big(\frac{n}{5}\big)+1$ & $8\big(\frac{n}{5}\big)+1$ \\
  			\hline
  			$1 \mod{5}$ & $8\big(\frac{n-1}{5}\big)+2$ & $8\big(\frac{n-1}{5}\big)+3$ &  $8\big(\frac{n-1}{5}\big)+3$ & $8\big(\frac{n-1}{5}\big)+2$ & $8\big(\frac{n-1}{5}\big)+2$ & $8\big(\frac{n-1}{5}\big)+2$ \\
  			\hline
  			$2 \mod{5}$ & $8\big(\frac{n-2}{5}\big)+4$ & $8\big(\frac{n-2}{5}\big)+4$ &  $8\big(\frac{n-2}{5}\big)+5$ & $8\big(\frac{n-2}{5}\big)+4$ & $8\big(\frac{n-2}{5}\big)+4$ & $8\big(\frac{n-2}{5}\big)+4$ \\
  			\hline
  			$3 \mod{5}$ & $8\big(\frac{n-3}{5}\big)+5$ & $8\big(\frac{n-3}{5}\big)+6$ &  $8\big(\frac{n-3}{5}\big)+6$ & $8\big(\frac{n-3}{5}\big)+5$ & $8\big(\frac{n-3}{5}\big)+6$ & $8\big(\frac{n-3}{5}\big)+5$ \\
  			\hline
  			$4 \mod{5}$ & $8\big(\frac{n-4}{5}\big)+7$ & $8\big(\frac{n-4}{5}\big)+7$ &  $8\big(\frac{n-4}{5}\big)+8$ & $8\big(\frac{n-4}{5}\big)+7$ & $8\big(\frac{n-4}{5}\big)+7$ & $8\big(\frac{n-4}{5}\big)+7$ \\
  			\hline
  	\end{tabular}}
  	\caption{}
  	\label{tab: min Pn5}
  \end{table}

\end{proposition}
\begin{proof}
    We prove the statement by using induction on $n$ for the graphs in $\{\Gamma_{n,5},\Lambda_{n,5}, A_{n,5},B_{n,5}$, $C_{n,5},D_{n,5},E_{n,5},F_{n,5}\}$ for $n\geq 0$. 
    The statement holds for $n=0$ from \Crefrange{homG_{n,5}}{homFn}.  Assume that the statement holds for all non-negative integers less than $n$. For $\Gamma_{n,5}$, \Cref{homG_{n,5}} implies that the statement holds for $n=1$ as well. Further, for $n\geq 2$, we have 
    \begin{equation}{\label{eq: gamma-n-5 rec}}
        \ind(\Gamma_{n,5}) \simeq {\vee}^2\Sigma^{3}\ind(D_{n-2,5}) \vee \Sigma^3\ind(A_{n-2,5}).     \end{equation}
    Therefore, $\dimax{\Gamma_{n,5}}$ is given by
    \begin{equation*}
    \begin{split}
        \max\{ & 3+\dimax{A_{n-2,5}}, 3+\dimax{D_{n-2,5}} \}.
    \end{split}
    \end{equation*}
    Let $n \equiv 0 \mod{4}$. Then $n-2$ is congruent to $2$  modulo $4$. By the induction hypothesis, we have
    \begin{equation*}
        \begin{split}
            \dimax{\Gamma_{n,5}} = &
            \max\{ 3+(\sfrac{7(n-4)}{4}+4), 3+(\sfrac{7(n-4)}{4}+4) \} \\
            = & 7\left(\frac{n}{4}\right).
        \end{split}
    \end{equation*}
    % For $n \equiv 1 \mod{4}$, $n-2$ and $n-3$ are congruent to $3$ and $2$ modulo $4$, respectively. Thus by induction hypothesis, we have
    % \begin{equation*}
    %     \begin{split}
    %         \dimax{\Gamma_{n,5}} = &
    %         \max\{ 3+(\sfrac{7(n-5)}{4}+6),4+(\sfrac{7(n-5)}{4}+5),\\ & \phantom{\max\{}3+(\sfrac{7(n-5)}{4}+6), 3+(\sfrac{7(n-4)}{4}+6) \} \\
    %         = & 7\left(\frac{n-5}{4}\right) + 9 = 7\left(\frac{n-1}{4}\right) + 2.
    %     \end{split}
    % \end{equation*}
    % For $n \equiv 2 \mod{4}$, $n-2$ and $n-3$ are congruent to $0$ and $3$ modulo $4$, respectively. By induction hypothesis, we have
    % \begin{equation*}
    %     \begin{split}
    %         \dimax{\Gamma_{n,5}} = &
    %         \max\{ 3+(\sfrac{7(n-2)}{4}+1),4+(\sfrac{7(n-2)}{4}),\\ & \phantom{\max\{}3+(\sfrac{7(n-2)}{4}+1), 3+(\sfrac{7(n-2)}{4}+1) \} \\
    %         = & 7\left(\frac{n-2}{4}\right)+4.
    %     \end{split}
    % \end{equation*}
    % For $n \equiv 3 \mod{4}$, $n-2$ and $n-3$ are congruent to $1$ and $0$ modulo $4$, respectively. By induction hypothesis, we have
    % \begin{equation*}
    %     \begin{split}
    %         \dimax{\Gamma_{n,5}} = &
    %         \max\{ 3+(\sfrac{7(n-3)}{4}+3),4+(\sfrac{7(n-3)}{4})+2,\\ & \phantom{\max\{}3+(\sfrac{7(n-3)}{4}+2), 3+(\sfrac{7(n-3)}{4}+2) \} \\
    %         = & 7\left(\frac{n-3}{4}\right)+6.
    %     \end{split}
    % \end{equation*}
    The proof or $n\equiv 1,2,3 \mod{4}$ follows similarly.

On the other hand, $\dimin{\Gamma_{n,5}}$ is given by 
    \begin{equation*}
    \begin{split}
        \min\{ & 3+\dimin{A_{n-2,5}}, 3+\dimin{D_{n-2,5}} \}.
    \end{split}
    \end{equation*}
For $n\equiv 0 \mod{5}$, $n-2$ is congruent to $3$  modulo $5$. Therefore, using the induction hypothesis, we have
\begin{equation*}
        \begin{split}
            \dimin{\Gamma_{n,5}} = &
            \min\{ 3+(\sfrac{8(n-5)}{5}+6), 3+(\sfrac{8(n-5)}{5}+6) \} \\
            = & 8\left(\frac{n}{5}\right)+1.
        \end{split}
    \end{equation*}
    Similarly, we prove the required equality for $n\equiv 1,2,3,4 \mod{5}$.

The proof follows analogously for the maximum and minimum dimension of the wedge in independence complexes of $\Lambda_{n,5},A_{n,5},B_{n,5},C_{n,5},D_{n,5},E_{n,5}$ and $F_{n,5}$.
\end{proof}

\begin{proposition}
   For $n\geq 0$, the maximum and minimum dimension of the spheres in the homotopy type of the independence complexes of graphs in $\{\widetilde{\Gamma}_{n, 5}, \widetilde{\Lambda}_{n, 5}, \widetilde{A}_{n, 5}, \widetilde{B}_{n, 5}, \widetilde{C}_{n, 5}, $  $\widetilde{D}_{n, 5}, \widetilde{E}_{n, 5}, \widetilde{F}_{n, 5}\}$  are given in $\Cref{tab: max 2nd Pn5}$ and $\Cref{tab: min 2nd Pn5}$, respectively.
     \begin{table}[h!]
 	{\tiny
 		\setlength{\tabcolsep}{5pt} % horizontal spacing
 		\renewcommand{\arraystretch}{1.7} % vertical spacing
 		\centering
 		\begin{tabular}{|c|c|c|c|c|c|}
 			\hline
 			\multirow{2}{*}{\centering $n \equiv$} & \multicolumn{5}{c|}{$d_{\mathit{max}}$ of independence complexes of graphs} \\
 			\cline{2-6}
 			& $\widetilde{\Gamma}_{n,5}, \widetilde{E}_{n,5}$ & $\widetilde{\Lambda}_{n,5}$ & $\widetilde{A}_{n,5}, \widetilde{F}_{n,5}$ & $\widetilde{B}_{n,5}$ & $\widetilde{C}_{n,5}, \widetilde{D}_{n,5}$ \\
 			\hline
 			$0 \mod{4}$ & $7\big(\frac{n}{4}\big)+1$ & $7\big(\frac{n}{4}\big)+2$ & $7\big(\frac{n}{4}\big)+2$ & $7\big(\frac{n}{4}\big)+3$ & $7\big(\frac{n}{4}\big)+2$ \\
 			\hline
 			$1 \mod{4}$ & $7\big(\frac{n-1}{4}\big)+3$ & $7\big(\frac{n-1}{4}\big)+4$ & $7\big(\frac{n-1}{4}\big)+4$ & $7\big(\frac{n-1}{4}\big)+4$ & $7\big(\frac{n-1}{4}\big)+3$ \\
 			\hline
 			$2 \mod{4}$ & $7\big(\frac{n-2}{4}\big)+5$ & $7\big(\frac{n-2}{4}\big)+6$ & $7\big(\frac{n-2}{4}\big)+5$ & $7\big(\frac{n-2}{4}\big)+6$ & $7\big(\frac{n-2}{4}\big)+5$ \\
 			\hline
 			$3 \mod{4}$ & $7\big(\frac{n+1}{4}\big)$ & $7\big(\frac{n+1}{4}\big)$ & $7\big(\frac{n+1}{4}\big)$ & $7\big(\frac{n+1}{4}\big)+1$ & $7\big(\frac{n+1}{4}\big)$ \\
 			\hline
 	\end{tabular}}
 	\vspace{-0.2 cm}
 	\caption{}
 	\label{tab: max 2nd Pn5}
 \end{table}
 %--------------------------Table 9 Updated-----------------------------
 \begin{table}[h!]
 	{\tiny
 		\setlength{\tabcolsep}{5pt} % horizontal spacing
 		\renewcommand{\arraystretch}{1.6} % vertical spacing
 		\centering
 		\begin{tabular}{|c|c|c|c|c|c|}
 			\hline
 			\multirow{2}{*}{\centering $n \equiv$} & \multicolumn{5}{c|}{$d_{\mathit{min}}$ of independence complexes of graphs} \\
 			\cline{2-6}
 			& $\widetilde{\Gamma}_{n,5}, \widetilde{E}_{n,5}$ & $\widetilde{\Lambda}_{n,5}, \widetilde{A}_{n,5}$  & $\widetilde{B}_{n,5}$ & $\widetilde{C}_{n,5}, \widetilde{D}_{n,5}$ & $\widetilde{F}_{n,5}$ \\
 			\hline
 			$0 \mod{5}$ & $8\big(\frac{n}{5}\big)+2 \ (n \neq 0)$ & $8\big(\frac{n}{5}\big)+2$  & $8\big(\frac{n}{5}\big)+3$ & $8\big(\frac{n}{5}\big)+2$ & $8\big(\frac{n}{5}\big)+2$\\
 			\hline
 			$1 \mod{5}$ & $8\big(\frac{n-1}{5}\big)+3$ & $8\big(\frac{n-1}{5}\big)+4$  & $8\big(\frac{n-1}{5}\big)+4$ & $8\big(\frac{n-1}{5}\big)+3$ & $8\big(\frac{n-1}{5}\big)+4$\\
 			\hline
 			$2 \mod{5}$ & $8\big(\frac{n-2}{5}\big)+5$ & $8\big(\frac{n-2}{5}\big)+5$  & $8\big(\frac{n-2}{5}\big)+6$ & $8\big(\frac{n-2}{5}\big)+5$ & $8\big(\frac{n-2}{5}\big)+5$\\
 			\hline
 			$3 \mod{5}$ & $8\big(\frac{n-3}{5}\big)+6$ & $8\big(\frac{n-3}{5}\big)+7$  & $8\big(\frac{n-3}{5}\big)+7$ & $8\big(\frac{n-3}{5}\big)+7$ & $8\big(\frac{n-3}{5}\big)+7$\\
 			\hline
 			$4 \mod{5}$ & $8\big(\frac{n+1}{5}\big)$ & $8\big(\frac{n+1}{5}\big)+1$  & $8\big(\frac{n+1}{5}\big)+1$ & $8\big(\frac{n+1}{5}\big)$ & $8\big(\frac{n+1}{5}\big)$\\
 			\hline
 	\end{tabular}}
 	\vspace{-0.2cm}
 	\caption{}
 	\label{tab: min 2nd Pn5}
 \end{table}

\end{proposition}
\begin{proof}
    The proof is essentially verbatim to the proof of \Cref{propn: dim-min-max-Pn5}.
\end{proof}

\begin{rem}

The maximum and minimum dimension of spheres in the homotopy type of $\M(P_n \times P_5)$ as given in \Cref{proposition:dim_PnP5} can be concluded from \Crefrange{tab: max Pn5}{tab: min 2nd Pn5}. 
\end{rem}

%---------------------------------------------------------------------------
\section{Future Directions}\label{section:futureplan}
We proved that for $n \geq 3$ and $3 \leq m \leq 5$, the matching complex of $P_n \times P_m$ is homotopy equivalent to a wedge of spheres. For $m =3$, we explicitly compute the number and dimension of spheres appearing in the wedge. Further, for $m \in \{4, 5\}$,  we have given the minimum and maximum dimension of spheres appearing in the wedge in $\mathsf{M}(P_n \times P_m)$. This exploration naturally leads to the following question:
\begin{ques}
    Can we determine the dimensions of the spheres appearing in the wedge representing the homotopy types of the matching complexes of categorical products $P_n\times P_4$ and $P_n\times P_5$?
\end{ques}
Based on the main result of this article and our computer-based computations, we propose the following:

\begin{conj}
    For $n\geq 3$ and  $m\geq 6$, the matching complex of categorical product $P_n\times P_m$ is homotopy equivalent to a wedge of spheres.
\end{conj}

For odd $m$ and $n$, one of the components of $P_n \times P_m$ is an induced subgraph of $P_n \Box P_m$  (see \Cref{fig:cat-cart-paths}). Since the homotopy type of the matching complexes of $P_n \Box P_m$ is not known for $m \geq 3$, it will be interesting to see if there is any relation between the topology of $\M(P_n \times P_n)$ and $\M(P_n \Box P_m)$. 

\begin{ques}
Is there any correlation between the matching complexes of the cartesian product of path graphs and those of the categorical product of path graphs?
\end{ques}

\section{Acknowledgment}
The authors would like to thank the anonymous referees for many useful comments and suggestions. We also would like to thank an anonymous referee for pointing out a gap in the proof of Claim 5.4 in the earlier version.  The authors express gratitude to NCM and CMI for organizing the workshop titled ``Cohen Macaulay simplicial complexes in graph theory", which provided an opportunity for the authors to convene and initiate the current work. The first author is financially supported by CSIR(India) at IIT Delhi and Inspire Faculty Fellow research grant of Subhajit Ghosh (IFA 23 MA 198). The second author is supported by the Prime Minister's Research Fellowship (PMRF/1401215), India, and also thanks SERB (Now ANRF) India for the CRG Grant (CRG/2023/000239). The third author is supported by HTRA-fellowship by IIT Madras, India. The fourth author is supported by the seed grant project IITM/SG/SMS/95 by IIT Mandi, India.
%-----------------------------------------------------------

\section{Declarations}
 \subsection{Ethical Approval} Not applicable
 \subsection{Funding} Not applicable
  \subsection{Availability of data and materials }  Not applicable

\end{document}